\documentclass[a4paper,reqno]{amsart}

\pdfoutput=1

\usepackage{nicefrac,amsmath,amssymb,amsthm,url,verbatim,enumitem,stmaryrd,mathtools,microtype,bbm}
\usepackage[utf8]{inputenc}
\usepackage[T1]{fontenc}
\usepackage{libertine}
\usepackage{dsfont}
\usepackage[libertine,cmintegrals,cmbraces]{newtxmath}
\usepackage[dvipsnames,svgnames,x11names]{xcolor}
\definecolor{darkred}{RGB}{139,0,0}
\definecolor{darkblue}{RGB}{0,0,139}
\definecolor{darkgreen}{RGB}{0,100,0}
\usepackage[mathscr]{euscript}
\usepackage[pagebackref]{hyperref}
\hypersetup{
	colorlinks   = true, 
	urlcolor     = darkred, 
	linkcolor    = darkred, 
	citecolor   = darkgreen}
\usepackage{tikz}
\usepackage{tikz-cd}
\usetikzlibrary{decorations.markings,patterns}
\usepackage[capitalise,noabbrev]{cleveref}

\calclayout

\AtBeginDocument{%
   \def\MR#1{}
}

\setenumerate[1]{label=(\roman*),}
\setitemize[1]{label=\raisebox{0.25ex}{\tiny$\bullet$}}

\setlist[enumerate]{leftmargin=*}
\setlist[itemize]{leftmargin=*}

\newcommand\smallsquare{{\mathbin{\text{\raise0.17ex\hbox{\scalebox{.7}{$\blacksquare$}}}}}}

\newcommand{\circled}[1]{\raisebox{.5pt}{\textcircled{\raisebox{-.9pt} {#1}}}}

\newcommand{\MT}{\ensuremath{\mathbf{MT}}}
\newcommand{\M}{\ensuremath{\mathbf{M}}}

\newcommand{\SO}{\ensuremath{\mathrm{SO}}}

\newcommand{\BO}{\ensuremath{\mathrm{BO}}}
\newcommand{\BSO}{\ensuremath{\mathrm{BSO}}}

\newcommand{\Diff}{{\ensuremath{\mathrm{Diff}}}}

\newcommand{\BDiff}{\ensuremath{\mathrm{BDiff}}}

\newcommand{\Homeo}{\ensuremath{\mathrm{Homeo}}}
\newcommand{\BlockHomeo}{\ensuremath{\widetilde{\mathrm{Homeo}}}}
\newcommand{\BHomeo}{\ensuremath{\mathrm{BHomeo}}}

\newcommand{\hAut}{\ensuremath{\mathrm{hAut}}}

\newcommand{\Map}{\ensuremath{\mathrm{Map}}}

\newcommand{\Emb}{\ensuremath{\mathrm{Emb}}}

\newcommand{\BEmb}{\ensuremath{\mathrm{BEmb}}}
\newcommand{\BAut}{\ensuremath{\mathrm{BAut}}}
\newcommand{\BSAut}{\ensuremath{\mathrm{BSAut}}}

\newcommand{\BTop}{\ensuremath{\mathrm{BTop}}}
\newcommand{\BSTop}{\ensuremath{\mathrm{BSTop}}}
\newcommand{\STop}{\ensuremath{\mathrm{STop}}}

\newcommand{\interior}{\ensuremath{\mathrm{int}}}

\newcommand{\inc}{\ensuremath{\mathrm{inc}}}
\newcommand{\ev}{\ensuremath{\mathrm{ev}}}
\newcommand{\res}{\ensuremath{\mathrm{res}}}

\newcommand{\Fun}{\ensuremath{\mathrm{Fun}}}

\newcommand{\sm}{\ensuremath{\mathrm{sm}}}

\newcommand{\PSh}{\mathrm{PSh}}

\newcommand{\seg}{\mathrm{seg}}
\newcommand{\GT}{\mathrm{GT}}
\newcommand{\tohofib}{\mathrm{tohofib}}
\newcommand{\init}{\mathds{I}}

\DeclareMathAlphabet{\mathpzc}{OT1}{pzc}{m}{it}
\newcommand{\catsingle}[1]{\ensuremath{\mathscr{#1}}}
\newcommand{\cat}[1]{\ensuremath{\mathsf{#1}}}
\newcommand{\icat}[1]{\ensuremath{\mathscr{#1}}}
\newcommand{\half}{{\nicefrac{1}{2}}}

\newcommand{\oH}{\ensuremath{\mathrm{H}}}

\newcommand{\oO}{\ensuremath{\mathrm{O}}}
\newcommand{\oB}{\ensuremath{\mathrm{B}}}

\newcommand{\bfF}{\ensuremath{\mathbf{F}}}

\newcommand{\bfN}{\ensuremath{\mathbf{N}}}

\newcommand{\bfR}{\ensuremath{\mathbf{R}}}
\newcommand{\bfZ}{\ensuremath{\mathbf{Z}}}
\newcommand{\bfQ}{\ensuremath{\mathbf{Q}}}

\newcommand{\bfk}{\ensuremath{\mathbf{k}}}

\newcommand{\cC}{\ensuremath{\catsingle{C}}}
\newcommand{\cD}{\ensuremath{\catsingle{D}}}
\newcommand{\cE}{\ensuremath{\catsingle{E}}}

\newcommand{\cL}{\ensuremath{\catsingle{L}}}
\newcommand{\cM}{\ensuremath{\catsingle{M}}}
\newcommand{\cN}{\ensuremath{\catsingle{N}}}
\newcommand{\cO}{\ensuremath{\catsingle{O}}}
\newcommand{\cP}{\ensuremath{\catsingle{P}}}
\newcommand{\cQ}{\ensuremath{\catsingle{Q}}}
\newcommand{\cR}{\ensuremath{\catsingle{R}}}
\newcommand{\cS}{\ensuremath{\catsingle{S}}}

\newcommand{\cX}{\ensuremath{\catsingle{X}}}
\newcommand{\cY}{\ensuremath{\catsingle{Y}}}

\newcommand{\cAb}{\ensuremath{\catsingle{A}\mathrm{b}}}

\newcommand{\ccMod}{\ensuremath{\catsingle{M}\mathrm{od}}}

\newcommand{\Cosp}{\ensuremath{\mathrm{Cosp}}}

\newcommand{\colim}{\ensuremath{\mathrm{colim}}}

\newcommand{\ra}{\rightarrow}
\newcommand{\lra}{\longrightarrow}

\newcommand{\xra}[1]{\xrightarrow{#1}}

\newcommand{\xlra}[1]{\overset{#1}{\longrightarrow}}

\newcommand{\longhookrightarrow}{\lhook\joinrel\longrightarrow}

\newcommand{\longdashrightarrow}[1][1.75em]{\mathrel{
		\tikz[baseline]{\draw[dash pattern=on .25em off .1 em,->](0,.58ex)--(#1,.58ex)}}}

\newcommand{\GL}{\mathrm{GL}}

\newcommand{\Aut}{\mathrm{Aut}}
\newcommand{\SAut}{\mathrm{SAut}}

\newcommand{\End}{\mathrm{End}}
\newcommand{\Hom}{\mathrm{Hom}}

\renewcommand{\Top}{\mathrm{Top}}

\newcommand{\im}{\mathrm{im}}

\newcommand{\id}{\mathrm{id}}

\newcommand{\pr}{\mathrm{pr}}

\newcommand{\fib}{\mathrm{fib}}
\newcommand{\cofib}{\mathrm{cofib}}
\newcommand{\st}{\mathrm{st}}

\newcommand{\op}{\mathrm{op}}

\newcommand{\Pro}{\mathrm{Pro}}
\newcommand{\aug}{\mathrm{aug}}
\newcommand{\un}{\mathrm{un}}
\newcommand{\red}{\mathrm{red}}
\newcommand{\gc}{\mathrm{gc}}
\newcommand{\col}{\mathrm{col}}

\newcommand{\BV}{\ensuremath{\mathrm{BV}}}
\newcommand{\pre}{\ensuremath{\mathrm{pre}}}

\newcommand{\Mal}{\ensuremath{\mathrm{Mal}}}
\newcommand{\scp}{\ensuremath{\mathrm{scp}}}
\newcommand{\bscp}{\ensuremath{\mathrm{bscp}}}

\newcommand{\fQ}{{\mathrm{f}\bfQ}}

\newcommand{\Latch}{\mathrm{Latch}}
\newcommand{\Match}{\mathrm{Match}}

\newcommand{\BG}{\mathrm{BG}}

\newcommand{\TOP}{{\mathrm{Top}}}

\newcommand{\sSet}{\ensuremath{\cat{sSet}}}

\newcommand{\DiscInf}{\ensuremath{\icat{D}\mathrm{isc}}}

\newcommand{\Fin}{\ensuremath{\mathrm{Fin}}}
\newcommand{\ALG}{\ensuremath{\mathrm{ALG}}}
\newcommand{\Alg}{\ensuremath{\mathrm{Alg}}}
\newcommand{\AlgOpd}{\ensuremath{\underline{\mathrm{A}}\mathrm{lg}}}
\newcommand{\MapInt}{\ensuremath{\underline{\mathrm{M}}\mathrm{ap}}}

\newcommand{\ncBordInf}{\ensuremath{\mathrm{nc}\icat{B}\mathrm{ord}}}

\newcommand{\Cat}{\ensuremath{\icat{C}\mathrm{at}}}

\newcommand{\BMod}{\ensuremath{\mathrm{BMod}}}

\newcommand{\Ass}{\ensuremath{\mathrm{Ass}}}
\newcommand{\CatInf}{\ensuremath{\icat{C}\mathrm{at}}}

\newcommand{\ModInf}{\ensuremath{\mathpzc{r}\icat{M}\mathrm{od}}}

\newcommand{\Opd}{\ensuremath{{\icat{O}\mathrm{p}}}}
\newcommand{\OpdSet}{\ensuremath{{\cat{O}\mathrm{p}}}}

\newcommand{\Triv}{\ensuremath{{\icat{T}\mathrm{riv}}}}

\newcommand{\Mul}{\ensuremath{\mathrm{Mul}}}

\newcommand{\Tow}{\ensuremath{\mathrm{Tow}}}

\newcommand{\ini}{\ensuremath{\mathrm{ini}}}

\newtheorem{bigthm}{Theorem}
\newtheorem{bigcor}[bigthm]{Corollary}

\newtheorem{thm}{Theorem}[section]

\newtheorem{lem}[thm]{Lemma}
\newtheorem{prop}[thm]{Proposition}
\newtheorem{cor}[thm]{Corollary}

\theoremstyle{definition}
\newtheorem{dfn}[thm]{Definition}

\newtheorem*{nconvention}{Convention}
\theoremstyle{remark}
\newtheorem{ex}[thm]{Example}

\newtheorem{rem}[thm]{Remark}
\newtheorem*{nrem}{Remark}

\newcommand{\ul}[1]{\underline{#1}}

\begin{document}

\title[Pontryagin--Weiss classes and a rational decomposition of spaces of homeomorphisms]{Pontryagin--Weiss classes and a rational \\decomposition of spaces of homeomorphisms}

\author{Manuel Krannich}
\address{Department of Mathematics, Karlsruhe Institute of Technology, 76131 Karlsruhe, Germany}
\email{krannich@kit.edu}

\author{Alexander Kupers}
\address{Department of Computer and Mathematical Sciences, University of Toronto Scarborough, 1265 Military Trail, Toronto, ON M1C 1A4, Canada}
\email{a.kupers@utoronto.ca}

\begin{abstract}We construct a rational homotopy pullback decomposition for variants of the classifying space of the group of homeomorphisms for a large class of manifolds. This has various applications, including a rational section of the stabilisation map $\Top(d)\ra \TOP$ of the space of homeomorphisms of $\bfR^d$ for $d\ge6$, or a new method to construct topological bundles and detect characteristic classes thereof. Some steps in the proofs may be of independent interest, such as the construction of a nullhomotopy of the twice-iterated stabilisation map for the space of orientation-preserving derived automorphisms of the rational $E_d$-operad, results on recovering boundaries of manifolds from their interior in the context of embedding calculus, or a treatment of tensor products of truncated $\infty$-operads.
\end{abstract}

\maketitle

\vspace{-0.25cm}

\tableofcontents

\vspace{-0.6cm}
\section{Introduction}
From a homotopy-theoretic perspective, the study of fibre bundles with fibre a compact topological $d$-manifold $M$ and trivialised boundary bundle amounts to the study of the homotopy type of the classifying space $\BHomeo_\partial(M)$ of the topological group of homeomorphisms fixing the boundary. In this work we construct a rational homotopy pullback decomposition of $\BHomeo_\partial(M)$  with certain tangential structures for a large class of manifolds $M$. Instead of a stand-alone result, we consider this decomposition rather as a tool in the study of $\BHomeo_\partial(M)$ and thus believe that it is best appreciated through its applications. We begin by explaining three of them, and only then discuss the actual decomposition theorem.

\subsection*{Application I: Pontryagin--Weiss classes} The first application concerns \emph{topological Pontryagin classes}. Recall that the ring of rational characteristic classes for stable vector bundles---the rational cohomology ring of the corresponding classifying space $\BO$---is a polynomial ring $\oH^*(\BO;\bfQ)\cong\bfQ[p_1,p_2,\ldots]$ in the rational Pontryagin classes $p_k$. By construction, the $k$th Pontryagin class $p_k$ is trivial on vector bundles of fixed rank $d$ as long as $\smash{k>\frac{d}{2}}$, i.e.\,$\smash{p_k\in \oH^{4k}(\BO;\bfQ)}$ vanishes in this range when pulled back along the stabilisation map $\BO(d)\ra \BO$ that classifies the underlying stable vector bundle of a vector bundle of rank $d$. It turns out that rational Pontryagin classes can be defined more generally for $d$-dimensional \emph{Euclidean bundles} i.e.\,fibre bundles whose fibres are homeomorphic to $\bfR^d$, without linear structure \cite{Schafer,SiebenmannICM,Kahn}. Indeed, the map $\BO\ra \BTop$ that classifies the underlying stable Euclidean bundle of a stable vector bundle induces an isomorphism on rational cohomology, so Pontryagin classes extend uniquely to characteristic classes $p_k\in  \oH^{4k}(\BTop;\bfQ)$ for stable Euclidean bundles. However, with this definition it is not clear whether they have the same vanishing behaviour when evaluated on Euclidean bundles of fixed rank $d$ as in the case of vector bundles, that is, whether the pullback $p_k\in\oH^{4k}(\BTop(d);\bfQ)$ along the stabilisation map $\BTop(d)\ra \BTop$ vanishes for $\smash{k> \frac{d}{2}}$. In fact, Weiss showed that this is \emph{not} the case for many choices of pairs $(d,k)$ in this range \cite[Theorem 1.1.1]{WeissDalian}. He moreover proved that for these $(d,k)$, the class $p_k$ is nontrivial on a $d$-dimensional Euclidean bundle \emph{over a sphere}, i.e.\ that $p_k$ evaluates nontrivially even on the image of the Hurewicz homomorphism $\pi_{4k}(\BTop(d))\ra \oH_{4k}(\BTop(d))$. More recently, Galatius and Randal-Williams \cite[Theorem 1.1]{GRWIndependence} showed via a different method that the stabilisation map $\BTop(d)\ra \BTop$ is injective on rational cohomology in \emph{all} degrees for $d\ge6$, so in particular $p_k\neq 0\in\oH^{4k}(\BTop(d);\bfQ)$ for all $k\ge1$. Their argument however does not prove that these classes are nontrivial on a bundle over a sphere (see Remark 1.3 loc.cit.). Extending Weiss' original approach using the aforementioned pullback decomposition, we show that they are:

\begin{bigthm}\label{bigthm:geometric}For all $d\ge 6$ and $k\ge1$, there exists a fibre bundle over the $4k$-sphere with fibres homeomorphic to $\bfR^d$ whose $k$th rational Pontryagin class is nontrivial.
\end{bigthm}

 \cref{bigthm:geometric} has the following homotopy-theoretic reformulation:

\begin{bigcor}\label{bigcor:homotopy}For $d\ge6$ the stabilisation map $\BTop(d)\ra \BTop$ admits a rational section after taking loop spaces. Equivalently, it is surjective on rational homotopy groups. 
\end{bigcor}

Via Morlet's equivalence $\smash{\BDiff_\partial(D^d)\simeq\Omega^d_0\Top(d)/\oO(d)}$ in dimensions $d\neq 4$ \cite[V.3.4]{KirbySiebenmann}, \cref{bigcor:homotopy} has the following consequence for the classifying space $\BDiff_\partial(D^d)$ of the group of diffeomorphisms of a closed $d$-dimensional disc fixing the boundary pointwise, equipped with the smooth topology (see e.g.\,\cite[Section 8.2.1]{KRWOddDiscs} for a definition of the maps in the statement):

\begin{bigcor}\label{bigcor:PW-classes}For $n\ge3$, the following maps are surjective on rational homotopy groups:
\[
\def\arraystretch{1.5}
\begin{array}{c@{\hskip 0cm} l@{\hskip 0.1cm} c@{\hskip 0.1cm} l@{\hskip 0.1cm} l@{\hskip 0.1cm} }
\big(\Omega^{2n}(p_n-e^2)^\tau\times \sqcap_{k>n}\Omega^{2n}p_k^\tau\big)&\colon&\BDiff_\partial(D^{2n})&\lra&\bigsqcap_{k=n}^\infty K(\bfQ,4k-2n-1)\\
\big(\Omega^{2n+1}(p_n-E)^\tau\times \sqcap_{k>n}\Omega^{2n+1}p_k^\tau\big)&\colon&\BDiff_\partial(D^{2n+1})&\lra&\bigsqcap_{k=n}^\infty K(\bfQ,4k-2n-2).\end{array}
\]
\end{bigcor}

\begin{nrem}\ 
\begin{enumerate}
\item Most of the homotopy classes detected by \cref{bigcor:PW-classes} lie outside of the range where the groups $\pi_*(\BDiff_\partial(D^d))\otimes\bfQ$ have been computed in \cite{KRWEvenDiscs,KRWOddDiscs} and many lie in different degrees than Watanabe's ``graph classes'' detected by configuration space integrals \cite{WatanabeOdd, WatanabeErratum,WatanabeEven}. 
\item From the case of discs, one can obtain nontriviality of many rational homotopy groups for $\BDiff_\partial(M)$ for many other smooth compact $d$-manifolds $M$, e.g.\,whenever $M$ is closed, parallelisable, and admits a non-negatively curved metric \cite[Lemma 2.1]{BustamanteFarrellJiang}.
\item We do not show that there is a section as in \cref{bigcor:homotopy} before taking loop spaces, so our result is neither stronger nor weaker than the main result of \cite{GRWIndependence}: it does not show that the stabilisation $\BTop(d)\ra \BTop$ is surjective on rational homology groups, but it does show that it is surjective on rational homotopy groups. For applications such as \cref{bigcor:PW-classes}, the surjectivity on homotopy as opposed to homology groups is essential.
\end{enumerate}
\end{nrem}

\subsection*{Application II: A rationally unconstrained Burghelea--Lashof splitting}
The second application of our pullback decomposition concerns the forgetful map 
\begin{equation}\label{equ:nonblock-to-block}
	\Aut_\partial(M)/\Homeo_\partial(M)\lra\Aut_\partial(M)/\BlockHomeo_\partial(M)
\end{equation}
from the classifying space for fibre homotopically trivialised fibre bundles with a compact $d$-manifold $M$ as fibre and trivialised boundary bundle, to the corresponding classifying space for fibre homotopically trivialised block bundles. Burghelea and Lashof \cite{BurgheleaLashofSplitting} proved that after Postnikov truncation in a range growing with $d$, taking loop spaces, and localising away from $2$, the map \eqref{equ:nonblock-to-block} admits a section. Since the target of the map admits by surgery theory a complete description in terms of algebraic $L$-theory, this gives a convenient method to construct fibre bundles with specified properties; see e.g.\,\cite{KKRW,Frenck} for recent applications of this strategy in the smooth setting. Using the aforementioned pullback decomposition, we show that, when looped once, the forgetful map  \eqref{equ:nonblock-to-block} often admits a section \emph{before} Postnikov truncation as long as one is willing to rationalise:

\begin{bigthm}\label{bigthm:splitting-introduction} For a $2$-connected smoothable compact manifold $M$ of dimension $d \geq 13$ with $2$-connected boundary, the forgetful map \eqref{equ:nonblock-to-block} admits a rational section after taking loop spaces. \end{bigthm}

\begin{nrem}\ 
\begin{enumerate}[leftmargin=*]
\item The hypothesis on the dimension in \cref{bigthm:splitting-introduction} can often be improved, and under additional assumptions on $M$ the section also exists before taking loop spaces (see \cref{thm:splitting}).
\item In contrast to Burghelea--Lashof's section, the section in \cref{bigthm:splitting-introduction} without Postnikov truncation is a purely topological phenomenon and does not exist in the smooth category in general (this can for instance be deduced from \cite[Proposition 3.1]{RWupperbound}).
\end{enumerate}
\end{nrem}

\subsection*{Application III: Detecting characteristic classes}Our pullback decomposition can also serve as a tool to detect characteristic classes of topological manifold bundles, in particular \emph{tautological classes}, that is, fibre integrals of characteristic classes of the vertical topological tangent bundle. To illustrate the general method, we prove the following:

\begin{bigthm}\label{bigthm:detect-classes}Let $M$ be a $2$-connected smoothable compact manifold of dimension $d\ge8$ with $2$-connected boundary, such that $M$ contains $S^m\times S^{d-m}$ as a connected summand for some $m$.
 \begin{enumerate}[leftmargin=*]
\item\label{enum:smsn-infinite-i} The tautological class $\kappa_{p_i p_j}\in\oH^{4(i+j)-d}(\BHomeo_\partial(M);\bfQ)$ associated to a product of Pontryagin classes is nontrivial for all $i,j > \max(m,d-m)/4$.
\item\label{enum:smsn-infinite-ii} If $m$ and $d-m$ are even, then the total dimension of the cohomology ring $\oH^{*}(\BHomeo_\partial(M);\bfQ)$ in degree $*\le k$ grows with $k$ faster than any polynomial.
\end{enumerate}
\end{bigthm}

\subsection*{The pullback decomposition}\label{sec:pullback-intro}
We now explain the pullback decomposition on which the above applications rely. For a topological $d$-manifold $M$, taking topological derivatives yields a factorisation
\begin{equation}\label{equ:nondelooped-comp-before-theta}
	\Homeo_\partial(M)\lra \Aut_\partial(TM)\lra\Aut_\partial(M)
\end{equation}
of the forgetful map to the topological monoid $\Aut_\partial(M)$ of homotopy equivalences of $M$ fixed on $\partial M$, through the topological monoid $\Aut_\partial(TM)$ of homotopy equivalences of $M$ fixed on $\partial M$ covered by a bundle map of the topological tangent bundle $TM$ of $M$ fixed on $TM|_{\partial M}$. We show that the delooping of \eqref{equ:nondelooped-comp-before-theta} can ``often'' be ``partially'' completed to a homotopy pullback square
\[
\begin{tikzcd}[row sep=0.4cm,column sep=0.4cm]
\BHomeo_\partial(M)\rar\dar& Z_M\dar\\
\BAut_\partial(TM)\rar &\BAut_\partial(M)
\end{tikzcd}
\]
for some space $Z_M$. ``Often'' in that we assume $d\ge5$ and that $M$ is smoothable, $2$-connected, and has $2$-connected boundary, and ``partially''  in that we replace the sequence \eqref{equ:nondelooped-comp-before-theta} by a variant explained below, involving a \emph{weaker  boundary condition}, \emph{tangential structures}, and certain \emph{rationalisations}.

Regarding the \emph{weaker boundary condition}, we assume that $M$ has nonempty boundary, fix an embedded disc $D^{d-1}\subset \partial M$, and set $\half\partial M\coloneq \partial M\backslash \interior(D^{d-1})$. We then replace $\Aut_\partial(TM)$ in \eqref{equ:nondelooped-comp-before-theta} by the topological monoid $\smash{\Aut_\partial(TM,\half \partial)}$ defined in the same way as $\smash{\Aut_\partial(TM)}$, except that the bundle map is only required to fix $TM|_{\half\partial M}$ instead of $TM|_{\partial M}$. Regarding \emph{tangential structures}, we fix a map $\theta\colon B\ra \BTop(d)$ and a $\theta$-structure $\ell_{\half\partial}$ on $TM|_{\half\partial M}$ (a lift of its classifying map along $\theta$). Then the space of $\theta$-structures on $TM$ extending $\ell_{\half\partial}$ is acted upon by $\Homeo_\partial(M)$ and $\smash{\Aut_\partial(TM,\half \partial)}$, so by taking homotopy orbits we obtain a variant
$\smash{\BHomeo^\theta_\partial(M;\ell_{\half\partial})\ra \BAut^\theta_\partial(TM;\ell_{\half\partial})\ra\BAut_\partial(M)}$ of the sequence \eqref{equ:nondelooped-comp-before-theta}. Regarding \emph{rationalisations}, we replace the final map in the latter sequence by the first map in its Moore--Postnikov $1$-factorisation $\smash{\BAut^\theta_\partial(TM,\half \partial;\ell_{\half\partial})\ra \BAut_{\partial}(M)^{\ell_{\half \partial}} \ra\BAut_\partial(M)}$. We then replace $\smash{\BHomeo^\theta_\partial(M;\ell_{\half\partial})}$ as well as $\smash{ \BAut^\theta_\partial(TM;\ell_{\half\partial})}$ by their fibrewise rationalisation over $\smash{\BAut_{\partial}(M)^{\ell_{\half \partial}}}$, indicated by a $(-)_{\fQ}$-subscript. In total, we arrive at a sequence of the form
\begin{equation}\label{equ:introduction-compl-seq}
	\BHomeo^\theta_\partial(M;\ell_{\half\partial})_\fQ\lra \BAut^\theta_\partial(TM;\ell_{\half\partial})_\fQ\lra\BAut_{\partial}(M)^{\ell_{\half \partial}}
\end{equation}
which is the variant of \eqref{equ:nondelooped-comp-before-theta} that our homotopy pullback decomposition result is about:

\begin{bigthm}\label{bigthm:pullback}Let $M$ be a smoothable compact $d$-manifold $M$ with $d\ge5$ and $\partial M \neq\varnothing$, $\theta\colon B\ra \BSTop(d)$ an oriented tangential structure such that $B$ is nilpotent, and $\ell_{\half\partial}$ a $\theta$-structure on $TM|_{\half\partial}$. If
\begin{enumerate}
	\item  \label{enum:pullback-ii} $M$ and $\partial M$ are both $2$-connected,
	\item \label{enum:pullback-i} $\theta$ factors through the stabilisation map $\BSTop(d-2)\ra \BSTop(d)$ after rationalisation,
\end{enumerate}
then the sequence \eqref{equ:introduction-compl-seq} can be completed to a homotopy pullback square.
\end{bigthm}

Informally speaking and neglecting tangential structures, rationalisation, and boundary conditions, \cref{bigthm:pullback} says that specifying a topological fibre bundle with fibre $M$ is equivalent to specifying a fibration with fibre $M$ together with two pieces of data on it \emph{that are independent from each other}: a ``vertical tangent bundle'' i.e.~a $d$-dimensional Euclidean bundle on the total space that is equivalent to $TM$ on each fibre, and whichever extra data on a fibration the additional space in the pullback decomposition of \cref{bigthm:pullback} classifies. Since the two pieces of data can be chosen independently, this allows one for instance to manipulate the vertical tangent bundle of a fibre bundle without changing the underlying fibration.

\begin{nrem}\ 
\begin{enumerate}[leftmargin=*]
\item We prove a stronger version of \cref{bigthm:pullback} that allows more general boundary conditions and also applies to certain spaces of self-embeddings (see \cref{thm:pullback-decomp}).
\item The proof of \cref{bigthm:pullback} in principle also describes the additional space in the pullback extension of \eqref{equ:introduction-compl-seq}, but we prefer to state the result as above to emphasise that the mere existence of such a pullback decomposition suffices for many applications, in particular for Theorems \ref{bigthm:geometric}, \ref{bigthm:splitting-introduction}, and \ref{bigthm:detect-classes}. Roughly speaking, the additional space is constructed out of the homotopy types of $\Aut_\partial(M)$ and the limit of fibrewise rationalisations of the truncated particle embedding calculus approximations to a certain self-embedding space of $M$, in the sense of \cite{KKoperadic} and closely related to Boavida de Brito--Weiss' theory of configuration categories \cite{BoavidaWeissConfiguration}.
\end{enumerate}
\end{nrem}

\subsection*{Twice-iterated stabilisation of the rationalised $E_d$-operad}The proof of \cref{bigthm:pullback} is inspired by Weiss' work on topological Pontryagin classes \cite{WeissDalian} and relies on our earlier work \cite{KKoperadic}. One of the crucial additional ingredients is a result on the classifying space $\BAut(E_{d,\bfQ})$ of the derived automorphism group of the rationalisation of the operad $E_d$ of little $d$-discs. The space of $2$-ary operations of $E_{d,\bfQ}$ is homotopy equivalent to a rational $(d-1)$-sphere, so acting on its top homology yields a morphism $\Aut(E_{d,\bfQ})\ra \GL_1(\bfQ)$ whose kernel we denote by $\SAut(E_{d,\bfQ})$. A rational version of the additivity theorem of Dunn and Lurie yields a stabilisation map $\BSAut(E_{d-1,\bfQ})\ra \BSAut(E_{d,\bfQ})$ which we show to be nullhomotopic when iterated twice: 
 
\begin{bigthm}\label{bigthm:nullhomotopy}The stabilisation map $\BSAut(E_{d-2,\bfQ})\to \BSAut(E_{d,\bfQ})$ is nullhomotopic for all $d \geq 2$.
\end{bigthm}

\begin{nrem}\ 
\begin{enumerate}[leftmargin=*]
\item \cref{bigthm:nullhomotopy} is optimal in the sense that passing to the subgroup $\SAut(E_{d-2,\bfQ})\le \Aut(E_{d-2,\bfQ})$, rationalising, and stabilising twice instead of once are all necessary in general (see \cref{rem:nullhomotopy-optimal}).
\item \cref{bigthm:nullhomotopy} is related to work of Fresse--Willwacher \cite{FresseWillwacher} and Khoroshkin--Willwacher \cite{KhoroshkinWillwacher}, but our proof is independent of their results and methods (see \cref{rem:graph-complex-relation}).
\item The strategy of proof of \cref{bigthm:nullhomotopy} is inspired by work of Boavida de Brito and Horel \cite{BoavidaHorel} on a proof of formality for the $E_d$-operads in positive characteristic.
\item As an input to the proof of \cref{bigthm:nullhomotopy}, we establish various results on tensor products of truncated $\infty$-operads that may be of independent interest (see \cref{sec:truncation-of-operads}).
\end{enumerate}
\end{nrem}

\subsection*{Acknowledgments}
We thank Gijs Heuts and Ieke Moerdijk for answers to some questions on $\infty$-operads, Pedro Boavida de Brito for general comments, Geoffroy Horel for exchanges regarding \cref{sec:pro-operads}, John Klein for a conversation on $T_\infty$-boundaries, Oscar Randal-Williams for general comments and discussions on tautological classes and $T_\infty$-boundaries, Thomas Willwacher for a conversation on graph complexes, Søren Galatius for spotting an oversight in an earlier version, and the referees for many helpful comments. Above all, we owe a significant intellectual debt to Michael Weiss whose work on topological Pontryagin classes inspired several key ideas in this paper. MK was supported by the European Union through an ERC grant (MaFC, 101221003). AK acknowledges the support of the Natural Sciences and Engineering Research Council of Canada (NSERC) [funding reference number 512156 and 512250]. AK was supported by an Alfred P.~Sloan Research Fellowship.

\section{Preliminaries}We begin with some preliminaries on  rationalisation (\cref{sec:rationalisation}), operads (\cref{sec:operad-conventions}), and the Bousfield--Kan spectral sequence (\cref{sec:bousfield-kan-ss}).

\begin{nconvention}
We work in the setting of $\infty$-categories as in \cite{LurieHTT,LurieHA} throughout. In particular, a \emph{category} is an $\infty$-category, an \emph{operad} is an $\infty$-operad, a \emph{space} is an object in the $\infty$-category of spaces $\cS$, and a \emph{pointed space} is an object in the undercategory $\cS_{\ast/}$, etc.. Operations that we perform with (pointed) spaces---such as various (co)limits---are formed within these $\infty$-categories. An \emph{$E_1$-space} is a monoid object in $\cS$ and an \emph{$E_1$-group} is a group object in $\cS$. An action of an $E_1$-group $G$ on a space $X$ is a functor $BG\ra \cS$ whose values are equivalent to $X$, and the (co)invariants $X_G$ and $X^G$ are the (co)limit of this functor.
\end{nconvention}

\subsection{Rationalisation}\label{sec:rationalisation}
Throughout this work, by \emph{rationalisation of spaces}, we mean Bousfield--Kan $\bfQ$-completion, viewed as an endofunctor $(-)_\bfQ\colon \cS\ra \cS$ of the category of spaces, together with a natural transformation $\id_\cS\ra (-)_\bfQ$. It arises by considering the adjunction between $\cS$ and the category $\smash{\ccMod^{\aug}_\bfQ}$ of augmented $\mathrm{H}\bfQ$-module spectra whose left adjoint is given by mapping a space $X$ to the $\mathrm{H}\bfQ$-module $\Sigma_+^\infty X\otimes \mathrm{H}\bfQ$ equipped with the augmentation induced by $X\ra \ast$. The right adjoint maps an augmented $\mathrm{H}\bfQ$-module $E\ra \mathrm{H}\bfQ$ to the fibre $\fib_1(\Omega^\infty E\ra \Omega^\infty \mathrm{H}\bfQ\simeq \bfQ)$. The natural transformation $\smash{\id_\cS\ra (-)_\bfQ}$ is then obtained as the limit of the coaugmented canonical resolution of the monad associated to this adjunction. A classical reference for this is \cite[I.4.2]{BousfieldKan}, where the functor $(-)_\bfQ$ is denoted $\bfQ_\infty(-)$. We call this functor simply \emph{rationalisation}.

\subsubsection{Fibrewise rationalisation}\label{sec:fibrewise-rat}
For a fixed space $B$, by \emph{fibrewise rationalisation over $B$}, we mean the endofunctor on the overcategory $\cS_{/B}$ given as the composition $\cS_{/B}\simeq \Fun(B,\cS)\ra\Fun(B,\cS)\simeq \cS_{/B}$ induced by postcomposition with $(-)_\bfQ$. We denote this functor by $(-)_{\fQ}$; the space $B$ will always be clear from the context. The natural transformation $\id_\cS\ra (-)_\bfQ$ of endofunctors on $\cS$ induces an analogous natural transformation  $\id_{\cS/B}\ra (-)_\fQ$ of endofunctors on $\cS_{/B}$.

\subsubsection{Nilpotent spaces} \label{sec:nilpotent-completion}A space $X$ is \emph{nilpotent} if for all $x\in X$ the action of $\pi_1(X,x)$ on $\pi_k(X,x)$ is nilpotent for all $k\ge1$ in the sense of \cite[II.4.1]{BousfieldKan}. For $k=1$ this is equivalent to the group $\pi_1(X,x)$ being nilpotent for all $x\in X$. On the full subcategory of nilpotent spaces, the functor $(-)_\bfQ$ agrees with usual rationalisation of nilpotent spaces, so the map $X\ra X_\bfQ$ induces rationalisations in the algebraic sense on all homotopy and homology groups of each path component in positive degrees, see \cite[V.3.1]{BousfieldKan} (see V.2 loc.cit. for what it means to rationalise a nilpotent group).

\subsubsection{Non-nilpotent spaces} \label{sec:non-nilpotent-completion} 
If $X$ is not nilpotent, the map $X\ra X_\bfQ$ need not induce isomorphisms on rational homotopy or homology groups, but the rational (co)homology of $X_\bfQ$ still contains the rational (co)homology of $X$ as a natural summand. This because the coaugmented cosimplicial object in $\cS$ arising from the monad associated to the adjunction between $\cS$ and $\smash{\ccMod^{\aug}_\bfQ}$ admits an extra codegeneracy after applying the left adjoint $\Sigma_+^\infty (-)\otimes \mathrm{H}\bfQ$; see \cite[I.5.4]{BousfieldKan}. Working over a base space $B$, the rational (co)homology of the total space of the fibrewise rationalisation $X_{\fQ}\ra B$ of a map $X\ra B$ similarly contains the rational (co)homology of $X$ as a natural summand. This is because taking vertical fibrewise rationalisation of the square
\[
\begin{tikzcd}[column sep=0.4cm,row sep=0.4cm]
X\dar\arrow[r,equal]&X\dar\\
B\rar &*
\end{tikzcd}
\]
gives a map $\smash{X_\fQ\ra X_\bfQ}$ from the fibrewise to the nonfibrewise rationalisation of $X$ whose precomposition with the fibrewise rationalisation $X\ra X_\fQ$ gives the nonfibrewise rationalisation $\smash{X\ra X_\bfQ}$. Since the latter has a cosection after applying  $\Sigma_+^\infty (-)\otimes \mathrm{H}\bfQ$, so does the former.

\subsubsection{Inheritance properties}\label{sec:rat-inheritance}
The functor $(-)_\bfQ\colon \cS\ra \cS$ preserves coproducts, finite products, and connectivity of maps by \cite[I.7.1, I.7.2., IV.5.1]{BousfieldKan}. In addition, as the following lemma shows, it preserves pullbacks along $1$-connected maps $f\colon Z\ra W$ that are \emph{nilpotent} in that for all $x\in X$, the action of $\pi_1(X,x)$ on $\pi_k(\fib_{f(x)}(f),x)$ for every $k\ge 1$ is nilpotent in the sense of II.4.1 loc.cit..

\begin{lem}\label{lem:Q-completion-pullback}
For a pullback square of spaces
\[
\begin{tikzcd}[column sep=0.4cm,row sep=0.4cm]
X\dar{}\rar{}&Y\dar{}\\
Z\rar{f}&W.
\end{tikzcd}
\]
where the map $f$ is $1$-connected and nilpotent, the rationalisation of the square remains a pullback. The nilpotency assumption on $f$ is automatic if $Z$ and $W$ are both nilpotent.
\end{lem}

\begin{proof}
In the case where all four spaces are connected, the claim follows from \cite[II.5.3, II.5.4]{BousfieldKan}. To deduce the general case from this, note that since rationalisation preserves connectivity, the rationalised square has still $1$-connected rows. A square with $1$-connected rows is a pullback square if and only if the squares of connected spaces obtained by restriction to the components induced by points in the top-left corner are again pullback squares, so we may indeed assume that all spaces are connected. The addendum of the claim follows from II.4.5 loc.cit..
\end{proof}

\subsection{Operads}\label{sec:operad-conventions} \label{sec:operad-truncation}
We adopt the conventions and notation for operads from \cite[Section 1.4]{KKoperadic}. As in that reference, we consider the category $\Opd^\un$ of unital operads (i.e.\,operads $\cO$ with $\Mul_\cO(\varnothing;c)\simeq\ast$ for all colours $c$) and its truncated variants $\Opd^{\le k,\un}$ for $1\le k<\infty$. These categories fit into a tower 
\begin{equation}\label{equ:tower-of-operads}
	\Opd^\un=\Opd^{\le \infty,\un}\ra\cdots\ra\Opd^{\le 2,\un}\ra\Opd^{\le 1,\un}\simeq\Cat
\end{equation} 
of categories  with $\Opd^\un\simeq \lim_k\Opd^{\le k,\un}$; see Section 1.4.4 loc.cit.. The composition $\Opd^\un\ra \Cat$ sends an operad to its \emph{category of colours} $\cO^\col$. All functors in the tower $\tau\colon \Opd^{\le j,\un}\ra \Opd^{\le k,\un}$ for $1\le k\le j\le \infty$ admit fully faithful right and left adjoints $\smash{\tau_*\colon \Opd^{\le k,\un}\ra \Opd^{\le j,\un}}$ and $\smash{\tau_!\colon \Opd^{\le k,\un}\ra \Opd^{\le j,\un}}$ (see  \cite[Thm 1.2]{DubeyLiu} or \cite[Thm A.4]{{KKoperadic}}). We are mostly interested in the full subcategories $\smash{\Opd^{\le k,\red}\subset \Opd^{\le k,\un}}$ of \emph{reduced operads}, which are those (truncated) unital operads $\cO$ such that $\cO^\col$ is trivial, i.e.\ $\Mul_{\cO}(c;d)\simeq\ast$ for all colours $c,d$. Alternatively, the category $\smash{\Opd^{\le k,\red}}$ is the fibre $\smash{\fib_*(\Opd^{\le k,\un}\ra\Opd^{\le 1,\un}\simeq\Cat)}$. In particular, we have $\Opd^{\le 1,\red}\simeq\ast$. The category $\smash{\Opd^{\le 2,\red}}$ admits a simple description as well: taking spaces of $2$-ary operations yields an equivalence $\Opd^{\le 2,\red}\simeq \cS^{\Sigma_2}=\Fun(\mathrm{B}\Sigma_2,\cS)$ where $\Sigma_2$ is the symmetric group in two letters (see \cref{lem:2-truncated-reduced} below).

\subsubsection{Tensor products of (truncated) operads}\label{sec:operads-tensor-product} The category of all operads $\Opd$ admits a symmetric monoidal structure $\otimes$ which can be thought of as the $\infty$-categorical version of the classical Boardman--Vogt product for ordinary operads in sets \cite[2.2.5.13]{LurieHA}. By 2.1.4.6, 2.2.5.7 loc.cit., it turns $\Opd$ into a presentable symmetric monoidal category in the sense of 3.4.4.1 loc.cit.. In particular the tensor products of operads preserves colimits in both variables. The unit for the monoidal structure is the \emph{trivial operad} $\Triv$ (the initial object in $\Opd$) whose category of colours is the trivial category (which has a single object $*$ with contractible endomorphism space) and all whose spaces of operations are empty except $\Mul_{\Triv}(*;*)\simeq \ast$. In particular $\Triv$ is not unital. There is a closely related operad $E_0$ which only differs from $\Triv$ in that there is a unary operation $\Mul_{E_0}(\varnothing;*)\simeq *$. The functor $(-)\otimes E_0\colon \Opd \ra\Opd$ takes values in the full subcategory $\Opd^\un\subset \Opd$ and the resulting corestriction $(-)\otimes E_0\colon \Opd \ra\Opd^\un$ is left adjoint to the inclusion $\Opd^\un\subset \Opd$ (see p.\,246 loc.cit.), so in particular $\Opd^\un$ is a localisation of $\Opd$ (it is also a colocalisation; see p.\,246 loc.cit.). Since the localisation is given by taking tensor product with an object, it is compatible with the symmetric monoidal structure of $\Opd$ in the sense of 2.2.1.7 loc.cit.\,, so an application of 2.2.1.9 loc.cit.\,yields a symmetric monoidal structure on $\Opd^\un$ such that $(-)\otimes E_0\colon \Opd\ra \Opd^\un$ can be lifted to a symmetric monoidal functor. As $E_0\otimes E_0\simeq E_0$ by 2.3.1.6 loc.cit., it follows that the tensor product of two objects in $\Opd^\un$ is simply given by the tensor product in $\Opd$, and that the monoidal unit in $\Opd^\un$ is $E_0$ instead of $\Triv$ as in $\Opd$. In \cref{sec:truncation-of-operads}, we will show that there is an extension of this symmetric monoidal structure to a symmetric monoidal structure on the category of $k$-truncated operads $\Opd^{\le k,\un}$ for all $k$ which (a) agrees for $k=1$ with the cartesian structure on $\Opd^{\le 1,\un}\simeq \Cat$, (b) is closely related to the join of spaces for $k=2$, and (c) for which \eqref{equ:tower-of-operads} lifts to a tower of symmetric monoidal categories.

\subsubsection{Dunn--Lurie additivity}\label{sec:dl-add}Recall the operad $E_d$ of \emph{little $d$-discs} for $d\ge0$ which has a single colour $*$ and whose space of $k$-ary operations $E_d(k)\coloneq \Mul_{E_d}(\{1,\ldots ,k\}\times *;*)$ is homotopy equivalent to the space of ordered configurations of $k$ points in $\bfR^d$ \cite[5.1.0.3, 5.1.0.4, 5.1.1.3]{LurieHA}. In particular, $E_d(k)$ is contractible for $k \le 1$, so $E_d$ is a reduced operad. These operads are additive under taking tensor products, meaning that there are preferred equivalences $E_n\otimes E_m\simeq E_{n+m}$ \cite[5.1.1.2]{LurieHA} for $n,m\ge0$. Moreover, there is a preferred map $t\colon \BTop(d)\ra \BAut(E_d)$ where $\BTop(d)$ is the classifying space of the topological group of homeomorphisms of $\bfR^d$ \cite[Section 5.1.1]{KKoperadic}. This map turns out to be compatible with additivity: 

\begin{lem}There exists a commutative square of spaces
\begin{equation}\label{equ:additivity-direct-product}
\begin{tikzcd}[column sep=0.8cm, row sep=0.5cm]
\BTop(n)\times \BTop(m)\dar{\times}\rar{(t,t)}&\BAut(E_n)\times \BAut(E_m)\dar{\otimes}\\
\BTop(n+m)\rar{t}&\BAut(E_{n+m})
\end{tikzcd}
\end{equation}
whose vertical arrows are induced by taking direct products and tensor products respectively.
\end{lem}

\begin{proof}We write $\Opd^{\gc}\subset \Opd^\un$ for the full subcategory of \emph{groupoid-coloured operads}, which are unital operads $\cO$ whose category of colours $\cO^\col$ is a groupoid (equivalently this is the pullback of $(-)^\col\colon \Opd^\un\ra \Cat$ along $\cS\subset \Cat$). This subcategory features in \cite[Theorem 2.2]{KKoperadic} which says that taking colimits in $\Opd$ of functors $\theta\colon X\ra \Opd$ where $X$ is a groupoid and $\theta$ has values in reduced operads, induces an equivalence over $\cS$ between the cartesian unstraightening of the functor $\Fun(-,\Opd^\red)\colon \cS^\op\ra \Cat$ and the functor $(-)^\col\colon \Opd^\gc\ra \cS$. Using this equivalence and viewing $\BAut(E_n)$ as a subcategory of $\Opd$, it suffices to construct an equivalence of operads
\[\hspace{-0.15cm}\colim\Big(\BTop(n)\times \BTop(m)\xra{(t,t)}\Opd\times \Opd \xra{\otimes}\Opd\Big)\simeq \colim\Big(\BTop(n)\times \BTop(m)\xra{\times}\BTop(n+m)\xra{t} \Opd\Big)\]
which upon applying $(-)^\col$ yields $\id_{\BTop(n)\times \BTop(m)}$. Since the tensor product preserves colimits in both variables, the left-hand side agrees with $\colim(\BTop(n)\ra\Opd)\otimes \colim(\BTop(m)\ra \Opd)$. Moreover, $\colim(\BTop(n)\ra\Opd)$ agrees, by definition of the map $t\colon \BTop(n)\ra \BAut(E_n)$, with the operad denoted $\smash{\mathbb{E}_{\BTop(n)}^\otimes}$ in \cite[5.4.2.10]{LurieHA} (see \cite[Section 5.1.1]{KKoperadic} where this operad is denoted $E_n^t$). By a similar argument as in the proof of Lemma 2.7 loc.cit.\,the right-hand colimit in the asserted equivalence agrees with $\smash{\mathbb{E}_{\BTop(n)\times \BTop(m)}^\otimes}$, so it suffices to establish an equivalence \[\smash{\mathbb{E}_{\BTop(n)}^\otimes\otimes \mathbb{E}_{\BTop(m)}^\otimes\simeq \mathbb{E}_{\BTop(n)\times \BTop(m)}^\otimes}\] that yields $\id_{\BTop(n)\times \BTop(m)}$ on categories of colours. This follows from \cite[5.4.2.14]{LurieHA}.
\end{proof}

\subsubsection{Dendroidal models}\label{sec:prelim-dendroidal} In addition to Lurie's model for operads, we will also use the \emph{dendroidal model} for operads. We now summarise the relevant points. Writing $\OpdSet$ for the ordinary $1$-category of coloured operads in sets in the classical sense, there is a fully faithful inclusion \cite[2.1.1.7]{LurieHA}
\[
	\ell\colon \OpdSet\longhookrightarrow \Opd.
\] 
The category $\OpdSet$ has a full subcategory $\Omega\subset \OpdSet$, the \emph{dendroidal tree category}, which is spanned by operads encoded by finite rooted trees (see \cite[Sections 1.3 and 3.2]{HeutsMoerdijk}). There is also an intermediate full subcategory $\Omega\subset \Phi\subset \OpdSet$, obtained from $\Omega$ by formally adjoining finite coproducts, which can be thought of as finite forests of finite rooted trees (see \cite[Section 2.2.2]{HinichMoerdijk}). The Yoneda embedding of $\Opd$ followed by restriction along $\ell|_{\Phi}$ yields a functor $\delta_\Phi\colon \Opd\ra \PSh(\Phi)$. By \cite[Theorem 3.1.4]{HinichMoerdijk} this functor is fully faithful and its essential image $\PSh(\Phi)^{\seg,c}\subset \PSh(\Phi)$ (denoted \texttt{DOp} in loc.cit.) is the category of \emph{complete dendroidal Segal spaces}. This is the full subcategory on those presheaves that satisfy three conditions: a completeness one, a dendroidal Segal one, and a forest Segal one (see (D1)--(D3) in Section 2.2.3 loc.cit.). We obtain a left-hand equivalence in
\[\smash{\delta_\Phi\colon \Opd\xlra{\simeq}\PSh(\Phi)^{\seg,c},\quad \quad \quad\delta_\Omega\colon \Opd\xlra{\simeq}\PSh(\Omega)^{\seg,c}.}\]
As indicated by the right-hand equivalence, one can avoid using forests and only work with $\Omega$: this right-hand equivalence is given similarly by the Yoneda embedding of $\Opd$ followed by restriction along $\ell|_{\Omega}$. It differs from the right-hand equivalence by postcomposition with the restriction $\PSh(\Phi)\ra \PSh(\Omega)$ along $\Omega\subset\Phi$, which restricts---by design of the forest Segal condition---to an equivalence $\PSh(\Phi)^{\seg,c}\simeq\PSh(\Omega)^{\seg,c}$ to the full subcategory of presheaves on $\Omega$ that only satisfy the completeness and dendroidal Segal condition (see (D1) and (D2) in loc.cit.). Note that, since $\ell$ is fully faithful, the compositions $\delta_\Phi\circ\ell|_{\Phi}$ and $\delta_\Omega\circ\ell|_{\Omega}$ agree with the respective Yoneda embeddings.

To model the subcategory of unital operads $\Opd^\un\subset \Opd$ and the tower \eqref{equ:tower-of-operads}, one considers the full subcategory $\overline{\Omega}\subset \Omega$ of \emph{closed trees} (see \cite[p.\,92, 97]{HeutsMoerdijk}). Imposing a dendroidal Segal and a completeness condition yields a full subcategory $\PSh(\overline{\Omega})^{\seg,c}\subset \PSh(\overline{\Omega})$ which is equivalent to $\Opd^\un$ such that the inclusion $\Opd^\un\subset\Opd$ corresponds via the equivalence $\delta_\Omega$ to the functor $\PSh(\overline{\Omega})^{\seg,c}\ra \PSh(\Omega)^{\seg,c}$ induced by left Kan extension along $\overline{\Omega}\subset \Omega$ (see \cite[Sections A.1]{KKoperadic} for a summary). For the tower \eqref{equ:tower-of-operads}, one considers the sequence of full subcategories $\smash{\overline{\Omega}_{\le 1}\subset \overline{\Omega}_{\le 2}\subset\cdots \overline{\Omega}_{\le \infty}=\overline{\Omega}}$ where $\smash{\overline{\Omega}_{\le k}\subset \overline{\Omega}}$ are those closed trees whose vertices all have $\le k$ incoming edges. Imposing the Segal and completeness condition gives full subcategories $\smash{\Opd^{\le k,\un}\coloneq \PSh(\overline{\Omega}_{\le k})^{\seg,c}\subset \PSh(\overline{\Omega}_{\le k})}$ which are preserved by restriction along the inclusions $\overline{\Omega}_{\le k-1}\subset \overline{\Omega}_{\le k}$ as well as left and right Kan extension along them (see \cite{DubeyLiu} and \cite[Sections A.2]{KKoperadic}). The resulting tower yields \eqref{equ:tower-of-operads} and the fully faithful left and right adjoints are given by the left and right Kan extensions, respectively. Since left Kan extension preserves representables, this in particular implies that the inclusion $\smash{\ell|_{\overline{\Omega}_{\le k}}\colon \overline{\Omega}_{\le k}\hookrightarrow \Opd^\un}$ factors over the left adjoint $\smash{\tau_!\colon  \Opd^{\le k,\un}\hookrightarrow  \Opd^\un}$ to truncation. 

\subsubsection{Rationalisation of reduced operads}\label{sec:rationalisation-operads}In terms of the dendroidal model, the full subcategory of reduced operads $\smash{\Opd^{\le k,\red}\subset \Opd^{\le k,\un}}\subset \PSh(\overline{\Omega}_{\le k})$ corresponds to those presheaves on $\overline{\Omega}_{\le k}$ that satisfy the dendroidal Segal condition and whose values at the closed $n$-corolla $\smash{\overline{C}_n}$ (the unique closed tree with $1$ internal and $n$ external vertices) is the terminal object for $n=0,1$. The completeness condition is automatic in this case and the pullback that appears in the dendroidal Segal condition becomes a product. In particular, for any endofunctor $\varphi\colon \cS\ra\cS$ that preserves finite products, the functor $\varphi_*\colon \PSh(\overline{\Omega}_{\le k})\ra \PSh(\overline{\Omega}_{\le k})$ given by postcomposition preserves the full subcategory of reduced operads. It hence induces an endofunctor $\varphi_*\colon \Opd^{\le k,\red}\ra \Opd^{\le k,\red}$ by restriction, compatible with truncation. Moreover, if $\varphi$ receives a natural transformation $\id_{\cS}\ra \varphi$ from the identity then so does $\varphi_*$, compatible with truncation. The component $\cO\ra\varphi_*(\cO)$ of this natural transformation is on spaces of multi-operations given by the respective component of the natural transformation $\id_{\cS}\ra \varphi$. Applying this to the rationalisation functor $\varphi=(-)_\bfQ$ from \cref{sec:rationalisation}, this allows us to rationalise reduced $k$-truncated operads for any $1\le k\le\infty$, compatible with truncation.

\subsubsection{Decomposing maps of operads}\label{sec:goppl-weiss} 
In view of the equivalence $\Opd^\un\simeq \lim_k\Opd^{\le k,\un}$, the mapping space $\Map(\cO,\cP)\coloneqq \Map_{\Opd^\un}(\cO,\cP) $ for unital operads $\cO,\cP\in \Opd^\un$ agrees with the limit of the tower of mapping spaces
\[\cdots\lra\Map_{\le k}(\cO,\cP)\lra \Map_{\le k-1}(\cO,\cP)\lra \cdots \lra \Map_{\le 1}(\cO,\cP)\simeq \ast\]
for the corresponding truncated operads; here $\Map_{\le k}(\cO,\cP)\coloneqq \Map_{\Opd^{\le k,\un}}(\cO,\cP)$. In the case where $\cO$ and $\cP$ are reduced, Göppl and Weiss \cite{Goppl} gave a convenient description of the fibres 
\begin{equation}\label{equ:fibre-Goppl}
	L^f_k(\cO,\cP)\coloneq \fib_f\big(\Map_{\le k}(\cO,\cP)\ra \Map_{\le k-1}(\cO,\cP)\big)\quad\text{for }f\in\Map_{\le k-1}(\cO,\cP).
\end{equation}
To state a suitable version of their result, we consider the commutative square
\[
\begin{tikzcd}[column sep=4cm,row sep=0.6cm]
	\PSh(\overline{\Omega}_{\le k})\dar{\res}\rar{\big(\Latch_k(-)\ra (-)(k)\ra 	\Match_k(-)\big)}&\PSh(\mathrm{B}\Sigma_k)^{[2]}\dar{(0\le 2)^*}\\
	 \PSh(\overline{\Omega}_{\le k-1})\rar{\big(\Latch_k(-)\ra \Match_k(-)\big)}&\PSh(\mathrm{B}\Sigma_k)^{[1]}
\end{tikzcd}
\]
of categories. Here the left vertical arrow is given by restriction. The category $\PSh(\mathrm{B}\Sigma_k)^{[2]}$ is the category of sequences $X\ra Y\ra Z$ of spaces with an action of the symmetric group $\Sigma_k$. The latter maps to the category $\PSh(\mathrm{B}\Sigma_k)^{[1]}$ of maps $X\ra Z$ of $\Sigma_k$-spaces by composition; this is the right vertical functor. For $\cO\in \PSh(\overline{\Omega}_{\le k-1})$, the $k$th \emph{latching object} $\Latch_k(\cO)$ is the $\Sigma_k$-space given by left Kan extending $\cO$ to $\overline{\Omega}_{\le k}$ and then evaluating at the $k$-corolla $\overline{C}_k\in \overline{\Omega}_{\le k}$ with its $\Sigma_k$-action by permuting the $k$ leaves. The \emph{matching object} $\Match_k(X)$ is defined analogously using right Kan extensions. The universal property of Kan extensions gives a map $\Latch_k(X)\ra \Match_k(X)$ of $\Sigma_k$-spaces; this gives the bottom horizontal functor. If $X$ is the restriction of a presheaf of $\PSh(\overline{\Omega}_{\le k})$ then the universal property of Kan extensions gives a factorisation of the map of $\Sigma_k$-spaces $\Latch_k(X)\ra \Match_k(X)$ through the value $\cO(k)\coloneq \cO(\overline{C}_k)$ at the $k$-corolla (which corresponds to the space of $k$-ary operations if $\cO$ is a reduced operad); this defines the top horizontal functor. 

By \cite[Theorem 3.2.7]{Goppl}, this square of categories induces pullbacks on mapping spaces between reduced operads. By taking vertical fibres in the pullback on mapping spaces, it follows that the fibre \eqref{equ:fibre-Goppl} is for $\cO,\cP\in \Opd^{\le k,\red}$ naturally equivalent to the space of $\Sigma_k$-equivariant dashed fillers in
\[
\begin{tikzcd}[row sep=0.3cm,column sep=0.6cm]
\Latch_k(\cO)\rar{f_*}\dar& \Latch_k(\cP)\dar\\
\cO(k)\arrow[d]\arrow[r,dashed]&\cP(k)\dar\\
\Match_k(\cO)\rar{f_*}&\Match_k(\cP).
\end{tikzcd}
\] 
It thus fits into a fibre sequence of the form
\begin{equation}\label{eqref:fibre-sequence-layers}
	L^f_k(\cO,\cP)\lra\Map_{\Latch_k(\cO)}(\cO(k),\cP(k))^{\Sigma_k}\lra \Map_{\Latch_k(\cO)}(\cO(k),\Match_k(\cP))^{\Sigma_k}
\end{equation}
where $\Map_{\Latch_k(\cO)}(\cO(k),\Match_k(\cP))$ is the space of maps $\cO(k)\ra \Match_k(\cP)$ extending the map $\Latch_k(\cO)\ra \Match_k(\cP)$ (based at the composition of $\cO(k) \to \Match_k(\cO)$ with $f_* \colon \Match_k(\cO) \to \Match_k(\cP)$), the space $\smash{\Map_{\Latch_k(\cO)}(\cO(k),\cP(k))}$ is the space of maps $\cO(k)\ra \cP(k)$ extending the map $\Latch_k(\cO)\ra \cP(k)$, and $(-)^{\Sigma_k}$ denotes taking $\Sigma_k$-invariants.

\subsubsection{Matching and latching objects}\label{sec:latching-matching}For reduced operads $\cO$, Göppl and Weiss also gave a simpler description of $\Latch_k(\cO)$ and $\Match_k(\cO)$: For the \emph{latching object}, one considers the poset $(\Psi_k,\le )$ of finite rooted trees $T$ with $k$ leaves labeled by $\underline{k}=\{1,\ldots,k\}$ and no vertex of valence $2$, where $T\le T'$ if $T'$ can be obtained from $T$ by collapsing an interior forest. This has a subposet $(\Psi_k^{-},\le )$ on those labeled trees with more than one interior vertex (i.e. those that are not equivalent to the $k$-corolla). There is an evident way to identify $(\Psi_k,\le )$ with a subcategory of the overcategory $\overline{\Omega}_{/\overline{C}_k}$  (the labelings of the leaves describe maps to $\smash{\overline{C}_k}$ using the identification of the leaves of $\smash{\overline{C}_k}$ with $\underline{k}$), so we get a map $\smash{\colim_{T \in \Psi^-_k}\cO(T)\ra \cO(\overline{C}_k)=\cO(k)}$
which the argument in \cite[Example 3.2.6]{Goppl} shows to be naturally equivalent to $\Latch_k(\cO)\ra \cO$. Note that $\Psi^-_k$ is empty for $k\le 2$, so $\Latch_k(\cO)=\varnothing$ for $k\le 2$. For the \emph{matching object}, one uses that any $S\subseteq \underline{k}$ defines a map $\smash{\overline{C}_{|S|}\ra \overline{C}_k}$ in $\smash{\overline{\Omega}}$, so there is a $\underline{k}$-cubical diagram $\smash{\underline{k}\supseteq S\mapsto \cO(\overline{C}_{|S|})}$. By the argument of \cite[Lemma 3.4.7]{WeissDalian}, the resulting map $\smash{\cO(k)=\cO(\overline{C}_k)\ra \lim_{S\subsetneq \underline{k}}\cO(\overline{C}_{|S|})}$
is equivalent to the map $\cO(k)\ra\Match_k(\cO)$.

\begin{ex}\label{ex:latching-object-Ed}Using this, Göppl and Weiss \cite[Examples 2.1.6, 3.2.6]{Goppl} showed that for the little $d$-discs operad $\cO=E_d$, the map $\Latch_k(\cO)\ra \cO(k)$ is equivalent to the boundary inclusion of an explicit compact manifold of dimension $(k-1)d-1$ whose boundary is empty if and only if $k\le 2$.
\end{ex}

We conclude this section by justifying the description of the category $\Opd^{\le 2,\red}$ mentioned at the beginning of the section.
 \begin{lem}\label{lem:2-truncated-reduced}The functor $(-)(2)\colon \Opd^{\le 2,\red}\ra \cS^{\Sigma_2}$ that takes $2$-ary operations is an equivalence.
\end{lem}

\begin{proof}Using $\smash{\Opd^{\le1,\red}\simeq\ast}$, $\smash{\Latch_2(\cO)=\varnothing}$, and $\Match_2(\cO)=*$, the description of the fibres of the map on mapping spaces induced by $\smash{\Opd^{\le2,\red}\ra \Opd^{\le1,\red}}$ from \cref{sec:goppl-weiss} shows that the functor in the claim is fully faithful. Essential surjectivity is equivalent to the functor $\smash{(-)(2)\colon \Opd^{\red}\ra \cS^{\Sigma_2}}$ on nontruncated operads being essentially surjective, i.e.\,that any $\Sigma_2$-space $X$ arises as the space of $2$-ary operations of a reduced operad. Modelling $X$ by a Kan-complex with a strict $\Sigma_2$-action, there is a $1$-coloured operad in Kan complexes in the classical sense whose $n$-ary operations are given by $\prod_{J \subseteq \ul{n},|J|=2} X$. Its operadic nerve (see \cite[2.1.1.27]{LurieHA}) yields an operad as required.
\end{proof}

\subsection{The Bousfield--Kan spectral sequence of a tower}\label{sec:bousfield-kan-ss}
Given a tower of pointed spaces
\[
	\cdots \ra X_s\ra X_{s-1}\ra\cdots\ra X_1\ra X_0\lra X_{-1}\simeq \ast,
\]
the homotopy groups of the limit $X\coloneq \lim_s X_s$ can be studied by a spectral sequence known as the  \emph{Bousfield--Kan spectral sequence} \cite[IX.4]{BousfieldKan}. It is natural in maps of towers of pointed spaces, and its entries are defined for $t\ge s\ge0$, starting on the first page $E_1^{s,t}$. The differentials are of the form $\smash{d_r\colon E_r^{s,t}\ra E_r^{s+r,t+r-1}}$, so decrease $t-s$ by $1$. The first page is given by 
\[E_1^{s,t}=\pi_{t-s}(F_s)\quad\text{where}\quad F_s\coloneq\fib(X_s\ra X_{s-1}).\] 
Note that these entries are not always abelian groups: they are possibly non-abelian groups for $t-s=1$, and in general only pointed sets for $t-s=0$. In general, the spectral sequence is an \emph{extended spectral sequence} in the sense of IX.4.2 loc.cit., which means that:

\begin{enumerate}[leftmargin=*]
\item $E^{s,t}_{r}$ is a pointed set for $t-s=0$, a group for $t-s\ge1$, and an abelian group for $t-s\ge2$.
\item There is an action of the group $E_r^{s-r,s-r+1}$ on the pointed set $E_r^{s,s}$.
\item For $t-s\ge2$ the differential $d_r\colon E_r^{s,t}\ra E_r^{s+r,t+r-1}$ is a group homomorphism which lands in the centre of the target for $t-s=2$. For $t-s=1$, the differential is given by acting on the basepoint using the action of  $E_r^{s-r,s-r+1}$ on $E_r^{s,s}$.
\item We have $E_r^{s,t}=\ker(d_r)/\im(d_r)$ for $t-s\ge1$, and $E_{r+1}^{s,s}\subset E_{r}^{s,s}/(\text{action of }E_r^{s-r,s-r+1})$.
\end{enumerate}
General convergence properties of this spectral sequence are discussed in IX.5 loc.cit.. We will only need a special case which we describe now. Note that there are only finitely many differentials into each entry since $E_r^{s,t}$ is only defined for $t\ge s\ge 0$, so we have $E_{r+1}^{s,t}\subset E_{r}^{s,t}$ for fixed $0\le s\le t$ and all $r\gg0$. If for a fixed $i\ge 1$ we have $\smash{E_r^{s,s+i}= E_{r+1}^{s,s+i}}$ for each $s\ge0$ and $r\ge N(s)$ for some $N(s)\ge0$, then one says that the spectral sequence \emph{Mittag-Leffler converges in degree $i$} (see IX.5.5 loc.cit.). In this case, the tower of surjections of groups
\begin{equation}\label{eqref:limit-tower}
	\cdots\twoheadrightarrow \pi_i(X)_1\twoheadrightarrow \pi_i(X)_0\twoheadrightarrow \pi_i(X)_{-1}=\ast\quad \text{given by}\quad \pi_i(X)_s\coloneq\im\big(\pi_i(X)\ra \pi_i(X_s)\big)
\end{equation} satisfies 
\[\pi_i(X)\cong \lim_s\pi_i(X)_s\quad\text{and}\quad E_\infty^{s,s+i}\cong\ker\big(\pi_i(X)_s\twoheadrightarrow \pi_i(X)_{s-1}\big)\quad \text{for }s\ge0\] 
where $E_\infty^{s,s+i}\coloneq E_{r}^{s,s+i}$ for $r\ge N(s)$ (see IX.5.3 loc.cit.).

\begin{rem}\label{rem:stronger-convergence}
In practice, for fixed $i\ge1$ the entries $\smash{E_1^{s,s+i}}$ and $\smash{E_1^{s,s+i-1}}$ are often zero for $s\gg0$, and hence the same holds for $E_r^{s,s+i}$ and $E_r^{s,s+i-1}$ for all $1\le r\le \infty$. This in particular implies Mittag-Leffler convergence in degree $i$ and that the tower of surjections \eqref{eqref:limit-tower} is eventually constant. In particular, by setting $\pi_i(X)^s\coloneq \ker(\pi_i(X)\twoheadrightarrow\pi_i(X)_s)$, this property implies that $\pi_i(X)$ has a finite  normal series $\{e\}=\pi_i(X)^N\triangleleft\cdots \triangleleft \pi_i(X)^{-1}=\pi_i(X)$ with quotients $\pi_i(X)^{s-1}/\pi_i(X)^s\cong E_\infty^{s,s+i}$.
\end{rem}

We now discuss three examples of this spectral sequence that will play a role in this work.

\begin{ex}\label{ex:homotopy-fixed-points}
Given an $E_1$-group $G$ acting on a pointed space $X$, the Postnikov tower $(\cdots\ra X_{\le 1}\ra X_{\le 0}\ra X_{\le -1}\simeq \ast)$ of $X$ gives a limit decomposition of the fixed points $X^G\simeq \lim_k (X_{\le k})^{G}$. In this case $\smash{F_s\simeq K(\pi_s(X),s)^G}$, so the $E_1$-page of the associated spectral sequence satisfies
\[E_1^{s,t}\cong \oH^{2s-t}\big(\mathrm{B}G;\pi_s(X)\big)\quad\text{for }t\ge s\ge0.\]
These cohomology groups (and pointed sets) are interpreted to be $0$ for $2s<t$. In the case $s=0,1$ where the coefficients need not be abelian groups, the possibly nontrivial entries are $\oH^{1}(\mathrm{B}G;\pi_1(X))$,  $\oH^{0}(\mathrm{B}G;\pi_1(X))$, and $\oH^{0}(\mathrm{B}G;\pi_0(X))$, which are defined as the first nonabelian cohomology of $\pi_0(G)$ acting on $\pi_1(X)$, the fixed point group of this action, and the fixed point set of the action on the pointed set $\pi_0(X)$ respectively. There are two situations in which we are going to apply this spectral sequence: in the first $\BG$ is finite dimensional and in the second $G$ is discrete and finite and $\pi_s(X)$ admits for $s\ge1$ a $G$-invariant finite  normal series whose consecutive quotients are $\bfQ$-vector space. In both cases  $\oH^k(\BG;\pi_s(X))$ vanishes for large enough $k$ and $s$ (by an induction on the length of the  normal series in the second case), so the spectral sequences in particular converge in all degrees $i\ge1$ in the strong sense of \cref{rem:stronger-convergence}.
\end{ex}

\begin{ex}\label{ex:BKSS-mapping-space}
Given a sequence of maps $A\ra X\ra Y$ between nonempty spaces, we consider the space $\Map_A(X,Y)$ of maps $X\ra Y$ that extend the composition $A\ra Y$. This is the fibre at $A\ra Y$ of the map $\Map(X,Y)\ra \Map(A,Y)$ induced by precomposition with $A\ra X$. This is pointed by the given map $X\ra Y$. Applying $\Map_A(X,-)$ to the Postnikov tower $(\cdots\ra Y_{\le 1}\ra Y_{\le 0}\ra Y_{\le -1}\simeq \ast)$ of $Y$ provides a limit decomposition $\Map_A(X,Y)\simeq\lim\big(\cdots \ra \Map_A(X,Y_{\le 0})\ra\Map_A(X,Y_{\le -1})\simeq \ast\big)$ as pointed spaces. For simplicity, we assume that $Y$ is $1$-connected. In this case, the Postnikov tower of $Y$ consists of principal fibrations, so we have fibre sequence $Y_{\le s}\ra Y_{\le s-1}\ra K(\pi_{s}(Y),s+1)$ and hence the $E_1$-page of the Bousfield--Kan spectral sequence is given by
\[E_1^{s,t}=\pi_{t-s}\big(\fib(\Map_A(X,Y_{\le s})\ra \Map_A(X,Y_{\le s-1}))\big)\cong\oH^{2s-t}(X/A;\pi_s(Y))\quad\text{for }t\ge s\ge0.\]
Here $X/A$ is the cofibre of $A\ra X$ and the groups $\oH^{2s-t}(X/A;\pi_s(Y))$ are interpreted to be $0$ for $s\le 1$ or $2s-t<0$.  In particular, if $X/A$ is homologically finite-dimensional, then the spectral sequence converges in all degrees $i\ge1$ in the strong sense of \cref{rem:stronger-convergence}.
\end{ex}

\begin{ex}\label{ex:BKSS-automorphism-space}
Given a map $A\ra X$ to a $1$-connected space $X$, we consider the $E_1$-group $\Aut_A(X)$ of automorphisms of $X$ under $A$, i.e.\,the fibre of the map $\Aut(X)\ra \Map(A,X)$ at the given map $A\ra X$. Postnikov truncating $X$ gives a limit decomposition as $E_1$-groups $\Aut_A(X)\simeq\lim_k\Aut_A(X_{\le k})$. This deloops to a limit decomposition $\BAut_A(X)\simeq(\lim_k\BAut_A(X_{\le k}))_\id$
where $(-)_\id$ denotes the basepoint component (to see this, note that since both sides are path-connected, the statement can be checked after applying loop spaces). To identify the $E_1$-page of the associated Bousfield--Kan spectral sequence, it is convenient to pass to the subgroup $\SAut_A(X)\subset\Aut_A(X)$ of those homotopy automorphisms that induce the identity on all homotopy groups of $X$. The latter has a similar limit decomposition as pointed spaces $\BSAut_A(X)\simeq\lim(\cdots\ra  \BSAut_A(X_{\le 0})\ra \BSAut_A(X_{\le -1})\simeq \ast)_\id$. The $E_1$-page of the Bousfield--Kan spectral sequence of this tower satisfies
\begin{equation}\label{eqref:e1-page-aut}E_1^{s,t}\cong \begin{cases}\text{pointed subset of }\oH^{s+1}(X_{\le s-1}/A;\pi_{s}X)&\text{for }t-s=0,\text{ as pointed sets}\\
\oH^{2s-t+1}(X_{\le s-1}/A;\pi_{s}X)&\text{for }t-s\ge1, \text{ as groups.}
\end{cases}\end{equation}
This identification becomes clearest when being deduced from a more general statement: fixing a space $A$ and an integer $s\ge2$, we write $\smash{\cS^{\le s,1\text{-ctd},\simeq }_{A/}\subset \cS^{\simeq}_{A/}}$ for the full subcategory of the core of the category of spaces under $A$ on those maps $A\ra X$ for which $X$ is $s$-truncated and $1$-connected. Writing $\cAb^\simeq$ for the core of the category of abelian groups, there is a fibre sequence 
\[\Map_A(Z,K(M,s+1))\lra \cS^{\le s,1\text{-ctd},\simeq}_{A/}\lra \cS^{\le s-1,1\text{-ctd},\simeq}_{A/}\times\cAb^\simeq\]
in which the fibre is taken over $\smash{(A\ra Z,M)\in  \cS^{\le s-1,1\text{-ctd}}_{/ A}\times\cAb}$. Here $\Map_A(Z,K(M,s+1))$ is the space of maps $Z\ra K(M,s+1)$ that agree with the constant map on $A$, the first map is given by taking fibres, and the second map sends $A\ra Y$ to the composition $A\ra Y\ra Y_{\le s-1}$ in the first argument and to $\pi_s(Y)$ in the second (see \cite[Remark 3.9]{BlancDwyerGoerss} for a non-$\infty$-categorical reference in the case $A=\ast$ and \cite[Section 2]{Pstragowski} for an $\infty$-categorical reference in the case $A=\varnothing$; the argument for general $A$ is similar). Restricting components and taking fibres at $(\pi_i(X))_{i\ge1}$ of the functors to $\cAb^{\times s}$ given by taking homotopy groups, we obtain  a fibre sequence $\Map_A(X_{\le s-1},K(\pi_s(X),s+1))_{X_{\le s}}\ra \BSAut_A(X_{\le s})\ra \BSAut_A(X_{\le s-1})$. Here $\Map_A(X_{\le s-1},K(\pi_s(X),s+1))_{X_{\le s}}\subset \Map_A(X_{\le s-1},K(\pi_s(X),s+1))$ is the collection of those components of maps $f\colon X_{\le s-1}\ra K(\pi_s(X),s+1)$ trivialised on $A$ with the following property: the fibre inclusion $\fib(f)\ra X_{\le s-1}$ is equivalent, as a map under $A$, to the truncation $X_{\le s}\ra X_{\le s-1}$. Taking homotopy  groups gives the identification \eqref{eqref:e1-page-aut}. Now note that $E_1^{s,s+k}=\oH^{s-k+1}(X_{\le s-1}/A;\pi_{s}(X))$ injects for $k\ge1$ into $\oH^{s-k+1}(X/A;\pi_{s}(X))$ since $X\ra X_{\le s-1}$ is $s$-connected. Consequently, if $X/A$ is homologically finite-dimensional, these groups vanish for $s\gg0$, so the spectral sequence converges in degrees $i\ge2$ in the strong sense of \cref{rem:stronger-convergence}. In dimension $i=1$, the situation is more subtle since the pointed sets $E_1^{s,s}\subset \oH^{s+1}(X_{\le s-1}/A;\pi_{s}X)$ could be nontrivial for infinitely many $s$. However, one still gets that $\pi_0(\SAut_A(X_{\le s}))=\pi_1(\BSAut_A(X_{\le s}))$ admits for \emph{each fixed finite $s\ge 1$} a finite normal series with quotients isomorphic to $\oH^{s}(X_{\le s-1}/A;\pi_{s}(X))\le  \oH^{s}(X/A;\pi_{s}(X))$ for varying $s$, but this need not be the case for $s=\infty$. However, note that if $X/A$ is homologically finite-dimensional, then it follows from obstruction theory that $\pi_0(\SAut_A(X))\ra \pi_0(\SAut_A(X_{\le s}))$ is injective for $s\gg0$. Thus, even though $\pi_0(\SAut_A(X))$ need not admit a finite normal series with quotients of the form $\oH^{s}(X_{\le s-1}/A;\pi_{s}(X))$, it is a subgroup of a group that does. In the one case that will play a role for us (see the proof of \cref{prop:auted-cyclotomic}), this will be sufficient. \end{ex}

\section{Twice-iterated stabilisation for $\Aut(\smash{E_{d,\bfQ}})$}\label{sec:doublestab}
As mentioned in the final part of the introduction, a crucial ingredient for the proof of our main results is that the twice-iterated stabilisation map for the group of orientation preserving automorphisms of the rationalised $E_d$-operad is nullhomotopic. Let us now be more precise about this. Recall that the space of $k$-ary operations of the little $d$-discs operad $E_d$ is equivalent to the ordered configuration space $F_k(\bfR^d)$ (see \cref{sec:dl-add}). The latter is contractible for $k\le 1$, so $E_d$ is a reduced operad and we may apply the rationalisation procedure described in \cref{sec:rationalisation-operads} to define $E_{d,\bfQ}\coloneq (E_d)_\bfQ$. As $F_2(\bfR^d)$ is equivalent to a $(d-1)$-sphere, the space of binary operations in $E_{d,\bfQ}$ is equivalent to a rational $(d-1)$-sphere, so acting on its top homology gives a morphism
\[
	\chi\colon \Aut(E_{d,\bfQ})\lra \Aut(\widetilde{\oH}_{d-1}(E_{d,\bfQ}(2)))\cong\GL_1(\bfQ),
\] 
whose kernel we denote as $\smash{\SAut(E_{d,\bfQ})\le \Aut(E_{d,\bfQ})}$. We will see shortly that there is a stabilisation $E_1$-map $\Aut(E_{d-1,\bfQ})\ra \Aut(E_{d,\bfQ})$ which is compatible with the map $\chi$, so it restricts to an $E_1$-map $\SAut(E_{d-1,\bfQ})\ra \SAut(E_{d,\bfQ})$. The purpose of this section is to prove \cref{bigthm:nullhomotopy}, which says that its twice-iterate $\SAut(E_{d-2,\bfQ})\ra \SAut(E_{d,\bfQ})$ is nullhomotopic as an $E_1$-map.

\begin{rem}\label{rem:nullhomotopy-optimal} \cref{bigthm:nullhomotopy} is optimal in the following three senses.
\begin{enumerate}[leftmargin=*]
	\item The restriction to the kernel $\SAut(E_{d,\bfQ}) \le \Aut(E_{d,\bfQ})$ is necessary since the morphism $\chi \colon \Aut(E_{d,\bfQ}) \to \GL_1(\bfQ)$ is compatible with stabilisation and surjective (see \cref{sec:e2-strategy}).
	\item The \emph{single} $E_1$-stabilisation map $\SAut(E_{d-1,\bfQ})\ra \SAut(E_{d,\bfQ})$ is often \emph{not} null: for any odd $d=2n+1\ge3$ the composition (involving the evident action of $\oO(d)$ on $E_{d}$)
	\[\hspace{0.7cm}\quad\BSO(2n)\ra \BSAut(E_{2n,\bfQ})\ra \BSAut(E_{2n+1,\bfQ})\ra \BSAut(E_{2n+1,\bfQ}(2))\simeq \BSAut(S^{2n})_\bfQ\]
	pulls back the Hopf class $E\in \oH^{4n}(\BSAut(S^{2n});\bfQ)$ to the (nontrivial) square of the Euler class $p_n=e^2\in \oH^{4n}(\BSO(2n);\bfQ)$ (see e.g.\,\cite[p.\,95]{Sullivan}), so the second map is not null.
	\item There is a version $\SAut(E_{d-2}) \to \SAut(E_d)$ of the double stabilisation map between the unrationalised $E_d$-operads which is compatible with the version for the rationalised operads. The nonrationalised version can be seen to often be \emph{not} nullhomotopic, by precomposing it with the map $\SO(d-2)\ra \SAut(E_{d-2})$ and postcomposing it with the composition $\SAut(E_{d})\ra \SAut(S^{d-1})\ra \SAut_*(S^{d})$ induced by restriction to arity 2 and unreduced suspension. The resulting map $\SO(d-2)\ra\SAut_*(S^{d})$ induces the appropriate unstable $J$-homomorphism on homotopy groups, so it is often nontrivial.
\end{enumerate}
\end{rem}

\subsection{Strategy of proof}\label{sec:e2-strategy}
The first step to proving \cref{bigthm:nullhomotopy} is to construct a dashed map in 
\begin{equation}\label{equ:filler-rational-bv}
\begin{tikzcd}[column sep=1cm,row sep=0.6cm]
	\Aut(E_{n})\times \Aut(E_{m})\rar{(-)_\bfQ}\arrow[d,"{\otimes}",swap]&\Aut(E_{n,\bfQ})\times \Aut(E_{m,\bfQ})\rar{(\chi,\chi)}\arrow[d,"{\otimes_\bfQ}",swap,dashed]& \dar{\otimes}\GL_1(\bfQ)\times \GL_1(\bfQ)\\
	\Aut(E_{n+m})\rar{(-)_\bfQ}&\Aut(E_{n+m,\bfQ}) \rar{\chi}&\GL_1(\bfQ),
\end{tikzcd}
\end{equation}
making this a commutative diagram of $E_1$-groups for $n,m\ge1$. Here the maps labeled $(-)_\bfQ$ are induced by the rationalisation functor from \cref{sec:rationalisation-operads} and the leftmost vertical map is induced by the tensor product of operads from \cref{sec:operads-tensor-product} and the additivity equivalence $E_n\otimes E_m\simeq E_{n+m}$ from \cref{sec:dl-add}. The right vertical map is induced by $\bfQ\otimes_{\bfQ}\bfQ\cong\bfQ$, so it is given by multiplication in the group $\GL_1(\bfQ)$. The construction of the dashed filler is somewhat subtle, since we do not know whether one has $E_{n,\bfQ}\otimes E_{m,\bfQ}\simeq E_{n+m,\bfQ}$, so the dashed filler will \emph{not} just be given by taking tensor products in $\Opd$. Instead it will use a tensor product in a certain category of \emph{rational pro-operads}.

\begin{rem}\label{rem:Pedro-Geoffroy}A similar diagram was constructed by Boavida de Brito--Horel, with the leftmost vertical map replaced by a certain tensor product for configuration categories from \cite{BoavidaWeissProduct} (see \cite[1.5, 5.5., 6.3, 7.2]{BoavidaHorel}; by the discussion in \cref{sec:the-rational-filler} below our definition of $\Aut(E_{d,\bfQ})$ agrees with theirs). Their construction inspired ours, and it seems likely that the resulting diagrams are in fact equivalent. Showing this would involve proving that the tensor product for configuration categories from \cite{BoavidaWeissProduct} is compatible with the one for operads from \cite{LurieHA}. We learned from Boavida de Brito that there is a natural comparison map between them, but it has not yet been shown to be an equivalence. 
\end{rem}

Once \eqref{equ:filler-rational-bv} is established, we can specialise the dashed map to $n=d-2$ and $m=2$ and restrict to $\SAut$-subgroups to obtain an $E_1$-map $(-)\otimes_\bfQ(-)\colon \SAut(E_{d-2,\bfQ})\times \SAut(E_{2,\bfQ})\ra \SAut(E_{d,\bfQ})$. This map is $\Aut(E_{2,\bfQ})$-equivariant with respect to the action on the source by acting trivially on  $\SAut(E_{d-2,\bfQ})$ and by conjugation on $\SAut(E_{2,\bfQ})$, and on the target by stabilisation followed by conjugation. Restricting the source to $\SAut(\smash{E_{d-2,\bfQ}})\times \{\id\}$ and taking invariants, it follows that the stabilisation map $(-)\otimes_\bfQ\id_{E_{2,\bfQ}}\colon \SAut(\smash{E_{d-2,\bfQ}})\ra \SAut(\smash{E_{d,\bfQ}})$ factors through the $\Aut(E_{2,\bfQ})$-invariants $\smash{\SAut(E_{d,\bfQ})^{\Aut(E_{2,\bfQ})}}$. It then also factors through the invariants $\smash{\SAut(E_{d,\bfQ})^G}$ of any $E_1$-group $G$ that acts on $\SAut(E_{d,\bfQ})$ through $\Aut(E_{2,\bfQ})$. In order to prove \cref{bigthm:nullhomotopy}, it will thus suffice to find such a $G$ for which the invariants $\SAut(E_{d,\bfQ})^G$ are trivial. Our choice of $G$ relies the following result going back to Drinfel'd (see \cite[Theorem 7.1]{BoavidaHorel} and \cite{Drinfeld}), which can also be used to give a proof of rational formality of the $E_d$-operads (see \cite{PetersenGT,BoavidaHorel}).

\begin{thm}[Drinfel'd]\label{thm:cyc-surj}The map $\chi\colon \Aut(E_{2,\bfQ})\ra \GL_1(\bfQ)$ is surjective.
\end{thm}

Fixing an automorphism $\Lambda\in \Aut(E_{2,\bfQ})$, we get a morphism $\langle\Lambda\rangle\ra \Aut(E_{2,\bfQ})$ from the free $E_1$-group generated by $\Lambda$. By the above discussion, \cref{bigthm:nullhomotopy} would follow if we could choose $\Lambda$ so that the invariants $\SAut(\smash{E_{d,\bfQ}})^{\langle \Lambda\rangle}$ are trivial. We will show that this is indeed possible: 

\begin{prop}
\label{prop:trivial-invariants}For $d\ge3$ and $\Lambda\in\Aut(E_{2,\bfQ})$ with $\chi(\Lambda)\neq\pm1$, we have $\SAut(\smash{E_{d,\bfQ}})^{\langle \Lambda\rangle}\simeq\ast$.
\end{prop}

To summarise the discussion, in order to prove \cref{bigthm:nullhomotopy}, it suffices to (a) construct the dashed filler in \eqref{equ:filler-rational-bv} , and (b) prove \cref{prop:trivial-invariants}. Part (a) will be taken care of in \cref{sec:the-rational-filler}, so the remainder of this section deals with (b), the proof of \cref{prop:trivial-invariants}.

\subsection{Cyclotomic actions} The key ingredient to the proof of \cref{prop:trivial-invariants} is to establish a certain qualitative property of the action of $\Aut(E_{2,\bfQ})$ on the homotopy groups of the automorphism spaces $\SAut_{\le k}(E_{d,\bfQ})$ of the truncations of $E_{d,\bfQ}$. For brevity, we write 
\[
	\GT\coloneq \pi_0(\Aut(E_{2,\bfQ})).
\]
This notation is motivated by a result of Fresse \cite[Theorem A]{Fresse} that says that  $\pi_0(\Aut(E_{2,\bfQ}))$ is isomorphic to Drinfel'ds pro-unipotent rational version of the Grothendieck--Teichmüller group, though we will not rely on this. The qualitative property of the $\GT$-action on the homotopy groups of $\SAut_{\le k}(E_{d,\bfQ})$ we prove is that they are \emph{gr-cyclotomic of nontrivial weight} in the following sense:

\begin{dfn}Let $G$ be a group with a $\GT$-action.
\begin{enumerate}[leftmargin=*]
	\item The $\GT$-action is \emph{cyclotomic} if it factors through the morphism $\chi\colon \GT\ra \GL_1(\bfQ)$. 
	\item A cyclotomic action is \emph{of weight $n$} for an integer $n\in \bfZ$ if $G$ is abelian, torsion-free, and divisible (equivalently: a $\bfQ$-vector space) and the action is given by $\varphi(g)=\chi(\varphi)^n\cdot g$ for all $\varphi\in\GT$ and $g\in G$. It is of \emph{nontrivial weight} if it is of weight $n$ for \emph{some} $n\neq0$.
	\item The action is \emph{gr-cyclotomic} if there exists a $\GT$-invariant finite  normal series $1=G_0\triangleleft \cdots\triangleleft G_k=G$ such that the action on all the subquotients $G_i/G_{i-1}$ is cyclotomic. Such a gr-cyclotomic action is of \emph{nontrivial weight} if the action on all the subquotients is of nontrivial weight (the weights may be different for the different subquotients).
\end{enumerate}
\end{dfn}

Being gr-cyclotomic of nontrivial weight enjoys various inheritance properties:

\begin{lem}\label{lem:cyclotomic-inherits}\ 
\begin{enumerate}[leftmargin=*]
	\item \label{enum:cyclotomic-inherits-i} Being gr-cyclotomic of nontrivial weight is preserved under taking $\GT$-invariant subgroups, quotients by $\GT$-invariant normal subgroups, and $\GT$-invariant extensions.
	\item \label{enum:cyclotomic-inherits-ii} Let $Q$ be a group and $A$ an abelian group with $(Q\times \GT)$-action. If $A$ is gr-cyclotomic of nontrivial weight then so is $\oH^k(\mathrm{B}Q;A)$ for $k\ge0$. If $A$ is nonabelian, then this also holds for $k=0$. 
\end{enumerate}
\end{lem}

\begin{proof}We only prove the part of \ref{enum:cyclotomic-inherits-i} about taking subgroups; the other properties are similar. Fix a group $G$ with $\GT$-action that is gr-cyclotomic of nontrivial weight, i.e.\,there is a  $\GT$-invariant finite normal series $1 = G_0 \triangleleft \cdots \triangleleft G_k = G$  such that $G_i/G_{i-1}$ is cyclotomic of weight $n_i\neq0$. Given a $\GT$-invariant subgroup $H\le G$ we consider the normal series $1 = H_0 \triangleleft \ldots \triangleleft H_k = H$ with $H_i\coloneq H\cap G_i$,  and claim that the $\GT$-invariant subgroup $H_i/H_{i-1}\le G_i/G_{i-1}$ is cyclotomic of nontrivial weight. That $H_i/H_{i-1}$ is cyclotomic, abelian, torsion free, and that the action is given by $\varphi(x)=\chi(\varphi)^{n_i}\cdot x$ follows from the corresponding properties for $G_i/G_{i-1}$. It remains to show that any $x\in H_i/H_{i-1}$ is divisible by any $m>0$. For this use \cref{thm:cyc-surj} to find an $\varphi \in \GT$ with $\chi(\varphi) = m^{-n_i}$, then for $x\in H_i/H_{i-1}$, we have by $\GT$-invariance of $H_i/H_{i-1}\le G_i/G_{i-1}$ and the fact that $G_i/G_{i-1}$ has weight $n_i$ that $\smash{\varphi(x)=\chi(\varphi)^{n_i}\cdot x= m^{-n_i^2}\cdot x \in H_i/H_{i-1} }$. This shows that $x$ is divisible by $\smash{m^{n_i^2}}$, so also by $m$.

The first part of \ref{enum:cyclotomic-inherits-ii} (where $A$ is abelian) follows from a straight-forward induction over the length of the normal series, using  \ref{enum:cyclotomic-inherits-i} and the long exact sequence in cohomology induced by a short exact sequence in the coefficients. The part about nonabelian $A$ follows directly from \ref{enum:cyclotomic-inherits-i}, since $\oH^0(\mathrm{B}Q;A)$ is given by the invariants of the $Q$-action on $A$, which is a $\GT$-invariant subgroup of $A$.
\end{proof}

Many of the $\GT$-actions related to $E_{d,\bfQ}$ turn out to be gr-cyclotomic of nontrivial weight:

\begin{lem}\label{lem:weight-lemma}Fix $d\ge3$ and $k\ge2$.
\begin{enumerate}[leftmargin=*]
	\item\label{enum:weight-i} The groups $\oH_\ast(\smash{E_{d,\bfQ}}(k))$ are trivial except possibly in degrees $\ast=n(d-1)$ for $0\le n\le k-1$ in which case the $\GT$-action is cyclotomic of weight $n$.
	\item\label{enum:weight-ii} The groups $\pi_\ast(\smash{E_{d,\bfQ}}(k))$ and $\pi_\ast(\Match_k(\smash{E_{d,\bfQ}}))$ are trivial except possibly in degrees $\ast=n(d-2)+1$ for $n\ge1$ in which case the $\GT$-action is cyclotomic of weight $n$.
	\item\label{enum:weight-iii} The group $\oH_{(k-1) d-1}(\smash{E_{d,\bfQ}}(k),\Latch_k(\smash{E_{d,\bfQ}}))$ is cyclotomic of weight $(k-1)$.
	\item\label{enum:weight-iv} The groups $\oH^\ast(\smash{E_{d,\bfQ}}(k),\Latch_k(\smash{E_{d,\bfQ}}))$ are trivial except possibly in degrees $\ast=(k-1)d-n(d-1)-1$ for $0\le n\le k-1$ in which case the $\GT$-action is cyclotomic of weight $n-k+1$.
	\item\label{enum:weight-v} For $k\ge3$ and all $j\ge i\ge0$, the $\GT$-action on the groups \[\hspace{0.9cm}\oH^i\big(\smash{E_{d,\bfQ}}(k),\Latch_k(\smash{E_{d,\bfQ}});\pi_j(\smash{E_{d,\bfQ}}(k))\big)\ \ \ \text{and}\ \ \ \oH^i\big(\smash{E_{d,\bfQ}}(k),\Latch_k(\smash{E_{d,\bfQ}});\pi_j(\Match_k(\smash{E_{d,\bfQ}}))\big)\]
	is cyclotomic of nontrivial weight. For $k=2$, the same holds except in degrees $i=j=d-1$.
\end{enumerate}
\end{lem}

\begin{proof}Throughout the proof, we will make use of two facts: (a) the map $\Latch_k(E_d)\ra E_d(k)$ is for each $k$ equivalent to the boundary inclusion of a compact $((k-1) d-1)$-dimensional manifold (see \cref{ex:latching-object-Ed}), and (b) the space $E_d(k)$ is $1$-connected for all $k$. The reason for the latter is that $\smash{E_{d}}(k)$ is equivalent to the space $\Emb(\underline{k},\bfR^d)$ of ordered configurations of $k$ points in $\bfR^d$ which is $1$-connected by transversality since we assumed $d\ge3$. In particular, the spaces $\smash{E_{d}}(k)$ are nilpotent, so their rationalisations $\smash{E_{d,\bfQ}}(k)$ are well-behaved in the sense of \cref{sec:nilpotent-completion}.

For \ref{enum:weight-i}, we use that the operad $\oH_*(\smash{E_{d,\bfQ}}) = \oH_*(\smash{E_d};\bfQ)$ in graded vector spaces is the $d$-Poisson operad (see e.g.\,\cite[Theorem 6.3, Section 5]{Sinha}) which is generated under operadic compositions in arity $2$. This reduces the claim to showing that the arity $2$ part is concentrated in degree $0$ and $d-1$ and acted upon by $\GT$ cyclotomically with weight $0$ in degree $0$ and weight $1$ in degree $d-1$. But $\oH_{*}(\smash{E_{d,\bfQ}}(2))\cong \oH_{*}(S^{d-1};\bfQ)$, so this holds by definition of the morphism $\chi$.

The first part of \ref{enum:weight-ii} follows from \ref{enum:weight-i} together with the fact that the homotopy groups $\pi_\ast(\smash{E_{d,\bfQ}}(k))=\pi_\ast(E_d(k))\otimes\bfQ$ are generated under taking Whitehead brackets by $\pi_{d-1}(E_d(k))\otimes\bfQ\cong \oH_{d-1}(E_d(k);\bfQ)$ which follows e.g.\,from \cite[Theorem 2.3]{CohenGitler} together with the Milnor--Moore theorem. The second part will follow from the first part by showing that the map $E_{d,\bfQ}(k)\ra\Match_k(E_{d,\bfQ})$ is surjective on homotopy groups. From the discussion in \cref{sec:latching-matching}, it follows that the map $E_{d,\bfQ}(k)\ra\Match_k(E_{d,\bfQ})$ is equivalent to the map $\smash{\Emb(\underline{k},\bfR^d)_\bfQ\ra \lim_{S\subsetneq\ul{k}}(\Emb(S,\bfR^d)_\bfQ)}$ induced by forgetting points, so it suffices to show that the fibre inclusion $\smash{\tohofib_{S\subseteq \ul{k}}(\Emb(S,\bfR^d)_\bfQ)\ra \Emb(\underline{k},\bfR^d)_\bfQ}$ is injective on homotopy groups. Adding new points to the configuration near $\infty$ shows that the cubical diagram $\ul{k}\supseteq S\mapsto \Emb(S,\bfR^d)_\bfQ$ is split up to homotopy in the sense of \cite[Definition 5.24]{KRWEvenDiscs}, so the claim follows from Lemma 5.25 loc.cit..

To show \ref{enum:weight-iii}, we distinguish two cases. If $k=2$, then $\Latch_k(\smash{E_{d,\bfQ}})=\varnothing$ (see \cref{sec:latching-matching}), so the group in question is $H_{d-1}(\smash{E_d}(2);\bfQ)$ on which $\GT$ acts indeed cyclotomically of weight $1=k-1$. If $k\ge 3$, then $\Latch_k(\smash{E_{d}})$ is nonempty (see \cref{ex:latching-object-Ed}), so $\oH_{(k-1) d-1}(\smash{E_d}(k);\bfQ)$ vanishes by Poincaré duality and hence the connecting map $\oH_{(k-1) d-1}(\smash{E_{d,\bfQ}(k)},\Latch_k(\smash{E_{d,\bfQ}}))\ra \oH_{(k-1) d-2}(\Latch_k(\smash{E_{d,\bfQ}}))$ is injective, so it suffices to prove that $\oH_{(k-1) d-2}(\Latch_k(\smash{E_{d,\bfQ}}))$ is cyclotomic of weight $(k-1)$. For this we use the colimit decomposition $\colim_{T\in\Psi^-_k}E_{d,\bfQ}(T)\simeq \Latch_k(\smash{E_{d,\bfQ}})$ from \cref{sec:latching-matching}. Filtering the bar-construction model for colimits by skeleta, we get a spectral sequence of $\GT$-modules \begin{equation}\label{equ:bar-ss}
	\textstyle{E_{p,q}^1=\bigoplus_{T_0\le \ldots\le T_p}\oH_q(\smash{E_{d,\bfQ}}(T_0))\implies \oH_{p+q}(\Latch_k(\smash{E_{d,\bfQ}}))}
\end{equation}
where the sum runs through non-degenerate chains of length $p$ in the poset $\Psi^-_k$. Note that any such chain has  length ${\le}(k-3)$ since rooted trees with $k$ leaves that have no interior vertex of valence $2$ and have at least one interior vertex differ from each other by collapsing at most $k-3$ edges. This implies that $E_{p,q}^1=0$ for $p>k-3$. By the dendroidal Segal property we have $\smash{E_{d,\bfQ}(T_0)\simeq \sqcap_{i=1}^lE_{d,\bfQ}(k_i)}$ with $\smash{\sum_{i=1}^l(k_i-1)=k-1}$, and as $\smash{E_{d,\bfQ}}(n)$ has homology concentrated in degrees in degrees $*\le (n-1)(d-1)$, the homology of $\smash{E_{d,\bfQ}}(T_0)$ is concentrated in degrees $*\le (k-1)(d-1)$. This implies that $\smash{E^1_{p,q}}=0$ for $q>(k-1)(d-1)$. As $(k-3) + (k-1)(d-1)$ agrees with $(k-1)d-2$, the only entry that can contribute to $\oH_{(k-1)d-2}(\Latch_k(\smash{E_{d,\bfQ}}))$ is $\smash{E^1_{k-3,(k-1)(d-1)}}$. Since the action of $\GT$ on the latter is cyclotomic of weight $k-1$ by \ref{enum:weight-i}, so \ref{enum:weight-iii} follows.

For \ref{enum:weight-iv}, we first compare the spectral sequence \eqref{equ:bar-ss} with the corresponding spectral sequence for $E_d$ to see that the map $ \Latch_k(\smash{E_{d}})\ra \Latch_k(\smash{E_{d,\bfQ}})$ is a $\bfQ$-homology isomorphism. This implies that the pair $(E_{d,\bfQ}(k),\Latch_k(\smash{E_{d,\bfQ}}))$ has rational Poincaré duality with relative fundamental class in degree $(k-1)d-1$. Capping with this class thus gives an isomorphism 
\[\oH^\ast\big(\smash{E_{d,\bfQ}}(k),\Latch_k(\smash{E_{d,\bfQ}})\big)\cong\oH_{(k-1) d-1-\ast}\big(\smash{E_{d,\bfQ}}(k)\big)\] which together with \ref{enum:weight-i} implies the first part of \ref{enum:weight-iv}. The second part follows by observing that since the right-hand side of the Poincaré duality isomorphism is cyclotomic of weight $n$ by \ref{enum:weight-i} and the fundamental class has weight $(k-1)$ by \ref{enum:weight-iii}, the left-hand side is cyclotomic of weight $n-(k-1)$.

For \ref{enum:weight-v} we note that the universal coefficient theorem provides a $\GT$-equivariant isomorphism
\[\smash{\oH^i\big(\smash{E_{d,\bfQ}}(k),\Latch_k(\smash{E_{d,\bfQ}});\pi_j(\smash{E_{d,\bfQ}}(k))\big) \cong \oH^i\big(\smash{E_{d,\bfQ}}(k),\Latch_k(\smash{E_{d,\bfQ}});\bfQ\big) \otimes \pi_j(\smash{E_{d,\bfQ}}(k))}\]
and similarly for $\oH^i\big(\smash{E_{d,\bfQ}}(k),\Latch_k(\smash{E_{d,\bfQ}});\pi_j(\Match_k(\smash{E_{d,\bfQ}}))\big)$; here the tensor products are acted upon diagonally. By combining \ref{enum:weight-ii} and \ref{enum:weight-iv} it thus follows that the groups in  \ref{enum:weight-v} are $0$ unless $(i,j)=((k-1)d-n(d-1)-1,m(d-2)+1)$ for some $n\ge0$ and $m\ge1$ and in this case the $\GT$-action is cyclotomic of weight $w \coloneq n-k+1+m$. It remains to show that $w\neq 0$. As $0\le j-i=w(d-1)-m-k+3$, the only case in which the weight could be $0$ is when $m+k \leq 3$. As $m \geq 1$ this does not happen for $k\ge3$ and it happens for $k=2$ exactly when $j-i=0$ and $m=1$. For $m=1$ and $k=2$ we have $n=0$ if $w=0$, and thus $i=j=d-1$, as claimed.
\end{proof}

\begin{prop}\label{prop:auted-cyclotomic}For $d\ge3$, the $\GT$-action on the groups
	$\pi_i(\SAut_{\le k}(\smash{E_{d,\bfQ}}))$
is gr-cyclotomic of nontrivial weight for all $i\ge0$ and $2\le k<\infty$.
\end{prop}

\begin{proof}During the proof, we will repeatedly use the fact that any automorphism in $\SAut_{\le k}(E_{d,\bfQ})=\ker(\chi_{\le k}\colon \Aut_{\le k}(E_{d,\bfQ})\ra\GL_1(\bfQ))$ for $k\ge2$ acts as the identity on all homotopy and homology groups of $E_{d,\bfQ}(\ell)$ for all $\ell\le k$. This follows from the proof of \cref{lem:weight-lemma} \ref{enum:weight-i} and \ref{enum:weight-ii}.

By the long exact sequence of homotopy groups and \cref{lem:cyclotomic-inherits}, it suffices to show that $\GT$ acts on the homotopy groups of $\SAut_{\le 2}(\smash{E_{d,\bfQ}})$ and of $\smash{L^{\times}_k(E_{d,\bfQ})\coloneq \fib_{\id}(\Aut_{\le k}(\smash{E_{d,\bfQ}})\ra  \Aut_{\le k-1}(\smash{E_{d,\bfQ}}))}$ for $k\ge3$, gr-cyclotomically of nontrivial weight. 

We begin with $L^{\times}_k(E_{d,\bfQ})$. Taking group-like components of the fibre sequence \eqref{eqref:fibre-sequence-layers} in the case $\cP=\cO=\smash{E_{d,\bfQ}}$ and looping once gives a fibre sequence of $E_1$-groups with $\GT$-action
\[
	\Omega \big(\Map_{\Latch_k(\smash{E_{d,\bfQ}})}\big(\smash{E_{d,\bfQ}}(k),\Match_k(\smash{E_{d,\bfQ}})\big)^{\Sigma_k}\big) \ra L^{\id}_k(E_{d,\bfQ})^\times \ra \Aut_{\Latch_k(\smash{E_{d,\bfQ}})}(\smash{E_{d,\bfQ}}(k))^{\Sigma_k},
\]
Since any automorphism in $\Aut_{\le k}(\smash{E_{d,\bfQ}})$ that lifts to the fibre $L^{\id}_k(E_{d,\bfQ})^\times$ lies in the kernel of $\chi_{\le k}$, the second map lands in the components $\smash{\SAut_{\Latch_k(\smash{E_{d,\bfQ}})}(\smash{E_{d,\bfQ}}(k))^{\Sigma_k}\le \Aut_{\Latch_k(\smash{E_{d,\bfQ}})}(\smash{E_{d,\bfQ}}(k))^{\Sigma_k}}$. By the long exact sequence in homotopy groups, it thus suffices to show that $\GT$ acts on \begin{equation}\label{equ:induction-groups}\hspace{-0.2cm}\smash{
	\pi_{i}\big(\Map_{\Latch_k(\smash{E_{d,\bfQ}})}(\smash{E_{d,\bfQ}}(k),\Match_k(\smash{E_{d,\bfQ}}))^{\Sigma_k}\big)\ \ \text{and}\ \  \pi_i\big(\BSAut_{\Latch_k(\smash{E_{d,\bfQ}})}(\smash{E_{d,\bfQ}}(k))^{\Sigma_k}\big)\ \text{for }i\ge1}
\end{equation}
gr-cyclotomically of nontrivial weight. We first show this for the homotopy groups of the corresponding spaces without the $\Sigma_k$-superscripts. For the first group, we apply the Bousfield--Kan spectral sequence as in \cref{ex:BKSS-mapping-space}. We have seen in the proof of \cref{lem:weight-lemma}  \ref{enum:weight-ii} that $E_{d,\bfQ}(k)\ra\Match_k(E_{d,\bfQ})$ is surjective on homotopy groups, so $\Match_k(E_{d,\bfQ})$ is $1$-connected. Moreover, since $\Latch_k(\smash{E_{d,\bfQ}})\ra \smash{E_{d,\bfQ}}(k)$ is $\bfQ$-homology equivalent to the boundary inclusion of a compact manifold (see the proof of \cref{lem:weight-lemma} \ref{enum:weight-iv}), the homology of the cofibre vanishes in large degrees. Hence, by the discussion in \cref{ex:BKSS-mapping-space} and \cref{rem:stronger-convergence}, the group $\pi_{i}(\Map_{\Latch_k(\smash{E_{d,\bfQ}})}(\smash{E_{d,\bfQ}}(k),\Match_k(\smash{E_{d,\bfQ}})))$ admits for $i\ge1$ a $\GT$-invariant  normal series whose quotients are subquotients of the groups  $\oH^{s-i}(\smash{E_{d,\bfQ}}(k)/\Latch_k(\smash{E_{d,\bfQ}});\pi_{s}(\Match_k(\smash{E_{d,\bfQ}})))$, so the claim follows from the fact that $\GT$ acts gr-cyclotomically with nontrivial weight on these groups for $k\ge3$ by \cref{lem:weight-lemma} \ref{enum:weight-v}. Similarly, to deal with the second group in \eqref{equ:induction-groups} (without the $\Sigma_k$-superscript) we apply the Bousfield--Kan spectral sequence as in \cref{ex:BKSS-automorphism-space}. Arguing as before, this shows the claim as a consequence of the fact that $\GT$ acts for $k\ge3$ on $\oH^{s-i+1}(\smash{E_{d,\bfQ}}(k)/\Latch_k(\smash{E_{d,\bfQ}});\pi_{s}(\smash{E_{d,\bfQ}}(k)))$ gr-cyclotomically with nontrivial weight by \cref{lem:weight-lemma} \ref{enum:weight-v} (the issue pointed out in the final sentence in \cref{ex:BKSS-automorphism-space} causes no trouble, since gr-cyclotomic actions with nontrivial weight are closed under taking subgroups). Finally, to show that $\GT$ acts gr-cyclotomically with nontrivial weights on the groups \eqref{equ:induction-groups} \emph{with} the $\Sigma_k$-superscript, we apply the Bousfield--Kan spectral sequence as in \cref{ex:homotopy-fixed-points} to see that these groups admit finite $\GT$-invariant normal series whose quotients are subquotients of the groups
\[\begin{aligned}\oH^{s-i}\big(\mathrm{B}\Sigma_k;\pi_{s}(\Map_{\Latch_k(\smash{E_{d,\bfQ}})}(\smash{E_{d,\bfQ}}(k),\Match_k(\smash{E_{d,\bfQ}}))\big)&\ \ \ \text{for }s\ge i\ge 1\\
\oH^{s-i}\big(\mathrm{B}\Sigma_k;\pi_s\big(\BSAut_{\Latch_k(\smash{E_{d,\bfQ}})}(\smash{E_{d,\bfQ}}(k))\big)\big)&\ \ \ \text{for }s\ge i\ge 1.\end{aligned}.\]
Consequently, by \cref{lem:cyclotomic-inherits} \ref{enum:cyclotomic-inherits-i} it suffices to show that both of these groups are gr-cyclotomically of nontrivial weight. Since we already showed this for the coefficient-groups, the claim follows from \cref{lem:cyclotomic-inherits} \ref{enum:cyclotomic-inherits-ii} (in fact, almost all of these $\Sigma_k$-cohomology groups are even trivial, since the coefficients are rational vector spaces for $s\ge2$).

It remains to show that the $\GT$-action on the homotopy groups of $\SAut_{\le 2}(\smash{E_{d,\bfQ}})$ is gr-cyclotomic of nontrivial weight. Since $\Match_2(\smash{E_{d,\bfQ}})$ and $\Aut_{\le 1}(E_{d,\bfQ})$ are contractible and $\Latch_2(\smash{E_{d,\bfQ}})$ empty (see \cref{sec:latching-matching}), we have $\SAut_{\le 2}(\smash{E_{d,\bfQ}})\simeq  \SAut(\smash{E_{d,\bfQ}}(2))^{\Sigma_2}$ by \eqref{eqref:fibre-sequence-layers}. By the same argument as in the final step above, it suffices to show that the $\GT$-action on $\pi_i(\SAut(\smash{E_{d,\bfQ}}(2)))$ is cyclotomic of nontrivial weight. As $\smash{E_{d,\bfQ}}(2)$ is the rationalisation of a $(d-1)$-sphere, the morphism $\chi\colon \pi_0(\Aut(\smash{E_{d,\bfQ}}(2)))\ra \GL_1(\bfQ)$ is an isomorphism, so $\pi_i(\SAut(\smash{E_{d,\bfQ}}(2)))$ vanishes for $i=0$. For $i\ge1$, we have that $\pi_i(\SAut(\smash{E_{d,\bfQ}}(2)))=\pi_i(\Map(\smash{E_{d,\bfQ}}(2),\smash{E_{d,\bfQ}}(2)))$, so applying the Bousfield--Kan spectral sequence as in \cref{ex:BKSS-mapping-space}, the claim follows from the final part of \cref{lem:weight-lemma} \ref{enum:weight-v}.
\end{proof}

\begin{proof}[Proof of \cref{bigthm:nullhomotopy} and \cref{prop:trivial-invariants}] We already explained in \cref{sec:e2-strategy} how \cref{prop:trivial-invariants} implies \cref{bigthm:nullhomotopy}. To show the former, note that since $\lim_k\SAut_{\le k}(E_{d,\bfQ})\simeq \SAut(E_{d,\bfQ})$ and taking invariants commutes with limits, it suffices to show $\smash{\SAut_{\le k}(E_{d,\bfQ})^{\langle \Lambda\rangle}\simeq\ast}$ for $2\le k<\infty$. We have $\smash{\SAut_{\le k}(E_{d,\bfQ})^{\langle \Lambda\rangle}\simeq \Omega \big(\BSAut_{\le k}(E_{d,\bfQ})^{\langle \Lambda\rangle})}$, so it suffices to show that $\pi_i(\BSAut_{\le k}(\smash{E_{d,\bfQ}})^{\langle \Lambda\rangle})$ vanishes for $i\ge1$. From the Bousfield--Kan spectral sequence as in \cref{ex:homotopy-fixed-points}, we see that it suffices to show that $\oH^{*}(\mathrm{B}(\langle \Lambda\rangle);\pi_s(\BSAut_{\le k}(\smash{E_{d,\bfQ}})))$ vanishes for all $s$. As $\pi_s(\BSAut_{\le k}(\smash{E_{d,\bfQ}}))$ is gr-cyclotomic of nontrivial weight by \cref{prop:auted-cyclotomic}, it admits a $\GT$-invariant finite normal series whose quotients are cyclotomic of nontrivial weight, so an induction over the length of the series shows that it suffices to show that $\oH^{\ast}(B(\langle \Lambda\rangle);A)$ is trivial for any cyclotomic $\GT$-module $A$ of fixed weight $n\neq0$. Since  $B(\langle \Lambda\rangle)\simeq B\bfZ$, the only potentially nontrivial cohomology groups occur for $*=0,1$ where they are given by the (co)invariants of $\langle\Lambda\rangle$ acting on $A$. But by assumption, $A$ is a rational vector space and $\Lambda$ acts by multiplication with $\chi(\Lambda)^n\neq 1$, so the (co)invariants vanish.
\end{proof}

\begin{rem}\label{rem:graph-complex-relation}\ 
\begin{enumerate}[leftmargin=*]
\item The proof of \cref{bigthm:nullhomotopy} can likely be shortened by relying on the description of $\BAut(E_{d,\bfQ})$ in terms of a graph complex \cite{FresseWillwacher} to show that the $\GT$-action on the homotopy groups of $\SAut(E_{d,\bfQ})$ is cyclotomic of nontrivial weight (the weight should correspond to the loop order in the graph complex). Such an argument was suspected by Willwacher. We decided against pursuing this route when we discovered the more self-contained argument given above.
\item \cref{bigthm:nullhomotopy} in particular implies that the map $\BSO(d-2)\ra \BSAut(E_{d,\bfQ})$ induced by the $\SO(d)$-action on $ \BSAut(E_{d,\bfQ})$ is nullhomotopic, because it factors over $\BAut(E_{d-2,\bfQ}) \to \BAut(E_{d,\bfQ})$ as a result of \eqref{equ:additivity-direct-product} and \eqref{equ:filler-rational-bv}. This corollary of \cref{bigthm:nullhomotopy} ought to be closely related to work of Khoroshkin--Willwacher \cite{KhoroshkinWillwacher} which suggests that the map $\BSO(d)\ra \BSAut(E_{d,\bfQ})$ factors  for $d=2n$ through the Euler class $e\colon \BSO(2n)\ra K(2n,\bfQ)$ and for $d=2n+1$ through the Pontryagin class $p_n\colon \BSO(2n+1)\ra K(4n,\bfQ)$.
\end{enumerate}
\end{rem}

\subsubsection{A consequence of \cref{bigthm:nullhomotopy}}\label{equ:consequence-thm-g}For the later applications of  \cref{bigthm:nullhomotopy}, we will mostly be using a particular consequence of it. To state it, recall that---as a result of the stable homeomorphism theorem---the subspace $\STop(d)\subset \Top(d)$ of orientation-preserving homeomorphisms of $\bfR^d$ is exactly the path-component of the identity of $\Top(d)$. Consequently, the map $t\colon \BTop(d)\ra \BAut(E_d)$ from \cref{sec:dl-add} induces on universal covers a map $t\colon \BSTop(d)\ra \BAut^\id(E_d)$ where $\Aut^\id(E_d)\subset \Aut(E_d)$ is the component of the identity. This map turns out to be nullhomotopic if one precomposes it with the map $s\colon \BSTop(d-2)\ra \BSTop(d)$ induced by taking products with $\bfR^2$, and postcomposes it with the map $(-)_\bfQ\colon \BAut^\id(E_{d})\ra \BAut^\id(E_{d,\bfQ})$ induced by rationalisation:

\begin{cor}\label{cor:double-stab-top-null-extra}For $d\ge2$, the compostion
\[\smash{\BSTop(d-2)\xlra{s} \BSTop(d)\xlra{t} \BAut^\id(E_{d})\xlra{(-)_\bfQ}\BAut^\id(E_{d,\bfQ})}\] is nullhomotopic. 
\end{cor}
\begin{proof}
We consider the commutative diagram
\begin{equation}\label{equ:fcatorisation-square-extra}\begin{tikzcd}[column sep=0.6cm,row sep=0.5cm]
	\BSTop(d-2)\arrow[d,"{s=(-)\times\id_{\bfR^2}}",swap]\rar{t}&\dar{(-)\otimes \id_{E_2}} \BAut^{\id}(E_{d-2})\rar{(-)_\bfQ}&\BAut^{\id}(E_{d-2,\bfQ})\dar{(-)\otimes_\bfQ \id_{E_{2,\bfQ}}}\\[5pt]
	\BSTop(d)\rar{t}& \BAut^{\id}(E_{d})\rar{(-)_\bfQ}&\BAut^{\id}(E_{d,\bfQ})
\end{tikzcd}\end{equation}
whose upper squares are induced by \eqref{equ:additivity-direct-product} and \eqref{equ:filler-rational-bv} respectively. By taking universal covers, it follows from  \cref{bigthm:nullhomotopy}  that the  rightmost vertical map $\smash{(-)\otimes_\bfQ \id_{E_{2,\bfQ}}}$ is nullhomotopic, so it follows that the composition $\BSTop(d-2)\ra \BAut^\id(E_{d,\bfQ})$ is also nullhomotopic, as claimed.\end{proof}

\subsubsection{Alternative nontriviality of  $S^{\DiscInf}_\partial(D^d)$}
We briefly mention an application of \cref{bigthm:nullhomotopy} that is independent of the remainder of this work. There is an equivalence (see \cite[Remark 5.39]{KKoperadic}) \begin{equation}\label{equ:topaut-sdisc}
	\Omega^{d+1}\Aut(E_d)/\Top(d)\simeq S^{\DiscInf}_\partial(D^d)\quad\text{for }d\neq 4
\end{equation}
between the $(d+1)$st loop space of the fibre $\Aut(E_d)/\Top(d)$ of the map $t\colon \BTop(d) \to \BAut(E_d)$ from \cref{sec:dl-add} and the $\DiscInf$-structure space $S^{\DiscInf}_\partial(D^d)$ in the sense of \cite{KKDisc}. In Theorems E and 8.13 loc.cit.\ we proved that $\Aut(E_d)/\Top(d)$ and $\smash{S^{\DiscInf}_\partial(D^d)}$ are nontrivial for most values of $d$:
\begin{thm}\label{thm:sdiscnontrivial}We have $\Aut(E_d)/\Top(d)\not\simeq \ast$ for $d\ge3$, and $S^{\DiscInf}_\partial(D^{d'})\not\simeq\ast$ for $d'=3$ and $d'\ge5$.
\end{thm}
 The proof in \cite{KKDisc} relied on a computation of the homotopy groups of $\smash{\Aut(E_{d,\bfQ})}$ in terms of a graph complex due to Fresse--Turchin--Willwacher (see Theorem 8.3 loc.cit.). Using \cref{bigthm:nullhomotopy}, we can offer (for most values of $d$ and $d'$) a simpler proof which is independent of their work:
\begin{proof}[Alternative proof of \cref{thm:sdiscnontrivial} for $d\ge6$ and $d'\ge8$]
Consider the commutative diagram \eqref{equ:fcatorisation-square-extra} from the previous proof. Assume there is a class in $\pi_i(\BSTop(d-2))_\bfQ$ that maps nontrivially to $\pi_i(\BSTop(d))_\bfQ$. Since the rightmost column is nullhomotopic as a result of \cref{bigthm:nullhomotopy}, the class is trivial in $\pi_i(\BAut^\id(E_{d,\bfQ}))_ \bfQ$. The argument in \cite[Section 8.3.1]{KKDisc} shows that the class is either zero in $\pi_i(\BAut^\id(E_d))_\bfQ$ and thus $\pi_i(\Aut(E_{d,\bfQ})/\Top(d))_\bfQ\neq 0$, or $\pi_*(\Aut(E_{d,\bfQ})/\Top(d))$ is uncountable for $*=i-1$ or $*=i-2$. Thus, as soon as there is an element that is not in the kernel of  $\pi_i(\BSTop(d-2))_\bfQ\ra \pi_i(\BSTop(d))_\bfQ$, we have that  $\pi_*(\Aut(E_{d})/\Top(d))$ is nontrivial in at least one of the degrees $*=i,i-1,i-2$. For $d\ge6$, the composition $\BSO(4)\ra \BSTop(d-2)\ra \BSTop(d)\ra \BTop$ is nontrivial on rational homotopy groups, so there is indeed such an element and the claim regarding $\Aut(E_{d})/\Top(d)$ follows. To apply the same strategy to show that $S^{\DiscInf}_\partial(D^{d'})\simeq \Omega^{d'+1}\Aut(E_{d'})/\Top(d')$ is nontrivial, we need a class that is not in the kernel $\pi_i(\BSTop(d'-2))_\bfQ\ra \pi_i(\BSTop(d'))_\bfQ$ for some $i-2\ge d'+1$. The proof finishes by explaining why such classes exist as long as $d'-2\ge6$. For $d'-2\ge7$ this is classical: write $d'-2=2n$ or $d'-2=2n+1$  and use that the composition
\[\smash{\BSO(d'-2)\ra \BSTop(d'-2)\ra \BSTop(d')\xlra{p_n}K(\bfQ,4n)}\]
is known to be nontrivial on $\pi_{4n}(-)$. In the remaining case $d'-2=6$, the existence of a class as required follows from \cite[Corollary 6.7, Example 6.8]{KRWEvenDiscs}, or from our \cref{bigcor:homotopy}.
\end{proof}

\section{Boundary conditions in embedding calculus}
In this section, as a further ingredient to the proof of our main result \cref{bigthm:pullback}, we discuss how the homotopy type of the boundary of a manifold $M$ can often be recovered in a natural way from the interior of $M$ in the context of Goodwillie--Weiss' embedding calculus. We first fix some notation. Fix a compact $d$-dimensional manifold triad $(M,\partial^vM,\partial^hM)$, by which we mean a compact topological $d$-manifold $M$ together with a decomposition $\partial M=\partial^vM\cup \partial^hM$ of its boundary into two codimension zero submanifolds that intersect in their common boundary $\partial^{vh}M\coloneq\partial(\partial^vM)=\partial^vM\cap \partial^hM=\partial(\partial^hM)$  (the $v$ stands for \emph{vertical}, the $h$ for \emph{horizontal}; this choice of terminology will become clear later). We consider the commutative square of $E_1$-groups 
\[
\begin{tikzcd}[column sep=0.4cm,row sep=0.4cm]
	\Homeo_{\partial^v}(M)\dar\rar&\Emb_{\partial^v}(M)^\times\dar\\
	\Aut_{\partial^v}(M,\partial^h)\rar&\Aut_{\partial^v}(M) 
\end{tikzcd}\quad \mathrm{where}
\]
\begin{itemize}[leftmargin=*]
	\item $\Homeo_{\partial^v}(M)$ is the space of homeomorphisms of $M$ that fix a neighbourhood of $\partial^vM$ pointwise, viewed as an $E_1$-group by composition,
	\item $\Emb_{\partial^v}(M)^\times$ is the space of topological self-embeddings of $M$ that are the identity in a neighbourhood of $\partial^vM$ and are invertible up to isotopy fixed on $\partial^vM$, viewed as an $E_1$-group by composition (these are the group-like elements of the $E_1$-space $\Emb_{\partial^v}(M)$ of all self-embeddings),
	\item $\Aut_{\partial^v}(M,\partial^h)$ is the $E_1$-group of homotopy self-equivalences of $M$ that ``fix $\partial^vM$ pointwise and $\partial^hM$ setwise'', i.e.\,the automorphism group of the square induced by the inclusions
	\begin{equation}\label{equ:bdy}
	\begin{tikzcd}[column sep=0.4cm,row sep=0.4cm]
		\partial^{vh}M\rar\dar&\partial^hM\dar{}\\
		\partial^vM\rar& M
	\end{tikzcd}
	\end{equation}
	considered as an object in the undercategory $(\cS^{[1]})_{(\partial^{vh}M\ra \partial^vM)/}$ where $\cS$ is the category of spaces,
	\item $\Aut_{\partial^v}(M)$ is the $E_1$-group of self-equivalences of $M$ that  only ``fix $\partial^vM$ pointwise'', i.e.\,the automorphism group of $(\partial^vM\ra M)\in \cS_{\partial^vM/}$, and
	\item the maps in \eqref{equ:initial-forgetful-square-bdy} are the evident forgetful maps. For the left vertical map, this uses that homeomorphisms of a manifold always preserve the boundary setwise, by invariance of domain.
\end{itemize}

As a warmup to the main content of this section, we give a proof of the following statement, which informally says that a self-embedding that fixes $\partial^vM$ pointwise also preserves $\partial^hM$ setwise ``in a homotopical sense'', even though it does not need to actually preserve it in a point-set sense, unlike for homeomorphisms. 

\begin{prop}\label{prop:embedding-fix-bdy-setwise}For a compact $d$-dimensional manifold triad $(M,\partial^vM,\partial^hM)$, there exists a dashed filler in the diagram of $E_1$-groups
	\begin{equation}\label{equ:initial-forgetful-square-bdy}
		\begin{tikzcd}[column sep=0.4cm,row sep=0.4cm]
			\Homeo_{\partial^v}(M)\dar\rar&\Emb_{\partial^v}(M)^\times\dar\arrow[dl,dashed]\\
			\Aut_{\partial^v}(M,\partial^h)\rar&\Aut_{\partial^v}(M).
		\end{tikzcd}\end{equation}
\end{prop}

The main point of this section will then be to show that after replacing $\partial^hM$ and $\partial^{vh}M$ with certain approximations resulting from variants of Goodwillie--Weiss' embedding calculus \cite{WeissEmbeddings}, the dashed map in \eqref{equ:initial-forgetful-square-bdy} admits a further factorisation 
\[
	\Emb_{\partial^v}(M)^\times\lra T_\infty\Emb_{\partial^v}(M)^\times \lra T_\infty\Emb^{p}_{\partial^v}(M)^\times \longdashrightarrow  \Aut_{\partial^v}(M,\partial^h).
\] 
Here the first two maps are given by the topological embedding calculus approximation followed by the map to particle embedding calculus approximation, both in the sense of \cite{KKoperadic}. (we omit the $t$-superscripts from the notation in loc.cit.\,since we will only consider topological embeddings and topological embedding calculus in this work). We will see that the approximations to $\partial^hM$ and $\partial^{vh} M$ recover the homotopy types of $\partial^hM$ and $\partial^{vh} M$ if all of the following conditions are satisfied: (i) the triad $(M,\partial^vM,\partial^hM)$ is smoothable, (ii) $d\ge5$, (iii) the inclusion $\partial^h M\subset M$ is an equivalence on tangential $2$-types (i.e.\,it induces an equivalence on fundamental groupoids and for all components the  Stiefel--Whitney class $w_2\colon\pi_2(M)\ra\bfZ/2$ is nontrivial if and only if $w_2\colon\pi_2(\partial^hM)\ra\bfZ/2$ is), and (iv) the same holds for the inclusion $\partial^{vh} M\subset \partial^vM$. This will then lead to a proof of:

\begin{thm}\label{thm:fake-boundary-2type}Let $(M,\partial^vM,\partial^hM)$ be a compact smoothable $d$-manifold triad with $d\ge5$. If $\partial^h M\subset M$ and $\partial^{vh} M\subset \partial^vM$ are equivalences on tangential $2$-types, then there exists a dashed filler in the diagram
\[\begin{tikzcd}[column sep=0.3cm,row sep=0.6cm]
	\Homeo_{\partial^v}(M)\dar\rar&\Emb_{\partial^v}(M)^\times\rar&T_\infty\Emb_{\partial^v}(M)^\times\rar&T_\infty\Emb^{p}_{\partial^v}(M)^\times\dar\arrow[dlll,dashed]\\
	\Aut_{\partial^v}(M,\partial^h)\arrow[rrr]&&&\Aut_{\partial^v}(M)
\end{tikzcd}\]
of $E_1$-groups.
\end{thm}

\begin{rem}\label{rem:extends-Weiss-construction}
\,
\begin{enumerate}[leftmargin=*]
\item In the smooth setting and under slightly more restrictive assumptions on the triad, a variant of this result was proved by Weiss as the main result of \cite{WeissConAut}. The \emph{statement} of \cref{thm:fake-boundary-2type} is strongly inspired by Weiss' work, but the \emph{proofs} are considerably different. The main results of \cite{TillmannWeiss} can be recovered and slightly extended by a similar approach. One advantage of our construction is that it has better naturality properties. We believe that (extensions of) this construction will see further applications in the future (see also \cref{rem:presheaf-boundaries}). 
\item Under the hypothesis of \cref{thm:fake-boundary-2type} on $\partial^hM\subset M$, the map $\Emb_{\partial^v}(M)\ra T_\infty\Emb_{\partial^v}(M)$ is an equivalence by convergence of topological embedding calculus \cite[Theorem 6.4]{KKoperadic}, so \eqref{equ:initial-forgetful-square-bdy} already gives a dashed filler with domain $T_\infty\Emb_{\partial^v}(M)^\times$. The nontrivial part of \cref{thm:fake-boundary-2type} is that it can be extended to $T_\infty\Emb^{p}_{\partial^v}(M)^\times$ under the additional assumption on $\partial^{vh}M\subset \partial^{v}M$.
\end{enumerate}
\end{rem}

\subsection{Embeddings}\label{sec:emb-boundary-square} In this section we prove \cref{prop:embedding-fix-bdy-setwise}, i.e.\,construct the dashed filler in \eqref{equ:initial-forgetful-square-bdy}. To do so, instead of considering compact manifold triads $(M,\partial^vM,\partial^hM)$ as above, we work in the slightly more general setting of compact bordisms $W\colon O\leadsto Q$ of manifolds with boundary, i.e.\,compact manifold triads $(W,\partial^vW, \partial^hW)$ with an ordered decomposition $\partial^vW=O\sqcup Q$ of the vertical boundary as a disjoint union of submanifolds; see \cref{fig:bordism-notation}. We view manifold triads as bordisms of the form $\partial^vM\leadsto \varnothing$. It will be convenient to fix a closed collar $\partial W\times I\subset W$ of $\partial W=\partial W\times\{0\}$ in $W$ and replace $\Homeo_{\partial^v}(W)$ by the equivalent $E_1$-group of homeomorphisms $\phi$ with $\phi|_{\partial^vW\times I}=\id$ and $\phi|_{\partial^hW\times I}=\phi|_{\partial^hW}\times\id_I$. The strategy for obtaining the filler in \eqref{equ:initial-forgetful-square-bdy} is to replace the square \eqref{equ:bdy} of inclusions $\Homeo_{\partial^v}(M)$-equivariantly by an equivalent square of $\Homeo_{\partial^v}(M)$-spaces whose $\Homeo_{\partial^v}(M)$-action visibly factors over $\Emb_{\partial^v}(M)^\times$, in a way that is compatible with the a priori action of $\Emb_{\partial^v}(M)^\times$ on the bottom row (which factors through $\Aut_{\partial^v}(M)$). In doing so, we will repeatedly use the Kister--Mazur theorem \cite{Kister} which says that the inclusion $\Emb(\bfR^d)\subset \Top(d)=\Homeo(\bfR^d)$ is an equivalence. In particular, the $E_1$-space $\Emb(\bfR^d)$ of self-embeddings of $\bfR^d$ under composition is an $E_1$-group. 

\begin{figure}
	\begin{tikzpicture}[scale=1.3]
		\draw[fill=black!5!white] (0,0) rectangle (3,2);
		\node at (0,1) [left,darkgreen] {$O$};
		\fill[pattern=north west lines, pattern color=darkgreen!40!white] (0,0) -- (.3,.3) -- (.3,1.7) -- (0,2);
		\node at (3,1) [right,darkred] {$Q$};
		\fill[pattern=north west lines, pattern color=black!40!white] (0,0) -- (.3,.3) -- (2.7,.3) -- (3,0);
		\fill[pattern=north west lines, pattern color=black!40!white] (0,2) -- (.3,1.7) -- (2.7,1.7) -- (3,2);
		\fill[pattern=north west lines, pattern color=darkred!40!white] (3,0) -- (2.7,.3) -- (2.7,1.7) -- (3,2);
		\draw[very thick] (0,0) rectangle (3,2);
		\draw[very thick,darkgreen] (0,0) -- (0,2);
		\draw[very thick,darkred] (3,0) -- (3,2);
		\node at (.7,2.7) {$\partial^h W$};
		\draw [->,darkgray,dashed] (1.1,2.68) to[in=90,out=0] (1.5,2.1);
		\draw [->,darkgray,dashed] (.7,2.4) to[in=90,out=270] (1.5,.1);
		\node at (2.3,-.7) {$\partial^v W$};
		\draw [->,darkgray,dashed] (1.8,-.7) to[in=0,out=180] (.1,1);
		\draw [->,darkgray,dashed] (2.3,-.4) to[in=180,out=90] (2.9,1);
		\node at (1.5,1) {$W$};
	\end{tikzpicture}
	\caption{A bordism $W \colon O \leadsto Q$ of manifolds with boundary. A collar is dashed. }
	\label{fig:bordism-notation}
\end{figure}

\subsubsection{Recovering horizontal boundaries from embeddings}\label{sec:recover-hor-bdy-emb}
The first observation leading to the dashed filler in \eqref{equ:initial-forgetful-square-bdy} is that the $\Emb(W)^\times$-equivariant homotopy type of $W$ can be recovered as the orbits $\smash{\Emb(\bfR^d,W)_{\Emb(\bfR^d)}}$ of the space of topological embeddings $\bfR^d\hookrightarrow W$ with respect to the $\Emb(\bfR^d)$-action by precomposition. This is a consequence of the Kister--Mazur theorem (cf.\,\cite[Proof of Proposition 5.10]{KKoperadic}). Here $\Emb(W)$ acts on the orbits by postcomposition. More precisely, we have equivalences $\smash{\Emb(\bfR^d,W)_{\Emb(\bfR^d)} \simeq \smash{\Map_{\cS}(\bfR^d,W)_{\Aut_{\cS}(\bfR^d)}} \simeq  W}$ where the first equivalence is induced by the forgetful map and the second by the inclusion $W \to \smash{\Map_{\cS}(\bfR^d,W)}$ as the constant maps. The second observation is that there is a homotopy fibre sequence of the form
\begin{equation}\label{equ:first-isot-ext-sequ}\begin{tikzcd}
	\Emb(\bfR^d,\partial^h W\times I)\lra \Emb_{\partial^v}\big(W\sqcup \bfR^d,W\big)\lra \Emb_{\partial^v}(W),
\end{tikzcd}\end{equation}
where the left-hand map is induced by taking disjoint union with $W$, followed by postcomposition with an embedding $\partial^h W\times I\sqcup W\hookrightarrow W$ which ``pushes in the collar''. The right-hand map is given by restriction, and the homotopy fibre is taken at $\id_W$. That this homotopy fibre is indeed $\Emb(\bfR^d,\partial^h W\times I)$ follows from the isotopy extension theorem: to apply it, we replace the manifold $W$ up to homeomorphism with $W \cup_{\partial^h} (\partial^h W \times I)$ and use that $W\subset W \cup_{\partial^h} (\partial^h W \times I)$ is an isotopy equivalence relative to $\partial^vM$ to identify the right-hand map in \eqref{equ:first-isot-ext-sequ} up to equivalence with the map $\Emb_{\partial^v}(W \sqcup \bfR^d,W \cup_{\partial^h} (\partial^h W \times I)) \to \Emb_{\partial^v}(W,W \cup_{\partial^h} (\partial^h W \times I))$ induced by restriction along $W\subset W \sqcup \bfR^d$. This restriction map is a fibration by the isotopy extension theorem, and its strict fibre inclusion is $\Emb(\bfR^d,\partial^hW\times I)$. 

The entire homotopy fibre sequence \eqref{equ:first-isot-ext-sequ} is equivariant with respect to a $\Homeo_{\partial^v}(W)$-action, by conjugation on the middle term and the right term, and by postcomposition with the restriction to $\partial^h\times I$ on the left term (which is preserved by the convention on $\Homeo_{\partial^v}(W)$ from the beginning of the subsection). Combining these observations, we see that the $\Homeo_{\partial^v}(W)$-equivariant homotopy type of $\partial^hW\subset W$ can be described as the following map induced by restriction
\begin{equation}\label{equ:recover-boundary-embeddings}\hspace{-1.2cm}
\begin{tikzcd}
	\partial^h_EW&[-1cm]\coloneq&[-1cm]
	\fib_{\id_W}\big(\Emb_{\partial^v}(W\sqcup \bfR^d,W)\ra \Emb_{\partial^v}(W)\big)_{\Emb(\bfR^d)}\dar\\[-0.3cm]
	{}_EW&\coloneq& \Emb_{\partial^v}(\bfR^d,W)_{\Emb(\bfR^d)}.
 \end{tikzcd}
\end{equation}
Note that, written like this, the action on $\partial^h W\subset W$ in the arrow category $\cS^{[1]}$ of the $\infty$-category of spaces visibly factors through $\Homeo_{\partial^v}(W)\ra \Emb_{\partial^v}(W)^\times$, in a way that is compatible with the a priori action of $\Emb_{\partial^v}(W)^\times$ on ${}_EW\simeq W$. Moreover, this construction is natural with respect to gluing bordisms: if $W$ arises a composition $ W=U\cup V$ of two bordisms $U\colon O\leadsto P$ and  $V\colon P\leadsto Q$, then we have two equivalent squares in $\cS$ 
\begin{equation}\label{equ:embedding-fake-squares}
	\left(\begin{tikzcd}[row sep=0.25cm, column sep=0.4cm]
	\partial_E^hU\dar\rar &\partial_E^hW\dar\\
	{}_EU\rar&{}_EW
	\end{tikzcd}\right)\hspace{1cm}\simeq\hspace{1cm}\left( 
	\begin{tikzcd}[row sep=0.4cm, column sep=0.4cm]
	\partial^hU\arrow[d]\arrow[r] &{\partial^hW}\arrow[d]\\
	U\arrow[r]&W
	\end{tikzcd}\right)
\end{equation}
where the left square induced by extension of embeddings along $U \subset W$ and the right square by inclusion. The equivalence between these two  squares is equivariant with respect to the $E_1$-group $\Homeo_{\partial^v}(U)\times \Homeo_{\partial^v}(V)\ra \Homeo_{\partial^v}(W)$ in $\cS^{[1]}$, where the action on the left column is through the projection to $\Homeo_{\partial^v}(U)$. It is then clear that the action on the left square factors through the forgetful map  from $\Homeo_{\partial^v}(U)\times \Homeo_{\partial^v}(V)\ra \Homeo_{\partial^v}(W)$ to $\Emb_{\partial^v}(U)^\times \times \Emb_{\partial^v}(V)^\times\ra \Emb_{\partial^v}(W)^\times$, compatibly with the a priori action of the latter on the bottom row.

\subsubsection{Recovering vh-boundaries from embeddings}\label{sec:recover-triad-emb}
To construct the filler in \eqref{equ:initial-forgetful-square-bdy} from this and thus prove \cref{prop:embedding-fix-bdy-setwise}, we consider the situation of compact triads $(M,\partial^vM,\partial^hM)$, so $W=M$, $O=\partial^v M$, $Q=\varnothing$. We replace $M$ up to homeomorphism relative to $\partial^vM$ with the composition $M^c=\partial^vM\times I\cup_{\partial^v}M$ of the trivial bordism $\partial^vM\times I \colon \partial^vM\leadsto \partial^vM$ with the nullbordism $M\colon \partial^vM \leadsto \varnothing$. Since $\partial^h(\partial^vM\times I)=\partial^{vh}M\times I\simeq \partial^{vh}M$ (see \cref{fig:bordism-triad}), the equivalence \eqref{equ:embedding-fake-squares} specialises to
\begin{equation}\label{equ:embedding-fake-squares-triad-situation}
	\left(\begin{tikzcd}[row sep=0.3cm, column sep=0.4cm]
	\partial_E^{vh}M\coloneq \partial_E^h(\partial^vM\times I)\dar\rar &\partial_E^hM\dar\\
	\partial_E^{v}M\coloneq {}_E(\partial^vM\times I)\rar&{}_EM
	\end{tikzcd}\right)\hspace{1cm}\simeq\hspace{1cm} \left(
	\begin{tikzcd}[row sep=0.4cm, column sep=0.4cm]
	\partial^{vh}M\arrow[d]\arrow[r] &{\partial^hM}\arrow[d]\\
	\partial^vM\arrow[r]&M
	\end{tikzcd}\right)
\end{equation}
of squares. Moreover, the action by $\Homeo_{\partial^v}(M^c)\simeq \Homeo_{\partial^v}(M) $ on the right column of the right-square extends to an $\Homeo_{\partial^v}(M^c)$-action on the square considered as an object in the undercategory $\smash{\cS^{[1]}_{(\partial^{vh}M\ra \partial^vM)/}}$, since the restriction ${\{\id\}\times \Homeo_{\partial^v}(M)\ra \Homeo_{\partial^v}(M^c)}$ of the gluing morphism $\smash{\Homeo_{\partial^v\times\{0,1\}}(\partial^vM\times I)\times \Homeo_{\partial^v}(M)\ra \Homeo_{\partial^v}(M^c)}$ is an equivalence. The same holds compatibly for the $\Emb_{\partial^v}(M)^\times$-action on the left square, in a way that is compatible with the a priori action on the bottom row. The equivariant equivalence \eqref{equ:embedding-fake-squares-triad-situation} thus yields the filler in \eqref{equ:initial-forgetful-square-bdy}.

\begin{figure}
	\begin{tikzpicture}[scale=1.2]
		\clip (-6,-.5) rectangle (4.5,2.5);
		\draw[fill=darkgreen!5!white] (-2,0) rectangle (0,2);
		\fill[darkred!5!white] (0,0) rectangle (1.5,2);
		\fill[darkred!5!white] (1.5,1) circle (1cm);
		\draw[thick,darkgreen] (-2,0) -- (-2,2);
		\draw[thick,darkred] (0,0) -- (0,2);
		\draw (0,0) -- (1.5,0) arc(90:270:-1) (1.5,2) -- (0,2);
		\node at (-2.5,1) [left] {$\partial^h(\partial^v M \times I) = \partial^{vh} M \times I$};
		\draw [->,darkgray,dashed] (-2.9,1.4) to[in=135,out=60] (-1,2.1);
		\draw [->,darkgray,dashed] (-2.9,0.6) to[in=225,out=-60] (-1,-.1);
		\node at (2.5,1) [right] {$\partial^hM$};
		\node at (1.25,1) {$M$};
		\node at (-1,1) {$\partial^v M \times I$};
	\end{tikzpicture}
	\caption{The bordism $M^c \colon \partial^v M \leadsto \varnothing$ is obtained by composing the left-hand bordism $\partial^v M \times I \colon \partial^v M \leadsto \partial^v M$ with the right-hand bordism $M \colon \partial^v M \leadsto \varnothing$.}
	\label{fig:bordism-triad}
\end{figure}

\subsection{A generalisation to double categories}\label{sec:dbl-cat-boundary-square}
We now explain an abstraction of the previous constructions to the situation of a symmetric monoidal double category $\cM$ satisfying two conditions. This formulation will include the previous constructions as a special case, and will be general enough to also carry out analogous constructions for several variants of embedding calculus (such as the particle version appearing in \cref{thm:fake-boundary-2type}).

We refer to \cite[Section 1.1]{KKoperadic} for a recollection on double categories and an explanation of the notation and terminology we use. To state the conditions on the symmetric monoidal double category $\cM$ that we impose, note that the unit of the symmetric monoidal structure yields an object $\varnothing\in \cM_{[0]}$ in the category of objects of $\cM$, and an object $\varnothing_m\in 
\cM_{\varnothing,\varnothing}$ in the mapping category from $\varnothing$ to itself. The two conditions on $\cM$ are:
\begin{enumerate}[label=\textbf{C\arabic*}]
	\item\label{enum:condition-i-doublecat} the mapping category $\cM_{A,B}$ admits an initial object $\init_{A,B}$ for all $A,B\in\cM_{[0]}$, and
	\item\label{enum:condition-ii-doublecat} the object $\varnothing_m\in \cM_{\varnothing,\varnothing}$ is an initial object of $\cM_{\varnothing,\varnothing}$, so $\varnothing_m\simeq \init_{\varnothing,\varnothing}$.
\end{enumerate}
In particular, for objects $A,B\in\cM_{[0]}$, we have functors
\[\begin{aligned}(-)_\varnothing&\coloneq \big(\cM_{A,B}\xra{(-)\cup \init_{B,\varnothing}}\cM_{A,\varnothing}\big)\\ {}_\varnothing(-)_\varnothing&\coloneq\big( \cM_{A,B}\xra{ \init_{\varnothing,A}\cup(-)\cup \init_{B,\varnothing}}\cM_{A,\varnothing}\big)\\  {}_\varnothing(-)&\coloneq \big(\cM_{A,B}\xra{\init_{\varnothing,A}\cup(-)}\cM_{\varnothing,B}\big)\end{aligned}\]
where $\cup$ denotes the composition functors in $\cM$ (the notation is motivated by \cref{ex:manifold-example-of-doubl-cat} below).

To explain the analogue of \eqref{equ:recover-boundary-embeddings}, we fix an object $R\in \cM_{\varnothing,\varnothing}$ that will play the role of $\bfR^d$, and objects $O,Q\in \cM_{[0]}$ as well as $W\in \cM_{O,Q}$ that will play the role of the same-named manifolds and the bordism between them. Abbreviating $\Aut(R)\coloneq \Aut_{ \cM_{\varnothing,\varnothing}}(R,R)$, the analogue of \eqref{equ:recover-boundary-embeddings} is
\begin{equation}\label{equ:recover-boundary-double-cat}\hspace{-1.2cm}
\begin{tikzcd}
	\partial_{\cM}^hW&[-1cm]\coloneq&[-1cm]
	\fib_{\id_W}\big(\Map_{\cM_{O,Q}}(W\sqcup R,W)\ra \Map_{\cM_{O,Q}}(W,W)\big)_{\Aut(R)}\dar\\[-0.3cm]
	{}_{\cM}{W}&\coloneq& \Map_{\cM_{\varnothing,\varnothing}}(R,
	 {}_{\varnothing}W_{\varnothing})_{\Aut(R)},
\end{tikzcd}
\end{equation}
where $\sqcup$ denotes the symmetric monoidal structure. The horizontal map we take fibres of is induced by precomposition with $\id_W\sqcup (\varnothing_m\ra R)$. The vertical map is induced by the composition of the inclusion of the fibre inclusion into $\smash{\Map_{\cM_{O,Q}}(W\sqcup R,W)}$ followed by\vspace{-0.1cm}\begin{equation}\label{equ:auxiliary-comp-double-cat}
	\Map_{\cM_{O,Q}}(W\sqcup R,W)\xra{{}_{\varnothing}
 	(-)_\varnothing}\Map_{\cM_{\varnothing,\varnothing}}( {}_{\varnothing}W_{\varnothing}\sqcup R, {}_{\varnothing}W_{\varnothing})\xra{((\varnothing_m\ra {}_{\varnothing}W_{\varnothing})\sqcup\id_R)^*}\Map_{\cM_{\varnothing,\varnothing}}(R, {}_{\varnothing}W_{\varnothing})
\end{equation} 
where the superscript $(-)^*$ indicates precomposition. Conjugation on the source and postcomposition on the target induces an $\Aut_{\cM_{O,Q}}(W)$-action on the map \eqref{equ:recover-boundary-double-cat}. If the bordism $W$ arises as a composition $W = U \cup V$ of $U\in\cM_{O,P}$ and $V\in\cM_{P,Q}$, then there is a commutative square of the same form as the left square in \eqref{equ:embedding-fake-squares}
\begin{equation}\label{equ:embedding-fake-squares-double-cat}
	\begin{tikzcd}[row sep=0.25cm, column sep=0.4cm]
	\partial_{\cM}^hU\dar\rar &\partial_{\cM}^hW\dar\\
	{}_{\cM}U\rar&{}_{\cM}W,
	\end{tikzcd}
\end{equation}
 together with an action of the gluing morphism $\Aut_{\cM_{O,P}}(U)\times\Aut_{\cM_{P,Q}}(V)\ra \Aut_{\cM_{O,Q}}(W)$ considered as a $E_1$-group in $\smash{\cS^{[1]}}$. This square is induced by the following commutative diagram whose unlabeled arrows are instances of the composition \eqref{equ:auxiliary-comp-double-cat}
 \[\begin{tikzcd}[column sep=2cm]
	 \Map_{\cM_{O,P}}(U,U)\dar{(-)\cup V}&  \Map_{\cM_{P,Q}}(U\sqcup R,U)\dar{(-)\cup V}\arrow[l,swap,"{(\id_{U}\sqcup (\varnothing_m\ra R))^*}"]\rar &[-1cm]\Map_{\cM_{\varnothing,\varnothing}}(R, {}_{\varnothing}U_{\varnothing})\arrow[d,"{( {}_{\varnothing}U\cup(\init_{P,\varnothing}\ra V_{\varnothing}))_*}"]\\
 	\Map_{\cM_{O,Q}}(W,W) &\Map_{\cM_{\varnothing,\varnothing}}( {}_{\varnothing}W_{\varnothing}\sqcup R, {}_{\varnothing}W_{\varnothing})\lar{(\id_W\sqcup (\varnothing_m\ra R))^*} \rar&\Map_{\cM_{\varnothing,\varnothing}}(R, {}_{\varnothing}W_{\varnothing}).
\end{tikzcd}\]

By design, this construction is natural in functors $\varphi\colon \cM\ra \cN$ of symmetric monoidal double categories that preserve the initial objects $\init_{A,B}$ of the mapping categories. In particular, there is a map of squares from \eqref{equ:embedding-fake-squares-double-cat} to the same square for the composition $\varphi(W)=\varphi(U)\cup \varphi(V)$ in $\cN$. This map of squares comes with an action of the square of $E_1$-groups 
\[\begin{tikzcd}[row sep=0.6cm,column sep=0.7cm]
	\Aut_{\cM_{O,P}}(U)\times\Aut_{\cM_{P,Q}}(V)\rar{\cup}\dar{\varphi}& \Aut_{\cM_{O,Q}}(W)\dar{\varphi}\\
	\Aut_{\cN_{\varphi(O),\varphi(P)}}(\varphi(U))\times\Aut_{\cN_{\varphi(P),\varphi(Q)}}(\varphi(V))\rar{\cup}& \Aut_{\cN_{\varphi(O),\varphi(Q)}}(\varphi(W)).
\end{tikzcd}\]

\begin{ex}\label{ex:recover-triad-emb-double-cat}The analogue in the context of double categories of the special case discussed in \cref{sec:recover-triad-emb} is the following: given objects $\partial^vM\in \cM_{[0]}$ and $M\in \cM_{\partial^vM,\varnothing}$, we can view $M$ as a composition $M^c\coloneq \partial^vM\times I\cup M\simeq M$ where  $\partial^vM\times I\in \cM_{\partial^vM,\partial^vM}$ now stands for the unit with respect to the composition $\cup$ in $\cM$. By the same argument as in \cref{sec:recover-triad-emb} and setting 
\[
	\big(\partial^{vh}_\cM M\ra \partial^v_{\cM}M\big)\coloneq \big(\partial^h_\cM(\partial^vM\times I)\ra {}_\cM(\partial^vM\times I)\big),
\] 
the $\Aut_{\cM_{\partial^vM,\varnothing}}(M)$-action on the right-column of the square resulting from the construction, extends to an action in the overcategory $\smash{\cS^{[1]}_{/(\partial^{vh}_\cM M\ra \partial^v_{\cM}M)}}$ since $\smash{\Aut_{\cM_{\partial^vM,\varnothing}}(M)\simeq \Aut_{\cM_{\partial^vM,\varnothing}}(M^c)}$.\end{ex}

\smallskip

\begin{ex}\label{ex:manifold-example-of-doubl-cat}
An example of a symmetric monoidal double category $\cM$ satisfying \ref{enum:condition-i-doublecat} and \ref{enum:condition-ii-doublecat} is the topological noncompact bordism double category $\ncBordInf^t(d)$ introduced in \cite[Section 5.3]{KKoperadic}. Its category of objects $\ncBordInf^t(d)_{[0]}$ has as objects $(d-1)$-dimensional (potentially noncompact) topological manifolds without boundary and mapping spaces are embeddings between them. Its mapping categories have as objects bordisms and mapping spaces are spaces of embeddings between bordisms fixing the boundary. The composition functors are induced by gluing bordisms, the composition in the mapping categories is given by composing embeddings, and the symmetric monoidal structure is given by taking disjoint unions. The unit $\varnothing $ is given by the empty manifold, the object $\varnothing_m$ by the empty bordism, the initial objects in condition \ref{enum:condition-i-doublecat} are given by bordisms $(A\times[0,1)\sqcup (-1,0]\times B)\colon A\leadsto B$, and \ref{enum:condition-ii-doublecat} is satisfied since $\varnothing=\varnothing\times[0,1)\sqcup (-1,0]\times \varnothing$.

Given a compact bordism of manifolds with boundary $W\colon O\leadsto Q$ as in \cref{sec:emb-boundary-square}, we may cut out the horizontal part of the boundary to obtain a bordism $W^\circ\coloneq (W\backslash\partial^hW)\coloneq \interior(O)\leadsto \interior(Q)$ in $\ncBordInf^t(d)$. Note that the left-hand square in \eqref{equ:embedding-fake-squares} only depends on $W^\circ$ and its decomposition $W^\circ=U^\circ\cup_{\interior(P)}V^\circ$ since $U$, $V$, and $W$ are compatibly isotopy equivalent (relative to the interior of their vertical boundary) to the respective versions with a $(-)^\circ$ subscript. Moreover, going through the construction, one sees that the constructions from this section in the case $\cM=\ncBordInf^t(d)$ and $R=\bfR^d\coloneq \varnothing\leadsto \varnothing$ specialise to those in Sections~\ref{sec:recover-hor-bdy-emb} and \ref{sec:recover-triad-emb}.\end{ex}

\begin{ex}\label{ex:cospan-example}A second example of a symmetric monoidal double category that satisfies \ref{enum:condition-i-doublecat} and \ref{enum:condition-ii-doublecat} is the cospan double category $\Cosp(\cC)$ of a category $\cC$ admitting finite colimits (this will be relevant when dealing with the first stage of the embedding calculus tower). Its category of objects is $\cC$, the mapping category between objects $A$ and $B$ is the overcategory $\cC_{A\sqcup B/}$, the composition functor are given by taking pushouts, and the symmetric monoidal structure by taking coproducts (see \cite[Section 1.7.1]{KKoperadic} and \cite[Section 2.10]{KKDisc} for more on $\Cosp(\cC)$). The initial object $\init_{A,B}$ of  $\Cosp(\cC)_{A,B}\simeq \cC_{A\sqcup B/}$ is $\id_{A\sqcup B}$. In the case $\cM=\Cosp(\cC)$, we have \[\smash{\Map_{\cM_{O,Q}}(W\sqcup R,W)=\Map_{\cC_{/O\sqcup Q}}(W\sqcup R,W)\simeq \Map_{\cC}(R,W)\times \Map_{\cC_{/O\sqcup Q}}(W,W)},\] so it follows that the map $\smash{\partial^h_{\Cosp(\cC)}(W)\ra {}_{\Cosp(\cC)}W=\Map_{\cC}(R,W)_{\Aut_\cC(R)}}$ is always an equivalence.

If $\cC=\cS_{/B}$ is the category of spaces over a connected space $B$ and we choose for $R=(\ast\ra B)$ the inclusion of a point, then $\Aut(R)\simeq \Omega B$ and for any space $W\ra B$ over $B$, we get 
\[
	\partial^h_{\Cosp(\cC)}(W)\simeq {}_{\Cosp(\cC)}W=\Map_{\cC}(R,W)_{\Aut_\cC(R)}\simeq \fib(W\ra B)_{\Omega B}\simeq W.
\]
\end{ex}

\begin{ex}\label{ex:morita-cat-example} Generalising \cref{ex:cospan-example}, one may consider any symmetric monoidal category $\cC$ whose tensor product is compatible with geometric realisations in the sense of \cite[Section 1.6.4]{KKoperadic} and which is \emph{unital}, i.e.\,the monoidal unit $\varnothing$ is an initial object. Its Morita category $\ALG(\cC)$ in the sense of Section 1.7 loc.cit.\,is then a symmetric monoidal double category which satisfies  \ref{enum:condition-i-doublecat} and \ref{enum:condition-ii-doublecat}, since the free $(A,B)$-bimodule $F_{A,B}(\varnothing)\in \BMod_{A,B}(\cC)=\ALG(\cC)_{A,B}$ on the monoidal unit $\varnothing$ for associative algebras $A,B\in\Ass(\cC)=\ALG(\cC)_{[1]}$ is an initial object, using that $\varnothing$ is initial in $\cC$. Such Morita categories will appear for us when dealing with the higher stages of the embedding calculus tower. Given another symmetric monoidal category $\cD$ satisfying the above conditions, and a symmetric monoidal functor $\varphi\colon \cC\ra \cD$ compatible with geometric realisations in the sense of Section 1.6.4 loc.cit.\,the induced functor $\varphi\colon \ALG(\cC)\ra \ALG(\cD)$ preserves the initial objects in the mapping categories, since $\varphi(F_{A,B}(\varnothing))\simeq F_{\varphi(A),\varphi(B)}(\varnothing)$ (see Lemma 1.10 loc.cit. for an explanation).
\end{ex}

\subsection{Embedding calculus}\label{sec:bdy-emb-calc} In \cite[Section 5.3.2]{KKoperadic}, we introduced a version of Goodwillie--Weiss embedding calculus for spaces of topological embeddings. It arose as the special case of a construction for any unital operad $\cO$---namely the case when $\cO=E^t_d$ is the $\BTop(d)$-framed $E_d$-operad in the sense of Section 5.1.1 loc.cit.. It took the form of a tower of symmetric monoidal double categories under the bordism double category $\ncBordInf^t(d)$ from \cref{ex:manifold-example-of-doubl-cat}
\begin{equation}\label{equ:tower-bordism-truncated-morita}
	\smash{\ncBordInf^t(d)\xra{E} \ModInf^\un_\infty(E_d^t)\ra\cdots \ra \ModInf^\un_2(E_d^t)\ra \ModInf^\un_1(E_d^t)\simeq \Cosp(\cS_{/\BTop(d)})}
\end{equation}
which yields the topological embedding calculus tower upon taking mapping spaces in the mapping categories. The double categories in this sequence all satisfy conditions \ref{enum:condition-i-doublecat} and \ref{enum:condition-ii-doublecat} by Examples \ref{ex:manifold-example-of-doubl-cat} and \ref{ex:morita-cat-example}. This is because $\ModInf^\un_k(E_d^t)$ is, by definition, the Morita category of a unital symmetric monoidal category, namely the category $\smash{\PSh^\un(\DiscInf^t_{d,\le k})}$ of unital space-valued presheaves on a symmetric monoidal category $\smash{\DiscInf^t_{d,\le k}}$ of topological $d$-manifolds homeomorphic to $\sqcup^n\bfR^d$ for some $n\le k$ and topological embeddings between those, equipped with the Day convolution symmetric monoidal structure induced by disjoint union. Here unital means that the value of the presheaf at $\varnothing\in\DiscInf_{d,\le k}$ is contractible. Moreover, all functors in \eqref{equ:tower-bordism-truncated-morita} preserve initial objects in mapping categories: for all but the first functor this was explained in \cref{ex:morita-cat-example} since the functors are by definition induced by symmetric monoidal functors compatible with geometric realisations. For the first functor it follows from the argument in the proof of \cite[Lemma 5.7]{KKoperadic}. As a result, combining the discussion in \cref{ex:manifold-example-of-doubl-cat} with the naturality of the construction in \cref{sec:dbl-cat-boundary-square} and adopting the notation regarding towers of Section 1.2 loc.cit., given a compact $d$-dimensional bordism $W\colon O\leadsto Q$, we obtain a tower $\smash{(\partial_{\le \bullet}^hW\ra {}_{\le \bullet}{W})\in\Tow(\cS^{[1]}_{/\partial^hW\ra W})}$ in $\cS^{[1]}$ under $\partial^hW\ra W$ 
 \begin{equation}\label{equ:bdy-approx-tower}
\begin{tikzcd}[row sep=0.4cm]
	\partial^hW\arrow[d]&[-1cm]\simeq&[-1cm] \partial_E^hW\dar\rar&\partial_{\le \infty}^hW\dar\rar&\cdots\rar&\partial^h_{\le 2}W\rar\dar&\partial_{\le 1}^hW\dar{\simeq}\\
	W&\simeq& {}_E{W}\rar{\simeq}&{}_{\le \infty}{W}\rar{\simeq}&\cdots\rar{\simeq} &{}_{\le 2}{W}\rar{\simeq}&{}_{\le 1} {W}
\end{tikzcd}
\end{equation}
where the column involving the ${\le k}\text{-}$subscripts is the instance of \eqref{equ:recover-boundary-double-cat} for the image of $W^\circ\in\ncBordInf^t(d)_{\interior(O),\interior(Q)}$ in $\ModInf^\un_k(E_d^t)$. The tower \eqref{equ:bdy-approx-tower} is acted upon by the tower $\Emb_{\partial^v}(W)^\times\ra T_\bullet\Emb_{\partial^v}(W)^\times$ of $E_1$-groups under $\Emb_{\partial^v}(W)^\times$ consisting of the   automorphism groups of the images of $W^\circ$ in $\ModInf^\un_\bullet(E_d^t)$. The fact that the rightmost vertical arrow in \eqref{equ:bdy-approx-tower} is an equivalence was explained in \cref{ex:cospan-example} since $\ModInf^\un_1(E_d^t)\simeq \Cosp(\cS_{/\BTop(d)})$. Regarding the horizontal maps, we can deduce the following lemma from convergence of embedding calculus:

\begin{lem}\label{lem:convergent-boundary} Fix a compact $d$-dimensional bordism $W\colon O\leadsto Q$ of manifolds with boundary.
\begin{enumerate}[leftmargin=*]
	\item\label{item:convergence-lem-i} The map $W\ra {}_{\le k}W$ is an equivalence for all $1\le k\le \infty$.
	\item\label{item:convergence-bdy} If $d\ge5$, the bordism $W$ is smoothable, and $\partial^hW\subset W$ is $2$-connected, then the map $\smash{\partial^hW\ra \partial^h_{\le k}W}$ is $(k-d+3)$-connected for $1\le k\le\infty$. In particular, it is an equivalence if $k=\infty$.
	\item\label{item:convergence-bdy-better} For the final conclusion in \ref{item:convergence-bdy}, the $2$-connectivity assumption can be weakened to only assuming that the inclusion $\partial^hW\subset W$ is an equivalence on tangential $2$-types.
\end{enumerate}
\end{lem}

\begin{proof}Adopting the notation from \cite[Section 5.3.2]{KKoperadic}, the map $\Emb(S\times \bfR^d,N)\ra T_k\Emb(S\times \bfR^d,N)$ is an equivalence for any finite set $S$ of cardinality $\le k$, and all topological $d$-manifolds $N$, e.g.\ by combining Theorem 5.3 (iii) loc.cit.\,with the Yoneda lemma. This implies that the map 
\[
	W\simeq \smash{{}_EW=\Emb(\bfR^d,W)_{\Emb(\bfR^d)^\times}\lra T_k\Emb(\bfR^d,W)_{T_k\Emb(\bfR^d)^\times}={}_{\le k}W}
\] 
from Part \ref{item:convergence-lem-i} is an equivalence. Under the hypothesis of Part \ref{item:convergence-bdy}, the maps $\Emb_{\partial^v}(W\sqcup \bfR^d,W)\ra T_k\Emb_{\partial^v}(W\sqcup \bfR^d,W)$ and $\Emb_{\partial^v}(W,W)\ra T_k\Emb_{\partial^v}(W,W)$ are $(k-d+4)$-connected by Theorem 6.3 and Remark 6.11 loc.cit., so the map between fibres
\[
	\fib_{\id_W}\big(\Emb_{\partial^v}(W\sqcup \bfR^d,W)\ra \Emb_{\partial^v}(W,W)\big)\ra \fib_{\id_W}\big(T_k\Emb_{\partial^v}(W\sqcup \bfR^d,W)\ra T_k\Emb_{\partial^v}(W,W)\big)
\]
is $(k-d+3)$-connected. After taking orbits (which preserves connectivity) with respect to the action by $\Emb(\bfR^d)^\times\simeq T_k\Emb(\bfR^d)^\times$, this is precisely the map ${\partial^hW\simeq \partial^h_EW\ra \partial^h_{\le k}W}$ in the claim, so Part \ref{item:convergence-bdy} follows. The final Part \ref{item:convergence-bdy-better} follows from the same argument using Theorem 6.3 loc.cit..
\end{proof}

\subsection{Particle embedding calculus}
As explained in \cite[Section 5.7]{KKoperadic}, there is a map of towers of symmetric monoidal double categories $\ModInf^\un_\bullet(E_d^t)\ra \ModInf^\un_\bullet(E_{d,\le \bullet}^p)$ such that the composition $\smash{\ncBordInf^t(d)\ra \ModInf^\un_\bullet(E_d^t)\ra \ModInf^\un_\bullet(E_{d,\le \bullet}^p)}$ yields the particle embedding calculus tower in the sense of loc.cit.\,upon taking  mapping spaces in mapping categories. This induces a map of towers in $\smash{\cS^{[1]}_{\partial^hW\ra W/}}$ from $\smash{(\partial_{\le \bullet}^hW\ra {}_{\le \bullet}W)}$ in \eqref{equ:bdy-approx-tower} to the analogous tower $\smash{(\partial_{\le \bullet,p}^hW\ra {}_{\le \bullet,p}W)}$ obtained by replacing $\smash{\ModInf^\un_\bullet(E_d^t)}$ with $\smash{\ModInf^\un_\bullet(E_{d,\le \bullet}^p)}$. This map turns out to consist of equivalences:

\begin{lem}\label{lem:embcalc-and-particle-bdy-agree}
The maps $\partial_{\le k}^hW\ra \partial_{\le k,p}^hW$ and ${}_{\le k}{W}\ra {}_{\le k,p}{W}$ are equivalences for all $1\le k\le \infty$.
\end{lem}

\begin{proof}Recall from \cite[Section 5.7]{KKoperadic} that there is a pullback square 
\begin{equation}\label{equ:pullback-double-geo}
\begin{tikzcd}[column sep=0.4cm, row sep=0.4cm]
	\ModInf^\un_k(E_d^t)\dar\rar&\Cosp(\cS_{/\BTop(d)})\dar\\
	\ModInf^\un_k(E_{d,\le k}^p)\rar&\Cosp(\cS_{/\BAut_{\le k}(E_{d})})
\end{tikzcd}
\end{equation}
of symmetric monoidal double categories whose right vertical row is induced by postcomposition with the composition of the map $t\colon \BTop(d)\ra \BAut(E_d)$ from \cref{sec:dl-add} with the map $\BAut(E_d)\ra \BAut_{\le k}(E_d)$ induced by truncation (see \cref{sec:operad-truncation}). Naturality of the construction of \cref{sec:dbl-cat-boundary-square} induces commutative squares 
\begin{equation}\label{equ:pullbacks-true-vs-cosp}
\begin{tikzcd}[column sep=0.5cm, row sep=0.5cm]
	\partial_{\le k}^hW\dar\rar&\partial^h_{\Cosp(\cS_{/\BTop(d)})}W\dar{\circled{1}}\\
	\partial_{\le k,p}^hW\rar&\partial^h_{\Cosp(\cS_{/\BAut_{\le k}(E_{d}})}W
\end{tikzcd}\quad\text{and}\quad 
\begin{tikzcd}[column sep=0.5cm, row sep=0.65cm]
	_{\le k}W\dar\rar&_{\Cosp(\cS_{/\BTop(d)})}W\dar{\circled{2}}\\
	_{\le k,p}W\rar&_{\Cosp(\cS_{/\BAut_{\le k}(E_{k})})}W,
\end{tikzcd}
\end{equation}
which we claim to be pullback squares. We prove this for the right-hand square; the left one is analogous. Spelling out the definition of the entries of  from \eqref{equ:recover-boundary-double-cat}, the square is given by
\[
\hspace{-0.3cm}\begin{tikzcd}[column sep=0.3cm, row sep=0.3cm]
	\Map_{\ModInf^\un_k(E_d^t)_{\varnothing,\varnothing}}(\bfR^d,
	 W)_{\Aut_{\ModInf^\un_k(E_d^t)_{\varnothing,\varnothing}}(\bfR^d)}\dar\rar&\Map_{\Cosp(\cS_{/\BTop(d)})_{\varnothing,\varnothing}}(\bfR^d,
	 W)_{\Aut_{\Cosp(\cS_{/\BTop(d)})_{\varnothing,\varnothing}}(\bfR^d)}\dar\\
	\Map_{\ModInf^\un_k(E_{d,\le k}^p)_{\varnothing,\varnothing}}(\bfR^d,
	W)_{\Aut_{\ModInf^\un_k(E_{d,\le k}^p)_{\varnothing,\varnothing}}(\bfR^d)}\rar&\Map_{\Cosp(\cS_{/\BAut_{\le k}(E_{d})})_{\varnothing,\varnothing}}(R,
	 W)_{\Aut_{\Cosp(\cS_{/\BAut_{\le k}(E_{d})})_{\varnothing,\varnothing}}(\bfR^d)}.
\end{tikzcd}
\]
Since mapping categories of pullbacks of double categories are pullbacks of categories, and mapping spaces of pullbacks of categories are pullbacks of spaces, the fact that \eqref{equ:pullback-double-geo} is a pullback implies that this square is a pullback before taking the orbits. Moreover, the maps between automorphism groups
\[\begin{aligned}\Aut_{\ModInf^\un_k(E_d^t)_{\varnothing,\varnothing}}(\bfR^d)&\lra \Aut_{\Cosp(\cS_{/\BTop(d)})_{\varnothing,\varnothing}}(\bfR^d)\\\Aut_{\ModInf^\un_k(E_{d,\le k}^p)_{\varnothing,\varnothing}}(\bfR^d)&\lra \Aut_{\Cosp(\cS_{/\BAut_{\le k}(E_{d})})_{\varnothing,\varnothing}}(\bfR^d)\end{aligned}\]
are both equivalences (the first are both equivalent to $\Top(d)$, the second both to $\Aut_{\le k}(E_d)$), so the two $E_1$-groups we take orbits with respect to in the top row are the same, and similarly for the bottom row. The horizontal fibres are thus unaffected by taking these orbits, so the square after taking orbits remains a pullback, as claimed.

Since \eqref{equ:pullbacks-true-vs-cosp} are pullbacks, to show the claim it suffices to show that the maps $\circled{1}$ and $\circled{2}$ are equivalences. By the discussion towards the end of \cref{ex:cospan-example}, the map $\circled{1}$ is equivalent to $\circled{2}$ which is in turn equivalent to $\id_W$, so it is in particular an equivalence.
\end{proof}

\cref{lem:embcalc-and-particle-bdy-agree} implies that the action of the tower of $E_1$-groups $T_\bullet\Emb_{\partial^v}(W)^\times$ on $\smash{(\partial_{\le \bullet}^hW\ra _{\le \bullet}W)}$ factors through the map of towers of $E_1$-groups $T_\bullet\Emb_{\partial^v}(W)^\times\ra T_\bullet\Emb^p_{\partial^v}(W)^\times$ induced by the map of towers of symmetric monoidal double categories $\ModInf^\un_\bullet(E_d^t)\ra \ModInf^\un_\bullet(E_{d,\le \bullet}^p)$. 

\subsection{Proof of \cref{thm:fake-boundary-2type}}
Applying the above discussion to the case of triads $M\colon \partial^vM\leadsto \varnothing$ from \cref{sec:recover-triad-emb} and \cref{ex:recover-triad-emb-double-cat}, we obtain map of towers of squares under \eqref{equ:embedding-fake-squares-triad-situation}
\begin{equation}\label{equ:triad-structure-towers}
\begin{tikzcd}[column sep=0.3cm, row sep=0cm]
	\partial^{vh}M\arrow[dd]\arrow[dr]\arrow[rr]&&\partial^{vh}_{\le \bullet}M\arrow[dd]\arrow[dr]\arrow[rr,"\simeq"]&&\partial^{vh}_{\le \bullet,p}M\arrow[dd]\arrow[dr]\\
	&\partial^{h}M\arrow[dd]\arrow[rr,crossing over]&&\partial^{h}_{\le \bullet}M\arrow[dd]\arrow[rr,crossing over,"\simeq",near start]&&\partial^{h}_{\le \bullet,p}M\arrow[dd]\\
	\partial^vM\arrow[dr]\arrow[rr,"\simeq",near end]&&\partial^v_{\le \bullet}M\arrow[dr]\arrow[rr,"\simeq",near end]&&\partial^v_{\le \bullet,p}M\arrow[dr]\\
	&M\arrow[rr,"\simeq",near end]&&_{\le \bullet}M\arrow[rr,"\simeq"]&&_{\le \bullet,p}M.
\end{tikzcd}
\end{equation}
This is acted upon by 
$\Emb_{\partial^v}(M)^\times\ra T_\bullet\Emb_{\partial^v}(M)^\times\ra T_\bullet\Emb^p_{\partial^v}(M)^\times$, compatibly with the a priori action of $T_\bullet\Emb^p_{\partial^v}(M)^\times$ on $(\partial^vM\ra M)$ induced by the composition of symmetric monoidal double categories $\smash{\ModInf^\un_\bullet(E_{d,\le \bullet}^p)\ra \Cosp(\cS_{/\BAut_{\le \bullet}(E_{d})})\ra \Cosp(\cS)}$. The arrows labeled by $\simeq$ are equivalences as a result of Lemmas \ref{lem:embcalc-and-particle-bdy-agree}, and \ref{lem:convergent-boundary} \ref{item:convergence-lem-i} in the cases $W=\partial^v M$ or $W=M$. Moreover, by \cref{lem:convergent-boundary} \ref{item:convergence-bdy-better}, if the hypothesis of \cref{thm:fake-boundary-2type} is satisfied, the two unlabeled horizontal maps are for $\bullet=\infty$ also equivalences, so we obtain the asserted filler in \cref{thm:fake-boundary-2type}.

\subsubsection{Truncated variant of \cref{thm:fake-boundary-2type}} \label{sec:bdy-truncated-variant} One might think that \cref{thm:fake-boundary-2type} ought to, under suitable hypothesis, be extendable to a factorisation of maps \emph{of towers} when replacing $T_\infty\Emb_{\partial^v}(M)^\times$,  $T_\infty\Emb^p_{\partial^v}(M)^\times$, and $\Aut_{\partial^v}(M,\partial^h)$ by towers $T_\bullet\Emb_{\partial^v}(M)^\times$,  $T_\bullet\Emb^p_{\partial^v}(M)^\times$, and $\Aut_{\partial^v}(M,\partial_{\le \bullet}^h)$. Note, however, that the final tower $\Aut_{\partial^v}(M,\partial_{\le \bullet}^h)$ is not even defined: the existence of the tower of squares given as the middle diagonal face of \eqref{equ:triad-structure-towers} does not yield a tower on automorphism spaces. However, there is the following workaround to this issue: Given $k\in \bfN\cup\{\infty\}$, we may consider the endofunctor $\rho_k\colon \cS^{[1]}\ra \cS^{[1]}$ that sends a map $X\ra Y$ to the second map $ \rho_{k}X \to Y$ in the Moore--Postnikov $k$-factorisation of $X\ra Y$ (the unique factorisation into a $k$-connected map followed by a $k$-coconnected map). This defines a localisation of the category $\cS^{[1]}$  \cite[5.2.8.16,5.2.8.19]{LurieHTT}, so we obtain a localisation $\cS^{[1]\times [1]}=\Fun([1],\cS^{[1]})\ra \Fun([1],\cS^{[1]})=\cS^{[1]\times [1]}$ of the category of squares $\cS^{[1]\times [1]}$, given by postcomposition with $\rho_k$. This satisfies $\rho_n\circ \rho_k\simeq \rho_{\min(n,k)}$. In particular, for any monotonically increasing function $c\colon \bfN\cup\{\infty\}\ra \bfZ\cup\{\infty\}$, we can apply $\rho_{c(k)}$ to the two vertical maps in leftmost diagonal square in \eqref{equ:triad-structure-towers} to obtain a map
$(\rho_{c(k)} \partial^{vh}(M) \to \partial^v M) \ra (\rho_{c(k)} \partial^h M \to M)$  in $\cS^{[1]}$.
We then consider the space of automorphisms $\smash{\Aut_{\partial^v}(M,\rho_{c(k)}\partial^h)}$ of the right term as an object in the undercategory $\smash{{\cS^{[1]}}_{(\rho_{c(k)} \partial^{vh}M)\ra \partial^vM)/}}$. Varying $k$ this yields a tower of $E_1$-groups $\Aut_{\partial^v}(M,\partial^h)\ra \Aut_{\partial^v}(M,\rho_{c(\bullet)}\partial^h)$ under $\Aut_{\partial^v}(M,\partial^h)$. Using this, a minor extension of the proof of \cref{thm:fake-boundary-2type} in which the role of \cref{lem:convergent-boundary} \ref{item:convergence-bdy-better} is replaced by \cref{lem:convergent-boundary} \ref{item:convergence-bdy} shows:

\begin{thm}\label{thm:fake-boundary-truncated}Let $(M,\partial^vM,\partial^hM)$ be a compact smoothable $d$-dimensional manifold triad with $d\ge5$ and  $c\colon \bfN\cup\{\infty\}\ra \bfZ\cup\{\infty\}$ monotonically increasing function. If
\begin{enumerate}
	\item $c(k)\le k-d+2$ for all $k$ and
	\item the inclusions  $\partial^h M\subset M$ and $\partial^{vh} M\subset \partial^vM$ are  $2$-connected,
\end{enumerate}
then there exists a dashed filler in the diagram
\[\begin{tikzcd}[column sep=0.3cm,row sep=0.6cm]
	\Homeo_{\partial^v}(M)\dar\rar&\Emb_{\partial^v}(M)^\times\rar&T_\bullet\Emb_{\partial^v}(M)^\times\rar&T_\bullet\Emb^{p}_{\partial^v}(M)^\times\dar\arrow[dlll,dashed]\\
	\Aut_{\partial^v}(M,\rho_{c(\bullet)}\partial^h)\arrow[rrr]&&&\Aut_{\partial^v}(M).
\end{tikzcd}
\]
of towers of $E_1$-groups. 
\end{thm}

\subsection{Digression: $T_\infty$-boundaries}\label{sec:digression} We conclude this section with some remarks on the embedding calculus approximation $\partial^hM\ra \partial^h_{\le \infty}M$ to the ``moving'' part of the boundary of a compact $d$-dimensional manifold triad. \cref{lem:convergent-boundary} \ref{item:convergence-bdy-better} implies that this map is an equivalence if $M$ is smoothable, $d\ge5$, and the inclusion $\partial^hM\subset M$ is a tangential $2$-type equivalence. It would be interesting to determine $\smash{\partial^h_{\le \infty}M}$ when these assumptions are not satisfied. Here are some observations regarding this, for simplicity phrased in the case $\smash{\partial^hM=\partial M}$ where we write $\partial^h_{\le \infty}M\coloneq \partial_{\le \infty}M$:
\begin{enumerate}[leftmargin=0.6cm]
	\item For $d \leq 2$, we have $\partial M\simeq \partial_{\leq \infty} M$ as a consequence of the main result of \cite{KKSurfaces}.
	\item\label{equ:descent-expression} Using that $T_\infty\Emb(-,M)$ satisfies descent for complete Weiss $\infty$-covers \cite[Lemma 5.9]{KKoperadic}, we can identify the restriction map $T_\infty \Emb(M \sqcup \bfR^d,M)  \to T_\infty \Emb(M,M)$ with the map
	\[ \textstyle{\lim_{D \subset M} \Emb(D \sqcup \bfR^d,M) \lra  \lim_{D \subset M} \Emb(D,M)}\]
	induced by restriction. Here the (homotopy) limit is taken over the opposite of the poset of submanifolds $D\subset \interior(M)$ homeomorphic to $\sqcup^k\bfR^d$ for some $k\ge0$, ordered by inclusion. Comparing the definition of $\partial M \simeq \partial_E M$ with that of $\partial_{\leq \infty} M$ and using the isotopy extension theorem, it follows that the map $\partial M\ra \partial_{\leq \infty} M$ is equivalent to the map
	\begin{equation}\label{equ:explicit-fake-boundary}
		\textstyle{\partial M\lra \lim_{D\subset \interior(M)}M\backslash D}
	\end{equation}
	induced by the inclusions $\partial M \to M\backslash D$ of the original boundary.

	\item\label{enum:homology-spheres} There are examples for which \eqref{equ:explicit-fake-boundary} \emph{is not} an equivalence: Fix a homology $(d-1)$-sphere $\Sigma$ with nontrivial fundamental group and $d \geq 3$ which bounds a contractible manifold $W_\Sigma$ (e.g.~$\Sigma$ could be the Poincar\'e homology sphere) and consider the commutative square
	\[\begin{tikzcd}[column sep=0.4cm, row sep=0.4cm]
		\Sigma\dar \rar&\lim_{D\subset \interior(W_\Sigma)}W_\Sigma\backslash D\dar\\
		\Sigma^+\rar&\lim_{D\subset \interior(W_\Sigma)}(W_\Sigma\backslash D)^+
	\end{tikzcd}\]
	whose vertical arrows are induced by taking Quillen's plus-construction. Since $d\ge3$, the inclusion $W_\Sigma\backslash D\ra W_\Sigma$ is $2$-connected by transversality, so $W_\Sigma\backslash D$ is simply connected which implies that the right vertical map is an equivalence. In particular, as $\Sigma^+\simeq S^{d-1}$, the upper horizontal map is trivial on fundamental groups, so it cannot be an equivalence.
	
	\item\label{enum:closed-manifolds} There are also examples for which \eqref{equ:explicit-fake-boundary} \emph{is} an equivalence but that are not covered by \cref{lem:convergent-boundary} \ref{item:convergence-bdy-better}, for instance any nonempty closed manifold $M$. To see this, it suffices to prove that $\partial_{\le \infty}M$ is empty if $M$ is closed. Assume for a contradiction that it is not, so there is an element of $T_\infty \Emb(M\sqcup \bfR^d,M)$ that hits the identity component of $T_\infty \Emb(M,M)$. By isotopy extension for embedding calculus \cite[Section 5.4.2]{KKoperadic}, we have a fibre sequence $T_\infty \Emb(M,M \backslash \bfR^d)\ra T_\infty \Emb(M\sqcup\bfR^d,M)\ra T_\infty \Emb(\bfR^d,M)$, so we conclude that there is an element of $T_\infty \Emb(M,M \backslash \bfR^d)$ that maps to the identity component of $T_\infty \Emb(M,M)$. By considering the underlying maps of spaces, this shows that the identity on $M$ factors up to homotopy over $M\backslash \bfR^d\subset M$. This is impossible if $M$ is closed (consider the top $\bfF_2$-homology).

	\item The map $\partial M \to \partial_{\le \infty} M$ is a retract after applying double suspension $\Sigma^2_+(-)$ if $M$ has no closed components and $d \geq 3$. The argument for this uses the \emph{double $M$-suspension} of a map $X\ra M$, defined as the pushout $\Sigma^2_M X\coloneq M\times S^1\cup_{X\times S^1}X\times D^2$. For the inclusion $Y\subset M$ of a reasonable subspace of the interior, the map $\Sigma^2_M (M\backslash Y)=M \times S^1 \cup_{(M \backslash Y) \times S^1} (M \backslash Y) \times D^2  \ra (M \times D^2)\backslash(Y \times \{0\})$ induced by inclusion is an equivalence, using that the target $(M \times D^2)\backslash(Y \times \{0\})$ is the union of the subspaces $M\times (D^2\backslash\{0\})$ and $(M\backslash Y)\times D^2$ that intersect in $(M\backslash Y)\times D^2\backslash\{0\}$, and that this is a (homotopy) pushout decomposition under mild point-set topological conditions on $Y\subset M$ (e.g.\,it suffices to assume that $M\backslash Y\subset M$ is a strong deformation retract of an open neighbourhood of it). Combining this with \ref{equ:descent-expression} and descent for Weiss $\infty$-covers, one checks that the composition $\smash{\Sigma^2_M \partial M \ra \Sigma^2_M \lim_{D \subset \interior(M)} M \backslash D \ra \lim_{D \subset \interior(M)} \Sigma^2_M(M \backslash D)}$ agrees with the map $\partial(M \times D^2) \to \partial^h_{\le \infty}(M \times D^2)$, which is an equivalence as a result of \cref{lem:convergent-boundary} since $\partial(M\times D^2)\subset M\times D^2$ is $2$-connected. This implies that $\Sigma^2_M \partial M \to \Sigma^2_M \partial_{\le \infty} M$ admits a cosection compatible with the inclusions of $M \times S^1$. Collapsing  $M \times S^1$ then gives the statement. We suspect that $\partial M \to \partial_{\le \infty} M$ is an equivalence after applying $\Sigma^2_+(-)$, not just a retract, and it also seems possible that suspending once suffices.
\end{enumerate}

\begin{rem}[Boundaries and complements of $\DiscInf$-presheaves]\label{rem:presheaf-boundaries}
The construction of the embedding calculus approximation $\smash{\partial^hM\ra \partial^h_{\le k}M}$ to the moving part of the manifold can be generalised in various directions. We briefly outline two of them, for simplicity in the case $\smash{\partial^hM=\partial M}$ and $k=\infty$:
\begin{enumerate}[leftmargin=*]
	\item Given a presheaf $X\in \PSh(\DiscInf^t_{d})$, its \emph{underlying homotopy type} is the orbit space $X(\bfR^d)_{\Emb(\bfR^d)^\times}$. With this terminology, $\partial_{\le \infty}M$ arises as the underlying homotopy type of a presheaf $\partial E_{M}$, namely $\DiscInf_{d}^t\ni D\mapsto \fib_{\id}\big(T_\infty\Emb(M\sqcup D,M) \to T_\infty \Emb(M,M)\big)$. The latter arises as a fibre
	\[
		\partial E_{M}\simeq \fib_\id\big(\MapInt(E_M,E_M)\ra \Map(E_M,E_M))\in \PSh(\DiscInf^t_{d})
	\] 
	where $\MapInt(-,-)$ is the internal mapping object in $\smash{\PSh(\DiscInf^t_{d})}$ and the map to the constant presheaf on the mapping space  in $\smash{\PSh(\DiscInf^t_{d})}$ from $E_M$ to itself  is induced by the fact that $\varnothing\in \DiscInf^t_{d}$ is initial. The latter formula also makes sense when $E_{M}$ is replaced by an arbitrary presheaf $X$ on $\DiscInf^t_{d}$, not necessarily induced by a manifold. The maps $\Emb(-,-)\ra T_\infty\Emb(-,-)$ induce to a map of presheaves $E_{\partial M \times I}\ra \partial E_{M}$ which yields the map $\smash{\partial M\ra \partial_{\le \infty}M}$ considered above on underlying homotopy types. Moreover, the internal composition in $\MapInt(-,-)$ equips $\partial E_{M}$ (more generally: $\partial X$) with an associative algebra structure such that $E_{\partial M \times I}\ra \partial E_{M}$ is an algebra map through which the $E_{\partial M \times I}$-module structure on $E_M$ given by ``stacking'' factors (more generally: any presheaf $X$ is a module over $\partial X$). Much of the discussion in this chapter on $\partial_{\le \infty} M$ can be generalised to the level of presheaves, e.g.\,the conditions we discussed under which $\partial M\ra \partial_{\le \infty}M$ is an equivalence in fact imply that the map of presheaves $E_{\partial M \times I}\ra \partial E_{M}$ is an equivalence. One can also define particle embedding calculus versions of this ``presheaf boundary''. There are further variants of the construction, e.g.\,when part of the boundary is fixed.
	\item More generally, given a map of $\DiscInf_d^t$-presheaves $\varphi\colon X\ra Y$, one can define a ``complement'' $\smash{Y\backslash^\varphi X\coloneq \fib_{\varphi}(\MapInt(X,Y)\ra \Map(X,Y))\in \PSh(\DiscInf_d^t)}$ which generalises $\partial X$ in that $\partial X\simeq X\backslash^{\id}X$. This complement admits a module structure over $\partial X$, and is the analogue on the level of presheaves of taking the complement of a codimension $0$-embedding of manifolds. In fact, given a codimension $0$ embedding $e\colon M\hookrightarrow \interior(N)$ between compact manifolds, there is a map of presheaves $E_{M\backslash e(\interior(N))}\ra E_{N}\backslash^{E_e}E_{M}$. This is a module map with respect to the algebra map $E_{\partial M\times I}\ra \partial E_M$ and an equivalence in favourable circumstances.
\end{enumerate} 
These ``boundaries'' and ``complements'' of presheaves, as well as extensions of these constructions, allow one to mimic various constructions for manifolds on the level of presheaves. We expect this to find further applications in future work.
\end{rem}

\section{A pullback decomposition of spaces of self-embeddings} 
This section serves to prove a version of the pullback decomposition as in the introduction for certain spaces of self-embeddings instead of homeomorphisms. \cref{bigthm:pullback} will be a special case.

\subsection{The pullback decomposition}For a compact $d$-dimensional topological manifold triad $(M,\partial^vM,\partial^hM)$, we saw in \cref{prop:embedding-fix-bdy-setwise} that the forgetful map $\Homeo_{\partial^v}(M)\ra \Aut_{\partial^v}(M,\partial^h)$ factors through the group-like components $\Emb_{\partial^v}(M)^\times $ of the space of self-embeddings of $M$ that fix a neighbourhood of $\partial^vM$. By taking derivatives, the latter maps to the $E_1$-group of self-equivalences $\Aut_{\partial^v}(TM)$ of the topological tangent bundle that fix $TM|_{\partial^v}$. More precisely, $\Aut_{\partial^v}(TM)$ is defined as the automorphisms of a tangent classifier $M\ra\BTop(d)$ in the under-/overcategory $(\cS_{/\BTop(d)})_{\partial^vM/}$. The derivative map induces a factorisation 
\begin{equation}\label{equ:emb-derivative}
	\Emb_{\partial^v}(M)^\times\lra \Aut_{\partial^v}(TM,\partial^h)\lra \Aut_{\partial^v}(M,\partial^h)
\end{equation}
through the pullback $\smash{\Aut_{\partial^v}(TM,\partial^h)\coloneq  \Aut_{\partial^v}(TM)\times_{\Aut_{\partial^v}(M)}\Aut_{\partial^v}(M,\partial^h)}$. This pullback can be thought as the space of self-equivalences of $TM$ as an Euclidean bundle that fix $TM|_{\partial^vM}$ pointwise and $TM|_{\partial^hM}$ setwise. The general version of \cref{bigthm:pullback} will show that the delooping of this sequence can often completed to a pullback square once we add appropriate tangential structures and rationalise. We will now make these modifications precise.

\subsubsection{Tangential structures and rationalisations}\label{sec:tangential-structures}
A tangential structure is a map $\theta\colon B\ra \BTop(d)$, and a $\theta$-structure $\ell^v$ on $TM|_{\partial^vM}$ is a lift of the restriction $TM|_{\partial^vM}\colon \partial^vM\ra \BTop(d)$ along $\theta$. Fixing such $B$ and $\ell^v$, the first two $E_1$-groups in \eqref{equ:emb-derivative} act on the space of $\theta$-structures on $TM$ that agree with $\ell^v$ on $TM|_{\partial^v M}$, that is, the space of lifts $\ell \colon M\ra B$ of $TM\colon M\ra \BTop(d)$ along $\theta$ that extend $\ell^v \colon TM|_{\partial^vM} \to B$. Taking orbits, we obtain a sequence of the form
\begin{equation}\label{equ:emb-derivative-tangential}
	\BEmb^\theta_{\partial^v}(M;\ell^v)^\times\lra \BAut^\theta_{\partial^v}(TM,\partial^h;\ell^v)\lra \BAut_{\partial^v}(M,\partial^h),
\end{equation} 
which we will now further modify. First we replace the second map by the initial map in its Moore--Postnikov $1$-factorisation
$\smash{
	\BAut_{\partial^v}^\theta(TM,\partial^h;\ell^v)\ra \BAut_{\partial^v}(M,\partial^h)^{\ell^v}\ra \BAut_{\partial^v}(M,\partial^h).
}$
This modification is mild in that the homotopy type of $\BAut_{\partial^v}(M,\partial^h)^{\ell^v}$ is given by \[\textstyle{\BAut_{\partial^v}(M,\partial^h)^{\ell^v}\simeq \bigsqcup_{[\ell]\in \theta\text{-}\st}\BAut_{\partial^v}^{\ell}(M,\partial^h).}\]
Here $\theta\text{-}\st$ is the set of components of $\smash{\BAut_{\partial^v}^\theta(TM,\partial^h;\ell^v)}$, namely the set of orbits of the action of $\pi_0(\Aut_{\partial^v}(TM,\partial^h))$ on the path components of the space of $\theta$-structures that agree with $\ell^v$ on $TM|_{\partial^v M}$. The space $\smash{\BAut_{\partial^v}^{\ell}(M,\partial^h)}$ is defined as the covering space of $\BAut_{\partial^v}(M,\partial^h)$ corresponding to the image on fundamental groups based at $\ell$ of the map $\smash{\BAut_{\partial^v}^\theta(TM,\partial^h;\ell^v)\ra \BAut_{\partial^v}(M,\partial^h)}$. As a second modification, we replace $\BEmb^\theta_{\partial^v}(M;\ell^v)^\times$ and  $\BAut_{\partial^v}^\theta(TM,\partial^h;\ell^v)$  by their fibrewise rationalisation over $\BAut_{\partial^v}(M,\partial^h)^{\ell^v}$ in the sense of \cref{sec:fibrewise-rat} (indicated by a $(-)_{\fQ}$-subscript), arriving at a variant of the sequence \eqref{equ:emb-derivative-tangential} that has the form
\begin{equation}\label{equ:extandable-composition}
	\BEmb^\theta_{\partial^v}(M;\ell^v)^\times_{\fQ}\lra \BAut^\theta_{\partial^v}(TM,\partial^h;\ell^v)_{\fQ}\lra \BAut_{\partial^v}(M,\partial^h)^{\ell^v}.
\end{equation}

\subsubsection{Statement of the pullback decomposition}\label{sec:statement-pullback}With these explanations out of the way, we can state the main result of this section. We say that a sequence $X\ra Y\ra Z$ of spaces \emph{can be completed to a pullback square} if there exists a space $W$ and a pullback square in $\cS$
\[\begin{tikzcd}[column sep=0.4cm, row sep=0.4cm]
	X\rar\dar&W\dar\\
	Y\rar &Z
\end{tikzcd}\]
which extends the two given maps. An \emph{oriented tangential structure} is a map of the form $\theta\colon B\ra \BSTop(d)$. Postcomposition with $\BSTop(d)\ra \BTop(d)$ gives a  tangential structure in above sense.

\begin{thm}\label{thm:pullback-decomp}Let $(M,\partial^v M,\partial^h M)$ a smoothable compact $d$-manifold triad with $d\ge5$, $\theta\colon B \ra \BSTop(d)$ an oriented tangential structure with $B$ a nilpotent space, and $\ell^v$ a $\theta$-structure on $TM|_{\partial^v}$. If
\begin{enumerate}
	\item \label{enum:pullback-decomp-ii} the inclusions $\partial^h M\subset M$ and $\partial^{vh} M \subset \partial^v M$ are both $2$-connected, where $\partial^{vh} M=\partial^vM\cap \partial^hM$,
	\item \label{enum:pullback-decomp-i} the composition
	\[\smash{B\xlra{\theta} \BSTop(d)\xlra{t}\BAut^\id(E_d)\xlra{(-)_\bfQ} \BAut^\id(E_{d,\bfQ})}\]
	is nullhomotopic after rationalisation (see \cref{equ:consequence-thm-g} for the notation),
\end{enumerate}
then the sequence of spaces \eqref{equ:extandable-composition} can be completed to a pullback square. Moreover, by \cref{cor:double-stab-top-null-extra}, the second condition \ref{enum:pullback-decomp-i} is implied by the following condition
\begin{enumerate}
\item[(ii)'] \label{enum:pullback-decomp-iiprime} $\theta$ factors through the stabilisation map $\BSTop(d-2)\ra \BSTop(d)$ after rationalisation.
\end{enumerate} 
\end{thm} 

\begin{rem}\label{rem:improvement-even-so} If $d$ is even, then condition \ref{enum:pullback-decomp-i} in \cref{thm:pullback-decomp} is also implied by the condition
\begin{enumerate}
\item[(ii)''] $\theta$ factors through the stabilisation map $\BSO(d-1)\ra \BSTop(d)$ after rationalisation.
\end{enumerate}
This is neither stronger nor weaker than condition (ii)'. To see that (ii)'' implies (ii), it suffices to show that the map $\BSO(d-1)\ra \BAut^\id(E_{d,\bfQ})$ is nullhomotopic. After precomposing it with $\BSO(d-2)\ra \BSO(d-1)$, it is nullhomotopic since it factors through the map $\BSTop(d-2)\ra \BAut^\id(E_{d,\bfQ})$ which is nullhomotopic by \cref{cor:double-stab-top-null-extra}. Now the claim follows from the general fact that for $d$ even, any map $f\colon \BSO(d-1)\ra X$ to a $1$-connected rational space $X$ is nullhomotopic if the composition $\BSO(d-2)\ra\BSO(d-1)\ra X$ is rationally nullhomotopic. This follows from obstruction theory: the obstructions to extending $f$ over a point lie in $\oH^{i}(\BSO(d-1);\pi_i(X))$ for $i\ge1$ and these vanish if they vanish in $\oH^{i}(\BSO(d-2);\pi_i(X))$ since the map $\BSO(d-2)\ra\BSO(d-1)$ is for $d$ even injective on rational cohomology. This last statement is false for $\BSTop(d-2)\ra\BSTop(d-1)$, e.g.\,as a consequence of \cite[Theorem C]{KRWOddDiscs}.
 \end{rem}

We prove \cref{thm:pullback-decomp} in \cref{sec:proof-pullback-emb} below. But first we explain why the variant of this result for spaces of homeomorphisms stated in the introduction  as \cref{bigthm:pullback} follows from it.

\begin{proof}[Proof of \cref{bigthm:pullback}] We argue that \cref{bigthm:pullback} follows from \cref{thm:pullback-decomp} applied to $(M,\partial^vM,\partial^hM)=(M,\half\partial M,D^{d-1})$. To see this, note that for $d\ge5$ the inclusions $\partial^h M=D^{d-1}\subset M$ and $\partial^{vh} M=S^{d-2} \subset \half\partial M=\partial^v M$ are $2$-connected if and only if $M$ and $\half\partial M$ are $2$-connected, and that the inclusion $\half\partial M=\partial M\backslash \interior(D^{d-1})\subset \partial M$ is $(d-1)$-connected. Thus, for $d\ge5$, condition \ref{enum:pullback-ii} in \cref{bigthm:pullback} and condition \ref{enum:pullback-decomp-ii} in \cref{thm:pullback-decomp} are equivalent. Condition \ref{enum:pullback-i} in \cref{bigthm:pullback} is the same as condition (ii)' in \cref{thm:pullback-decomp}, so it remains to prove that the  sequences in the two results, \eqref{equ:introduction-compl-seq} and \eqref{equ:extandable-composition}, are equivalent. This follows by showing that the forgetful maps 
\[
	\Homeo_{\partial}(M)\lra \Emb_{\half\partial}(M) \quad \text{ and } \quad \Aut_\partial(M)\lra \Aut_{\half\partial}(M,D^{d-1})
\]
are equivalences (this is even slightly stronger, since it shows $\Emb_{\half\partial}(M)=\Emb_{\half\partial}(M)^\times$). 

For the second map, this follows from the contractibility of $\Aut_\partial(D^{d-1})$. For the first map, we argue that all its fibres are contractible. To do so, we first replace $M$ up to homeomorphism with $M\cup_{D^{d-1}}(D^{d-1}\times I)$. Then we note that since $\smash{M\subset M\cup_{D^{d-1}}D^{d-1}\times I}$ is an isotopy equivalence relative to $\half\partial M$, the first map is equivalent to the map \begin{equation}\label{equ:thickened-restrict-homeo-map}\smash{\Homeo_\partial(M\cup_{D^{d-1}}(D^{d-1}\times I))\lra \Emb_{\half\partial}(M,M\cup_{D^{d-1}}(D^{d-1}\times I))^{D}}\end{equation} given by restriction along $\smash{M\subset M\cup_{D^{d-1}}D^{d-1}\times I}$; here the $D$-superscript indicates that we are restricting to the embeddings that send $\interior(D^{d-1})\subset M$ into the interior of $M\cup_{D^{d-1}}(D^{d-1}\times I)$. It thus suffices to show that the (homotopy) fibres of \eqref{equ:thickened-restrict-homeo-map} at all  $\smash{e\in\Emb_{\half\partial}(M,M\cup_{D^{d-1}}(D^{d-1}\times I))^{D}}$ are contractible. By the isotopy extension theorem, this restriction map is a fibration, so the homotopy fibre is equivalent the actual fibre, which is the space of homeomorphisms relative to the boundary between $D^{d-1}\times I$ and $(M\cup_{D^{d-1}}(D^{d-1}\times I))\backslash e(M\backslash D^{d-1})$. The latter is a torsor over $\Homeo_\partial(D^{d-1}\times I)$ which is contractible by the Alexander trick, so it suffices to show that there is \emph{some} homeomorphism $D^{d-1}\times I\cong (M\cup_{D^{d-1}}(D^{d-1}\times I))\backslash e(M\backslash D^{d-1})$ relative to the boundary. Using \[\smash{M\cup_{D^{d-1}}(D^{d-1}\times I)\cong \Big(\big(M\cup_{D^{d-1}}(D^{d-1}\times I)\big)\backslash e(M\backslash D^{d-1})\Big)\cup_{D^{d-1}} e(M)}\] and the Seifert--van Kampen and Mayer--Vietoris theorem, one deduces that $(M\cup_{D^{d-1}}(D^{d-1}\times I))\backslash e(M\backslash D^{d-1})$ is $1$-connected and acyclic, so it is indeed homeomorphic to $D^d\cong D^{d-1}\times I$ by the topological Poincar\'e conjecture and the claim follows.
\end{proof}

\begin{rem}\label{rem:nilpotency-issue} There is a subtlety in the statement of \cref{thm:pullback-decomp} (and thus also in \cref{bigthm:pullback}) having to do with rationalisations. These results involve the fibrewise rationalisations in the sense of \cref{sec:rationalisation} (i.e.\,fibrewise $\bfQ$-completion) of the spaces $\smash{\BEmb^\theta_{\partial^v}(M;\ell^v)^\times}$ and $\smash{\BAut^\theta_{\partial^v}(TM,\partial^h;\ell^v)}$ over $\BAut_{\partial^v}(M,\partial^h)^{\ell^v}$. Fibrewise rationalisation behaves well for $\smash{\BAut^\theta_{\partial^v}(TM,\partial^h;\ell^v)}$: it induces isomorphisms on rational (co)homology and rationalised higher homotopy groups, since it is not hard to see that this space is fibrewise nilpotent over $\BAut_{\partial^v}(M,\partial^h)^{\ell^v}$ (see the proof of \cref{prop:factorisation}). For $\smash{\BEmb^\theta_{\partial^v}(M;\ell^v)^\times}$ this is not so clear. This issue can be addressed in two ways:
\begin{enumerate}[leftmargin=*]
	\item\label{enum:future-nilpotency} Extending the strategy of proof of \cite[Theorem C]{KRWAlgebraic}, one can prove the missing fibrewise nilpotency statement of $\smash{\BEmb^\theta_{\partial^v}(M;\ell^v)^\times}$ for all manifolds $M$ as in \cref{thm:pullback-decomp} and \cref{bigthm:pullback}, so there is no issue after all. We intend to revisit this point as  part of future work.
	\item\label{enum:no-nilpotency-needed} For all results in this work, this nilpotency statement is actually not necessary. In particular, the proofs of \cref{thm:pullback-decomp} and \cref{bigthm:pullback} do not rely on it, and neither do the applications such as Theorems~\ref{bigthm:geometric} or \ref{bigthm:detect-classes}. This is because all we use in the proofs (see e.g.\,\cref{rem:subtlety-summand}) is that the $\bfQ$-cohomology of fibrewise rationalisation contains the $\bfQ$-cohomology of the unrationalised space as a summand (see \cref{sec:non-nilpotent-completion}), without any assumption (see also \cref{rem:subtlety-summand}).
\end{enumerate}
\end{rem}

\subsection{The proof of \cref{thm:pullback-decomp}} \label{sec:proof-pullback-emb}
To prove \cref{thm:pullback-decomp}, we first formulate a variant of this result in the context of embedding calculus, then explain why \cref{thm:pullback-decomp} follows from this variant, and finally prove the variant. Throughout the section we fix a triad $(M,\partial^vM,\partial^hM)$ satisfying the assumptions in \cref{thm:pullback-decomp}. To state the variant, we define a version $T_\bullet\Emb_{\partial^v}(M,\partial^h)^\times$ of the embedding calculus tower $T_\bullet\Emb_{\partial^v}(M)^\times$ from \cref{sec:bdy-emb-calc} by the pullback in towers of spaces
\begin{equation}\label{equ:defining-pullback-bdy-tk}
\begin{tikzcd}[column sep=0.4cm, row sep=0.4cm]
	T_\bullet\Emb_{\partial^v}(M,\partial^h)^\times\rar\dar & T_\bullet\Emb_{\partial^v}(M)^\times\dar\\
	\Aut_{\partial^v}(M,\partial^h)\rar & \Aut_{\partial^v}(M,\rho_{c(\bullet)}\partial^h),
\end{tikzcd}
\end{equation}
where the right vertical map is the one from \cref{thm:fake-boundary-truncated} and the bottom horizontal map is induced by truncation. Here and henceforth, we adopt the convention that a space (in this case  $\Aut_{\partial^v}(M,\partial^h)$) is viewed as a constant tower. This definition depends on the choice of a function $c\colon\bfN\cup\{\infty\}\ra\bfZ\cup\{\infty\}$ as in \cref{thm:fake-boundary-truncated} which we fix once and for all.

The $E_1$-map $\Emb_{\partial_0}(M)^\times\ra \Aut_{\partial^v}(M,\partial^h)$ from \eqref{equ:initial-forgetful-square-bdy} and the tower of $E_1$-maps $\Emb_{\partial^v}(M)^\times\ra T_\bullet\Emb_{\partial^v}(M)^\times$ combine to a tower of $E_1$-maps $\Emb_{\partial_0}(M)^\times\ra T_\bullet\Emb_{\partial^v}(M,\partial^h)^\times$ into the pullback. Moreover, the $E_1$-map $\Emb_{\partial^v}(M)^\times\ra \smash{\Aut_{\partial^v}(TM,\partial^h)}$ used in the construction of \eqref{equ:extandable-composition} factors as 
\begin{equation}\label{equ:factorisation-over-modi-embcalc}
\Emb_{\partial_0}(M)^\times\lra T_\bullet\Emb_{\partial^v}(M,\partial^h)^\times\lra \Aut_{\partial^v}(TM,\partial^h)=\Aut_{\partial^v}(TM)\times_{\Aut_{\partial^v}(M)}\Aut_{\partial^v}(M,\partial^h)
\end{equation}
where the second map is induced by the map to the first stage $T_1\Emb_{\partial^v}(M)^\times\simeq \Aut_{\partial^v}(TM)$ in the tower (see \cite[Section 5]{KKoperadic} for this equivalence) and the map $T_\bullet\Emb_{\partial^v}(M,\partial^h)^\times\ra \Aut_{\partial^v}(M,\partial^h)$ in \eqref{equ:defining-pullback-bdy-tk}. Replacing the role of the $E_1$-map $\Emb_{\partial^v}(M)^\times\ra\smash{\Aut_{\partial^v}(TM,\partial^h)}$ in the construction of \eqref{equ:extandable-composition} by the second map in \eqref{equ:factorisation-over-modi-embcalc} results in a sequence of towers of spaces
\begin{equation}\label{equ:modified-tk-sequence}
	\oB T_\bullet\Emb^\theta_{\partial^v}(M,\partial^h;\ell^v)_{\fQ}^\times\lra  \BAut^\theta_{\partial^v}(TM,\partial^h;\ell^v)_{\fQ}\lra \BAut_{\partial^v}(M,\partial^h)^{\ell^v}.
\end{equation}
The first map in \eqref{equ:factorisation-over-modi-embcalc} induces a map of towers of spaces
\begin{equation}\label{equ:map-to-limit}
	\BEmb^\theta_{\partial^v}(M;\ell^v)_{\fQ}^\times\lra \oB T_\bullet\Emb^\theta_{\partial^v}(M,\partial^h;\ell^v)_{\fQ}^\times
\end{equation} 
that relates \eqref{equ:extandable-composition} to \eqref{equ:modified-tk-sequence}. The map \eqref{equ:map-to-limit} turns out to become an equivalence after taking limits:

\begin{lem}\label{lem:reduction-to-tk}The connectivity of the maps in the map of tower \eqref{equ:map-to-limit} diverges with $\bullet$. In particular, 
\[
	\textstyle{\BEmb^\theta_{\partial^v}(M;\ell^v)_{\fQ}^\times{\simeq} \lim_k\big(\oB T_k\Emb^\theta_{\partial^v}(M,\partial^h;\ell^v)_{\fQ}^\times\big)}.
\]
\end{lem}

\begin{proof}
We first show that the connectivity of the analogous maps without the ${\fQ}$-subscripts diverges. Both sides are the orbits of an action on the same space of $\theta$-structures, so comparing the two fibre sequence involving the orbits, it suffices to prove the connectivity of the map $\Emb_{\partial^v}(M)^\times\ra T_k\Emb_{\partial^v}(M,\partial^h)^\times$ between the $E_1$-groups that act diverges with $k$. Moreover, since $\Emb_{\partial^v}(M)^\times\ra T_k\Emb_{\partial^v}(M)^\times$ is at least $(k-d+4)$-connected by \cite[Theorem 6.3 and Remark 6.11]{KKoperadic}, it is enough to show that the connectivity of the upper horizontal map $T_k\Emb_{\partial^v}(M,\partial^h)^\times\ra T_k\Emb_{\partial^v}(M)^\times$ in the pullback \eqref{equ:defining-pullback-bdy-tk} diverges in $k$, or equivalently that the connectivity of the bottom map $\Aut_{\partial^v}(M,\partial^h)\ra\Aut_{\partial^v}(M,\rho_{c(k)} \partial^h)$ diverges. To see this, first recall that by obstruction theory, postcomposition $\phi_\ast\colon \Map(A,X)\ra \Map(A,Y)$ with an $n$-connected map $\phi\colon X\ra Y$ is $(n-a)$-connected if $A$ is homotopically $a$-dimensional. In particular, the connectivity of the two maps \[
\Map(\partial^hM,\partial^hM)\lra \Map(\partial^hM,\rho_{c(k)}\partial^hM)\quad\text{and}\quad \Map(\partial^{vh}M,\partial^{h}M)\lra \Map(M^{vh},\rho_{c(k)}\partial^{h}M),
\]
induced by postcomposition with the $c(k)$-connected maps $\partial^hM\ra \rho_{c(k)}\partial^hM$, both diverge in $k$. This implies that the connectivity of the two analogous maps \[\hspace{-0.2cm}
	\Map_{\cS^{[1]}}\Big(\mathrel{\substack{\partial^hM\\\downarrow\\M}},\mathrel{\substack{\partial^hM\\\downarrow\\M}}\Big)\ra \Map_{\cS^{[1]}}\Big(\mathrel{\substack{\partial^hM\\\downarrow\\M}},\mathrel{\substack{\rho_{c(k)}\partial^hM\\\downarrow\\M}}\Big)\ \ \text{and}\ \  \Map_{\cS^{[1]}}\Big(\mathrel{\substack{\partial^{vh}M\\\downarrow\\\partial^{v}M}},\mathrel{\substack{\partial^hM\\\downarrow\\M}}\Big)\ra \Map_{\cS^{[1]}}\Big(\mathrel{\substack{\partial^{vh}M\\\downarrow\\\partial^{v}M}},\mathrel{\substack{\rho_{c(k)}\partial^hM\\\downarrow\\M}}\Big),
\]
also diverge with $k$, where $\cS^{[1]}$ is the arrow category of $\cS$. Since $\rho_{c(k)}\colon \cS^{[1]}\ra \cS^{[1]}$ is a localisation, these two maps are equivalent to the two vertical maps in the commutative square 
\[
\begin{tikzcd}[column sep=1.5cm]
	\Map_{\cS^{[1]}}\Big(\mathrel{\substack{\partial^hM\\\downarrow\\M}}\ ,\ \mathrel{\substack{\partial^hM\\\downarrow\\M}}\Big)\rar{\inc^*}\dar{\rho_{c(k)}}&\Map_{\cS^{[1]}}\Big(\mathrel{\substack{\partial^{vh}M\\\downarrow\\\partial^vM}}\ ,\ \mathrel{\substack{\partial^hM\\\downarrow\\M}}\Big)\dar{\rho_{c(k)}}\\
	\Map_{\cS^{[1]}}\Big(\mathrel{\substack{\rho_{c(k)}\partial^hM\\\downarrow\\M}}\ ,\ \mathrel{\substack{\rho_{c(k)}\partial^hM\\\downarrow\\M}}\Big)\rar{(\rho_{c(k)}\inc)^*}&\Map_{\cS^{[1]}}\Big(\mathrel{\substack{\rho_{c(k)}\partial^{vh}M\\\downarrow\\\partial^vM}}\ ,\ \mathrel{\substack{\rho_{c(k)}\partial^hM\\\downarrow\\M}}\Big),
\end{tikzcd}
\]
so these vertical maps also have diverging connectivities. It follows that the induced map  \[
\Map_{\cS^{[1]}_{(\partial^{vh}M\ra \partial^{v}M)/}}\Big(\mathrel{\substack{\partial^hM\\\downarrow\\M}}\ ,\ \mathrel{\substack{\partial^hM\\\downarrow\\M}}\Big)\xlra{\rho_{c(k)}} \Map_{\cS^{[1]}_{(\rho_{c(k)}\partial^{vh}M\ra \partial^{v}M)/}}\Big(\mathrel{\substack{\rho_{c(k)}\partial^hM\\\downarrow\\M}}\ ,\ \mathrel{\substack{\rho_{c(k)}\partial^hM\\\downarrow\\M}}\Big)
\]
on horizontal fibres has diverging connectivity too. Restricting to group-like components (which preserves the property of having diverging connectivity) gives the map $\Aut_{\partial^v}(M,\partial^h)\ra\Aut_{\partial^v}(M,\rho_{c(k)} \partial^h)$ we are interested in, so its connectivity indeed diverges with $k$.

To deduce the claim for the map \emph{with} the ${\fQ}$-subscripts, we consider the commutative triangle
\[\begin{tikzcd}[row sep=0.1cm]
	\BEmb^\theta_{\partial^v}(M;\ell^v)^\times\arrow[rr]\arrow[dr, bend right=10]&& \oB T_k\Emb^\theta_{\partial^v}(M,\partial^h;\ell^v)^\times\arrow[dl,bend left=10]\\
	&\BAut_{\partial^v}(M,\partial^h)^{\ell^v}.&
\end{tikzcd}\]
We have already shown that the connectivity of the horizontal map diverges with $k$, so the same holds for the map between the diagonal fibres at all basepoints. Since rationalisation preserves connectivity (see \cref{sec:rat-inheritance}), this implies that the connectivity of the map between fibres after rationalisation diverges with $k$. The latter in turn implies that the connectivity of the horizontal map in the triangle after fibrewise rationalisation diverges with $k$, as claimed.
\end{proof}

Given \cref{lem:reduction-to-tk}, \cref{thm:pullback-decomp} follows from the following version of the statement for the sequence \eqref{equ:modified-tk-sequence} by taking limits. Its proof occupies the remainder of this section.

\begin{thm}\label{thm:pullback-decomp-Too}Under the assumptions of \cref{thm:pullback-decomp} the sequence \eqref{equ:modified-tk-sequence} of towers can be completed to a levelwise pullback square of towers.
\end{thm}

\subsubsection{Proof of \cref{thm:pullback-decomp-Too}}
We begin by considering the commutative diagram of towers
\begin{equation}\label{equ:initialdiagramfordecom}
\begin{tikzcd}[column sep=0.4cm, row sep=0.4cm]
	T_\bullet\Emb_{\partial^v}(M)^\times\rar\dar&T_\bullet \Emb^{p}_{\partial^v}(M)^\times \dar\rar&\Aut_{\partial^v}(M,\rho_{c(\bullet)}\partial^h)\dar\\
	 \Aut_{\partial^v}(TM) = \Aut_{\partial^v}^{/\BTop(d)}(M) \rar& \Aut_{\partial^v}^{/\BAut_{\leq \bullet}(E_d)}(M) \rar{\mathrm{forget}}& \Aut_{\partial^v}(M)
\end{tikzcd}
\end{equation}
of $E_1$-groups whose left-hand square is the pullback square relating topological and particle embedding calculus from \cite[Section 5.7]{KKoperadic}. Concretely, this pullback is the pullback of group-like components of the pasting of the pullbacks $\circled{3}$ and $\circled{4}$ in Equation (127) in loc.cit.\,in the case where $M$ and $N$ in loc.cit.\,are both given as the noncompact nullbordism $M\backslash(\partial^hM)\colon \mathrm{int}(\partial^vM)\leadsto\varnothing$. Its bottom row is induced by postcomposition with the map of towers $\BTop(d)\ra \BAut_{\le \bullet}(E_{d})$ that already featured in the proof of \cref{lem:embcalc-and-particle-bdy-agree}. The right-hand square in \eqref{equ:initialdiagramfordecom} results from  \cref{thm:fake-boundary-truncated}. 

We now extend \eqref{equ:initialdiagramfordecom} to a commutative diagram of towers
\[\hspace{-0.2cm}
\begin{tikzcd}[column sep=-0.4cm, row sep=0.1cm]
&T_\bullet\Emb_{\partial^v}(M,\partial^h)^\times\arrow[rr]\arrow[dd]\arrow[dl]&&T_\bullet\Emb^{p}_{\partial^v}(M,\partial^h)\arrow[dd]\arrow[rr]\arrow[dl]&&\Aut_{\partial^v}(M,\partial^h)\arrow[dl]\arrow[dd,equal]\\
	T_\bullet\Emb_{\partial^v}(M)^\times\arrow[rr,crossing over]\arrow[dd]&&T_\bullet \Emb^{p}_{\partial^v}(M)^\times \arrow[dd]\arrow[rr,crossing over]&&\Aut_{\partial^v}(M,\rho_{c(\bullet)}\partial^h)\arrow[dd]&\\
	&\Aut_{\partial^v}^{/\BTop(d)}(M,\partial^h)\arrow[rr]\arrow[dl]&&\Aut^{/\BAut_{\leq \bullet}(E_d)}_{\partial^v}(M,\partial^h)\arrow[rr]\arrow[dl]&&\Aut_{\partial^v}(M,\partial^h)\arrow[dl]\\
	  \Aut_{\partial^v}^{/\BTop(d)}(M) \arrow[rr]&& \Aut_{\partial^v}^{/\BAut_{\leq \bullet}(E_d)}(M) \arrow[rr]&& \Aut_{\partial^v}(M)&
\end{tikzcd}
\]
obtained as follows: the front two squares are \eqref{equ:initialdiagramfordecom}, the two rightmost diagonal maps from $\Aut_{\partial^v}(M,\partial^h)$ to $\Aut_{\partial^v}(M,\rho_{c(\bullet)}\partial^h)$ and $\Aut_{\partial^v}(M)$ are induced by truncation and forgetting respectively, and all other entries and maps are defined by taking pullbacks along these two maps. In particular both bottom and both top faces are pullbacks by definition. Since the front left face is a pullback, it follows that the back left face is a pullback too. We now consider a variant of this pullback
\begin{equation}\label{equ:pm-square-2}
	\begin{tikzcd}[column sep=0.4cm, row sep=0.4cm] T_\bullet\Emb_{\partial^v}(M,\partial^h)^\times\rar\dar&T_\bullet\Emb^{p}_{\partial^v}(M,\partial^h)^t\dar \\
	\Aut_{\partial^v}(TM,\partial^h) = \Aut_{\partial^v}^{/\BTop(d)}(M,\partial^h)\rar&\Aut^{/\BAut_{\leq \bullet}(E_d)}_{\partial^v}(M,\partial^h)^t\end{tikzcd}
\end{equation}
which is obtained by restricting to certain components of the spaces $\smash{\Aut^{/\BAut_{\leq \bullet}(E_d)}_{\partial^v}(M,\partial^h)}$ and $\smash{T_\bullet\Emb^{p}_{\partial^v}(M,\partial^h)^\times}$, indicated by the $t$-superscripts. Namely, we restrict $\smash{\Aut^{/\BAut_{\leq \bullet}(E_d)}_{\partial^v}(M,\partial^h)}$ to the components hit by $\smash{\Aut^{/\BTop(d)}_{\partial^v}(M,\partial^h)}$, and $\smash{T_\bullet\Emb^{p}_{\partial^v}(M,\partial^h)^\times}$ to the components that are mapped to these components via the map to $\smash{\Aut^{/\BAut_{\leq \bullet}(E_d)}_{\partial^v}(M,\partial^h)}$. This restriction of components is done so that the square obtained from \eqref{equ:pm-square-2} by delooping remains a pullback. Adding tangential structures to the left column in \eqref{equ:pm-square-2}, we obtain a pullback of the form
\[
\begin{tikzcd}[column sep=0.4cm, row sep=0.4cm]
	\oB T_\bullet\Emb^\theta_{\partial^v}(M,\partial^h;\ell^v)^\times \rar\dar &\oB T_\bullet\Emb^{p}_{\partial^v}(M,\partial^h)^t\dar\\
	\BAut^\theta_{\partial^v}(TM,\partial^h;\ell^v)\rar&\BAut^{/\BAut_{\leq \bullet}(E_d)}_{\partial^v}(M,\partial^h)^t
\end{tikzcd}
\]
which we factor as a pasting of two pullbacks by taking horizontal Moore--Postnikov $1$-factorisations
\begin{equation}\label{pasting}
\begin{tikzcd}[column sep=0.4cm, row sep=0.4cm]
	\oB T_\bullet\Emb^\theta_{\partial^v}(M,\partial^h;\ell^v)^\times\rar\dar&\oB T_\bullet\Emb^{p}_{\partial^v}(M,\partial^h)^{\ell^v}\rar\dar&\oB T_\bullet\Emb^{p}_{\partial^v}(M,\partial^h)^t\dar\\
	\BAut^\theta_{\partial^v}(TM,\partial^h;\ell^v)\rar{\phi}&\BAut^{/\BAut_{\leq \bullet}(E_d)}_{\partial^v}(M,\partial^h)^{\ell^v} \rar &\BAut^{/\BAut_{\leq \bullet}(E_d)}_{\partial^v}(M,\partial^h)^t
\end{tikzcd}
\end{equation}
(here the rows are \emph{defined} as the Moore--Postnikov $1$-factorisations). By the functoriality of Moore--Postnikov 1-factorisations, the left pullback square in \eqref{pasting} maps to the middle space $\BAut_{\partial^v}(M,\partial^h)^{\ell^v}$ in the Moore--Postnikov $1$-factorisation of $\smash{\BAut^\theta_{\partial^v}(TM,\partial^h;\ell^v)\ra \BAut_{\partial^v}(M,\partial^h)}$, so we may fibrewise rationalise the left square over $\smash{\BAut_{\partial^v}(M,\partial^h)^{\ell^v}}$ to arrive at a square
\begin{equation}\label{pre-factored-square}
\begin{tikzcd}[column sep=0.4cm, row sep=0.4cm]
	\oB T_\bullet\Emb^\theta_{\partial^v}(M,\partial^h;\ell^v)^\times_{\fQ}\rar\dar&\oB T_\bullet\Emb^{p}_{\partial^v}(M,\partial^h)^{\ell^v}_{\fQ}\dar\\
	\BAut^\theta_{\partial^v}(TM,\partial^h;\ell^v)_{\fQ}\rar{\phi_{\fQ}}&\BAut^{/\BAut_{\leq \bullet}(E_d)}_{\partial^v}(M,\partial^h)^{\ell^v}_{{\fQ}}.
\end{tikzcd}
\end{equation}
In the next subsection, we will establish the following properties of the bottom horizontal map $\phi_\fQ$:
\begin{prop}\label{prop:factorisation}
Under the assumptions of \cref{thm:pullback-decomp}, we have that
\begin{enumerate}
\item both source and target of the $1$-connected map $\smash{\phi\colon \BAut^\theta_{\partial^v}(TM,\partial^h;\ell^v)\ra \BAut^{/\BAut_{\leq \bullet}(E_d)}_{\partial^v}(M,\partial^h)^{\ell^v} }$ are fibrewise nilpotent over $\BAut_{\partial^v}(M,\partial^h)^{\ell^v}$, and
\item  the fibrewise rationalisation $\smash{\phi_\fQ\colon \BAut^\theta_{\partial^v}(TM,\partial^h;\ell^v)_{\fQ}\ra \BAut^{/\BAut_{\leq \bullet}(E_d)}_{\partial^v}(M,\partial^h)^{\ell^v}_{{\fQ}}}$, considered as a map over $ \BAut_{\partial^v}(M,\partial^h)^{\ell^v}$, admits a factorisation of the form 
\[
	\smash{\BAut^\theta_{\partial^v}(TM,\partial^h;\ell^v)_{\fQ}\lra  \BAut_{\partial^v}(M,\partial^h)^{\ell^v}\lra \BAut^{/\BAut_{\leq \bullet}(E_d)}_{\partial^v}(M,\partial^h)^{\ell^v}_{{\fQ}}}.
\] 
\end{enumerate}
\end{prop}
Assuming \cref{prop:factorisation}, we finish the proof of \cref{thm:pullback-decomp-Too}. First, we show that the fibrewise rationalisation \eqref{pre-factored-square} of the left pullback square in \eqref{pasting} is still a pullback. This follows once we show that for all components $\ell\in \BAut_{\partial^v}(M,\partial^h)^{\ell^v}$, the rationalised square of fibres
\begin{equation}\label{equ:square-of-fibres}
\begin{tikzcd}[column sep=0.4cm, row sep=0.4cm]
\fib_\ell\Big(\mathrel{\substack{\oB T_\bullet\Emb^\theta_{\partial^v}(M,\partial^h;\ell^v)^\times\\\downarrow\\\BAut_{\partial^v}(M,\partial^h)^{\ell^v}}}\Big)_\bfQ \rar\dar&\fib_\ell\Big(\mathrel{\substack{ \oB T_\bullet\Emb^{p}_{\partial^v}(M,\partial^h)^{\ell^v}\\\downarrow\\\BAut_{\partial^v}(M,\partial^h)^{\ell^v}}}\Big)_\bfQ\dar\\
	\fib_\ell\Big(\mathrel{\substack{\BAut^\theta_{\partial^v}(TM,\partial^h;\ell^v)\\\downarrow\\\BAut_{\partial^v}(M,\partial^h)^{\ell^v}}}\Big)_\bfQ\rar{\phi_{\fQ}}&\fib_\ell\Big(\mathrel{\substack{\BAut^{/\BAut_{\leq \bullet}(E_d)}_{\partial^v}(M,\partial^h)^{\ell^v}\\\downarrow\\\BAut_{\partial^v}(M,\partial^h)^{\ell^v}}}\Big)_\bfQ
\end{tikzcd}
\end{equation}
is a pullback. Before rationalisation this is a pullback, because it arose by taking fibres of a map from a pullback square to a fixed space. By \cref{lem:Q-completion-pullback}, it suffices to show that the bottom map in \eqref{equ:square-of-fibres} before rationalisation is $1$-connected and has nilpotent source and target. This holds by the first part of \cref{prop:factorisation}. Now that we know that \eqref{pre-factored-square} is a pullback square, we may factor the bottom map $\phi_\fQ$ in it according to the second part of \cref{prop:factorisation} and take a pullback to write the square \eqref{pre-factored-square} as the pasting of two squares
\[
\begin{tikzcd}[column sep=0.3cm, row sep=0.3cm]
\oB T_\bullet\Emb^\theta_{\partial^v}(M,\partial^h;\ell^v)^\times_{\fQ}\rar\dar&\varphi^*(T_\bullet\Emb^{p}_{\partial^v}(M,\partial^h)^{\ell^v}_{\fQ})\rar\dar&\oB T_\bullet\Emb^{p}_{\partial^v}(M,\partial^h)^{\ell^v}_{\fQ}\dar\\
\BAut^\theta_{\partial^v}(TM,\partial^h;\ell^v)_{\fQ}\rar&  \BAut_{\partial^v}(M,\partial^h)^{\ell^v}\rar{\varphi} &\BAut^{/\BAut_{\leq \bullet}(E_d)}_{\partial^v}(M,\partial^h)^{\ell^v}_{{\fQ}}.
\end{tikzcd}
\]
The outer square is a pullback by the first part and the right square is a pullback by definition, so the left square is a pullback square as well. This left square is the pullback extension promised in \cref{thm:pullback-decomp-Too}, so the claim follows. However this assumed \cref{prop:factorisation} which we are yet to prove.

\subsubsection{Proof of \cref{prop:factorisation}}\label{sec:proof-lemma}
Consider the map $\phi$ over $\BAut_{\partial^v}(M,\partial^h)^{\ell^v}$ the claim is about 
\begin{equation}\label{equ:triangle1}
\begin{tikzcd}[column sep=-0.05cm, row sep=0cm]
	\BAut^\theta_{\partial^v}(TM,\partial^h;\ell^v)\arrow[dr, bend right=10]\arrow[rr,"\phi"]&[10pt] & \BAut^{/\BAut_{\leq \bullet}(E_d)}_{\partial^v}(M,\partial^h)^{\ell^v} \arrow[dl, bend left=10]\\
	&\BAut_{\partial^v}(M,\partial^h)^{\ell^v}.&
\end{tikzcd}
\end{equation}
The horizontal map $\phi$ and the left diagonal map are both $1$-connected by construction (they were defined as Moore--Postnikov 1-factorisations), so the right diagonal map is $1$-conneected as well. In particular, all maps in \eqref{equ:triangle1} have connected fibres and induce bijections on components. We may thus split the diagram into its components:
\[
\begin{tikzcd}[column sep=-0.05cm, row sep=0cm]
	\bigsqcup_{[\ell]\in\theta\text{-}\st}\BAut^\theta_{\partial^v}(TM,\partial^h;\ell^v)_\ell\arrow[dr, bend right=10]\arrow[rr,"\phi"]&[3pt] & \bigsqcup_{[\ell]\in\theta\text{-}\st}\BAut^{/\BAut_{\leq \bullet}(E_d),\ell}_{\partial^v}(M,\partial^h)\arrow[dl, bend left=10]\\
	&\bigsqcup_{[\ell]\in\theta\text{-}\st}\BAut^\ell_{\partial^v}(M,\partial^h).&
\end{tikzcd}
\]
Here $\theta\text{-}\st$ is defined as in \cref{sec:tangential-structures}, $\BAut^\theta(TM,\partial^h;\ell^v)_\ell\subset \BAut^\theta(TM,\partial^h;\ell^v)$ is the component corresponding to $[\ell]\in \theta\text{-}\st$, and the $\ell$-superscripts on the other two spaces indicate that we are passing to the covering spaces that make the map from $\BAut^\theta(TM,\partial^h;\ell^v)_\ell$ be $1$-connected. By general bundle yoga, we have equivalences
\begin{equation}\label{equ:bundle-yoga-equivalences}\textstyle{\BAut^\theta_{\partial^v}(TM,\partial^h;\ell^v)_\ell\simeq \frac{\Map_{\partial^v}(M,B)_{\ell}}{\Aut^\ell_{\partial^v}(M,\partial^h)}\quad\text{and}\quad \BAut^{/\BAut_{\leq \bullet}(E_d),\ell}_{\partial^v}(M,\partial^h)\simeq \frac{\overline{\Map}_{\partial^v}(M,\BAut^\id_{\leq \bullet}(E_d))_{TM}}{\Aut^\ell_{\partial^v}(M,\partial^h)}}
\end{equation}
where 
\begin{itemize}
\item the $\ell$-subscript on $\Map_{\partial^v}(M,B)_{\ell}$ indicates the component of $\ell\colon M\ra B$, 
\item the $TM$-subscript on $\Map_{\partial^v}(M,\BAut^\id_{\leq \bullet}(E_d))_{TM}$ indicates the component of the composition of $\ell\colon M\ra B$ with the composition \begin{equation}\label{equ:nullhomotopic-composition}B\xlra{\theta}\BSTop(d)\xlra{t} \BAut^\id(E_d)\xlra{\tau} \BAut^\id_{\le \bullet}(E_d),\end{equation}
\item $\overline{\Map}_{\partial^v}(M,\BAut^\id_{\leq \bullet}(E_d))_{TM}$ is the covering space of $\Map_{\partial^v}(M,\BAut^\id_{\leq \bullet}(E_d))_{TM}$ that makes the map $\smash{\Map_{\partial^v}(M,B)_{\ell}\ra \Map_{\partial^v}(M,\BAut^\id_{\leq \bullet}(E_d))_{TM}}$ given by postcomposition with \eqref{equ:nullhomotopic-composition} be $1$-connected. Note that $\Map_{\partial^v}(M,\BAut^\id_{\leq \bullet}(E_d))_{TM}\simeq \Map_{\partial^v}(M,\BAut_{\leq \bullet}(E_d))_{TM}$, since $\Map_{\partial^v}(M,-)$ sends covering spaces to inclusions of path-components.
\end{itemize}
In view of the equivalences \eqref{equ:bundle-yoga-equivalences}, to show the first part of \cref{prop:factorisation}, it suffices to show that the mapping spaces $\Map_{\partial^v}(M,B)_{\ell}$ and $\smash{\overline{\Map}_{\partial^v}(M,\BAut^\id_{\leq \bullet}(E_d))_{TM}}$ are nilpotent. This is the case: $B$ is nilpotent by assumption, $\BAut^\id_{\leq \bullet}(E_d)$ is nilpotent as it is $1$-connected, and covering spaces of relative mapping spaces from finite CW-pairs into nilpotent spaces are nilpotent by \cite[V.§5.1]{BousfieldKan}.

Regarding the second part of \cref{prop:factorisation}, we may construct the claimed factorisation of $\phi_\fQ$ one component at the time. Thus the task is to factor the left map in
\begin{equation}\label{equ:composed-map-aux}\textstyle{\frac{(\Map_{\partial^v}(M,B)_{\ell})_\bfQ}{\Aut^\ell_{\partial^v}(M,\partial^h)}\xlra{\phi} \frac{(\overline{\Map}_{\partial^v}(M,\BAut^\id_{\leq \bullet}(E_d))_{TM})_\bfQ}{\Aut^\ell_{\partial^v}(M,\partial^h)}\lra\frac{({\Map}_{\partial^v}(M,\BAut^\id_{\leq \bullet}(E_d))_{TM})_\bfQ}{\Aut^\ell_{\partial^v}(M,\partial^h)}}\end{equation}
over $\smash{\BAut^\ell_{\partial^v}(M,\partial^h)}$, as a map in the overcategory $\smash{\cS_{/\BAut^\ell_{\partial^v}(M,\partial^h)}}$. By obstruction theory, since the right map in \eqref{equ:composed-map-aux} is a covering space and the map $\smash{(\Map_{\partial^v}(M,B)_{\ell})_\bfQ/\Aut^\ell_{\partial^v}(M,\partial^h)\ra\BAut^\ell_{\partial^v}(M,\partial^h)}$ is $1$-connected by construction, it suffices to produce a factorisation of the \emph{composition} in \eqref{equ:composed-map-aux} over $\smash{\BAut^\ell_{\partial^v}(M,\partial^h)\simeq \ast/\Aut^\ell_{\partial^v}(M,\partial^h)}$, i.e.\ an $\Aut^\ell_{\partial^v}(M,\partial^h)$-equivariant nullhomotopy of the map $(\Map_{\partial^v}(M,B)_{\ell})_\bfQ\ra ({\Map}_{\partial^v}(M,\BAut^\id_{\leq \bullet}(E_d))_{TM})_\bfQ$. The latter map is induced by postcomposition with \eqref{equ:nullhomotopic-composition}, so using the natural equivalence $(\Map_{\partial^v}(M,X)_f)_\bfQ\simeq \Map_{\partial^v}(M,X_\bfQ)_{f_\bfQ}$ for nilpotent spaces $X$ and maps $f\colon M\ra X$ \cite[V.5.1]{BousfieldKan}, it suffices to show that the rationalisation of \eqref{equ:nullhomotopic-composition} is nullhomotopic. As the rationalisation map $(-)_\bfQ\colon \BAut^{\id}_{\le \bullet}(E_d)\ra  \BAut^{\id}_{\le \bullet}(E_{d,\bfQ})$ is a levelwise rational equivalence as a result of \cite[Theorems 7.9-7.10]{KKDisc}, it suffices to show that the rationalisation of 
\[\smash{B\xlra{\theta} \BSTop(d)\xlra{t}\BAut^\id(E_d)\xlra{(-)_\bfQ} \BAut^\id(E_{d,\bfQ})\lra\BAut^\id(E_{d,\bfQ})}\] is nullhomotopic. This is indeed the case, since the rationalised composition of the first three maps is nullhomotopic by assumption.

\begin{rem}\label{rem:variant}There is a variant of \cref{thm:pullback-decomp} which can in some situations be more convenient to use. It is based on the observation that, since the second map in the sequence \eqref{equ:extandable-composition} that can be completed to a pullback is $1$-connected by construction, the sequence of connected spaces resulting from \eqref{equ:extandable-composition} by restriction to the components induced by a fixed $\theta$-structure $[\ell]\in\theta\text{-}\st$ can be completed to a pullback too. By the discussion in \cref{sec:proof-lemma} this sequence has the form
\begin{equation}\label{equ:pullback-squ-single-comp}
	\smash{\BEmb^\theta_{\partial^v}(M;\ell^v)^\times_{{\fQ},\ell}\lra
	\Map_{\partial^v}(M,B_\bfQ)_{\ell_\bfQ}{/}\Aut^{\ell}_{\partial^v}(M,\partial^h)\lra \BAut^{\ell}_{\partial^v}(M,\partial^h).}
\end{equation}
Here $\ell_\bfQ$ is the composition of the map $\ell\colon M\ra B$ with the rationalisation map $B \to B_\bfQ$, and $\Aut^{\ell}_\partial(M)\subset \Aut_\partial(M)$ consists of those components that stabilise $[\ell]\in \pi_0(\Map_{\half\partial}(M,B))$, before rationalisation. Since \eqref{equ:pullback-squ-single-comp} can be completed to a pullback, the same applies to the variant
\begin{equation}\label{equ:pullback-squ-single-comp-mod}
	\smash{\BEmb^\theta_{\partial^v}(M;\ell^v)^\times_{{\fQ},\ell}\lra
	\Map_{\partial^v}(M,B_\bfQ)_{\ell_\bfQ}{/}\Aut^{\ell_\bfQ}_{\partial^v}(M,\partial^h)\lra \BAut^{\ell_\bfQ}_{\partial^v}(M,\partial^h)}
\end{equation}
involving the stabiliser $\Aut^{\ell_\bfQ}_{\partial^v}(M,\partial^h)$ of the rationalised map $[\ell_\bfQ]\in \pi_0(\Map_{\half\partial}(M,B_\bfQ))$. This uses that the second map in \eqref{equ:pullback-squ-single-comp} is pulled back from the second map in \eqref{equ:pullback-squ-single-comp-mod} along the covering space $\smash{\BAut^{\ell}_{\partial^v}(M,\partial^h)\ra \BAut^{\ell_\bfQ}_{\partial^v}(M,\partial^h)}$. The reason why the sequence \eqref{equ:pullback-squ-single-comp-mod} is sometimes more convenient is that its middle space agrees with the pathcomponent induced by $\ell_\bfQ$ of the \emph{full} orbit space $\Map_{\partial^v}(M,B_\bfQ)/\Aut_{\partial^v}(M,\partial^h)$. In the special case of \cref{bigthm:pullback}, the variant \eqref{equ:pullback-squ-single-comp-mod} takes the form 
\begin{equation}\label{equ:pullback-squ-single-comp-mod-homeo-pre}
	\smash{\BHomeo^\theta_{\partial}(M;\ell_{\half\partial})_{{\fQ},\ell}\lra
	\big(\Map_{\half\partial}(M,B_\bfQ){/}\Aut_{\partial}(M)\big)_{\ell_\bfQ}\lra 	\BAut^{\ell_\bfQ}_{\partial}(M),}
\end{equation}
which is the version that is suited best for the applications presented in the next section. One of the consequences we will in particular use is the following: since \eqref{equ:pullback-squ-single-comp-mod-homeo-pre} can be completed to a pullback square, taking fibres and using $\Map_{\half\partial}(M,B_\bfQ)_{\ell_\bfQ}\simeq \fib_{\ell_\bfQ}((\Map_{\half\partial}(M,B_\bfQ){/}\Aut_{\partial}(M))_{\ell_\bfQ}\ra \smash{\BAut^{\ell_\bfQ}_{\partial}(M)})$ implies that there exists a dashed lift in the diagram
\begin{equation}\label{equ:dashed-lift-mainthm}
\begin{tikzcd}[row sep=0.4cm]
&\BHomeo^\theta_{\partial}(M;\ell_{\half\partial})_{{\fQ},\ell}\dar\\
\Map_{\half\partial}(M,B_\bfQ)_{\ell_\bfQ}\rar\arrow[ur,dashed, bend left=10]&\big(\Map_{\half\partial}(M,B_\bfQ){/}\Aut_{\partial}(M)\big)_{\ell_\bfQ}.
\end{tikzcd}
\end{equation}
\end{rem}

\section{Applications}
We explain various applications of the pullback decomposition established in the previous section to the study of topological fibre bundles. For simplicity and since we consider it the most interesting case, we phrase them for spaces of homeomorphisms as in \cref{bigthm:pullback}, but similar arguments apply to spaces of self-embeddings as in the more general \cref{thm:pullback-decomp}. 

\subsection{Tautological classes}\label{sec:tautological-classes}
As preparation for the applications of the pullback decomposition, we discuss a source of characteristic classes of fibrations with certain extra structure, inspired by \cite[Section 4.3]{RWSignatureNote} and \cite[Section 2.2]{WeissDalian}. All cohomology groups are taken with coefficients in a fixed commutative ring $\bfk$ which we omit from the notation. 

\medskip

\noindent Fix a compact oriented topological $d$-manifold $M$ and a map $\ell\colon M\ra B$ to a fixed space $B$. We first describe characteristic classes associated to the following three pieces of data:
\begin{enumerate}[leftmargin=*]
	\item an oriented relative fibration $\smash{(M,\partial M)\ra (E,\partial M\times X)\xra{\pi} X}$,
	\item a map $\ell^\pi\colon E\ra B$,
	\item a homotopy $\smash{h^\pi}$ of maps $\partial M\times X \to B$ between $\smash{\ell^\pi|_{\partial  M\times X}}$ and $(\ell|_{\partial M}\circ\pr_1)$.
\end{enumerate}
Given such a triple $\smash{(\pi,\ell^\pi,h^\pi)}$ and a cohomology class $c\in \oH^*(B)$, there is a characteristic class $\kappa_c(\pi)\in\oH^{*-d}(X)$ defined as follows: writing $-M$ for $M$ equipped with the opposite orientation, form the oriented fibration $\overline{\pi}\colon \smash{\overline{E}}\coloneq (E\cup_{\partial M\times X}(-M\times X))\ra X$ with fibre the double $M\cup_{\partial M}(-M)$, extend the map $\ell^\pi$ to a map $\overline{\ell} \colon \overline{E}\ra B$ using the map $(\ell\circ \pr_1)\colon M\times X\ra B$ and the homotopy $h^\pi$, pull back $c$ to a class in $\oH^*(\overline{E})$, and fibre integrate this class along $\overline{\pi}$ using the orientation to obtain a class $\kappa_c(\pi)\in\oH^{*-d}(X)$. This gives rise to a morphism $\kappa_{(-)}(\pi)\colon \oH^*(B)\ra\oH^{*-d}(X)$. A priori it depends on the choice of $\ell$, but actually turns out to only depend on the restriction $\ell|_{\partial M}$ as long as $*-d>0$ (to see this, note that for another map $\ell'\colon M\ra B$ with a homotopy between $\ell'|_{\partial M}$ to $\ell|_{\partial M}$, the resulting fibration $(\overline{\pi},\overline{\ell}')$ is as bordant---as an oriented fibration over $X$ with a map to $B$---to the fibrewise disjoint union of $(\overline{\pi},\overline{\ell}')$ and the trivial fibration $((M\cup_\partial (-M) \times X),(\ell\cup\ell') \circ \pr_2)$.)

Triples $\smash{(\pi,\ell^\pi,h^\pi)}$ as above are classified by the orbits $\Map_{\partial}(M,B){/} \Aut^+_{\partial}(M)$ of the precomposition action of the $E_1$-group $\Aut^+_{\partial}(M)$ of orientation-preserving homotopy automorphisms extending the boundary inclusion on the space of maps $M\ra B$ extending $\ell|_{\partial M}$. Applied to the universal triple, the above construction will thus give a morphism \[\kappa_{(-)}\colon \oH^*(B) \lra \oH^{*-d}\big(\Map_{\partial}(M,B){/} \Aut_{\partial}^+(M)\big).\]
Now suppose we have fixed an embedded disc $D^{d-1}\subset \partial M$ in the boundary and the homotopy $\smash{h^{\pi}}$ is only defined on the subspace $\half\partial M\times X\subset \partial M\times X$ where $\half\partial M\coloneq \partial M\backslash\interior(D^{d-1})$ ; we denote such a homotopy on this subspace by $\smash{h^{\pi}_{\half}}$. Given only the partial homotopy $\smash{h^{\pi}_{\half}}$, the class $\kappa_c(\pi)$ for $c\in \oH^*(B)$ is a priori no longer defined, since we made use of the full homotopy $\smash{h^{\pi}}$ to extend $\smash{\ell^\pi}$ to $\smash{\overline{E}}$. However, we will show that as long as $c$ comes with a cup-product decomposition which is trivialised over a point. More precisely, equipping $B$ with the basepoint given as the image under $\ell$ of a fixed point $\ast\in \partial D^{d-1}\subset \half\partial M$, we will construct a morphism  $\kappa_{(-)\otimes(-)}(\pi)\colon \oH^*(B,\ast)^{\otimes 2}\ra \oH^{*-d}(X)$ which only requires the partial homotopy $\smash{h^{\pi}_{\half}}$ and agrees, whenever the full homotopy is defined, with the composition
of the morphism $\oH^*(B,\ast)^{\otimes 2} \ra \oH^{*}(B)$ induced by restriction and taking cup-product, with the morphism $\kappa_{(-)}(\pi)\colon \oH^*(B)\ra\oH^{*-d}(X)$ from above. Triples $\smash{(\pi,\ell^\pi,h^\pi_{\half})}$ are classified by the orbits $\Map_{\half\partial}(M,B){/} \Aut^+_{\partial}(M)$ of the space of maps from $M$ to $B$ extending $\ell|_{\half\partial}$. The universal instance of the construction thus gives the upper row in a commutative diagram
\begin{equation}\label{equ:compatibility}
\begin{tikzcd}[row sep=0.35cm,column sep=1.2cm]
	\oH^*(B,\ast)^{\otimes 2} \rar{\kappa_{(-)\otimes(-)}}\arrow[d,"(-)\cup(-)",swap]& \oH^{*-d}\big( \Map_{\half\partial}(M,B){/} \Aut^+_{\partial}(M)\big)\arrow[d]\\
	\oH^*(B)\rar{\kappa_{(-)}}&\oH^{*-d}\big( \Map_{\partial}(M,B){/} \Aut^+_{\partial}(M)\big)
\end{tikzcd}
\end{equation}
whose right column is induced by the forgetful map $\Map_{\partial}(M,B)\ra \Map_{\half\partial}(M,B)$. 

We now explain the promised construction of the morphism $\smash{\kappa_{(-)\otimes(-)}(\pi)\colon \oH^*(B,\ast)^{\otimes 2}\ra \oH^{*-d}(X)}$. Writing $\smash{\overline{E}_{\half}}$ for the total space of the oriented relative fibration 
\vspace{-0.2cm}
\[(M\cup_{\half\partial M}(-M),D^{d-1}\cup_{\partial D^{d-1}}(-D^{d-1}))\ra (E\cup_{\half\partial M\times X}(-M\times X),(D^{d-1}\cup_{\partial D^{d-1}}(-D^{d-1}))\times X)\overset{\overline{\pi}_{\half}}{\ra} X\] and viewing $\smash{\overline{E}_{\half}}$ as a subspace of $\smash{\overline{E}}$ with complement $D^d\times X$, the morphism is defined as the long composition in the commutative diagram
\begin{equation}\hspace{-0.5cm}\label{equ:constructed-decomposed-kappa}
	\begin{tikzcd}[column sep=0.6cm,row sep=0.6cm]
		\oH^*(B,\ast)^{\otimes 2}\rar &\oH^*(\overline{E}_{\half},\ast\times X)^{\otimes 2} \arrow[dr,dashed]&&[5pt]&\\
		&\oH^*(\overline{E},\ast\times X)^{\otimes  2}\arrow[u,"\res^{\otimes 2} "]\arrow["\cup",r]&\oH^*(\overline{E},\ast\times X)\arrow["\res",r]& \oH^*(\overline{E})\rar{\int_{\overline{\pi}}}&\oH^{*-d}(X).
	\end{tikzcd}
\end{equation}
Here the upper-left horizontal map is defined by pulling back along the extension of $\ell^\pi\colon E\ra B$ to a map $\overline{E}_{\half}\ra B$ by $\ell\circ\pr_1\colon M\times X\ra B$ using the homotopy $\smash{h_{\half}^{\pi}}$. The dashed map results from:

\begin{lem}\label{lem:res-surj} The middle vertical map $\res^{\otimes2}$ in \eqref{equ:constructed-decomposed-kappa} is surjective and there exists a (necessarily unique) dashed map as indicated that makes the diagram commutative.
\end{lem}

\begin{proof}To show surjectivity of $\res$ and thus of $\res^{\otimes 2}$, since $\oH^*(\overline{E},\ast\times X)\cong \oH^*(\overline{E},D^d\times X)\cong \oH^*(\overline{E}_{\half},S^{d-1}\times X)$ by excision, it suffices by the long exact sequence of a triple to show that the upper horizontal map in the commutative diagram
\[
	\begin{tikzcd}[row sep=0.5cm]
		\oH^*(\overline{E}_{\half},\ast\times X)\rar{\res} \arrow["\res",d,swap]&\oH^*(S^{d-1}\times X,\ast\times X)\arrow["\res",d,swap]\arrow[dr,"\cong"]&[10pt]\\
		\oH^*(\overline{E}_{\half})\rar{\res}&\oH^*(S^{d-1}\times X)\arrow[r,"(-)/{[S^{d-1}]}",swap]&\oH^{*-(d-1)}(X)
	\end{tikzcd}
\]
is zero. Thinking of the slant product as fibrewise integration along the trivial bundle, the bottom composition is trivial by the ``fibrewise Stokes theorem'' (a direct consequence of the construction of fibrewise integration via the Serre spectral sequence), so the upper horizontal map is trivial as well. Using the surjectivity, the existence of the dashed map in \eqref{equ:constructed-decomposed-kappa} follows once we show that the kernel of $\res^{\otimes 2}$ has trivial cup-products. Identifying  $\oH^*(\overline{E},\ast\times X)$ with $\oH^*(\overline{E}_{\half},S^{d-1}\times X)$ as above, this kernel is spanned by elements of the form $x\otimes \delta(y)$ and $\delta(y)\otimes x$ where $\delta\colon \oH^{*-1}(S^{d-1}\times X,*\times X)\ra \oH^*(\overline{E}_{\half},S^{d-1}\times X)$ is the connecting homomorphism, so the claim follows from the general fact that cup products of two relative cohomology classes vanish if one is in the image of the connecting homomorphism or, equivalently,  trivial in absolute cohomology.
\end{proof}

\begin{rem}Going through the construction, one sees that $\kappa_{(-)\otimes(-)}(\pi)\colon \oH^*(B,\ast)^{ \otimes 2}\ra \oH^{*-d}(X)$ descends to a morphism on $\oH^*(B,\ast)\otimes_{\oH^*(B,\ast)}\oH^*(B,\ast)$. Note that as $R\coloneq\oH^*(B,\ast)$ does not have a unit unless it is trivial, $R\otimes_RR$ is typically not isomorphic to $R$.
\end{rem}
 
\begin{rem}\label{rem:formula-kappa-trivial-bundle}For later applications, we record a formula for the pullback of the class $\kappa_{c\otimes c'}\in\oH^{*-d}(\Map_{\half\partial}(M,B)/\hAut^+_\partial(M))$ for $\smash{c\otimes c'\in\oH^{*}(B,\ast)^{\otimes 2}}$ along the quotient map $\Map_{\half\partial}(M,B)\ra \Map_{\half\partial}(M,B)/\hAut^+_\partial(M)$, for simplicity in the case $\bfk=\bfQ$. The formula involves the map \[\smash{\ev_\ell\colon (M\cup_{\half\partial}(-M))\times \Map_{\half\partial}(M,B)\lra B}\] given by the evaluation map on $M$ and $(\pr_1\circ \ell)$ on $-M$, a choice of a homogenous basis $(b_i)_i\subset{\widetilde{\oH}_*(M\cup_{\half\partial}(-M);\bfQ)}$, and its dual basis ${(b^\vee_i)_i}\subset \widetilde{\oH}^*(M\cup_{\half\partial}(-M);\bfQ)$. In these terms, viewing $\smash{\widetilde{\oH}_*(M\cup_{\half\partial}(-M);\bfQ)}$ as a subspace of $\smash{\widetilde{\oH}_*(D(M);\bfQ)}$ via the inclusion $M\cup_{\half\partial}(-M)\subset (M\cup_{\half\partial}(-M)\cup_{S^{d-1}}D^d\cong D(M)$, we have the following identity in $\oH^{*-d}(\Map_{\half\partial}(M,B);\bfQ)$
\[\textstyle{{\kappa_{c\otimes c'}=\sum_{\{(i,j)\colon |b_i|+|b_j|=d\}}(-1)^{|b_j||c|-d}\langle b^\vee_i\cup b^\vee_j,[D(M)]\rangle\cdot (\ev_\ell^*(c)/b_i\cup \ev_\ell^*(c')/b_j)}},\]
 where $-/-$ denotes the slant product. This follows from a straight-forward check, using that fibre integration along a trivial bundle is given by the slant product with the fundamental class of the fibre. Moreover, if $\partial M\cong S^{d-1}$, so $\half\partial M\cong D^{d-1}\simeq\ast$ and $D(M)\cong \overline{M}\sharp (-\overline{M})$ for $\smash{\overline{M}\coloneq M\cup_{S^{d-1}}D^d}$, the above formula can be rewritten in terms of a choice of homogenous basis  $(a_i)_i\subset\smash{\widetilde{H}_*(M;\bfQ)}$ and the evaluation map $\ev\colon M\times \Map_{\ast}(M,B)\ra B$ as \[\smash{\textstyle{\kappa_{c\otimes c'}=\sum_{\{(i,j)\colon |a_i|+|a_j|=d\}}(-1)^{|a_j||c|-d}\langle a^\vee_i\cup a^\vee_j,[\overline{M}]\rangle\cdot (\ev^*(c)/a_i\cup \ev^*(c')/a_j)}}.\]
\end{rem}

\subsection{Modifying vertical tangent bundles} \label{sec:modifying-vertical-tangent} Recall that \cref{bigthm:pullback} in its variant from \eqref{equ:pullback-squ-single-comp-mod-homeo-pre} said the following: For any smoothable compact $d$-manifold $M$ of dimension $d\ge5$ with $M$ and $\partial M$ $2$-connected, a choice of disc $D^{d-1} \subset \partial M$ in its boundary, an oriented tangential structure $\theta \colon B \to \BSTop(d)$ such that $B$ is nilpotent and $\theta$ factors through $\BSTop(d-2) \to \BSTop(d)$ after rationalisation, and a $\theta$-structure $\ell\colon M\ra B$ on $TM$, the composition
\begin{equation}\label{equ:pullback-squ-single-comp-mod-homeo}
	\BHomeo^\theta_{\partial}(M;\ell_{\half\partial})_{{\fQ},\ell}\lra
	\big(\Map_{\half\partial}(M,B_\bfQ){/}\Aut_{\partial}(M)\big)_{\ell_\bfQ}\lra \BAut^{\ell_\bfQ}_{\partial}(M)
\end{equation}
can be completed to a pullback square. The action of $\Aut_{\partial}(M)$ on the mapping space is by precomposition, so it commutes with the action by postcomposition of the endomorphisms $\End_{\half\partial}(B_\bfQ)$ of $(\ell_{\half})_\bfQ\colon \half\partial M\ra B_\bfQ$ in the undercategory $\cS_{\half\partial M/}$. Restricted to the components \[\smash{\End^{\ell_\bfQ}_{\half\partial}(B_\bfQ)\subset \End_{\half\partial}(B_\bfQ)}\] that stabilise $\smash{[\ell_\bfQ]\in\pi_0(\Map_{\half\partial}(M,B_\bfQ))}$, the action restricts to an action on the middle space of \eqref{equ:pullback-squ-single-comp-mod-homeo} in $\cS_{/\smash{\BAut^{\ell_\bfQ}_{\partial}(M)}}$. The action can also be described in terms what it classifies: the middle space in \eqref{equ:pullback-squ-single-comp-mod-homeo} classifies triples $\smash{(\pi,\ell^\pi,h_{\half}^\pi)}$ as in \cref{sec:tautological-classes} (set $B=B_\bfQ$ and $\ell=\ell_\bfQ$) and $\smash{\chi\in \End^{\ell_\bfQ}_{\half\partial}(B_\bfQ)}$ acts by fixing $\pi$ and postcomposing $\smash{\ell^\pi}$ and $\smash{h_{\half\partial}^\pi}$ with $\chi$. Since \eqref{equ:pullback-squ-single-comp-mod-homeo} can be completed to a pullback, we may use the map $\End_{\cS_{/W}}(Z)\ra \End_{\cS_{/Y}}(X)$ for pullbacks $X\simeq Z\times_W Y$ to conclude the following:
	
\begin{prop}\label{prop:actions}
Under the assumptions of \cref{bigthm:pullback}, the action of the endomorphisms $\smash{\End^{\ell_\bfQ}_{\half\partial}(B_\bfQ)}$ on the space $\smash{(\Map_{\half\partial}(M,B_\bfQ){/}\Aut_{\partial}(M))_{\ell_\bfQ}}$ by postcomposition lifts to an action on $\BHomeo^{\theta}_\partial(M;\ell_{\half})_{\fQ,\ell}$.
\end{prop}

\begin{rem}Roughly speaking and ignoring boundary conditions, \cref{prop:actions} allows us to modify vertical tangent bundles of fibre bundles. Suppose we are given a topological $M$-bundle $\pi\colon E\ra X$ with $\theta$-structure $\ell^\pi\colon E\ra B$ on its vertical tangent bundle and a self-map $\chi\colon B_\bfQ\ra B_\bfQ$ with $\chi\circ \ell_\bfQ\simeq  \ell_\bfQ$. Then using the action we obtain a new topological $M$-bundle (in a certain rational sense), whose underlying fibration is the same as that of the original bundle but now whose rational vertical tangent bundle turned from  $(\theta \circ \ell^\pi)\colon E\ra \BSTop(d)_\bfQ$ to $(\theta \circ\chi\circ \ell^\pi)\colon E\ra \BSTop(d)_\bfQ$.\end{rem}

\subsubsection{Action on tautological classes}\label{label:action-on-kappa} Next we explain the effect of the action on $\BHomeo^{ \theta}_\partial(M;\ell_\half)_{{\fQ},\ell}$ by $\smash{\End^{\ell_\bfQ}_{\half\partial}(B_\bfQ)}$ on various characteristic classes. First we consider the classes in the image of 
\[\smash
	{\oH^*(B,\ast;\bfQ)^{\otimes 2}\xrightarrow{\kappa_{(-)\otimes(-)}} \oH^{\ast-d}(\Map_{\half\partial}(M,B_\bfQ){/}\Aut_{\partial}(M);\bfQ)\lra\oH^\ast(\BHomeo^{ \theta}_\partial(M;\ell_\half)_{{\fQ},\ell};\bfQ)}
\]
where the first map is the construction in \cref{sec:tautological-classes} applied to $B=B_\bfQ$ and $\ell=\ell_\bfQ$, and the second map is induced by the second map in \eqref{equ:pullback-squ-single-comp-mod-homeo}. By construction, both maps are equivariant with respect to the action of $\smash{\End^{\ell_\bfQ}_{\half\partial}(B_\bfQ)}$, so for $c,c' \in \oH^*(B,\ast;\bfQ)$ and $\chi \in \smash{\End^{\ell_\bfQ}_{\half\partial}(B_\bfQ)}$ we have \[\chi^*\kappa_{c \otimes c'} = \kappa_{(\chi^*c) \otimes (\chi^*c')}\]
in the rational cohomology of  $\smash{\Map_{\half\partial}(M,B_\bfQ){/}\Aut_{\partial}(M)}$ and thus also in that of $\BHomeo^{\theta}_\partial(M;\ell_\half
)_{\ell,{\fQ}}$. 

Another source of characteristic classes for $\smash{\BHomeo^{ \theta}_\partial(M;\ell_\half)_{{\fQ},\ell}}$ is the composition
\[\hspace{-.1cm}\smash
{\oH^*(\BSTop(d);\bfQ)\xrightarrow{\kappa_{(-)}} \oH^{\ast-d}(\smash{\Map_{\partial}(M,\BSTop(d)_\bfQ){/}\Aut_{\partial}(M)};\bfQ)\ra\oH^{\ast-d}(\BHomeo^{ \theta}_\partial(M;\ell_\half)_{\fQ,\ell};\bfQ)}
\]
where the first map is an instance of the first construction from \cref{sec:tautological-classes} and the second map is induced by the fibrewise rationalisation over $\BAut_\partial(M)$ of the forgetful composition
\begin{equation}\label{equ:forgetful-compo}
	\BHomeo^{ \theta}_\partial(M;\ell_\half)_{\ell}\lra \BHomeo_\partial(M)\lra \Map_{\partial}(M,\BSTop(d)){/}\Aut_{\partial}(M).
\end{equation} 
Here the second map is induced by taking  vertical tangent bundles. As a result of \eqref{equ:compatibility}, the two sources of classes fit into a commutative diagram
\[\begin{tikzcd}[row sep=-0.15cm]
	&\arrow[dl,"\theta^\ast\otimes\theta^\ast",swap, bend right=5]\oH^*(\BSTop(d),\ast;\bfQ)^{\otimes 2} \arrow[dr,"\cup", bend left=5]&\\
	\oH^*(B,\ast;\bfQ)^{\otimes 2} \arrow[dr,"\kappa_{(-)\otimes(-)}",bend right=5,swap]&
	&\oH^*(\BSTop(d);\bfQ)\arrow[dl,"\kappa_{(-)}",bend left=5]\\
	&\oH^{*-d}(\BHomeo^{ \theta}_\partial(M;\ell_\half)_{\ell,{\fQ}};\bfQ).&
\end{tikzcd}\]
This gives a procedure to compute the action of $\chi\in \smash{\End^{\ell_\bfQ}_{\half\partial}(B_\bfQ)}$-action on those classes $c\in \oH^{*}(\BSTop(d);\bfQ) $ that are decomposable: for $c=\sum c_i \cup c_i'$ with $c_i,c_i'\in \oH^{*>0}(\BSTop(d);\bfQ)$, we have 	\begin{equation}\label{equ:decom-action}\textstyle{\chi^* \kappa_c = \sum \kappa_{(\chi^* \theta^* c_i) \otimes (\chi^* \theta^* c'_i)}}\in\oH^{*-d}(\smash{\BHomeo^{\theta}_\partial(M;\ell_\half)_{{\fQ},\ell}}; \bfQ).\end{equation}

\begin{ex}\label{ex:action-on-stable-classes}The previous discussion allows us in particular to determine the action of $\smash{\End^{\ell_\bfQ}_{\half\partial}(B_\bfQ)}$ on all classes $\kappa_c\in  \smash{\oH^{*-d}(\BHomeo^{ \theta}_\partial(M;\ell_{\half})_{\bfQ,\ell};\bfQ)}$ for which $c\in\oH^{*}(\BSTop(d);\bfQ)$ pulls back along the stabilisation map $\BSTop(d) \to \BSTop$. To explain how, recall that the ring $\oH^*(\BSTop;\bfQ)\cong\bfQ[\cL_1,\cL_2,\ldots]$ is polynomial on the Hirzebruch $L$-classes $\cL_i\in \oH^{4i}(\BSTop;\bfQ)$. If $c$ is a sum of decomposable monomials in the classes $\cL_i$ of positive degrees, then the action is described by \eqref{equ:decom-action}. Indecomposable classes $\kappa_{\cL_i}\in \oH^{4i-d}(\BHomeo^{ \theta}_\partial(M;\ell_{\half})_{{\fQ},\ell};\bfQ)$ on the other hand are acted upon trivially, since these classes pulls back from $\BAut_\partial(M)$. To see the latter claim, note that $\kappa_{\cL_i}$ is by construction pulled back along the composition \eqref{equ:forgetful-compo}, fibrewise rationalised over $\BAut_\partial(M)$. By the family signature theorem (see e.g.\,\cite[Theorem 2.6]{RWfamilysignature}), the class $\kappa_{\cL_i}\in \oH^{4i-d}(\BHomeo_\partial(M)_{\fQ};\bfQ)$ is pulled back from $\BAut_\partial(M)$, so the same holds for its image in  $\oH^{4i-d}(\BHomeo^{ \theta}_\partial(M;\ell_{\half})_{\fQ,\ell};\bfQ)$.
 \end{ex}
 
\begin{rem}\label{rem:subtlety-summand} There is a subtlety in the final part in \cref{ex:action-on-stable-classes} that we have not yet addressed. The family signature theorem applies to $\kappa_{\cL_i}$ viewed as a class in $\oH^{4i-d}(\BHomeo_\partial(M);\bfQ)$, not when considered as a class in the cohomology $\oH^{4i-d}(\BHomeo_\partial(M)_{\fQ};\bfQ)$ of the fibrewise rationalisation. A priori these groups differ since we have not shown that $\BHomeo_\partial(M)$ is fibrewise nilpotent over $\BAut_\partial(M)$. There are several ways to circumvent this issue, for example one could rely on \cref{rem:nilpotency-issue} \ref{enum:future-nilpotency} for the missing fibrewise nilpotency claim. Alternatively, one can argue along the lines of \cref{rem:nilpotency-issue} \ref{enum:no-nilpotency-needed} as follows: it suffices to show that $\kappa_{\cL_i}$ lies in the canonical $\oH^{4i-d}(\BHomeo_\partial(M);\bfQ)$-summand of $\oH^{4i-d}(\BHomeo_\partial(M)_{\fQ};\bfQ)$ (see \cref{sec:non-nilpotent-completion}), which---by the naturality of this summand---follows from the same statement for any of the spaces in the composition \eqref{equ:forgetful-compo} along which $\kappa_{\cL_i}$ pulls back. The space $\Map_{\partial}(M,\BSTop(d)){/}\Aut_{\partial}(M)$ is fibrewise nilpotent over $\BAut_\partial(M)$, so its cohomology consists entirely of the summand in question.
 \end{rem}

\subsection{Pontryagin--Weiss classes}\label{sec:PW-classes}

One of the main application of \cref{bigthm:pullback} we carry out in this work is the following result. It is a reformulation of  \cref{bigthm:geometric}.

\begin{thm}\label{thm:PW-classes}
For all $d\ge6$ and $k\ge1$, the $k$th rational Pontryagin class $p_k\in \oH^{4k}(\BSTop(d);\bfQ)$ is nontrivial and evaluates nontrivially on the image of $\pi_{4k}(\BSTop(d))$ under the Hurewicz morphism.
\end{thm}

In addition to \cref{bigthm:pullback}, the proof of \cref{thm:PW-classes} relies on the work of Galatius and Randal-Williams on stable moduli spaces of even-dimensional manifolds \cite{GRWII} and the following auxiliary lemma, inspired by \cite[Section 6]{WeissDalian}:

\begin{lem}\label{lem:hurewicz-nontriviality}Fix a cohomology class $c\in \oH^k(\BSTop(d);\bfQ)$ of degree $k\ge 2$. If there is a tangential structure $\theta\colon B\ra \BSTop(d)$ with $B$ nilpotent, a compact $d$-manifold $M$ with $d\ge3$, an embedded disc $D^{d-1}\subset \partial M$, and a $\theta$-structure $\ell$ on $M$ such that
\begin{enumerate}[leftmargin=*]
	\item \label{enum:hurewicz-nontriviality-i}the pullback of $c$ along $\theta$ vanishes,
	\item  \label{enum:hurewicz-nontriviality-i.5}the map $\theta\colon B\ra \BSTop(d)$ is injective on $\pi_{k-1}(-)_\bfQ$, and
	\item \label{enum:hurewicz-nontriviality-ii}the class $\kappa_c\in \oH^{k-d}(\Map_\partial(M,\BSTop(d)){/} \Aut_\partial(M);\bfQ)$ is nontrivial when pulled back to
	\[\frac{\Map_\partial(M,\BSTop(d)_\bfQ)\times_{\Map_{\half\partial}(M,\BSTop(d)_\bfQ) }\Map_{\half\partial}(M,B_\bfQ)}{\Aut_\partial(M)},\]	
	
\end{enumerate}
then the class $c$ evaluates nontrivially on $\pi_k(\BSTop(d))$.
\end{lem}

\begin{proof}
Represent $c$ by a map $c\colon \BSTop(d)\ra K(\bfQ,k)$ and use \ref{enum:hurewicz-nontriviality-i} to choose a nullhomotopy of $(c\circ\theta)$ to get a class $\widetilde{c}\in \oH^{k-1}(\fib(\theta);\bfQ)$ that transgresses to $c$. By \ref{enum:hurewicz-nontriviality-i.5}, the map $\pi_{k}(\BSTop(d))_\bfQ \to \pi_{k-1}(\fib(\theta))_\bfQ$ is surjective, so it suffices to show that $\widetilde{c}$ is detected on $\pi_{k-1}(\fib(\theta))$. To do so, we write $X$ choose a map from a space $X$ into the pullback in \ref{enum:hurewicz-nontriviality-ii} such that the pullback of $\kappa_c$ along this map is nontrivial (for instance $X$ could be the pullback itself). This map classifies a relative fibration $\pi\colon (E,\partial M\times X)\ra X$ with fibre $(M,\partial M)$, together with a map $\ell^\pi\colon E\ra B_\bfQ$, a homotopy $\smash{h^{\pi}_{\half}}$ of maps  $ \half\partial M\times X\ra B_\bfQ$ between $\ell^\pi|_{\half\partial M\times X}$ and $\ell|_{\half\partial M}\circ \pr_1$, and an extension of the postcomposition of this homotopy with $\theta$ to a homotopy of maps $ \partial M\times X\ra \BSTop(d)_\bfQ$ between $\theta \circ\ell^\pi|_{ \partial M\times X}$ and $\theta\circ \ell|_{\partial M}\circ \pr_1$. From the construction in \cref{sec:tautological-classes}, we obtain a commutative diagram
\[
\begin{tikzcd}[column sep=0.5cm, row sep=0.5cm]
*\times X \rar{\subset}\arrow[d,equal]&  (D^{d-1}\cup_{\partial D^{d-1}}D^{d-1})\times X\rar\dar{\subset}& E\cup_{\half\partial M\times X}(-M\times X)\rar\dar{\subset}&B_\bfQ\dar{\theta}\rar&*\dar\\
*\times X \rar{\subset}& D^{d}\times X\rar& E\cup_{\partial M\times X}(-M\times X)\rar&\BSTop(d)_\bfQ\rar{c}&K(\bfQ,k),
\end{tikzcd}
\]
together with a compatible nullhomotopy of the composition $*\times X\ra B_\bfQ$. By construction the class $\widetilde{c}\in \oH^{k-1}(\fib(\theta);\bfQ)$ is represented by the map on vertical fibres of the rightmost square. Now recall that $\kappa_c\in\oH^{k-d}(X;\bfQ)$ is obtained from the class in $\oH^k(E\cup_{\partial M\times X}(-M\times X);\bfQ)$ by integration along the fibres, so in particular the latter is nontrivial and hence also the relative class given by the composition of the two rightmost squares in the diagram is nontrivial. As the second square is a pushout, also the relative class given by the composition of the three rightmost squares has to be nontrivial. Using the nullhomotopy of $*\times X\ra B_\bfQ$ and taking cofibres, we obtain a diagram
\[
\begin{tikzcd}[column sep=0.5cm, row sep=0.3cm]
\Sigma^{d-1} (X_+)\simeq \cofib\big( *\times X\ra  (D^{d-1}\cup_{\partial D^{d-1}}D^{d-1})\times X\big)\rar\dar&B_\bfQ\dar{\theta}\rar&*\dar\\
\ast\simeq \cofib\big(*\times X\ra  D^d\times X\big)\rar&\BSTop(d)_\bfQ\rar{c}&K(\bfQ,k)
\end{tikzcd}
\]
whose outer rectangle is a nontrivial relative cohomology class. Taking vertical fibres yields a map $\smash{\Sigma^{d-1}(X_+)\ra \fib(\theta)}$ that detects $\widetilde{c}\in \oH^{k-1}(\fib(\theta);\bfQ)$. As $d\ge3$, $\Sigma^{d-1}(X_+)$ is a $1$-connected space which is the suspension of a connected space, so its rational Hurewicz homomorphism is surjective (see e.g.\,\cite[Theorem 24.5]{FelixHalperinThomas}). Hence the class $\widetilde{c}$ is nontrivial on $\pi_{k-1}(\Sigma^{d-1}(X_+))$, so also on $\pi_{k-1}(\fib(\theta))$ and the claim follows.
\end{proof}

\begin{proof}[Proof of \cref{thm:PW-classes}]We may assume $k\ge3$ and $d=6$: the former since $p_k$ evaluates nontrivially on $\pi_{4k}(\BSO(6))$ for $k\le 2$, so also on $\pi_{4k}(\BSTop(6))_\bfQ$, and the latter since $p_k$ pulls back along the stabilisation map $\BSTop(d)\ra \BSTop$. For $k\ge3$ and $d=6$, we show the claim by an application of \cref{lem:hurewicz-nontriviality} to $p_k\in \oH^{4k}(\BSTop(6);\bfQ)$. As tangential structure we choose the composition $\theta\colon \tau_{>3}\BSO(4) \ra \BSTop(6)$ of the $3$-connected cover $\tau_{>3}\BSO(4)\ra \BSO(4)$ followed by the stabilisation and forgetful map $\BSO(4)\ra \BSTop(6)$. This satisfies conditions \ref{enum:hurewicz-nontriviality-i} and \ref{enum:hurewicz-nontriviality-i.5} in the lemma since $p_k$ is trivial in $\BSO(4)$ as $k\ge 3$ and $\BSO(4)_\bfQ \simeq K(\bfQ,4)^{\times 2}$. As the manifold $M$ we choose $\smash{W_{g,1}\coloneq (S^3\times S^3)^{\sharp^g}\backslash \interior(D^6)}$ for some $g\ge0$ that will be specified later. The manifold $W_{g,1}$ is parallelisable, so the constant map $\ell\colon W_{g,1}\ra \tau_{>3}\BSO(4)$ can be enhanced to a $\theta$-structure $\ell$ on $W_{g,1}$. We are left to show that $\kappa_{p_k}$ is nontrivial in the pullback appearing in \ref{enum:hurewicz-nontriviality-ii}. We will show the stronger claim that the pullback of $\kappa_{p_k}$ along the map from $\BHomeo^{\theta}_{\partial}(W_{g,1};\ell_{\half})_{\ell,\fQ}$ to the pullback in  \cref{lem:hurewicz-nontriviality} \ref{enum:hurewicz-nontriviality-ii} induced by the first map in \eqref{equ:pullback-squ-single-comp-mod-homeo} and the composition \eqref{equ:forgetful-compo} is nontrivial. To do so, we first argue that it suffices to prove instead that $\smash{\kappa_{\cL_k}\in \oH^{4k-6}(\BHomeo^{\theta}_{\partial}(W_{g,1};\ell_{\half})_{{\fQ},\ell};\bfQ)}$ is nontrivial. The manifold $\smash{W_{g,1}}$ and the tangential structure $\theta$ satisfy the assumptions of \cref{bigthm:pullback}, so writing  $B\coloneq \tau_{>3}\BSO(4)$ \cref{prop:actions} equips $\smash{\oH^{4k-6}(\BHomeo^{\theta}_{\partial}(W_{g,1};\ell_{\half})_{\ell,{\fQ}};\bfQ)}$ with an $\smash{\End^{\ell_\bfQ}_{\half\partial}(B_\bfQ)}$-action. As the underlying map of $\ell$ is constant, the constant map $B_\bfQ\ra B_\bfQ$ yields an element $\chi\in \smash{\End^{\ell_\bfQ}_{\half\partial}(B_\bfQ)}$. Using that $\kappa_{\cL_k}$ is fixed by the action of $\chi$ (see \cref{ex:action-on-stable-classes}) and writing $\kappa_{\cL_k}=A\cdot \kappa_{p_k}+B\cdot \kappa_{c}$ as a linear combination of $\kappa_{p_k}$ and $\kappa_{c}$ for $c$ a sum of decomposable monomials in Pontryagin classes, we get $\kappa_{\cL_k} = \chi^* \kappa_{\cL_k}=A\cdot \chi^* \kappa_{p_k}+B\cdot \chi^* \kappa_{c} =A\cdot \chi^* \kappa_{p_k}$, since $\chi^*\kappa_{c}=0$ by \eqref{equ:decom-action}. So if $\kappa_{\cL_k}$ is nontrivial, then  $\kappa_{p_k}$ has to be nontrivial as well.

This leaves us with showing that $\smash{\kappa_{\cL_k}\in \oH^{4k-6}(\BHomeo^{\theta}_{\partial}(W_{g,1};\ell_{\half})_{\ell,{\fQ}};\bfQ)}$ is nontrivial. To do so, we write $\theta^{\sm}\colon \tau_{>3}\BSO(4)\ra \BSO(6)$ for the composition of the $3$-connected cover with the stabilisation map $ \BSO(4)\ra  \BSO(6)$, and consider the composition \begin{equation}\label{equ:detect-smooth}\smash{\BDiff_{\partial}^{\theta^\sm}(W_{g,1};\ell_{\partial})_{\ell}\lra \BHomeo^{\theta}_{\partial}(W_{g,1};\ell_{\half})_{\ell}\lra \BHomeo^{\theta}_{\partial}(W_{g,1};\ell_{\half})_{\ell,{\fQ}}}\end{equation} whose first space is the component induced by $\ell$ of the $\Diff_{\partial}(W_{g,1})$-orbits of the space of $\theta^\sm$-structures on the smooth tangent bundle $T^{\sm}W_{g,1}$ that agree with $\ell$ on $T^{\sm}W_{g,1}|_{\partial W_{g,1}}$. The latter classifies smooth relative bundles $(W_{g,1},\partial W_{g,1})\ra (E,\partial W_{g,1}\times X)\ra X$ together with a lift $\ell^\pi\colon E\ra \tau_{>3}\BSO(4)$ of its vertical tangent bundle $T_\pi \colon E\ra \BSO(4)$ and a homotopy $h^\pi\colon \partial W_{g,1}\times X\ra \tau_{>3}\BSO(4)$ between $\ell^\pi|_{\partial W_{g,1}\times X}$ and $\ell|_{\partial W_{g,1}}\circ\pr_1$ such that $\ell^\pi$ is on each fibre homotopic to $\ell$ relative to $\partial W_{g,1}$. In these terms the first map in \eqref{equ:detect-smooth} is given by remembering the underlying topological bundle, viewing $\ell^\pi$ as a lift of the \emph{topological} vertical tangent bundle, and restricting the homotopy $h^\pi$ to $\half\partial W_{g,1}\times X\subset \partial W_{g,1}\times X$. The second map in \eqref{equ:detect-smooth} is the fibrewise rationalisation. We finish the proof by applying Galatius and Randal--Williams' work to show that $\smash{\kappa_{\cL_k} \in \oH^*(\BHomeo^{\theta}_{\partial}(W_{g,1};\ell_{\half})_{\ell};{\bfQ})}$ is nontrivial for  $g\gg0$ by proving that its pullback along \eqref{equ:detect-smooth} is nontrivial. 

The parametrised Pontryagin--Thom construction gives a map $\mathrm{PT}\colon\BDiff_{\partial}^{\theta^\sm}(W_{g,1})_{\ell}\ra \Omega^{\infty}_0\MT\theta^{\sm}$ to the $0$-component of the infinite loop space of the Thom spectrum $\MT\theta^{\sm}$ of the inverse of the vector bundle classified by $\theta^{\sm}$. By an application of \cite[Corollary 1.8]{GRWII}, this map induces an cohomology isomorphism in a range increasing with $g$. Now consider the composition
\[\smash{\oH^{\ast}(\tau_{>3}\BSO(4);\bfQ)\overset{\text{Thom}}{\cong}\oH^{\ast-6}(\MT\theta^{\sm};\bfQ)\ra \oH^{\ast-6}(\Omega^{\infty}_0\MT\theta^{\sm};\bfQ)\overset{\mathrm{PT}^*}{\ra} \oH^{\ast-6}(\BDiff_{\partial}^{\theta^\sm}(W_{g,1})_{\ell};\bfQ)}\]
whose middle map is induced by the counit of the $(\Sigma^\infty_+,\Omega^\infty)$-adjunction. As for any spectrum, the latter is injective in positive degrees, so the composition is injective in a range of degrees growing with $g$. Moreover, by an exercise in the Pontryagin--Thom construction, given a class $c\in \oH^{*}(\BSTop(d);\bfQ)$, the image of the pullback along $\theta$ under this composition agrees with the class $\kappa_c\in  \oH^{\ast-6}(\BDiff_{\partial}^{\theta^\sm}(W_{g,1})_{\ell};\bfQ)$. As a result, the pullback of $\kappa_{\cL_k}$ along \eqref{equ:detect-smooth} is nontrivial for large enough $g\gg0$ if $\cL_k$ is nontrivial in $\oH^{4k}(\tau_{>3}\BSO(4);\bfQ)$. But this is the case: it is even nontrivial in $\oH^{4k}(\tau_{>3}\BSO(3);\bfQ)$ because all monomials in Pontryagin classes except $p_1^k$ vanish in this group and the coefficient of $p_1^k$ in $\cL_k\in\oH^*(\BSO;\bfQ)$ is nontrivial by \cite[Corollary 3]{BerglundBergstrom}.
\end{proof}

\begin{proof}[Proof of \cref{bigcor:homotopy}] To construct a rational section of $\STop(d)\ra \STop$ for $d\ge6$, we use \cref{thm:PW-classes} to choose $x_i\in \pi_{4i}(\BSTop(d)_\bfQ)$ with $p_i(x)=1$ for $i\ge1$ and consider the composition \[\smash{\textstyle{
\vee_{i\ge 1} S^{4i}_\bfQ\xlra{\vee_i x_i}\BSTop(d)_\bfQ\xlra{\text{stab}}\BTop_\bfQ\overset{\sqcap_{i}p_i}{\simeq}
\bigsqcap_{i\ge 1}K(\bfQ,4i)}.}\]
Its composition represents the dual of the fundamental classes of the spheres. It suffices to show that the composition admits a section after looping once which follows from the fact that $\Omega(\vee_{i=1}^\infty S^{4i}_\bfQ)$ splits as a product of Eilenberg--MacLane spaces.
\end{proof}

\begin{proof}[Proof of \cref{bigcor:PW-classes}]We explain the proof in even dimensions; the one in odd dimensions is analogous. Consider for $2n\ge6$ the maps of horizontal fibre sequences
\[
\begin{tikzcd}[column sep=0.3cm]
\Omega^{2n+1}\BO(2n)\rar\arrow[d,equal]&\BDiff^\mathrm{fr}(D^{2n})\rar\arrow[d,hook]&\BDiff_\partial(D^{2n})\arrow[d,hook]\\[-.3cm]
\Omega^{2n+1}\BO(2n)\rar \arrow[d,"\simeq_\bfQ", "\sqcap_{i=1}^{n-1}\Omega^{2n+1}p_i"']&\Omega^{2n+1}\BTop(2n)\rar \arrow[d, "\sqcap_{i=1}^\infty\Omega^{2n+1}p_i"']&\Omega^{2n}\Top(2n)/\oO(2n)\dar\arrow[d, "\Omega^{2n}(p_n-e^2)^\tau\times \sqcap_{i>n}^\infty\Omega^{2n}p_i^\tau"']\\[-0.1cm]
\sqcap_{i=1}^{n-1}K(\bfQ,4i-2n-1)\rar& \sqcap_{i=1}^\infty K(\bfQ,4i-2n-1)\rar&\sqcap_{i=n}^\infty K(\bfQ,4i-2n-1)
\end{tikzcd}
\]
where the upper map of fibre sequences consists of vertical $0$-coconnected maps resulting from smoothing theory  (see e.g.\,\cite[Section 8.2]{KRWOddDiscs}) and the bottom one is as in the proof of Corollary 8.3 loc.cit.. By the long exact sequence and the $5$-lemma, it suffices to show that the middle lower vertical map is surjective on rational homotopy groups. This follows from \cref{thm:PW-classes}.
\end{proof}

\subsection{A rational Burghelea--Lashof splitting}\label{sec:BL-splitting} We now prove (a sharper version of) \cref{bigthm:splitting-introduction}. As mentioned in the introduction, a similar splitting was known after Postnikov truncation in a range depending on the dimension (see \cite[Corollary D]{BurgheleaLashofSplitting} or \cite[Theorem 2.5]{HsiangJahren}). The novelty of \cref{thm:splitting} is that it does not require any truncations.

\begin{thm}\label{thm:splitting} For $M$ a $2$-connected smoothable compact manifold of dimension $d \geq 8$ with $2$-connected boundary $\partial M$, the rationalised forgetful map between the components of the identity
\[
	\smash{\big(\Aut_\partial(M)/\Homeo_\partial(M)\big)_{\id,\bfQ}\lra \big(\Aut_\partial(M)/\BlockHomeo_\partial(M)\big)_{\id,\bfQ}}
\]
has the following properties:
\begin{enumerate}[leftmargin=*]
	\item\label{enum:section-i} It admits a section if the rationalised inclusion $(\half \partial M)_\bfQ\ra M_\bfQ$ desuspends once.
	\item\label{enum:section-ii} After taking loop spaces it admits a section if one of the following conditions are satisfied:
\begin{enumerate}[leftmargin=*]
	\item $d=2n+1$ and the Pontryagin class $p_{n-i}(M)\in\oH^{4(n-i)}(M;\bfQ)$ vanishes for $1\le i\le 3$
	\item $d=2n$ and the Pontryagin class $p_{n-i}(M)\in\oH^{4(n-i)}(M;\bfQ)$ vanishes for $1\le i\le 3$ and is a square for $i=4$.
\end{enumerate}
These conditions are always satisfied for $d\ge13$.
\end{enumerate}
\end{thm}

\begin{ex}If $\partial M\cong S^{d-1}$, then the condition in the first part of \cref{thm:splitting} is equivalent to the rational Hurewicz morphism $\pi_{*>0}(M)\otimes \bfQ\ra \oH_{*>0}(M;\bfQ)$ being surjective, since $\half\partial M\simeq\ast$ and a $1$-connected space is rationally equivalent to a suspension if and only if its rational Hurewicz morphism is surjective (see e.g.\ \cite[Theorem 24.5]{FelixHalperinThomas}). A concrete class of manifolds that satisfies this assumption are iterated connected sums of the form $(\sharp_{i=1}^g(S^{m_i}\times S^{d-m_i}))\backslash\interior(D^d)$ for $3\le m_i\le d-3$.\end{ex}

\begin{proof}[Proof of \cref{thm:splitting}]We consider the tangential structure $\theta\colon \BSTop(d)\times_{\BSTop(d)_\bfQ}\BSTop(d-2)_\bfQ\ra \BSTop(d)$ and first argue that $M$ admits a $\theta$-structure $\ell$. Since $M$ is smoothable, this follows from showing that the rationalised  tangent bundle $M\ra \BSO(d)_\bfQ$ lifts along $\BSO(d-2)_\bfQ\ra\BSO(d)_\bfQ$. By obstruction theory this follows if the Euler class $e$ the Pontryagin class $\smash{p_{\lfloor \frac{d}{2}\rfloor}}$ both rationally vanish. This is the case: they lie in $\oH^{i}(M;\bfQ)\cong \oH_{d-i}(M,\partial M;\bfQ)$ for $i\ge d$ which vanishes. Now choose a disc $D^{d-1}\subset\partial M$, set $\half\partial M\coloneq M\backslash \interior(D^{d-1})$ and consider the commutative diagram
\[\hspace{-0.2cm}
\begin{tikzcd}[column sep=0.3cm, row sep=0.5cm]
\fib_\ell\big(\BHomeo^\theta_\partial(M;\ell_{\half})\ra \BAut_\partial(M)^{\ell_\half}\big)_\ell\dar{\mathrm{forget}}\rar& \Map_{\half\partial}(M,\BSTop(d)\times_{\BSTop(d)_\bfQ}\BSTop(d-2)_\bfQ)_\ell\dar{\theta_*} \arrow[l,dashed, bend left=15]\\
\big(\Aut_\partial(M)/\Homeo_\partial(M)\big)_{\id}\dar{\mathrm{forget}}\rar&\Map_{\half\partial}(M,\BSTop(d))_{TM}\dar{s_*}\\
\big(\Aut_\partial(M)/\BlockHomeo_\partial(M)\big)_{\id}\rar{\simeq_\bfQ}&\Map_{\half\partial}(M,\BSTop)_{T^sM}.
\end{tikzcd}
\]
Here $s\colon \BSTop(d)\ra\BSTop$ is the stabilisation map and the subscripts indicate that we restrict to the respective path-components. The horizontal maps are induced by taking vertical tangent bundles; for the bottom map this uses that a block bundle still has a \emph{stable} vertical tangent bundle. After rationalisation, the top horizontal arrow admits a dashed section as indicated since it agrees with the first map in \eqref{equ:pullback-squ-single-comp-mod-homeo} after taking fibres of the maps to $\smash{\BAut_\partial^{\ell_\bfQ}(M)}$ and restricting to components, so it admits a section since \eqref{equ:pullback-squ-single-comp-mod-homeo} can be completed to a pullback square as a result of \cref{bigthm:pullback} and \cref{rem:variant}. The fact that the bottom horizontal map is a rational equivalence is a consequence of surgery theory (see \cite[Section 2.1]{KrannichHomological} for a fitting reminder of surgery theory in the smooth setting;  the topological case is analogous).

A diagram chase shows that the (once looped) map in the statement admits a rational section if this holds for the (once looped) right-vertical composition. As $\BSTop(d)\times_{\BSTop(d)_\bfQ}\BSTop(d-2)_\bfQ\ra\BSTop(d-2)$ is a rational equivalence this is the case if the map
\begin{equation}\label{equ:simplified-section-prob}\smash{ \Map_{\half\partial}(M,\BSTop(d-2)_\bfQ)_{\ell_\bfQ}\xlra{s_*}\Map_{\half\partial}(M,\BSTop_\bfQ)_{T^sM_\bfQ}} \end{equation} admits a rational section (after looping once). If the inclusion $\half\partial M\subset M$ is rationally the suspension of a map $A\ra X$, then \eqref{equ:simplified-section-prob} has rationally the form $\smash{\Map_{A}(X,\STop(d-2)_\bfQ)_{\overline{\ell}}\ra \Map_{A}(X,\STop_\bfQ)_{\overline{T^sM}}}$
where $\overline{\ell}$ and $\overline{T^sM}$ are the adjoints of $\ell_\bfQ$ and $T^sM_\bfQ$. As $d\ge8$, the stabilisation map $\STop(d-2)_\bfQ\ra \STop_\bfQ$ admits by \cref{bigcor:homotopy} a section $\sigma \colon \STop_\bfQ\ra \STop(d-2)_\bfQ$. Using this and the $E_1$-group-structure of $\STop(d-2)_\bfQ$, we see that \eqref{equ:simplified-section-prob} has indeed a section, namely
\vspace{-0.2cm}
\[
\Map_{A}(X,\STop_\bfQ)_{\overline{T^sM}}\xra{\sigma_*} \Map_{A}(X,\STop(d-2)_\bfQ)_{\sigma \circ \overline{T^sM}}\xra{\overline{\ell}\cdot(-)\cdot(\sigma \circ \overline{T^sM})^{-1}}\Map_{A}(X,\STop(d-2)_\bfQ)_{\overline{\ell}},
\]
so part \ref{enum:section-i} of the claim follows.

To see \ref{enum:section-ii}, note since we assumed that $M$ is smoothable we may, after possibly rechoosing $\ell$ assume that $\ell_\bfQ$ factors over $\BSO(d-2)_\bfQ$. The assumptions in \ref{enum:section-ii} are made exactly to ensure that $\ell_\bfQ$ factors further over $\BSO(d-8)_\bfQ$ and thus also over $\BSTop(d-8)_\bfQ$. In order to use this to show that \eqref{equ:simplified-section-prob} admits a section after looping once, we extend the looped map to a commutative square
\[
\begin{tikzcd}[row sep=0.5cm, column sep=0.4cm]
\Omega \big(\Map_{\half\partial}(M,\BSTop(6)_\bfQ)_{\mathrm{const}}\big)\dar{(-)+\ell_\bfQ}\rar& \Omega \big(\Map_{\half\partial}(M,\BSTop_\bfQ)_{\mathrm{const}}\big)\dar["(-)+T^sM","\simeq"']\\
\Omega \big(\Map_{\half\partial}(M,\BSTop(d-2)_\bfQ)_{\ell_\bfQ}\big)\rar & \Omega \big(\Map_{\half\partial}(M,\BSTop_\bfQ)_{T^sM_\bfQ}\big)
\end{tikzcd}
\]
whose vertical maps are induced by the external multiplication maps $\BSTop(k)\times \BSTop(n)\ra \BSTop(k+n)$ induced by taking direct products. By commutativity, it suffices to show that the top horizontal map has a section. As the loop spaces in base and target of this top map are based at the constant map, we may ``pull $\Omega$ inside the mapping space'' to see that this map is equivalent to the map $ \Map_{\half\partial}(M,\STop(6)_\bfQ)_{\mathrm{const}}\ra  \Map_{\half\partial}(M,\STop_\bfQ)_{\mathrm{const}}$ induced by stabilisation. The latter admits a section since $\STop(6)_\bfQ\ra \STop_\bfQ$ does, by \cref{bigcor:homotopy}, so this gives the claim.\end{proof}

\subsection{Detecting tautological classes}\label{sec:detect-tautological} \cref{bigthm:pullback} can also be used to show nontriviality of many tautological classes $\kappa_c\in \smash{\oH^{*-d}}(\BHomeo_\partial(M);\bfQ)$ for decomposable classes $c\in\oH^*(\BTop(d);\bfQ)$, and of products of such. To explain the strategy, fix an oriented smoothable $2$-connected $d$-manifold $M$ with $d\ge6$ and $2$-connected boundary $\partial M$. Fix a lift $\ell\colon M\ra \BSTop(d-2)_\bfQ$ of the rational tangent bundle $TM_\bfQ\colon M\ra \BSTop(d)_\bfQ$. To see that such a lift exists, note that $TM_\bfQ$ lifts to $\BSO(d)_\bfQ$ as $M$ is smoothable, then use obstruction theory to lift further to $\BSO(d-2)_\bfQ$ and thus to $\BSTop(d-2)_\bfQ$. 

\begin{thm}\label{thm:detect-kappas}Fix $M$ and $\ell$ as above. There exists a dashed lift
\begin{equation}\label{equ:lift-to-detect-kappa}\begin{tikzcd}[column sep=0.5cm, row sep=0.5cm]
&\BHomeo_\partial(M)_{\fQ}\dar{d}\\\Map_{\half\partial}(M,\BSTop(d-2)_\bfQ)_{\ell}\arrow[ur,dashed, bend left=0.4cm]\arrow[r,"q"]&\Map_{\half\partial}(M,\BSTop(d)_\bfQ)/\hAut_\partial(M)
\end{tikzcd}
\end{equation}
of the map $q$ induced by stabilisation and taking quotients, along the map $d$ induced by the topological derivative. In particular, if a decomposable class $c=c'\cup c''\in \oH^*(\BSTop(d);\bfQ)$ has the property that \[\smash{\kappa_{c'\otimes c''}\in \oH^{*-d}(\Map_{\half\partial}(M,\BSTop(d)_\bfQ)/\hAut_\partial(M);\bfQ)}\] pulls back nontrivially along the horizontal map in \eqref{equ:lift-to-detect-kappa}, then $\smash{\kappa_c\in   \oH^{*-d}(\BHomeo_\partial(M);\bfQ)}$ is nontrivial. The same applies to linear combinations of products of such classes.
\end{thm}

\begin{proof}
The second part follows from the first, since by the discussion in \cref{sec:tautological-classes}, $\kappa_{c'\otimes c''}$ pulls back to $\kappa_{c}$ along the vertical map in \eqref{equ:lift-to-detect-kappa}. To show the first part, we use the same tangential structure $\theta\colon \BSTop(d-2)_\bfQ\times_{\BSTop(d)_\bfQ}\BSTop(d)\ra \BSTop(d)_\bfQ$ as in the proof of \cref{thm:splitting}. The claim then follows from the commutative diagram
\[
\hspace{-.1cm}\begin{tikzcd}[column sep=0.3cm, row sep=0.3cm]
&\BHomeo_\partial^\theta(M;\ell_{\half\partial})_{\fQ,\ell}\arrow[d]\rar&\BHomeo_\partial(M)_{\fQ}\rar& \frac{\Map_{\partial}(M,\BSTop(d)_\bfQ)}{\Aut_\partial(M)}\arrow[d]\\
\Map_{\half\partial}(M,\BSTop(d-2)_\bfQ)_{\ell}\rar\arrow[ur,dashed, bend left=15]&\frac{\Map_{\half\partial}(M,\BSTop(d-2)_\bfQ)}{\Aut_\partial(M)}\arrow[rr]\arrow[u,leftarrow,crossing over]&&\frac{\Map_{\half\partial}(M,\BSTop(d)_\bfQ)}{\Aut_\partial(M)}
\end{tikzcd}
\]
whose dashed lift results from \cref{bigthm:pullback}: it is the case $B_\bfQ=\BSTop(d-2)_\bfQ$ of \eqref{equ:dashed-lift-mainthm} in \cref{rem:variant}.
\end{proof}

The condition in the second part of \cref{thm:detect-kappas} can often be checked by using the explicit description the pullback of $\kappa_{c\otimes c'}$ to $\Map_{\half\partial}(M,\BSTop(d-2)_\bfQ)_{\ell}$ resulting from \cref{rem:formula-kappa-trivial-bundle}, in the case $B=\BSTop(d-2)_\bfQ$. We will illustrate this method to detect tautological classes by proving \cref{bigthm:detect-classes} from the introduction, which said that for $M$ as above that also contains $S^m \times S^n$ as a connected summand (with $n = d-m$), we have
\begin{enumerate}
\item\label{enum:smsn-infinite-i} The class $\kappa_{p_i p_j}\in\oH^{4(i+j)-d}(\BHomeo_\partial(M);\bfQ)$ is nontrivial for all $i,j > \max(m,n)/4$.
\item\label{enum:smsn-infinite-ii} If $n$ and $m$ are both even, then the total dimension of the cohomology $\oH^{*\le k}(\BHomeo_\partial(M);\bfQ)$ grows with $k$ faster than any polynomial.
\end{enumerate}

\begin{proof}[Proof of \cref{bigthm:detect-classes}]By naturality, it suffices to prove the result for $M=D^d\sharp (S^m\times S^{n})$, which is homotopy equivalent to $S^m\vee S^{n}$. Choosing classes $a_m\in \oH_m(M;\bfQ)$ and $b_n\in \oH_n(M;\bfQ)$ corresponding to the fundamental classes of the two spheres, for \ref{enum:smsn-infinite-i}, it suffices by the second part of \cref{thm:detect-kappas} and \cref{rem:formula-kappa-trivial-bundle} to show that the class
\begin{equation}\label{eqn:class-non-triviality-smsn}\ev^*(p_i)/a_m\cup \ev^*(p_j)/b_n + \ev^*(p_i)/b_n\cup \ev^*(p_j)/a_m \in \oH^{4(i+j)-d}(\Map_{\ast}(M,\BSTop(d-2)_\bfQ)_0;\bfQ)\end{equation}
is nontrivial. To do so, we use \cref{thm:PW-classes} and the assumption that $d-2\ge6$ to find a map of the form $\smash{\rho \colon \bigvee_{k \geq 1} S^{4k} \to \BSTop(d-2)_\bfQ}$ with $\langle \rho^* p_k,[S^{4k}] \rangle = 1$ for all $k$. Pulling back \eqref{eqn:class-non-triviality-smsn} along \[\smash{\textstyle{\rho_*\colon \Omega^m_0(\bigvee_{k \geq 1} S^{4k}) \times \Omega^{n}_0(\bigvee_{k \geq 1} S^{4k})\simeq\Map_{\ast}(M,\bigvee_{k \geq 1} S^{4k})_0\lra \Map_{\ast}(M,\BSTop(d-2)_\bfQ)_0}}\]
yields the class \begin{equation}\label{equ:pullbac-classes-product-spheres}\smash{\smash{\textstyle{\ev^*(\lambda_{4i})/a_m \otimes \ev^*(\lambda_{4j})/b_{n} + (-1)^{nm} \ev^*(\lambda_{4j})/a_m \otimes \ev^*(\lambda_{4i})/b_{n}}}}\end{equation} in $\oH^{4(i+j)-d}(\Omega^m_0(\bigvee_{k \geq 1} S^{4k}) \times \Omega^{n}_0(\bigvee_{k \geq 1} S^{4k});\bfQ)$ where $\lambda_{4k}$ is the dual of the fundamental class of $S^{4k}$. It follows from the Milnor--Moore theorem that the cohomology of $\Omega^m_0(\bigvee_{k \geq 1} S^{4k})$ is a free graded-commutative algebra on a set of generators which include the classes $\ev^*(\lambda_{4k})/a_m$ for $k>m/4$, and similarly for $\Omega^n_0(\bigvee_{k \geq 1} S^{4k})$, so this pullback is nonzero and \ref{enum:smsn-infinite-i} follows. 

To show \ref{enum:smsn-infinite-ii}, note that by using the first part of \cref{thm:detect-kappas} it suffices to prove that the classes \eqref{equ:pullbac-classes-product-spheres} for $i,j>\max(m,n)/4$ span a subalgebra whose total dimension has the claimed growth. To show this, we consider an algebra morphism $\oH^*(\Omega^m_0(\textstyle{\bigvee_{k \geq 1} S^{4k}}) \times \Omega^{n}_0(\textstyle{\bigvee_{k \geq 1} S^{4k}});\bfQ) \ra \bfQ[x_i,y_j \mid i,j > \max(m,n)/4]$ to a free graded algebra with generators of even degrees $|x_{i}| = 4i-m$ and $|y_j| = 4j-n$, by sending the generators $\ev^*(\lambda_{4i})/a_m \otimes 1$ to $x_i$ and $1 \otimes \ev^*(\lambda_{4j})/b_{n}$ to $y_j$, and all other generators to zero. It thus suffices to show that the subalgebra generated by  the elements $x_iy_j+x_jy_i$ has the required growth. Assuming without loss of generality $m \leq n$, a crude lower bound for the growth of this subalgebra can be obtained by reindexing $y_j$ to be in degree $4j-m$ instead, and defining an algebra morphism to the polynomial algebra $\bfQ[z_i \mid i>n/4]$ in infinitely many generators by sending both $x_i$ and $y_i$ to $z_i$. This is surjective onto the polynomial subalgebra on generators $z_i^2$ for $i>n/4$. The growth of its total dimension can be bounded from below by the partition function which is known to grow faster than any polynomial.
\end{proof}

\begin{rem}\ 
\begin{enumerate}
\item Note that by choosing $M=D^8\sharp (S^4\times S^4)$, the first part of \cref{bigthm:detect-classes} in particular implies that $p_i p_j\neq 0\in\oH^{*}(\BTop(d);\bfQ)$ for $d\ge8$ and all $i,j>1$. By picking different choices of $M$ and extending the strategy of \cref{bigthm:detect-classes}, one can exclude more relations between products of Pontryagin classes in $\oH^{*}(\BTop(d);\bfQ)$ and thus partially recover \cite[Theorem 1.1]{GRWIndependence}. However it seems unlikely that this method can be extended to exclude all relations.
\item It may seem surprising at first sight that our method also shows that certain products of Pontryagin classes are nontrivial, even though the input \cref{thm:PW-classes} only shows that the individual $p_i$s are nontrivial and detected on the image of the Hurewicz map. The surprise comes from the implicit use of the Alexander trick and is similar to the following observation: suppose one only knew (i) surgery theory and (ii) that $p_i$ and $p_j$ for some $i,j>1$ are nontrivial in $\oH^*(\BTop;\bfQ)$ and detected on the image of the Hurewicz map, then one can deduce that their product $p_i p_j\in \oH^*(\BTop;\bfQ)$ is nontrivial, by considering the map \begin{equation}\label{equ:stable-deriative-blockhomeo}\smash{\big(\Aut_\partial(M)/\BlockHomeo_\partial(M)\big)_{\id}\lra\Map_{\half \partial}(M,\BSTop)_0}\end{equation}  induced by taking stable topological derivatives for $M=D^d\sharp (S^n\times S^m)$. Arguing as in the proof of \cref{bigthm:detect-classes}, the class \eqref{eqn:class-non-triviality-smsn} with $\BTop(d-2)$ replaced by $\BTop$ is nontrivial and it pulls back to $\smash{\kappa_{p_i p_j}}$ along \eqref{equ:stable-deriative-blockhomeo}. But it follows from surgery theory and the Alexander trick that \eqref{equ:stable-deriative-blockhomeo} is a rational equivalence, so $\kappa_{p_i p_j}$ is nontrivial and hence so is $p_ip_j$.
\item The nonzero classes produced in the proof of \cref{bigthm:detect-classes} all lie in the subalgebra generated by $\kappa_c$ for $c$ monomials in Pontryagin classes. In the smooth situation, the corresponding subalgebra of $\oH^*(\BDiff_\partial(M);\bfQ)$ is known to be finitely generated in many cases (so in particular to grow polynomially) by \cite[Theorem A]{RWtautological}, in particular for $M=D^{2n} \# (S^n \times S^n)^{\# g}$ for all $g$ and $n$ (if $n$ is odd, it is even finite-dimensional \cite[Corollary 1.4]{GGRW}).
\end{enumerate}
\end{rem}

\begin{rem}\cref{bigthm:pullback} can also be used to produce \emph{weight decompositions} on the rational cohomology of $\BHomeo_\partial^\theta(M;\ell_{\half \partial})_{{\fQ},\ell}$ for suitable choices of $\theta$, similar to the weight decompositions of the rational cohomology of block diffeomorphisms outlined in \cite{KKOWR}. We expect that these weight decompositions are useful to exclude relations between tautological classes and to resolve differentials in spectral sequences, but we do not pursue this further at this point.
\end{rem}

\appendix

\section{Tensor products of truncated operads} \label{sec:truncation-of-operads}
In this appendix, we construct the symmetric monoidal lift of the tower \eqref{equ:tower-of-operads} of categories of truncated operads from \cref{sec:operad-conventions} and establish several properties of it, summarised in \cref{thm:operad-truncation-tower-monoidal}. We adopt the notation of \cref{sec:operad-conventions}. We write $X\ast Y\coloneq X\cup_{X\times Y}Y$ for the \emph{join} of spaces and $\cS^{\Sigma_n}=\Fun(\mathrm{B}\Sigma_n,\cS)$ for the category of spaces with an action of the symmetric group on $n$ letters.

\begin{thm}\label{thm:operad-truncation-tower-monoidal}The tower of categories of truncated unital operads
\begin{equation}\label{equ:operad-truncation-tower-app}
	\smash{\Opd^\un\simeq\Opd^{\le \infty,\un}\ra\cdots\ra\Opd^{\le 2,\un}\ra\Opd^{\le 1,\un}\simeq\Cat}
\end{equation}
can be lifted to a tower of symmetric monoidal categories such that
\begin{enumerate}
\item \label{enum:trunmon-i} the symmetric monoidal structure on $\Opd^\un$ is the one described in \cref{sec:operad-conventions},
\item \label{enum:trunmon-i.5} objectwise, the symmetric monoidal structure on $\Opd^{\le k,\un}$ is given by natural equivalences
\[
	\cO\otimes \cP\simeq \tau(\tau_*(\cO)\otimes\tau_*(\cP))\simeq \tau(\tau_!(\cO)\otimes\tau_!(\cP))
\]
where $\tau\colon \Opd^{\un}\ra \Opd^{\le k,\un}$ is the truncation functor and $\tau_!$ and $\tau_*$ are its left and right adjoints. 
\item\label{enum:trunmon-ii}  the symmetric monoidal structure on $\smash{\Opd^{\le k,\un}}$ preserves colimits in both variables for all $k\le \infty$,
\item\label{enum:trunmon-iii}  the symmetric monoidal structure on $\smash{\Opd^{\le 1,\un}\simeq \Cat}$ is the cartesian one. In particular, as $\ast\times \ast \simeq \ast$, the full subcategories of reduced operads \[\smash{\Opd^{\le k,\red}=\fib_\ast(\Opd^{\le k,\un}\ra \CatInf)\subset \Opd^{\le k,\un}}\]   inherits symmetric monoidal structures for $1\le k\le \infty$,
\item\label{enum:trunmon-iv} with respect to the equivalence $\Opd^{\le 2,\red}\simeq \cS^{\Sigma_2}$ from \cref{lem:2-truncated-reduced}, the binary tensor product $\Opd^{\le 2,\red}\times \Opd^{\le 2,\red}\ra \Opd^{\le 2,\red} $ is given by the join, i.e.\,there is a natural  equivalence \[\smash{(-)(2) \ast (-)(2)\simeq \big((-)\otimes (-)\big)(2)}\]
between the join of the $2$-ary operation spaces to the $2$-ary operation space of the tensor product. 
\end{enumerate} 
\end{thm}

\begin{rem}We believe that the natural equivalence in \ref{enum:trunmon-iv} extends to an equivalence $\smash{\Opd^{\le 2,\red}\simeq\cS^{\Sigma_2}}$ of symmetric monoidal categories where the source is equipped with the tensor product and the target with the join, but this will not be needed for our applications.
\end{rem}

Throughout this appendix, we use the notation and terminology from \cite[Section 1.4]{KKoperadic}.

\subsection{Monoidality of truncation}Apart from Items \ref{enum:trunmon-iii} and \ref{enum:trunmon-iv}, which we deal with separately, \cref{thm:operad-truncation-tower-monoidal} will be proved by an application of the following general lemma:

\begin{lem}\label{lem:general-monoidal-tower}Fix a tower of categories
\begin{equation}\label{equ:truncation-tower-gen}
	\cC=\cC^{\le \infty}\ra\cdots\ra\cC^{\le 2}\ra\cC^{\le 1}
\end{equation}
such that all functors in this tower admit fully faithful left and right adjoints. If $\cC$ admits a symmetric monoidal structure such that the following condition is satisfied for each of the functors $\tau\colon \cC\ra\cC^{\le k}$
\begin{equation}\label{eqn:loc-condition-gen-tower}  
	\text{whenever $\tau(\varphi)$ is an equivalence for a morphism $\varphi$ in $\cC$, then so is $\tau(\varphi \otimes \id_c)$ for all $c \in \cC$},
\end{equation}
then the tower \eqref{equ:truncation-tower-gen} can be lifted to a tower of symmetric monoidal categories such that
\begin{enumerate}
\item\label{enum:trungen-i} objectwise the monoidal structure on $\cC^{\le k}$ is given by natural equivalences 
\[
	c\otimes d\simeq \tau(\tau_*(c)\otimes\tau_*(d))\simeq \tau(\tau_!(c)\otimes\tau_!(d))
\]
where $\tau_*$ and $\tau_!$ are the right and left adjoints of $\tau$, respectively, 
\item\label{enum:trungen-ii} the right and left adjoints to the functors in the tower $\cC^{\le k}\ra  \cC^{\le j}$ for $1\le k\le j\le \infty$ lift to lax respectively oplax symmetric monoidal functors,
\item \label{enum:trungen-iii} If the monoidal structure on $\cC$ preserves colimits in both variables, then so does the one on $\cC^{\le k}$.
\end{enumerate} 
\end{lem}

\begin{proof}
The first part of this proof is similar to the proof of \cite[Proposition 4.2]{KKoperadic}. One considers the symmetric monoidal structure $\cC^\otimes\ra\Fin_*$ on $\cC$ and the sequence $\cdots \subset\cC^{\le k,\otimes}\subset \cC^{\le k+1,\otimes}\subset \cdots \subset \cC^\otimes$ of full subcategories $\cC^{\le k,\otimes}\subset \cC^{\otimes}$ spanned by those objects $(c_s)_{s\in S}$ for finite sets $S$ for which all $c_s$ lie in the essential image of $\tau_{k*}$. By \cite[2.2.1.7, 2.2.1.9]{LurieHA}, condition \eqref{eqn:loc-condition-gen-tower} ensures that the restriction $\smash{\cC^{\le k,\otimes}\ra \Fin_*}$ is a symmetric monoidal category and the sequence of inclusions $\cdots \subset \smash{\cC^{\le k,\otimes}}\subset \smash{\cC^{\le k+1,\otimes}}\subset \cdots \subset \cC^\otimes$ is a sequence of lax symmetric monoidal functors. The argument in the proof of \cite[Proposition 4.2]{KKoperadic} then lifts \eqref{eqn:loc-condition-gen-tower} to a tower of symmetric monoidal categories with oplax symmetric monoidal functors between them. These oplax symmetric monoidal functors are in fact symmetric monoidal, by the argument in the proof of Lemma 3.3 (v) loc.cit., so we have produced the claimed symmetric monoidal lift of \eqref{equ:truncation-tower-gen} and the lax symmetric monoidal lift of the right adjoints as claimed in \ref{enum:trungen-ii}. Moreover, the argument in the proof of Lemma 3.3 (ii) loc.cit.~yields the first equivalence in \ref{enum:trungen-i}. Applying the same argument to the tower obtained from \eqref{equ:truncation-tower-gen} by taking opposites and then taking opposite symmetric monoidal structures of the result, we obtain the claimed oplax symmetric monoidal lift of the left adjoints in \ref{enum:trungen-ii} and the second equivalence in \ref{enum:trungen-i} \emph{as long as} we ensure that for all $k$ the two resulting symmetric monoidal structures on $\cC^{\le k}$ (the one using the right adjoints and the one using the left adjoints) agree. Denoting by $\cC^{\le k,\otimes_L}$ respectively $\cC^{\le k,\otimes_R}$ the symmetric monoidal structures resulting from using the left respectively right adjoint, we have by construction lax monoidal lifts $\cC^{\le k,\otimes_R}\ra\cC^\otimes$ and $\cC^\otimes\ra \cC^{\le k,\otimes_L}$ of $\tau_*$ and $\tau$ respectively (the second one is even monoidal). Their composition is a lax symmetric monoidal lift of $\tau^*\tau_*\simeq\id_{\cC^{\le k}}$, so it suffices to show that it is symmetric monoidal. To see this, recall that the tensor products are on objects given by $c\otimes_{R}d\simeq \tau(\tau_*(c)\otimes \tau_*(d))$ and $c\otimes_{L}d\simeq \tau(\tau_!(c)\otimes \tau_!(d))$. In terms of this, the components of the natural transformation which we have to show are equivalences, are 
\[\smash{\hspace{-0.3cm}\tau\tau_*\big[c\otimes_Rd\big]\simeq\tau\tau_*\big[\tau\big(\tau_*(c) \otimes \tau_*(d)\big)\big] \ra \tau(\tau_*(c) \otimes \tau_*(d)) \ra \tau\big[\tau_!\big(\tau\tau_*(c)\big) \otimes  \tau_!\big(\tau\tau_*(d)\big)\big]\simeq \tau\tau_*(c)\otimes_L\tau\tau_*(d)}\]
where the first map is induced by the counit $\tau_* \tau \to \id$, which is an equivalence after composing with $\tau$ by the triangle identities, and where the second map is induced by the inverse of the counit $\tau_! \tau \to \id$, which is an equivalence since $\tau_!$ is fully faithful. This concludes the proof of \ref{enum:trungen-i} and \ref{enum:trungen-ii}. Part \ref{enum:trungen-iii} holds by the second formula in \ref{enum:trunmon-i}, since $\tau$ and $\tau_!$ are left adjoints, so preserve colimits.
 \end{proof}
 
\subsubsection{The free operad on a category over $\Fin_*$}\label{sec:free-operad} To eventually verify the condition \eqref{eqn:loc-condition-gen-tower} for the truncation functors $\tau\colon \Opd^\un\ra \Opd^{\le k,\un}$, we will first establish some properties for the left adjoint 
\[
	F\colon\Cat_{/\Fin_*}\lra \Opd
\] 
to the forgetful functor $\inc \colon \Opd\ra \Cat_{/\Fin_*}$. This may be of independent interest as $F$ can be interpreted as taking a ``free operad'' on a category over $\Fin_*$ (see e.g.\,\cite[Corollary 4.2.3]{BHS} for the existence of this adjoint; alternatively see the proof of \cref{lem:f-tensor} below). 

Given an operad $\cO$, we denote the underlying category over $\Fin_*$ as $\cO^\otimes\ra\Fin_*$ and the fibre over $\langle r\rangle\in\Fin_*$ as $\smash{\cO^\otimes_{\langle r\rangle}}$ which we often identify with $\smash{(\cO^\otimes_{\langle 1\rangle})^{\times r}}$ by taking cocartesian lifts over inerts (c.f.\ \cite[2.1.1.14]{LurieHA}). A special case is the \emph{category of colours} $\smash{\cO^\col\coloneq \cO^{\otimes}_{\langle 1\rangle}}$

\smallskip

\noindent The first property of $F$ we discuss is an explicit description of its values on functors of the form $\alpha\colon [n]\ra \Fin_*$ for $[n]\in \Delta$, that is, simplices in $\Fin_*$. This uses the inclusion $\ell\colon\OpdSet\longhookrightarrow \Opd$ of the $1$-category of coloured operads in sets and the full subcategories $\Omega \subset\Phi\subset \OpdSet$ of trees and forests, as discussed in \cref{sec:prelim-dendroidal}. There is an explicit construction of a forest out of simplex $\alpha\colon [n]\ra \Fin_*$, which yields a functor 
\[
	\omega\colon \Delta_{/\Fin_*}\lra \Phi\subset \OpdSet
\] 
out of the category $\Delta_{/\Fin_*}$ of simplices in $\Fin_*$; see \cite[Section 3.1.1]{HinichMoerdijk}. It follows from Proposition 3.4.1 loc.cit.~that there is a natural map $\alpha\ra \ell(\omega(\alpha))$ whose adjoint is an equivalence
\begin{equation}\label{equ:f-on-simplices}
	F(\alpha)\overset{\simeq}\lra \ell(\omega(\alpha)).
\end{equation}
In other words: restricted along $\Delta_{/\Fin_*}\subset \Cat_{/\Fin_*}$, the adjoint $F\colon  \Cat_{/\Fin_*}\ra \Opd$ is given by $\ell\circ \omega$.

\begin{ex}\label{ex:trees-from-simplices}\label{ex:cn-description-fibre-seqs} Examples of the natural map $\alpha\ra \ell(\omega(\alpha))$ for a simplex $\alpha\colon [n]\ra\Fin_*$ include:
 \begin{enumerate}[leftmargin=0.7cm]
 \item\label{enum:cn-description-e1} For $\alpha = e_r \colon [0]\ra \Fin_*$ with image $\langle r\rangle\coloneq\{1,\ldots,r,\ast\}$, the operad $\omega(e_r)$ has $r$ colours and only identity $1$-ary operations. It corresponds to the forest $\sqcup_r \eta \in \Phi$ where $\eta\in\Omega$ is the unique tree with no vertices \cite[p.\,92]{HeutsMoerdijk}. The natural map $e_r \ra \ell(\omega(e_r))$ sends $[0]$ to the collection of colours in $\ell(\omega(e_r))^{\otimes}_{\langle r\rangle}\simeq (\ell(\omega(e_r))^\col)^{\times r}$. For an operad $\cO$, we have \[\smash{\Map_{/\Fin_*}(e_r,\cO^\otimes) \simeq (e_r \times_{\Fin_*} \cO^\otimes)^\simeq \simeq \cO_{\langle r\rangle}^{\simeq} \simeq (\cO^{\col,\simeq})^{\times r}}.\]
 
 \item \label{enum:cn-description-cr} For  $\alpha = c_r \colon [1]\ra \Fin_*$ given by the unique active map of the form $\langle r \rangle \to \langle 1 \rangle$, the operad $\omega(c_r)$ corresponds to the open corolla $C_r\in \Omega$ (see \cite[p.\,93]{HeutsMoerdijk}) and $c_r \ra \ell(\omega(c_r))$ is given by the unique $r$-ary operation, viewed as morphism in $\ell(\omega(c_r))^{\otimes}$. For an operad $\cO$, we have
	\[\Map_{\Cat_{/\Fin_*}}(c_r,\cO^\otimes)\simeq (c_r\times_{\Fin_*}\cO^\otimes)^\simeq.\]
	This is, informally speaking, the space of triples of $r$ input colours, a single output colour, and an $r$-ary multi-operation. More precisely: the restriction of $c_r \colon [1] \to \Fin_*$ to the boundary of the 1-simplex $[1]$ is given by $e_r \sqcup e_1 \colon [0] \sqcup [0] \to \Fin_*$ and restriction along its inclusion gives a map $\Map_{/\Fin_*}(c_r,\cO^\otimes) \ra \Map_{/\Fin_*}(e_r,\cO^\otimes) \times \Map_{/\Fin_*}(e_1,\cO^\otimes) \simeq(\cO^{\col,\simeq})^r \times \cO^{\col,\simeq}$,
	whose fibre over $((x_s)_{s \in \ul{r}},x)$ is, by definition (see 2.1.1.16 loc.cit.), the space of multi-operations $\smash{\Mul^\cO}((x_s)_{s \in \ul{r}},x)$. In particular, if $\cO$ is reduced, then $\smash{\Map_{/\Fin_*}(c_r,\cO^\otimes)\simeq\Mul_{\cO}((*)_{i \in \ul{r}};*)\eqcolon \cO(r)}$.
	
\item \label{enum:ex-trees-iii}\label{enum:cn-description-olcr} For $\alpha = \overline{c}_r \colon [2]\ra \Fin_*$ given by the unique composition $\langle0\rangle \ra \langle r\rangle \ra \langle 1\rangle$ where the second map is active, the operad $\omega(\overline{c}_r)$ corresponds to the closed corolla $\overline{C}_r\in \Omega$ (see \cite[p.\,93]{HeutsMoerdijk}) and $\overline{c}_r \ra \ell(\omega(\overline{c}_r))$ sends $(0<1)$ to the unique $0$-operation and $(1<2)$ to the unique $r$-ary one. For a general operad $\cO$, the space  $\Map_{/\Fin_*}(\overline{c}_r,\cO^\otimes)$ is informally given by adding to the data of \cref{enum:cn-description-cr} a $0$-ary operations for each of the $r$ input colours. More precisely, the restriction along the 1-simplex $c_r$ in the 2-simplex $\overline{c}_r$ yields a map 
	$\Map_{/\Fin_*}(\overline{c}_r,\cO^\otimes) \ra \Map_{/\Fin_*}(c_r,\cO^\otimes)$
	whose fibre over $m\in \Mul_\cO((x_s)_{s \in \ul{r}}; x)$ is given by $\sqcap_{s \in \ul{r}} \smash{\Mul^\cO}(\varnothing,x_s)$. Note that the latter is contractible if $\cO$ is unital, so in this case the restriction map is an equivalence.\end{enumerate}
 \end{ex}

\noindent The second property of the adjoint $F$ we discuss is how it interacts with tensor products of operads. This involves the smash product of finite sets  $\wedge\colon \Fin_\ast\times\Fin_*\ra \Fin_*$ (see \cite[2.2.5.1]{LurieHA}), the functor $\wedge_!\colon \Cat_{/\Fin_*\times \Fin_*}\ra \Cat_{/\Fin_*}$ given by postcomposition with $\wedge$, and its right adjoint $\wedge^*\colon \Cat_{/\Fin_*}\ra \Cat_{/\Fin_*\times\Fin_*}$ given by pullback along $\wedge$.

\begin{lem}\label{lem:f-tensor} The functor $F\colon \Cat_{/\Fin_*}\ra \Opd$ lifts to a monoidal functor with respect to a monoidal structure on $\Cat_{/\Fin_*}$, whose monoidal unit is $\id_{\Fin_*}$ and whose binary tensor product is the composition \[\smash{\Cat_{/\Fin_*}\times \Cat_{/\Fin_*}\xlra{\times} \Cat_{/\Fin_*\times\Fin_*}\xlra{\wedge_!}\Cat_{\Fin_*}}.\] The monoidal structure on the target $\Opd$ is the one underlying the symmetric monoidal structure on $\Opd$. In particular, we have for $X,Y\in\Cat_{/\Fin_*}$ natural equivalences $F(\wedge_!(X\times Y))\simeq F(X)\otimes F(Y)$.
\end{lem}

\begin{proof}
We start by some model-categorical preliminaries, based on Lurie's theory of \emph{categorical patterns} from \cite[Appendix B]{LurieHA}. By B.0.20 loc.cit., associated to a categorical pattern $\mathfrak{P} = (M_\cC,T,A)$ in the sense of B.0.19 on a simplicial set $\cC$, there is a left proper combinatorial simplicial model category $\smash{\sSet^+_{/\mathfrak{P} }}$ whose underlying category is the category $\smash{\sSet^+_{/(\cC,M_\cC)}}$ of marked simplicial sets $(X,M)$ over $(\cC,M_\cC)$ (see \cite[3.1]{LurieHTT}). The model structure is uniquely determined by the properties that (a) cofibrations are those maps that are monomorphisms on underlying simplicial sets (so in particular any object is cofibrant) and (b) fibrant objects are the $\mathfrak{P}$-fibred objects in the sense of \cite[B.0.19]{LurieHA}. Given a map $f \colon \cC \to \cC'$ compatible with categorical patterns $\mathfrak{P}$ and $\mathfrak{P}'$ on domain and target in the sense of B.2.8 loc.cit., then by B.2.9.\ we have a Quillen adjunction 
\vspace{-0.1cm}
\begin{equation}\label{equ:quill-adj-catpat}\begin{tikzcd} \sSet^+_{/\mathfrak{P}} \rar[shift left=.5ex]{f_!} & \sSet^+_{/\mathfrak{P}'} \lar[shift left=.5ex]{f^*} \end{tikzcd}\vspace{-0.1cm}\end{equation}
whose left adjoint $f_!$ precomposes with $f$ and right adjoint $f^*$ pulls back along $(\cC,M_\cC) \to (\cC',M_{\cC'})$. Moreover, by B.2.5 loc.cit., taking external products yields a left Quillen bifunctor
\begin{equation}\label{equ:external-products}\smash{\sSet^+_{/\mathfrak{P}} \times \sSet^+_{/\mathfrak{P}'} \overset{\times}\lra \sSet^+_{/\mathfrak{P} \times \mathfrak{P}'}}\end{equation} where $\mathfrak{P} \times \mathfrak{P}'$ is the categorical pattern on $\cC\times \cC'$ from B.1.8..

For each quasi-category $\cC$, there is categorical pattern $\mathfrak{P}_\text{slice}=(\text{equivalences}, \text{all }2\text{-simplices},\varnothing)$ on $\cC$ for which the $\mathfrak{P}_\text{slice}$-fibred objects are the categorical fibrations $\cE\ra\cC$ where $\cE$ is marked by its equivalences. The underlying $\infty$-category of $\smash{\sSet^+_{/\mathfrak{P}_\text{slice}}}$ is the overcategory $\Cat_{/\cC}$ (see \cite[Examples 3.2.11]{GepnerHaugsengEnriched}). It has the property that any $f \colon \cC \to \cC'$ is compatible with the $\mathfrak{P}_\text{slice}$-patterns, and the resulting Quillen adjunction $f_! \dashv f^*$ gives rise to the analogous adjunction $\Cat_{/\cC}\rightleftarrows\Cat_{/\cC'}$ on $\infty$-categories. Moreover, \eqref{equ:external-products} in the situation of the slice patterns yields the external product functor $\Cat_{/\cC}\times \Cat_{/\cC'}\ra \Cat_{/\cC\times \cC'}$ on $\infty$-categories.

On $\cC=\Fin_*$, there is also another categorical pattern $\mathfrak{P}_\text{opd}$ such that the $\mathfrak{P}_\text{opd}$-fibred objects are $\infty$-operads $\cO^\otimes\ra \Fin_*$ as in \cite[2.2.1.10]{LurieHA} marked by the inert morphisms in the sense of 2.1.2.3 loc.cit., and the underlying $\infty$-category is $\Opd$ (see 2.1.4.6 and its proof). The identity functor $\smash{\id\colon \sSet^+_{/\mathfrak{P}_\text{slice}}\ra \sSet^+_{/\mathfrak{P}_\text{opd}}}$ is compatible with the indicated categorical pattern, so it is left Quillen. On fibrant objects the right adjoint $\id^*$ sends $(\cO^\otimes,\mathrm{inerts})\ra \Fin_*$ to $(\cO^\otimes,\mathrm{equivalences})\ra \Fin_*$ (use \cite[2.4.1.5]{LurieHTT} to see this), so it models the forgetful functor $\Opd\ra \Cat_{/\Fin_*}$ and thus $\id_!$ models the left adjoint $F$. Moreover, the map $\wedge \colon \Fin_* \times \Fin_* \to \Fin_*$ is compatible with $\mathfrak{P} \times \mathfrak{P}$ and $\mathfrak{P}$ for either choices of $\mathfrak{P} \in\{\mathfrak{P}_\text{slice},\mathfrak{P}_\text{opd}\}$, so \eqref{equ:quill-adj-catpat} gives a Quillen adjunction $\wedge_!\dashv \wedge^*$. By the discussion above, for $\mathfrak{P} =\mathfrak{P}_\text{slice}$ this models the same-named adjunction $\wedge_!\dashv \wedge^*$ on $\infty$-categories. Moreover, the fibrous object $\id _{\Fin_*}$ in $ \sSet^+_{/\mathfrak{P}}$, marked with the equivalences in the slice case and the inerts in the operad case, together with the left Quillen functor 
\[\smash{\sSet^+_{/\mathfrak{P}} \times \sSet^+_{/\mathfrak{P}} \xlra{\times} \sSet^+_{/\mathfrak{P}\times \mathfrak{P}}\xlra{\wedge_!}\sSet^+_{/\mathfrak{P}}}\] turns $\smash{\sSet^+_{/\mathfrak{P}}}$ into a simplicial monoidal model category (see the proofs of 2.2.5.7 and 4.1.7.18 loc.cit.~for the proof in the operad case; the slice case is analogous). By construction, the identity  $\smash{\id\colon \sSet^+_{/\mathfrak{P}_\text{slice}}\ra \sSet^+_{/\mathfrak{P}_\text{opd}}}$ is a (strictly monoidal) functor of simplicial monoidal model categories. Applying 4.1.7.18 loc.cit.~, we obtain a lift of $F\colon \Cat_{/\Fin_*}\ra \Opd$ to a monoidal functor between monoidal $\infty$-categories where the monoidal structure on the source has by construction the properties asserted in the claim. By 4.1.7.17 loc.cit.~the monoidal structure on the target is the one underlying the symmetric monoidal structure on $\Opd$, so the claim follows.
\end{proof}

The third and final property of $F$ we establish is the following:
\begin{lem}\label{lem:colim}The category $\Opd$ of operads is generated under colimits by the values of $F$ at the $0$-simplex $e_1 \colon [0]\ra \Fin_*$ and the $1$-simplices $c_r \colon [1]\ra \Fin_*$ for $r\ge0$ as described in \cref{ex:trees-from-simplices}.
\end{lem}

\begin{proof}
As a result of \eqref{equ:f-on-simplices} and \cref{ex:trees-from-simplices}, the operads in the statement are sent under the equivalence $\delta_\Omega\colon \Opd\simeq \PSh(\Omega)^{\seg,c}$ from \cref{sec:prelim-dendroidal} to $y(\eta)$ and $y(C_r)$ for $r\ge0$ where $y$ is the Yoneda embedding for $\Omega$, and $\eta$ as well as $C_r$ are as in \cref{ex:trees-from-simplices}. It thus suffices that these presheaves generate $\PSh(\Omega)^{\seg,c}$ under colimits. As a localisation of the presheaf category $\PSh(\Omega)$, the category $\PSh(\Omega)^{\seg,c}$ is generated by $y(T)$ for $T\in\Phi$ under colimits. Moreover, the dendroidal Segal condition implies that any $y(T)$ is a finite colimit of $y(\eta)$ and $y(C_r)$ for $r\ge0$, so the claim follows.
\end{proof}

\subsubsection{The right adjoint to $(-)\otimes \cP$} As a next step towards proving \cref{thm:operad-truncation-tower-monoidal}, we establish two properties of the right adjoint 
\[
	\AlgOpd_{\cP}(-)\colon \Opd\lra  \Opd
\]
to the tensor product functor $(-)\otimes\cP\colon\Opd\ra \Opd$ for a fixed operad $\cP$. Its value $\AlgOpd_{\cP}(\cO)$ at $\cO\in\Opd$ has as category of colours the category $\Alg_{\cP}(\cO)$ of $\cP$-algebras in $\cO$ from \cite[2.1.3.1]{LurieHA} (see \cite[3.2.4.4]{LurieHA} for the construction of $\AlgOpd_{\cP}(\cO)$; the fact that $\AlgOpd_{\cP}(-)$ is right adjoint to $(-)\otimes \cP$ can be deduced by the universal property of $(-)\otimes \cP$ in 2.2.5 loc.cit., but is also spelled out in a more general setting in \cite[Proposition 3.7]{Stewart}). 

\smallskip

\noindent The first property is a description of $\AlgOpd_{\cP}(-)$ if $\cP$ is in the image of the left adjoint $F$:

\begin{lem}\label{prop:fopd-alg-opd}For  $K,L\in\Cat_{/\Fin_*}$ and $\cO\in\Opd$, there are natural equivalence 
\[
	\Map_{\Cat_{/\Fin_*}}(K,\AlgOpd_{F(L)}(\cO)^\otimes) \simeq \Map_{\Cat_{/\Fin_*}}(\wedge_!(K \times L),\cO^\otimes)\simeq  \Map_{\Cat_{/\Fin_* \times \Fin_*}}(K \times L,\wedge^*\cO^\otimes).
\]
\end{lem}

\begin{proof}Using the adjunctions $(-)\otimes \cP\dashv\AlgOpd_{\cP}(-)$ and $F\dashv\inc$, we compute
\[\Map_{\Cat_{/\Fin_*}}(K,\AlgOpd_{F(L)}(\cO)^\otimes) \simeq \Map_{\Opd}(F(K),\AlgOpd_{F(L)}(\cO))\simeq  \Map_{\Opd}(F(K) \otimes F(L),\cO) \]
which we can further rewrite, using \cref{lem:f-tensor} as well as $F \dashv \inc$, as
\[\Map_{\Opd}(F(K) \otimes F(L),\cO)\simeq  \Map_{\Opd}(F(\wedge_!(K \times L)),\cO)  \simeq \Map_{\Cat_{/\Fin_*}}(\wedge_!(K \times L),\cO^\otimes),
\]
proving the first equivalence in the claim. The second one follows from the adjunction $ \wedge_!\dashv \wedge^*$.
\end{proof}
  
 The second property of the right adjoint $\Alg_{\cP}(-)$ we establish is that it preserves \emph{$k$-truncated operads}. Recall from \cite[Section 1.4.4]{KKoperadic} that an operad $\cO\in\Opd$ is \emph{$k$-truncated} for some $k\ge 1$ if it is unital and if the counit $\cO\ra \tau_*\tau(\cO)$ of the adjunction involving the truncation functor $\smash{\tau\colon \Opd^\un\ra\Opd^{\le k,\un}}$ is an equivalence. As explained in Section 1.4.4 loc.cit., this is equivalent to $\cO$ being unital and having the property that for any finite collection $(x_s)_{s\in S}$ of colours in $\cO$ and another colour $c$, the following map induced by inserting the unique $0$-ary operation   is an equivalence:
 \[{\Mul_{\cO}((x_s)_{s\in S}; c)\lra \lim_{S'\subseteq S, |S'|\le k}\Mul_{\cO}((x_s)_{s\in S'}; c).}\]

 \begin{prop}\label{prop:alg-is-truncated}If an operad $\cO$ is $k$-truncated operad, then so is $\AlgOpd_\cP(\cO)$ for any operad $\cP$.
\end{prop}

\begin{proof}The functor $\AlgOpd_{(-)}(\cO)\colon \Opd^\op\ra\Opd$ sends colimits to limits, since using $(-)\otimes \cP\dashv\AlgOpd_{\cP}(-)$ and the fact that the tensor product in $\Opd$ preserves colimits in both variables, we have
\[\textstyle{\Map_{\Opd}(\cQ,\AlgOpd_{\colim_i\,\cP_i}(\cO))\simeq \Map_{\Opd}(\cQ \otimes (\colim_i\,\cP_i),\cO)\simeq \Map_{\Opd}(\colim_i\,(\cQ \otimes \cP_i),\cO)}\] which is turn equivalent to
\[\textstyle{\lim_i \Map_{\Opd}(\cQ \otimes \cP_i,\cO)\simeq \lim_i \Map_{\Opd}(\cQ,\Alg_{\cP_i}(\cO))\simeq \Map_{\Opd}(\cQ,\lim_i\AlgOpd_{\cP_i}(\cO)).}
\]
Note furthermore that $k$-truncated operads are closed under taking limits, since all functors featuring in the unit $\id \to \tau_* \tau$ are right adjoints and thus preserve limits. Combining these two properties with \cref{lem:colim}, it suffices to prove the statement when $\cP=F(\alpha)$ for $\alpha$ being the simplices $e_1 \colon [0]\ra \Fin_*$  and $c_r \colon [1]\ra \Fin_*$ for $r\ge0$ from \cref{ex:trees-from-simplices}. We prove this by a repeated application of \cref{prop:fopd-alg-opd}:

In the case $\alpha=e_1$, since $\wedge_!(-\times e_1)\simeq\id_{\Cat_{\Fin_*}}$,  \cref{prop:fopd-alg-opd} implies $\AlgOpd_{F(e_1)}(\cO)\simeq \cO$, so there is nothing to prove. In the case $\alpha=c_r$, we first note that since $\wedge_!(e_1\times c_r)\simeq c_r$, the core of the category of colours of $\AlgOpd_{F(c_r)}(\cO)$ is given by (cf.\,\cref{ex:trees-from-simplices} \ref{enum:cn-description-e1}) \[\Map_{\Cat_{/\Fin_*}}(e_1,\AlgOpd_{F(c_r)}(\cO)^\otimes)\simeq \Map_{\Cat_{/\Fin_*}}(\wedge_!(e_1\times c_r),\cO^\otimes)\simeq \Map_{\Cat_{/\Fin_*}}(c_r,\cO^\otimes)\simeq (c_r\times_{\Fin_*}\cO^\otimes)^\simeq,\]
 so the colours of $\AlgOpd_{F(c_r)}(\cO)$ can be described (cf.\,\cref{ex:trees-from-simplices} \ref{enum:cn-description-cr}) as triples $(\ul{x},x,m)$ of a collection $\ul{x} = (x^s)_{s \in \ul{r}}$ of colours in $\cO$, a colour $x$ in $\cO$, and a multi-operation $m \in \Mul_\cO((x^s)_{s \in \ul{r}},x)$. To determine the $n$-ary multi-operations in $\AlgOpd_{F(c_r)}(\cO)$ we consider the $1$-simplex $c_n \colon [1] \to \Fin_*$ given by the active map $\langle n \rangle \to \langle 1 \rangle$ whose two faces are given by the $0$-simplices $e_n,e_1 \colon [0] \to \Fin_*$ given by $\langle n \rangle,\langle 1 \rangle\in\Fin_*$. The $n$-ary multi-operations in $\AlgOpd_{F(c_r)}(\cO)$ from colours $(\ul{x}_i,x_i,m_i)_{i \in \ul{n}}$ to a colour $(\ul{x},x,m)$ are, by definition (see \cite[2.1.1.6]{LurieHA}), the fibre of the upper horizontal map in
\[
	\begin{tikzcd}[ar symbol/.style = {draw=none,"\textstyle#1" description,sloped},
	equivalent/.style = {ar symbol={\simeq}}, row sep=0.3cm]
	\Map_{/\Fin_*}(c_n,\Alg_{F(c_r)}(\cO)^\otimes) \rar \arrow[d,equivalent] & \Map_{/\Fin_*}(e_n,\AlgOpd_{F(c_r)}(\cO)^\otimes) \times \Map_{/\Fin_*}(e_1,\AlgOpd_{F(c_r)}(\cO)^\otimes) \arrow[d,equivalent]\\[-5pt]
	\Map_{/\Fin_*}(\wedge_!(c_n \times c_r),\cO^\otimes) \rar & \Map_{/\Fin_*}(\wedge_!(e_n \times c_r),\cO^\otimes) \times \Map_{/\Fin_*}(\wedge_!(e_1 \times c_r),\cO^\otimes)\end{tikzcd}
\]
at the point $((\ul{x}_i,x_i,m_i)_{i \in \ul{n}}, (\ul{x},x,m))$ in \[ 	\Map_{/\Fin_*}(e_n,\AlgOpd_{F(c_r)}(\cO)^\otimes) \times \Map_{/\Fin_*}(e_1,\AlgOpd_{F(c_r)}(\cO)^\otimes)\simeq \big(\sqcap^n \Alg_{F(c_r)}(\cO)^\simeq\big)\times \Alg_{F(c_r)}(\cO)^\simeq\] The fibre of the bottom horizontal map over the image of this point is the space of commutative squares in $\cO^\otimes$ of the form (the solid arrows are predetermined)
\[\begin{tikzcd}[row sep=0.4cm] (x^s_i)_{(i,s) \in \ul{n} \times \ul{r}} \rar{(m^s)_{s \in \ul{r}}} \dar[dashed] &[20pt]  (x_i)_{i \in \ul{n}} \dar[dashed] \\
(x^s)_{s \in \ul{r}} \rar{m} & x, \end{tikzcd}\]
 so we have a pullback of the form
	\[\begin{tikzcd}[column sep=0.3cm] \Mul_{\AlgOpd_{F(c_r)}(\cO)}\big((\ul{x}_i,x_i,m_i)_{i \in \ul{n}},(\ul{x},x,m)\big) \rar \dar &[30pt] \Mul_{\cO}\big((x_i)_{i \in \ul{n}},x\big) \dar{(-) \circ (m^1,\ldots,m^r)} \\
		\bigsqcap_{s \in \ul{r}} \Mul_{\cO}\big((x_i^s)_{i \in \ul{n}},x^s\big) \rar{m \circ(-)} & \Mul_{\cO}\big((x_i^s)_{(i,s) \in \ul{n} \times \ul{r}},x\big). \end{tikzcd}\]
Given this description of the multi-operations, we argue as follows: if $n=0$ then all terms are (products of) 0-ary operations, so contractible if $\cO$ is unital, which implies that $\AlgOpd_{F(c_r)}(\cO)$ is also unital. For $n>1$, we map the previous pullback by inserting $0$-ary operations to the pullback
	\[\begin{tikzcd}[row sep=0.3cm] \underset{\substack{J \subseteq \ul{n} , |J| \leq k}}{\lim} \Mul_{\AlgOpd_{F(c_r)}(\cO)}\big((\ul{x}_j,x_j,m_j)_{j \in J},(\ul{x},x,m)\big) \rar \dar[shorten=-.5ex,shift={(0,.1)}] &[20pt] \underset{\substack{J \subseteq \ul{n}, |J| \leq k}}{\lim} \Mul_{\cO}\big((x_j)_{j \in J},x\big) \dar[shorten=-.5ex,shift={(0,.1)}] \\[-5pt]
		\underset{J \subseteq \ul{n}, |J| \leq k}{\lim} \bigsqcap_{s \in \ul{r}} \Mul_{\cO}\big((x_j^s)_{j \in J},x^s\big) \rar & \underset{\substack{J \subseteq \ul{n} , |J| \leq k}}{\lim} \Mul_{\cO}\big((x_j^s)_{(j,s) \in J \times \ul{r}},x\big). \end{tikzcd}\] Since $\cO$ is $k$-truncated, the maps between all corners of the pullbacks except the top left one is an equivalence (for the bottom right one this uses  \cref{lem:cubical-lemma} below), so the map between the top left corners is an equivalence as well, which implies that $\AlgOpd_{F(c_r)}(\cO)$ is $k$-truncated as claimed.
\end{proof}

\begin{lem}\label{lem:cubical-lemma}Fix a functor $X \colon \cP(S\times T)^\op \to \cS$ for finite sets $S$ and $T$ where $\cP(S\times T)$ is the poset of subsets of $S\times T$ ordered by inclusion. If for fixed $k\ge1$ and all subsets $L \subseteq S \times T$ the left map in
	\[{X(L) \lra \lim_{\substack{K \subseteq L, |K| \leq k}} X(K)}\quad\quad\quad\quad{X(S \times T) \lra \lim_{\substack{J \subseteq S, |J| \leq k}} X(J \times T).}\]
	is an equivalence, then the right map is an equivalence as well.
\end{lem}

\begin{proof}Consider the commutative diagram
	\[\begin{tikzcd}[row sep=0.2cm, column sep=0.5cm] X(S\times T) \rar{\simeq} \dar & \underset{\substack{K \subseteq S \times T, |K| \leq k}}{\lim}  X(K) \dar \\
		\underset{\substack{J \subseteq S, |J| \leq k}}{\lim} X(J \times T) \rar{\simeq} & \underset{\substack{J \subseteq S, |J| \leq k}}{\lim} \Big(\underset{\substack{K \subseteq J \times T, |K| \leq k}}{\lim}  X(K)\Big) \simeq \underset{(J,K) \in \cat{C}}{\lim}  X(K) \end{tikzcd}\]
	where the horizontal maps are equivalences by hypothesis and $\cat{C}$ is the poset of pairs of sets $(J,K)$
with are both of cardinality $\le k$ and satisfy $J\subset S$ and $K\subset J\times T$. It remains to show that the right vertical map is an equivalence. The latter is induced by the functor $\pr_2(J,K)=K$. This has a left adjoint sending $K\subseteq S\times T$ to $(\pr_1(K),K)$ so $\pr_2$ is right cofinal and the claim follows.
\end{proof}

We have everything in place to prove \cref{thm:operad-truncation-tower-monoidal}, apart from \cref{enum:trunmon-iv}.
 
 \begin{proof}[Proof of \cref{thm:operad-truncation-tower-monoidal} except Item \ref{enum:trunmon-iv}]
As explained in \cref{sec:operad-conventions}, firstly, the functors in \eqref{equ:operad-truncation-tower-app} admit fully faithful left and right adjoint, secondly, $\Opd^\un$ inherits a symmetric monoidal structure from Lurie's symmetric monoidal structure on $\Opd$, and, thirdly, this monoidal structure preserves colimits in both variables. Apart from Items \ref{enum:trunmon-iii} and \ref{enum:trunmon-iv}, the claim thus follows from \cref{lem:general-monoidal-tower} once we establish the condition \eqref{eqn:loc-condition-gen-tower} for the truncation functors $\tau\colon \Opd^\un\ra \Opd^{\le k,\un}$. By the Yoneda lemma, this condition follows from showing that, given a map $\varphi\colon \cO\ra \cO'$ in $\Opd^\un$ with $\tau(\varphi)$ an equivalence, the map $\tau(\varphi\otimes \id_{\cP})^*\colon \Map_{\Opd^{\le k,\un}}(\tau(\cO'\otimes \cP), \cQ)\ra \Map_{\Opd^{\le k,\un}}(\tau(\cO\otimes \cP), \cQ)$ is an equivalence. From the adjunctions $\tau\dashv\tau^*$- and $(-)\otimes \cP\dashv \AlgOpd_{\cP}(-)$, we get natural equivalences 
\[\smash{
	\Map_{\Opd^{\le k,\un}}(\tau(\cO\otimes \cP), \cQ)   \simeq   \Map_{\Opd^{\un}}(\cO\otimes \cP, \tau_*(\cQ)) \simeq  \Map_{\Opd^{\un}}(\cO, \AlgOpd_{\cP}(\tau_*(\cQ))).}
\]
Since  $\tau_*(\cQ)$ is $k$-truncated as $\tau_*$ is fully faithful, the same holds for $\AlgOpd_{\cP}(\tau_*(\cQ))$ by \cref{prop:alg-is-truncated}, so the unit $\AlgOpd_{\cP}(\tau_*(\cQ))\ra \tau_*\tau(\AlgOpd_{\cP}(\tau_*(\cQ))$ is an equivalence and we get natural isomorphisms
\[\smash{
\Map_{\Opd^{\un}}(\cO, \AlgOpd_{\cP}(\tau_*(\cQ)))\simeq\Map_{\Opd^{\un}}(\cO, \tau_*\tau(\AlgOpd_{\cP}(\tau_*(\cQ))))\simeq \Map_{\Opd^{\le k, \un}}(\tau(\cO), \tau(\AlgOpd_{\cP}(\tau_*(\cQ)))).}\]
In total, we have $\Map_{\Opd^{\le k,\un}}(\tau(\cO\otimes \cP), \cQ)\simeq \Map_{\Opd^{\le k, \un}}(\tau(\cO), \tau(\AlgOpd_{\cP}(\tau_*(\cQ))))$, so it follows that $\tau(\varphi\otimes \id_{\cP})^* $ is indeed an equivalence since $\tau(\varphi)^* $ is one by assumption, so condition \eqref{eqn:loc-condition-gen-tower} follows.
  
To prove Item  \ref{enum:trunmon-iii}, we write $\Cat^{\otimes}$ and  $\Cat^{\times}$ for the category $\Cat\simeq \Opd^{\le1,\un}$ equipped with the symmetric monoidal structure produced in the first part, and with the cartesian structure induced by taking direct products as constructed in \cite[2.4.1]{LurieHA}, respectively. The goal is to lift $\id_{\Cat}$ to a symmetric monoidal equivalence $\Cat^{\otimes}\simeq \Cat^{\times}$. To do so, we consider the composition
\vspace{-0.185cm}
\begin{equation}\label{equ:mon-comparison-cat}\begin{tikzcd}[column sep=1.2cm]\Cat^\times \rar{\iota}& \Opd \rar{(-)\otimes E_0}&\Opd^\un\rar{\tau\simeq (-)^\col}&\Cat^{\otimes_{\BV}}\end{tikzcd}\end{equation}
 of symmetric monoidal functors where $\tau$ is the symmetric monoidal functor from the first part (whose underlying functor is  $(-)^\col \colon \Opd^\un\ra \Cat$ by construction) and $(-)\otimes E_0$ is the symmetric monoidal functor discussed in \cref{sec:operad-conventions}. On underlying categories $\iota$ is the left adjoint to $(-)^\col \colon \Opd\ra \Cat$ (see \cite[2.1.4.11]{LurieHA}). This admits a symmetric monoidal lift as indicated, as a result of 2.2.5.11 loc.cit.. We are left to show that this composition agrees with the identity on underlying categories. We can check this after taking right adjoints in which case \eqref{equ:mon-comparison-cat} becomes on underlying categories
\vspace{-0.185cm}
\[\begin{tikzcd}[column sep=1.2cm]\Cat \rar{(-)^\sqcup}& \Opd^\un\rar{\subset}&\Opd\rar{(-)^\col}&\Cat\end{tikzcd}\]
where $(-)^\sqcup$ is the functor that takes cocartesian operads (see 2.4.3.3 loc.cit.\,for the construction and 2.4.3.9 loc.cit. for the fact that it is right adjoint to $(-)^\col$). But $((-)^\sqcup)^\col\simeq \id_{\Cat}$ is clear from the construction of the cocartesian operad (see 2.4.3 loc.cit.), so the claim follows. 
 \end{proof}

\subsection{Tensor products of 2-truncated reduced operads} We now turn towards proving \cref{thm:operad-truncation-tower-monoidal} \ref{enum:trunmon-iv}, which says that under the equivalence $\smash{\Opd^{\le 2, \red}\simeq \PSh(\mathrm{B}\Sigma_2)}$ from \cref{lem:2-truncated-reduced}, the tensor product of $2$-truncated reduced operads $\smash{\otimes \colon \Opd^{\le 2, \red}\times \Opd^{\le 2, \red}\ra \Opd^{\le 2, \red}}$ is given by the join of $\Sigma_2$-spaces $\smash{\ast\colon \PSh(\mathrm{B}\Sigma_2)\times \PSh(\mathrm{B}\Sigma_2)\ra \PSh(\mathrm{B}\Sigma_2)}$. If these functors were to preserve colimits in both variables, then one could use the universal property of $\PSh(-)$ to see it suffices to prove this after restriction along the Yoneda embedding $\mathrm{B}\Sigma_2\subset \PSh(\mathrm{B}\Sigma_2)$. But this not the case: for example $\varnothing\in\PSh(\mathrm{B}\Sigma_2)$ is initial but $\varnothing\ast X\not\simeq \varnothing$ is not whenever $X\not\simeq\varnothing$. However, a variant of this strategy turns out to work, using that the join and tensor product \emph{do} preserve colimits in both variables once one takes suitable slices. We first make the latter precise, and then use it to prove \cref{thm:operad-truncation-tower-monoidal} \ref{enum:trunmon-iv}.

\subsubsection{Slice-colimit preserving functors} Fix categories $\cD$ and $\cE$ with colimits. 

\begin{dfn}A functor $F\colon \cD\ra \cE$ is \emph{slice-colimit preserving} if $\cD\simeq \cD_{d_0/}\ra \cE_{F(d_0)/}$ preserves colimits where $d_0\in \cD$ is the initial object in $\cD$. .\end{dfn} 
That is, $F$ is slice-colimit preserving if and only if the natural map $\colim_{i\in I}F(c_i)\cup_{\colim_{i\in I}F(c_0)}F(c_0) \to F(\colim_{i\in I}c_i)$ is an equivalence for all diagrams $I\ra \cC$.  We write \[\Fun^{\scp}(\cD,\cE)\subset \Fun(\cD,\cE)\] for the full subcategory on the slice-colimit preserving functors Since colimits commute with colimits, this subcategory is closed under colimits in $\Fun(\cD,\cE)$. Below we will give an alternative description of this subcategory for $\cD=\PSh(\cC)$ the category of space-valued presheaves on a category $\cC$. To explain this, recall that the \emph{left-cone} $\cC^\lhd$ on a category $\cC$ is the category obtained from $\cC$ by formally adjoining an initial object \cite[1.2.8.4]{LurieHTT}; this is also the pushout $\cC^\lhd\simeq (\cC\times[1])\cup_{\cC\times\{0\}}\ast$ (see 4.2.1 loc.cit.). It has the universal property that for a category $\cE$ with an initial object, restriction along $\cC\subset \cC^\lhd$ induces an equivalence $\Fun^\ini(\cC^\lhd,\cE)\simeq \Fun(\cC,\cE)$ where the $\ini$-superscript indicates that we restrict to the full subcategory of functors that preserve the initial object. The Yoneda embedding extends to a fully faithful functor $\cC^\lhd\hookrightarrow \PSh(\cC)$ by sending the initial object to the initial presheaf (the constant presheaf on $\varnothing\in \cS$).

\begin{lem}\label{lem:scp2}For categories $\cC$ and $\cE$ where $\cE$ has colimits, there is an equivalence
\[\Fun^\scp(\PSh(\cC),\cE)\simeq\Fun(\cC^\lhd,\cE)\]
induced by restriction along $\cC^\lhd\hookrightarrow \PSh(\cC)$.
\end{lem}

\begin{proof}We first observe that the evaluation $\ev_{\varnothing}\colon \Fun(\cC^\lhd,\cE)\ra \cE$ is a cocartesian fibration  that straightens to the functor $\Fun(\cC,\cE_{-/})$ where the functoriality of $\cE_{-/}\colon \cE\ra \Cat$ is induced by taking pushouts in $\cE$. To see this one can use $\cC^\lhd\simeq (\cC\times[1])\cup_{\cC\times\{0\}}\ast$ (see \cite[4.2.1]{LurieHTT}) to conclude that $\ev_{\varnothing}\colon \Fun(\cC^\lhd,\cE)\ra \cE$ is pulled back from the cocartesian fibration $\ev_0\colon \Fun([1],\Fun(\cC,\cE))\ra \Fun(\cC,\cE)$ that classifies $\Fun(\cC,\cE)_{-/}$ along the functor $\cE\ra \Fun(\cC,\cE)$ that takes constant functors. For a functor $\varphi\colon \cC\ra \cC'$, the restriction $(\varphi^\lhd)^*\colon \Fun(\cC'^\lhd,\cE)\ra \Fun(\cC^\lhd,\cE)$ is a map of cocartesian functors that straightens to $\varphi^*\colon \Fun(\cC',\cE_{-/})\ra \Fun(\cC,\cE_{-/})$.

Writing  $\smash{\PSh^{\neq \varnothing}(\cC)\subset \PSh(\cC)}$ for the full subcategory on presheaves that are not equivalent to the initial presheaf $\varnothing\in \PSh(\cC)$, the inclusion $\iota\colon \cC^\lhd\hookrightarrow \PSh(\cC)$ is equivalent to the result of applying $(-)^\lhd$ to the (corestricted) Yoneda embedding $\cC\hookrightarrow \PSh^{\neq \varnothing}(\cC)$. By the discussion above, the restriction $\Fun(\PSh(\cC),\cE)\ra \Fun(\cC^\lhd,\cE)$ is therefore a map of cocartesian fibrations over $\cE$ that straightens to
\[\smash{\Fun^\ini(\PSh(\cC),\cE_{-/})\simeq \Fun(\PSh^{\neq \varnothing}(\cC),\cE_{-/})\xlra{\iota^*} \Fun(\cC,\cE_{-/}).}\]
Since the pushout functors $\cE_{e/}\ra \cE_{f/}$ for morphisms $e\ra f$ preserve colimits, we have a objectwise full subfunctor $\Fun^L(\PSh(\cC),\cE_{-/})\subset \Fun^\ini(\PSh(\cC),\cE_{-/})$ given by the colimit-preserving functors. Unpacking the definitions, we see that after unstraightening, the full subfunctor $\Fun^L(\PSh(\cC),\cE_{-/})\subset \Fun^\ini(\PSh(\cC),\cE_{-/})$ corresponds to the subcategory $\Fun^\scp(\PSh(\cC),\cE)\subseteq \Fun(\PSh(\cC),\cE)$ over $\cE$, so it suffices to show that the composition
\[\smash{\Fun^L(\PSh(\cC),\cE_{-/})\subseteq \Fun^\ini(\PSh(\cC),\cE_{-/}){\simeq} \Fun(\PSh^{\neq \varnothing}(\cC),\cE_{-/})\xlra{\iota^*} \Fun(\cC,\cE_{-/})}\] is an objectwise natural equivalence. But this composition is given by restriction along the Yoneda embedding, so this follows from the universal property $\Fun^L(\PSh(C),-)\simeq \Fun(\cC,-)$ of $\PSh(\cC)$.
\end{proof}
Given another category $\cD'$, we call functor $F\colon \cD\times \cD'\ra\cE$ \emph{slice-colimit preserving in both variables} if $F(d,-)$ and $F(-,d')$ are slice-colimit preserving in the previous sense for all $d\in \cD$ and $d'\in \cD'$. We write $\Fun^\bscp(\cD\times \cD',\cE)\subset \Fun(\cD\times \cD',\cE)$ on the full subcategory on such functors.

\begin{cor}\label{cor:cpbv-cor}For categories $\cC$, $\cC'$, and $\cE$ where $\cE$ has colimits, there is an equivalence
\[\Fun^\bscp(\PSh(\cC)\times \PSh(\cC'),\cE)\simeq\Fun(\cC^\lhd\times \cC'^\lhd,\cE)\]
which is induced by restriction along the product of $\cC^\lhd\hookrightarrow \PSh(\cC)$ and $\cC'^\lhd\hookrightarrow \PSh(\cC')$.
\end{cor}
\begin{proof}By adjunction we have $\Fun^\bscp(\PSh(\cC)\times \PSh(\cC'),\cE)\simeq \Fun^\scp(\PSh(\cC),\Fun^\scp(\PSh(\cC'),\cE))$, so the claim follows from two applications of \cref{lem:scp2}.
\end{proof}

\begin{rem}\label{rem:if-also-initial}After passing to the full subcategory $\Fun^\ini(\cC^\lhd\times \cC'^\lhd,\cE)\subset \Fun(\cC^\lhd\times \cC'^\lhd,\cE)$, the right hand side of the equivalence in \cref{cor:cpbv-cor} can be reexpressed as follows: From $ \cC^\lhd\simeq (\cC\times[1])\cup_{\cC\times\{0\}}\ast$, one can deduce  $\cC^\lhd\times \cC'^\lhd\simeq (\cC\widehat{\times} \cC')^\lhd$ with $\cC\widehat{\times} \cC'\coloneq (\cC\sqcup \cC')\cup_{\cC\times\cC'\times\{0,1\}}(\cC\times \cC'\times\Lambda^2_2)$ where $\Lambda_2^2=(0<2>1)$ is the cospan category and the functor $\cC\times\cC'\times\{0,1\}\ra \cC\sqcup \cC'$ is given by the projection to $\cC$ on $\cC\times\cC'\times\{0\}$ and by that to $\cC'$ on $\cC\times\cC'\times\{0,1\}$. From this  we get
\[\hspace{-0.2cm}\smash{\Fun^\ini(\cC^\lhd\times \cC'^\lhd,\cE)\simeq \Fun(\cC\widehat{\times} \cC',\cE)\simeq \Fun(\cC',\cE)\times_{\Fun(\cC\times \cC',\cE)}\Fun(\cC\times \cC',\cE)^{\Lambda^2_2}\times_{\Fun(\cC\times \cC',\cE)}\Fun(\cC',\cE).}\]
In other words, the category $\Fun^\ini(\cC^\lhd\times \cC'^\lhd,\cE)$ is equivalent to the category of triples of functors $F_2\colon \cC\times \cC'\ra \cE$, $F_0\colon \cC\ra \cE$, and $F_1\colon \cC\ra\cE$ together with a zig-zag of natural transformations $(F_0\circ \pr_1)\ra F_2\leftarrow (F_1\circ \pr_2)$ in $\Fun(\cC\times \cC', \cE)$.

The relevant case for us is $\cC=\cC'=\mathrm{B}G$ and $\cE=\PSh(\mathrm{B}G)$ for a group $G$, so $\Fun(\mathrm{B}G\times \mathrm{B}G,\PSh(\mathrm{B}G))\simeq \Fun((\mathrm{B}G\times \mathrm{B}G)\times \mathrm{B}G,\cS)$ is the category of spaces with a $(G\times G)\times G$-action, and $\Fun(\mathrm{B}G\widehat{\times} \mathrm{B}G,\PSh(\mathrm{B}G))$ is the category of zig-zags of $(G\times G)\times G$-spaces $X_0\leftarrow X_2\rightarrow X_1$ with a factorisation of the action on $X_0$ through the projection $(G\times G)\times G\ra G\times G$ to the first and third factor, and a similar factorisation for the action on $X_0$ involving the second and third factor.
\end{rem}

The prototypical example of a functor that preserves slice-colimits in both variables is the join:

\begin{lem}\label{lem:join-scpv} For any category $\cC$, the functor 
 $(-) \ast (-) \colon \PSh(\cC) \times \PSh(\cC) \to \PSh(\cC)$ given by taking levelwise join is slice-colimit preserving in both variables.
\end{lem}
\begin{proof}As colimits in presheaf categories are computed objectwise it suffices to consider the case $\cC=*$, i.e.\,the functor $\ast\colon \cS \times \cS \to \cS$ giving by taking joins, so $X\ast Y=X\cup_{X\times Y}Y$. The task is to show that for a diagram $(X_i)_{i\in I}\colon I\ra \cS$ and a fixed space $Y$, the canonical map $\colim_i(X_i\ast Y)\cup_{\colim_iY} Y\ra (\colim_i\,X_i)\ast Y$ is an equivalence, i.e.\,that the right square in the  commutative diagram
\[\begin{tikzcd}[column sep=0.3cm, row sep=0.3cm]
\colim_i(X_i\times Y)\dar\rar&\colim_i\,Y\dar\rar& Y\dar\\
\colim_i\,X_i\rar&\colim_i(X_i\ast Y)\rar&(\colim_i\,X_i)\ast Y
\end{tikzcd}\]
induced by the projections and inclusions is a pushout. The left square is a colimit of pushouts, so a pushout. Since outer square is a pushout since $(-)\times Y$ preserves colimits (it is a left adjoint), the right square is also a pushout and the claim follows.
\end{proof}

Tensor products of reduced operads also preserve slice-colimits in both variables:

\begin{lem}\label{lem:bv-red-colimits}Fix $\cP\in\Opd^{\red}$.\begin{enumerate}
		\item\label{equ:red-colim-i} The operad $E_0$ is initial in $\Opd^\red$, so the forgetful functor $\smash{\Opd^{\red}_{E_0/}\ra \Opd^{\red}}$ is an equivalence.
		\item\label{equ:red-colim-ii} The inclusion ${\Opd^{\red}_{\cP/}\subset \Opd^{\un}_{\cP/}}$ is a fully faithful left adjoint, so preserves and detects colimits. In particular, colimits in the category of reduced operads $\smash{\Opd^{\red}\simeq \Opd^{\red}_{E_0/}}$ can be computed in $\smash{\Opd^{\un}_{E_0/}}$.
		\item\label{equ:red-colim-iii} The binary tensor product $\smash{ \Opd^{\red}\times \Opd^{\red}\ra \Opd^{\red}}$ is slice-colimit preserving in both variables.
	\end{enumerate}
	Moreover, \ref{equ:red-colim-i}-\ref{equ:red-colim-iii} holds when replacing $\smash{\Opd^{\red}}$ by  $\smash{\Opd^{\le k, \red}}$ and $\smash{\Opd^{\un}}$ by $\smash{\Opd^{\le k, \un}}$ for any $1\le k\le\infty$. 
\end{lem}

\begin{proof}Item \ref{equ:red-colim-i} and Item \ref{equ:red-colim-ii} for $\cP=E_0$ hold by  \cite[2.2.7--2.2.9]{SchlankYanovski}. Item \ref{equ:red-colim-ii} for general $\cP$ follows from the case $\cP=E_0$ in view of the equivalence of pairs \[\smash{\big((\Opd^{\red}_{E_0/})_{E_0 \to \cP/} \subset (\Opd^{\un}_{E_0/})_{E_0 \to \cP/}\big)\xra{\simeq }\big(\Opd^{\red}_{\cP/} \subset \Opd^{\un}_{\cP/}\big)}\] and the observation that if a functor $\varphi\colon \cC\ra\cD$ is a fully faithful left adjoint, then so is the induced functor $ \cC_{c/}\ra\cD_{\varphi(c)/}$ on undercategories for any $c\in\cC$ (this is a straightforward check). Using Item \ref{equ:red-colim-ii}, we compute 
	\[\smash{(\colim^{\Opd^{\red}}_i \cO_i) \otimes \cP \simeq (\colim^{\Opd^{\un}_{E_0/}}_i \cO_i) \otimes \cP\simeq (\colim^{\Opd^{\un}}_i \cO_i \cup_{\colim^{\Opd^{\un}}_i E_0} E_0) \otimes \cP} \]
	which we can---using that tensor product of unital operads preserves colimits in each entry (see \cref{thm:operad-truncation-tower-monoidal} \ref{enum:trunmon-iii})---further simplify to
	\[\hspace{-0.1cm}\smash{(\colim^{\Opd^{\un}}_i \cO_i \cup_{\colim^{\Opd^{\un}}_i E_0} E_0) \otimes \cP \simeq \colim^{\Opd^{\un}}_i (\cO_i\otimes \cP) \cup_{\colim^{\Opd^\un}_i \cP} \cP\simeq \colim^{\Opd^{\un}_{\cP/}}_i (\cO_i\otimes \cP)},\]
	so Item \ref{equ:red-colim-iii} follows. Moreover, the  same argument applies in the $k$-truncated case for any $1\le k\le \infty$, once we establish Item \ref{equ:red-colim-i} and Item \ref{equ:red-colim-ii} for $\cP=E_0$.  Item \ref{equ:red-colim-i} follows from the observation that $E_0=\overline{C}_0$ is in the image of the fully faithful left adjoint $\tau_!\colon \Opd^{\le k,\un}\ra \Opd^{\un}$, as observed as part of \cref{sec:prelim-dendroidal}. To see that $\smash{\Opd^{\le k,\red}_{E_0/} \subset  \Opd^{\le k,\un}_{E_0/}}$ is a left adjoint, we can argue as in the proof of \cite[Proposition 2.2.9]{SchlankYanovski} ($(-)^\col$ is denoted $L$ in loc.cit.), using that the functor 
	$\smash{\tau\colon \Opd^{\le k,\un} \to (\Opd^{\le 1,\un} \simeq \Cat)}$ has fully faithful right and left adjoints, given by $\tau_*$ and $\tau_!$ (see \cref{sec:operad-truncation}).\end{proof}

We now have all the ingredients together to prove  \cref{thm:operad-truncation-tower-monoidal} \ref{enum:trunmon-iv}---the comparison between the join and the tensor products of $2$-truncated reduced operads.

\subsubsection{Proof of \cref{thm:operad-truncation-tower-monoidal} \ref{enum:trunmon-iv}} We first discuss various objects and maps that will be used during the proof. For examples, we will use the standard equivalence to identify the join of spheres  \[\smash{s\colon S^n\ast S^m\xlra{\simeq} S^{n+m+1}\quad\text{for }n,m\ge0}\] which is induced by the map $S^n\times[0,1]\times S^m\ra S^{n+m+1}$ that sends $(x,t,y)$ to $(\cos(\tfrac{t\pi}{2})x,\sin(\tfrac{t\pi}{2})y)$. We will also use the explicit point-set model for the operad $E_d$ defined in terms of rectilinear embeddings between copies of $(-1,1)^d$ as in \cite[5.1.0.2]{LurieHA}. In this model there there is a preferred equivalence \[\smash{c\colon E_d(2)\xlra{\simeq} S^{d-1}\quad\text{for }d\ge0}\] by sending a rectilinear embedding $e\colon (-1,1)^d\times\{1,2\}\hookrightarrow (-1,1)^d$ to $\smash{\tfrac{e(0,2)-e(0,1)}{\parallel e(0,2)-e(0,1)\parallel}}$. There is also a canonical $(\Sigma_2)^{\times d}$-action on $E_d$ where the $i$th $\Sigma_2$-factor acts by conjugation with multiplication by $-1$ in the $i$th factor of $(-1,1)^d$. We will also make use of the explicit equivalences of operads
\[\smash{\gamma\colon E_n\otimes E_m\xlra{\simeq}E_{n+m}\quad\text{for }n,m\ge0}
\]
from \cite[5.1.2.2]{LurieHA} induced by the bifunctors of operads $E_n\times E_m\ra E_{n+m}$ from 5.1.2.1 loc.cit.\,given by taking cartesian products. The key ingredient in the proof of \cref{thm:operad-truncation-tower-monoidal} \ref{enum:trunmon-iv} will be the following commutative diagram of spaces with  $(\Sigma_2\times\Sigma_2)\times \Sigma_2$-action and equivariant maps between them
\begin{equation}\label{equ:big-diagram-2-truncated}
\begin{tikzcd}
S^0\ast\varnothing\arrow[d,"\id\ast\inc",swap]\arrow[r,"s","\simeq"']&S^0\arrow[d,"{x\mapsto(x,0)}"]&E_1(2)\arrow[d,"(-)\times\id_\bfR"]\arrow[l,"c","\simeq"',swap]&(E_1\otimes E_0)(2)\arrow[d,"\id\otimes\inc"]\arrow[l,"\gamma","\simeq"',swap]\\
S^0\ast S^0\arrow[r,"s","\simeq"']&S^1&E_2(2)\arrow[l,"c","\simeq"',swap]&(E_1\otimes E_1)(2)\arrow[l,"\gamma","\simeq"',swap]\\
\varnothing \ast S^0\arrow[u,"\inc\ast\id"]\arrow[r,"s","\simeq"']&S^0\arrow[u,"{x\mapsto(0,x)}",swap]&E_1(2)\arrow[u,"\id_\bfR\times (-)",swap]\arrow[l,"c","\simeq"',swap]&(E_0\otimes E_1)(2).\arrow[u,"\inc\otimes\id",swap]\arrow[l,"\gamma","\simeq"',swap]
\end{tikzcd}
\end{equation}
Here the $(\Sigma_2\times\Sigma_2)\times \Sigma_2$-action on the spaces in the diagram is as follows:
\begin{enumerate}
\item the first $\Sigma_2$-factor acts on the first (respectively last) column via multiplication by $-1$ on the first join-factor (respectively conjugation by $-1$ on the first tensor-factor), on the second (respectively third) column by multiplication (respectively conjugation) by $-1$ on the top, by $(-1,1)$ on the middle, and by $1$ on the bottom,
\item the second $\Sigma_2$-factor acts on the first (respectively last) column via multiplication by $-1$ on the second join-factor (respectively conjugation by $-1$ on the second tensor-factor), on the second (respectively third) column by multiplication (respectively conjugation) by $1$ on the top, by $(1,-1)$ on the middle, and by $-1$ on the bottom,
\item the third $\Sigma_2$-factor acts diagonally via the first two $\Sigma_2$-factors on first two columns, and via the canonical $\Sigma_2$-action on $\cO(2)$ for a reduced operad $\cO$ on the final two columns.
\end{enumerate}

With these preparations out of the way, we now commence the proof of \cref{thm:operad-truncation-tower-monoidal} \ref{enum:trunmon-iv}. The goal is to show that the the functor
\begin{equation}\label{equ:join-functor-1}\smash{\PSh(\mathrm{B}\Sigma_2)\times \PSh(\mathrm{B}\Sigma_2)\xrightarrow{(-)\ast(-)}\PSh(\mathrm{B}\Sigma_2)}\end{equation}
given by the join is equivalent to the functor 
\begin{equation}\label{equ:bv-functor-2}\smash{\PSh(\mathrm{B}\Sigma_2)\times \PSh(\mathrm{B}\Sigma_2)\simeq \Opd^{\le 2, \red}\times \Opd^{\le 2, \red}\xlra{\otimes}\Opd^{\le 2, \red}\simeq \PSh(\mathrm{B}\Sigma_2).}
\end{equation}
involving the equivalence $(-)(2)\colon \Opd^{\le 2, \red}\simeq \cS^{\Sigma_2}=\PSh(\mathrm{B}\Sigma_2)$ from \cref{lem:2-truncated-reduced} and the tensor product of truncated reduced operads from \cref{thm:operad-truncation-tower-monoidal} \ref{enum:trunmon-iii}. Both of the functors \eqref{equ:join-functor-1} and \eqref{equ:bv-functor-2} are slice-colimit preserving in both entries, by \cref{lem:join-scpv} and by \cref{lem:bv-red-colimits} \ref{equ:red-colim-iii} and its addendum. By \cref{cor:cpbv-cor}, it thus suffices to show that they are equivalent after restriction along the inclusion $\mathrm{B}\Sigma_2^\rhd \times \mathrm{B}\Sigma_2^\rhd\subset  \PSh(\mathrm{B}\Sigma_2)\times \PSh(\mathrm{B}\Sigma_2)$. The two functors also both preserve the initial objects: the first since $\varnothing\ast\varnothing\simeq \varnothing$ and the second since $E_0(2)=\varnothing$ and $E_0\otimes E_0\simeq E_0$. In view of \cref{ex:cospan-example}, it thus suffices to show they are equivalent after restricting further to $\mathrm{B}\Sigma_2\widehat{\times} \mathrm{B}\Sigma_2\subset \mathrm{B}\Sigma_2^\rhd \times \mathrm{B}\Sigma_2^\rhd$, i.e.\,in $\Fun( \mathrm{B}\Sigma_2\widehat{\times} \mathrm{B}\Sigma_2, \PSh(\mathrm{B}\Sigma_2))\simeq \Fun( (\mathrm{B}\Sigma_2\widehat{\times} \mathrm{B}\Sigma_2)\times \mathrm{B}\Sigma_2,\cS)$. By the discussion at the end of \cref{ex:cospan-example} such functors are equivalently given by zig-zags of $(\Sigma_2\times\Sigma_2)\times\Sigma_2$ spaces $X_0\rightarrow X_2\leftarrow X_1$ such that the second $\Sigma_2$-factor acts trivially on $X_0$ and the first $\Sigma_2$-factor acts trivially on $X_1$. Each of the columns in \eqref{equ:big-diagram-2-truncated} is such a zig-zag, and the horizontal equivalences are equivalences of such zig-zags. It thus suffices to show that the first column in \eqref{equ:big-diagram-2-truncated} is the zig-zag obtained by restricting \eqref{equ:join-functor-1} and the final column is the zig-zag obtained by restricting  \eqref{equ:bv-functor-2}. For \eqref{equ:join-functor-1} this follows by using that $\varnothing$ is initial in $\PSh(\mathrm{B}\Sigma_2)$, and that Yoneda embedding $\mathrm{B}\Sigma_2\hookrightarrow \PSh(\mathrm{B}\Sigma_2)$ viewed as a $\Sigma_2\times\Sigma_2$-space is given by $S^0$ with the multiplication by $-1$ action by both $\Sigma_2$-factors, since $S^0$ is the free $\Sigma_2$-space on a point. For \eqref{equ:bv-functor-2} this follows by using that $E_0$ is initial in $ \Opd^{\le 2, \red}$ and observing that the two $\Sigma_2$-actions on $E_1(2)$ given by conjugation by $-1$ and the canonical action on $\cO(2)$ for reduced operads $\cO$ are via the equivalence $c \colon E_1(2)\simeq S^0$ both given by the multiplication by $-1$ action. The latter  implies that the functor $\mathrm{B}\Sigma_2\ra \Opd^{\le 2, \red}$ given by $E_1$ with the $\Sigma_2$-action given by conjugation by $-1$ becomes $S^0$ with the two multiplication by $-1$ actions when postcomposed with $(-)(2)\colon \Opd^{\le 2, \red}\simeq \PSh(\mathrm{B}\Sigma_2)$.

\section{Tensor products of rational pro-operads}\label{sec:pro-operads}
The rationalisation of operads $(-)_\bfQ\colon \Opd^\red\ra \Opd^\red$ from \cref{sec:rationalisation-operads} could potentially fail to preserve tensor products of operads. This issue occurs in \cref{sec:e2-strategy} and to circumvent it, we will prove in this appendix that, once restricted to a certain full subcategory, the functor $(-)_\bfQ$ factors as
\vspace{-0.1cm}
\begin{equation}\label{equ:factorisation-rationalisation-operads}{\Opd^{\red} \xlra{(-)_\bfQ^\wedge}\Opd_\bfQ^\red\xlra{\lvert -\rvert} \Opd^\red}\end{equation}
through a certain category $\smash{\Opd_\bfQ^\red}$ of \emph{reduced rational pro-operads}, such that the first functor in \eqref{equ:factorisation-rationalisation-operads} preserves tensor products in suitable sense and the second functor is fully faithful on a large full subcategory. This will allow us in particular to construct the promised filler in \eqref{equ:filler-rational-bv}. The content of this appendix is inspired by work of Boavida de Brito--Horel \cite{BoavidaHorel} (see also \cref{rem:Pedro-Geoffroy}).

\subsection{Pro-rationalisation of spaces}\label{sec:pro-rationalisation} To construct the factorisation \eqref{equ:factorisation-rationalisation-operads}, we will make use of a different description of the rationalisation of spaces $X\ra X_\bfQ$ from \cref{sec:rationalisation}, using pro-homotopy theory. We write $\Pro(\cS)$ for the category of \emph{pro-spaces} in the sense of \cite[A.8.1.1]{LurieSAG}. Its objects are given by cofiltered diagrams of spaces $\{X_\alpha\}_{\alpha\in A}$ and its mapping spaces are given by
\begin{equation}\label{equ:pro-maps}\textstyle{\Map_{\Pro(\cS)}(\{X_\alpha\}_{\alpha\in A},\{Y_\beta\}_{\alpha\in B})\simeq \lim_{\beta}\colim_\alpha\,\Map_{\cS}(X_\alpha,Y_\beta);}\end{equation}
see A.8.1.5 loc.cit.. There is an adjunction $(-)^\wedge\colon \cS \rightleftarrows \Pro(\cS)\colon \lvert-\rvert$ where the left adjoint $(-)^\wedge$ takes the constant pro-space on a space and is fully faithful, and the right adjoint $\lvert-\rvert$ takes the limit of a pro-space. We will also consider a rational version of $\Pro(\cS)$: writing $\cR\subset \cS$ for the full subcategory generated by $K(\bfQ,n)$ for $n\ge0$ under taking finite limits and retracts (see \cite[Proposition 2.6]{BoavidaHorel} for an explicit description), we consider the full subcategory $\Pro(\cS)_\bfQ \subset \Pro(\cS)$ on those pro-spaces that are equivalent to $\{X_\alpha\}_{\alpha\in A}$ with $X_\alpha\in\cR$. As a result of \cite[Theorem 2.3]{BoavidaHorel} (or \cite[Proposition 3.7]{Hoyois}), the inclusion $\Pro(\cS)_\bfQ \subset \Pro(\cS)$ has a left-adjoint which exhibits $\Pro(\cS)_\bfQ$ as the localisation of $\Pro(\cS)$ at the maps that induce isomorphisms on rational cohomology groups $\smash{\oH^k(-;\bfQ)=\pi_0\,\Map_{\Pro(\cS)}(-,K(\bfQ,k))}$ for $k\ge0$. Composing the localisation adjunction $\Pro(\cS) \rightleftarrows \Pro(\cS)_\bfQ$ with  $(-)^\wedge\colon \cS \rightleftarrows \Pro(\cS)\colon \lvert-\rvert$, we obtain an adjunction \vspace{-0.1cm}
\begin{equation}\label{equ:adj-pro}\begin{tikzcd}  (-)^\wedge_\bfQ\colon \cS \rar[shift left=.5ex]{} & {\Pro(\cS)_\bfQ\colon \lvert-\rvert} \lar[shift left=.5ex]{} \end{tikzcd}\vspace{-0.1cm}.\end{equation}
We record some properties of \eqref{equ:adj-pro} in the next lemma. We write $\tau_{\le k}\colon \cS\ra \cS$ for the Postnikov $k$-truncation localisation functor (see \cite[5.5.6.18]{LurieHTT}), which induces a localisation $\tau_{\le k}\colon \Pro(\cS)\ra \Pro(\cS)$ that objectwise truncates pro-spaces (see \cite[A.8.1.8]{LurieSAG} for the existence of the functor and use \cite[5.2.7.4]{LurieHTT} to see that it is a localisation). We call a space $X$ of \emph{finite $\bfQ$-type} if  $\dim(\oH_k(X;\bfQ))<\infty$ for $k\ge0$.

\begin{lem}\label{lem:properties-pro-rationalisation}\ 
\begin{enumerate}
\item\label{enum:pro-prop-i} The category $\Pro(\cS)_\bfQ$ has small limits and colimits.
\item\label{enum:pro-prop-iii} The left adjoint $\smash{(-)_\bfQ^\wedge}$ preserves finite products.
\item\label{enum:pro-prop-ii} On spaces $X$ of finite $\bfQ$-type, we have a natural equivalence $\smash{X_\bfQ^\wedge\simeq \{\tau_{\le n}(\bfQ_n(X))\}_{n\in\bfN}}$ in $\Pro(\cS)$ where $\bfQ_\bullet(X)$ is the tower resulting from the cosimplicial resolution from \cref{sec:rationalisation}. In particular, the unit $\smash{X\ra \lvert X^\wedge_\bfQ\rvert}$ of \eqref{equ:adj-pro} is equivalent to the rationalisation $X\ra X_\bfQ$ from \cref{sec:rationalisation}. 
 \item\label{enum:pro-prop-iv} For connected spaces $X$ with $\dim(\oH_1(X;\bfQ))<\infty$, we have a natural equivalence in $\Pro(\cS)$ \[\smash{\tau_{\le 1}(X_\bfQ^\wedge)\simeq\{K((\pi/\Gamma_n(\pi))\otimes\bfQ,1)\}_{n\in\bfN}\quad\text{ for }\pi\coloneq \pi_1(X)}\]    where $\Gamma_\bullet(\pi)$ is the lower central series, and $(-)\otimes \bfQ$ denotes rationalisation of nilpotent groups. \end{enumerate}
\end{lem}
\begin{proof}
Item \ref{enum:pro-prop-i} follows from the fact that $\Pro(\cS)_\bfQ$ is the underlying $\infty$-category of a model category (see  \cite[Theorem 2.3]{BoavidaHorel}). Item \ref{enum:pro-prop-iii} is \cite[Proposition 2.9]{BoavidaHorel}. To show Item \ref{enum:pro-prop-ii}, we need to say more about this model category: one starts with a model structure on the category $\Pro(\sSet)$ of pro-objects in the $1$-category of simplicial sets, called the strict model structure, whose underlying $\infty$-category is $\Pro(\cS)$ (see \cite[Section 2.2]{IsaksenCompletion} and \cite[Section 7.1, Theorem 5.2.1]{BarneaHarpazHorel}). There is a Bousfield left localisation of this model category at the maps of pro-simplicial sets that induce isomorphisms of rational \emph{homology}, resulting in the \emph{homological model structure} (see \cite[Theorem 6.7]{IsaksenCompletion}). The latter can in turn be further Bousfield left localised at the maps of pro-simplicial sets that induce isomorphisms of rational \emph{cohomology} to arrive at the \emph{cohomological model structure} (see Theorem 6.3 loc.cit.). The underlying $\infty$-category of the latter is $\Pro(\cS)_\bfQ$ by \cite[Theorem 2.3]{BoavidaHorel}. Writing $\Pro(\cS)_{h,\bfQ}$ of the underlying $\infty$-category of the homological model structure, this yields a factorisation of $(-)_\bfQ^\wedge$ into three left adjoints $\smash{\cS\ra \Pro(\cS)\ra\Pro(\cS)_{h,\bfQ}\ra \Pro(\cS)_\bfQ}$. It follows from \cite[Proposition 7.3]{IsaksenCompletion} that the functor $\cS\ra \Pro(\cS)_{h,\bfQ}\subset\Pro(\cS)$ sends $X$ to $\smash{\{\tau_{\le n}(\bfQ_n(X))\}_{n\in\bfN}}$. Since the functors $\cS\ra \Pro(\cS)_{h,\bfQ}\subset\Pro(\cS)$ and  $\cS\ra \Pro(\cS)_{\bfQ}\subset\Pro(\cS)$ agree on spaces of finite $\bfQ$-type by the argument in \cite[Proposition 7.2, Remark 7.3]{HorelBinomial}, Item \ref{enum:pro-prop-ii} follows.

To prove Item \ref{enum:pro-prop-iv}, we first observe that the setup described at the beginning of this appendix equally makes sense using the category $\cS_\ast$ of pointed spaces, in particular we have adjunctions $\smash{\cS_* \rightleftarrows \Pro(\cS_*) \rightleftarrows \Pro(\cS_*)_\bfQ}$ where $\Pro(\cS_*)_\bfQ$ is the localisation of $\Pro(\cS_*)$ at the maps that are isomorphisms on reduced rational cohomology $\smash{\widetilde{\oH}^k(-;\bfQ)=\pi_0\,\Map_{\Pro(\cS_*)}(-,K(\bfQ,k))}$. The composition of the two left adjoints $\smash{(-)^\wedge_{\bfQ,*}\colon \cS_*\ra \Pro(\cS_*)_\bfQ}$ with the forgetful functor $\smash{\iota_\bfQ\colon \Pro(\cS_*)_\bfQ\ra \Pro(\cS)_\bfQ}$ is related to the composition of the forgetful map $\iota\colon \cS_*\ra \cS$ with $\smash{(-)^\wedge_{\bfQ}\colon \cS\ra \Pro(\cS)_\bfQ}$ by a Beck--Chevalley transformation $\smash{\iota(X)^\wedge_{\bfQ}\ra \iota_\bfQ(X^\wedge_{\bfQ,*})}$. This turns out to be an equivalence: it suffices to prove that $\smash{\iota(X)\ra \iota_\bfQ(X^\wedge_{\bfQ,*})}$ is an isomorphism on rational cohomology, which, since $\smash{X\ra X^\wedge_{\bfQ,*}}$ is an isomorphism on reduced cohomology, follows from the fact that filtered colimits in abelian groups are exact. Together with the fact that truncation of (pointed) spaces commutes with the forgetful map $\cS_*\ra \cS$, this in particular yields a pointed lift $\smash{\tau_{\le 1}(X_{\bfQ,*}^{\wedge})\in\Pro(\cS_*)}$ of $\smash{\tau_{\le 1}(X_{\bfQ}^{\wedge})}\in\Pro(\cS)$. It thus suffices to show that $\smash{\tau_{\le 1}(X_{\bfQ,*}^{\wedge})\simeq \{K(\pi/\Gamma_n(\pi)\otimes\bfQ,1)\}_{n\in\bfN}}$ in $\Pro(\cS_*)$. Both sides lie in the full subcategory $\smash{\Pro(\cS_*)^{\le 1}_{\bfQ}\subset \Pro(\cS_*)}$ of pointed pro-spaces that are equivalent to filtered diagrams $\{Y_\beta\}_{\beta\in B}$ of $1$-truncated spaces that lie in $\cR$ (for the left-hand side this holds by definition and for the right-hand side it follows from the second part of the proof of \cite[Proposition A.2]{BoavidaHorel} using the assumption on $X$). We will now prove the claimed equivalence by showing that the two sides have the same image under the Yoneda embedding of the opposite of $\smash{\Pro(\cS_*)^{\le 1}_{\bfQ}}$. For $\smash{\{Y_\beta\}_{\beta\in B}\in \Pro(\cS_*)^{\le 1}_{\bfQ}}$, it follows from adjunction and the localising property of truncation that 
$\smash{\Map_{\Pro(\cS_*)}(\tau_{\le 1}(X_{\bfQ,*}^{\wedge}),\{Y_\beta\}_{\beta\in B})\simeq \Map_{\Pro(\cS_*)}(\tau_{\le 1}(X),\{Y_\beta\}_{\beta\in B})}$ which using \eqref{equ:pro-maps} yields $\smash{\textstyle{\Map_{\Pro(\cS_*)}(\tau_{\le 1}(X_{\bfQ,*}^{\wedge}),\{Y_\beta\}_{\beta\in B})\simeq \lim_{\beta\in B}\Hom(\pi_1(X),\pi_1(Y_\beta))}}$.
On the other hand, using \eqref{equ:pro-maps}, we have ${\Map_{\Pro(\cS_*)}(\{K(\pi/\Gamma_n(\pi)\otimes\bfQ,1)\}_{n\in\bfN},\{Y_\beta\}_{\beta\in B})\simeq \lim_{\beta\in B}\colim_{n}\Hom(K(\pi/\Gamma_n(\pi)\otimes\bfQ,\pi_1(Y_\beta))}$, which combining with the first part of the proof of \cite[Propositions A.2 and 2.6]{BoavidaHorel} yields ${\Map_{\Pro(\cS_*)}(\{K(\pi/\Gamma_n(\pi)\otimes\bfQ,1)\}_{n\in\bfN},\{Y_\beta\}_{\beta\in B})\simeq \lim_{\beta\in B}\Hom(\pi,\pi_1(Y_\beta))}$. Since these equivalences are natural in the target, the claim follows from the Yoneda lemma.
\end{proof}

Next, we discuss a full subcategory on which $\smash{\lvert-\rvert\colon \Pro(\cS)_\bfQ\ra \cS}$ is fully faithful. We say that a space $X$ is \emph{$\bfQ$-good} if the unit $\smash{X \to \lvert X_\bfQ^\wedge\rvert}$ induces an isomorphism on rational cohomology (if $X$ is of finite $\bfQ$-type, this is by \cref{lem:properties-pro-rationalisation} \ref{enum:pro-prop-ii} equivalent to $\bfQ$-goodness in the sense \cite[I.5.1 (iii)]{BousfieldKan}, but note that it differs from the notion of $\bfQ$-goodness in \cite{BoavidaHorel}). We call a pro-space $P\in\Pro(\cS)$ \emph{k-truncated} if the objectwise truncation map $P\ra \tau_{\le k}(P)$ in $\Pro(\cS)$ is an equivalence.

\begin{lem}\label{lem:rat-map-comparison-spaces} For a space $X$ and a pro-space $P\in\Pro(\cS)_\bfQ$ such that either
\begin{enumerate}
\item $X$ is $\bfQ$-good, or
\item  $X_\bfQ^\wedge$ and $P$ are  1-truncated as pro-spaces and we have $\dim(\oH_1(X;\bfQ))<\infty$,
\end{enumerate}
the map $\smash{\lvert-\vert\colon \smash{\Map_{\Pro(\cS)_\bfQ}(X_\bfQ^\wedge,P) \ra \Map_{\cS}(\lvert X_\bfQ^\wedge\rvert,\lvert P\rvert)}}$
is an equivalence.
\end{lem}

\begin{ex}\label{ex:ex-realising-faithful}Here are some situations where the assumptions in \cref{lem:rat-map-comparison-spaces} are met:
\begin{enumerate}
\item\label{ex:ex-realising-faithful:i} If a space $X$ is of nilpotent of finite $\bfQ$-type, then it is $\bfQ$-good. This follows by combining \cref{lem:properties-pro-rationalisation} \ref{enum:pro-prop-ii} with the discussion in \cref{sec:nilpotent-completion}).
\item\label{ex:ex-realising-faithful:ii}  For a connected space $X$ with $\dim(\oH_1(X;\bfQ))<\infty$, it follows from \cref{lem:properties-pro-rationalisation} \ref{enum:pro-prop-iv} that the pro-space $\smash{X_\bfQ^\wedge}$ is $1$-truncated if and only if the map $\smash{\colim_n\oH^{k}(\pi_1(X)/\Gamma_n(\pi_1(X))\otimes\bfQ;\bfQ)\ra \oH^{k}(X;\bfQ)}$ induced by the $1$-truncation $X\ra K(\pi_1(X),1)$ is an isomorphism for all $k\ge0$. The case that is most relevant to us is $X=K(PB_m,1)$ for the pure braid group $PB_m$, where this isomorphism was verified in the proof of \cite[Proposition A.4]{BoavidaHorel}.
\item\label{ex:ex-realising-faithful:iii} The conditions on $X$ or $P$ appearing in \cref{lem:rat-map-comparison-spaces} are closed under taking finite products (use \cref{lem:properties-pro-rationalisation} \ref{enum:pro-prop-iii}, the Künneth theorem, and that truncation preserves products).
\end{enumerate}
\end{ex}

\begin{proof}The first case is straight-forward: by adjunction, the map in the claim is equivalent to the map $\smash{\Map_{\Pro(\cS)_\bfQ}(X^\wedge_\bfQ,P)\ra \Map_{\Pro(\cS)_\bfQ}(\lvert X^\wedge_\bfQ\rvert_\bfQ^\wedge,P)}$ induced by precomposition with the counit $\smash{|X^\wedge_\bfQ|_\bfQ^\wedge\ra X^\wedge_\bfQ}$. Using a triangle identity, the latter is an equivalence if the map $\smash{X_\bfQ^\wedge \to \lvert X_\bfQ^\wedge\rvert_\bfQ^\wedge}$ induced by the unit is an equivalence, i.e.\ if the unit $\smash{X\ra \lvert X_\bfQ^\wedge\rvert}$ is a rational cohomology isomorphism. This is exactly the $\bfQ$-goodness of $X$, so the first case follows. 
	
For the second case, as $\smash{X_\bfQ^\wedge}$ is $1$-truncated, it suffices to show that the map $\smash{\Map_{\Pro(\cS)}(\tau_{\le 1}(X_\bfQ^\wedge), P)} \ra \smash{\Map_{\cS}(\lvert \tau_{\le 1}(X_\bfQ^\wedge)\rvert, \lvert P\rvert)}$ is an equivalence. By adjunction and the fact that $P$ is $1$-truncated, it suffices to show that the map $\smash{\tau_{\le 1}(\lvert \tau_{\le 1}(X_\bfQ^\wedge)\rvert_\bfQ^\wedge)\ra  \tau_{\le 1}(X_\bfQ^\wedge)}$ induced by the counit is an equivalence. Using a triangle identity, this is an equivalence if the map  $\smash{\tau_{\le 1}(X_\bfQ^\wedge)\ra \tau_{\le 1}(\lvert \tau_{\le 1}(X_\bfQ^\wedge)\rvert_\bfQ^\wedge)}$ induced by the unit is an equivalence. By \cref{lem:properties-pro-rationalisation} \ref{enum:pro-prop-iv}, we have $\smash{\lvert \tau_{\le 1}(X_\bfQ^\wedge)\rvert\simeq K(\Mal_\bfQ(\pi_1(X)),1)}$ where we write $\Mal_\bfQ(\pi)\coloneq \lim_k((\pi/\Gamma^k(\pi))\otimes\bfQ)$ for groups $\pi$. As a result of \cite[Theorem 13.3 (iv)]{BousfieldHomological}, we have $\oH_1(\Mal_\bfQ(\pi_1(X);\bfQ)\cong \oH_1(X;\bfQ)$, so these vector spaces are finite-dimensional by assumption. Using \cref{lem:properties-pro-rationalisation} \ref{enum:pro-prop-iv} once more and writing $\pi\coloneq\pi_1(X)$, the map $\smash{\tau_{\le 1}(X_\bfQ^\wedge)\ra \tau_{\le 1}(\lvert \tau_{\le 1}(X_\bfQ^\wedge)\rvert_\bfQ^\wedge)}$  is equivalent to the map $\smash{\{K(\pi/\Gamma^n(\pi),1)\}_{n\in\bfN}\ra \{K(\Mal_\bfQ(\pi)/\Gamma^n(\Mal_\bfQ(\pi)),1)\}_{n\in\bfN}}$ induced levelwise by the canonical morphism $\pi\ra \Mal_\bfQ(\pi)$. This map of pro-spaces is a levelwise equivalence by \cite[Theorem 13.3 (iv)]{BousfieldHomological}, so the claim follows.
\end{proof}

\subsection{Pro-rationalisation of operads}\label{equ:pro}
With respect to the equivalence $\delta_\Phi\colon \Opd\simeq\PSh(\Phi)^{\seg,c}$ from \cref{sec:prelim-dendroidal}, the subcategory of reduced operads $\Opd^\red\subset\Opd\subset \PSh(\Phi)$ (see \cref{sec:operad-conventions}) corresponds to those presheaves whose values at $\smash{\overline{C}_0},\smash{\overline{C}_1\in\overline{\Omega}}$ and at the unique tree without vertices $\eta\in\Omega$ are contractible, and which satisfy the dendroidal Segal and forest Segal condition, both of which only involve finite products (the completeness condition is automatic). This is a localisation of $\PSh(\Phi)=\Fun(\Phi^\op,\cS)$, so the inclusion $\iota\colon \Opd^\red\subset \PSh(\Phi)$ admits left adjoint $L\colon \PSh(\Phi)\ra \Opd^\red$. Similarly, we can define the category of \emph{reduced rational pro-operads} as the full subcategory
\begin{equation}\label{equ:pro-operad-category}\smash{\Opd_\bfQ^\red\subset \Fun(\Phi^\op,\Pro(\cS)_\bfQ)}\end{equation} of those functors which have terminal value on $\smash{\overline{C}_0,\overline{C}_1}$, and $\eta$, and which satisfy the analogue of the dendroidal Segal and forest Segal condition (phrased in terms of finite products). Since $\Phi$ has finite morphism sets, it follows from \cite[Proposition 4.4]{BoavidaHorel} that the inclusion \eqref{equ:pro-operad-category} is a localisation as well, that is, it admits a left adjoint $L_\bfQ\colon \Fun(\Phi^\op,\Pro(\cS)_\bfQ)\ra \Opd_\bfQ^\red$, (c.f.~the proof of Proposition 6.3 loc.cit.). Since $\smash{(-)^\wedge_\bfQ}$ and $\smash{\lvert-\rvert}$ preserve finite products (see \cref{lem:properties-pro-rationalisation} \ref{enum:pro-prop-iii}), the adjunction $\smash{(-)^\wedge_\bfQ\colon \Fun(\Phi^\op,\cS)\leftrightarrows\Fun(\Phi^\op,\Pro(\cS)_\bfQ)\colon \vert-\rvert_\bfQ}$ restricts to an adjunction $\smash{(-)^\wedge_\bfQ\colon \Opd_\bfQ\leftrightarrows\Opd\colon \vert-\rvert_\bfQ}$. This in particular explains the  left-hand commutative square in 
\[\begin{tikzcd}[row sep=0.5cm] \Fun(\Phi^\op,\cS) &  \Fun(\Phi^\op,\Pro(\cS)_\bfQ) \lar[swap,pos=.3]{|{-}|} \\
	\Opd^\red \uar[hook]{i} &[5pt] \Opd^\red_\bfQ \uar[swap,hook]{i_\bfQ} \lar{|{-}|}\end{tikzcd}\qquad \begin{tikzcd}[row sep=0.5cm] \Fun(\Phi^\op,\cS) \rar{(-)^\wedge_\bfQ} \dar[swap]{L}&[5pt]  \Fun(\Phi^\op,\Pro(\cS)_\bfQ) \dar{L_\bfQ} \\
		\Opd^\red \rar{(-)^\wedge} & \Opd^\red_\bfQ,\end{tikzcd}\]
and the right-hand commutative square is obtained by taking left adjoints. By \cref{lem:properties-pro-rationalisation} \ref{enum:pro-prop-ii}, the composition $\smash{\lvert(-)^\wedge_\bfQ\rvert\colon \Opd\ra \Opd}$ agrees with the rationalisation $(-)_\bfQ\colon \Opd\ra \Opd$ from \cref{sec:rationalisation-operads} on the full subcategory of reduced operads whose multi-operation spaces have finite $\bfQ$-type.

\begin{lem}\label{lem:rat-map-comparison} Let $\cO,\cP \in \Opd^\red$ and assume that their spaces of multi-operations satisfy one of the following conditions for all $n\ge0$:
	\begin{enumerate}
		\item\label{condition-i-pro-opd} $\cO(n)$ is $\bfQ$-good, or 
		\item\label{condition-ii-pro-opd} $\smash{\cO(n)_\bfQ^\wedge}$ and $\cP(n)_\bfQ^\wedge$ are $1$-truncated as pro-spaces and $\dim(\oH_1(\cO(n);\bfQ))<\infty$.
	\end{enumerate}  
	Then the map $\lvert-\rvert\colon \smash{\Map_{\Opd^\red_\bfQ}(\cO_\bfQ^\wedge,\cP_\bfQ^\wedge) \ra \Map_{\Opd^\red}(\lvert\cO_\bfQ^\wedge\rvert,\lvert\cP_\bfQ^\wedge\rvert)}$
	is an equivalence.
\end{lem}

\begin{proof}Using the natural description of mapping spaces in functor categories as ends (see e.g.\ \cite[Proposition 5.1]{GepnerHaugsengNikolaus}), the claim follows by showing that for all forests $T,U\in \Phi$ the map $\smash{\Map_{\Pro(\cS)_\bfQ}(\cO(T)^\wedge_\bfQ,\cP(U)^\wedge_\bfQ) \ra \Map_\cS(|\cO(T)^\wedge_\bfQ|,|\cP(U)^\wedge_\bfQ|)}$
is an equivalence. The dendroidal Segal and the forest Segal condition implies that $\cO(T)$ is a finite product of spaces of the form $\cO(n)$ for $n\ge0$, and similarly for $\cP$, so the claim follows from \cref{ex:ex-realising-faithful} \ref{ex:ex-realising-faithful:iii} and \cref{lem:rat-map-comparison-spaces}.
\end{proof}

\begin{ex}\label{ex:ed-auts-all-the-same}For $\cO=\cP=E_d$ the little $d$-discs operad (see \cref{sec:dl-add}), we have \[\smash{\Map_{\Opd_\bfQ}((E_d)^\wedge_\bfQ,(E_d)^\wedge_\bfQ)\simeq\Map_{\Opd}(\lvert(E_d)^\wedge_\bfQ\rvert,\lvert(E_d)^\wedge_\bfQ\rvert)\simeq \Map_{\Opd}(E_{d,\bfQ},E_{d,\bfQ})\quad\text{for all }d\ge0.}\] Indeed, since $E_d(k)$ is equivalent to the ordered configuration space $F_k(\bfR^d)$ of $k$ points in $\bfR^d$ (see \cref{sec:dl-add}), it is always of finite $\bfQ$-type, so the second equivalence follows from \cref{lem:properties-pro-rationalisation} \ref{enum:pro-prop-i}.  For $d\neq 2$, the space $F_k(\bfR^d)$ is nilpotent since it is either homotopy discrete (for $d\le 1$) or simply connected (for $d\ge2$), so it is $\bfQ$-good by \cref{ex:ex-realising-faithful} \ref{ex:ex-realising-faithful:i} and thus condition \ref{condition-i-pro-opd} in \cref{lem:rat-map-comparison} is satisfied and yields the first equivalence. If $d=2$, then $E_d(k)\simeq K(PB_k,1)$, so condition \ref{condition-ii-pro-opd} in \cref{lem:rat-map-comparison} is satisfied by \cref{ex:ex-realising-faithful} \ref{ex:ex-realising-faithful:ii} and yields the first equivalence.
\end{ex}

\subsection{Tensor products of reduced rational pro-operads} \label{sec:bv-complete-operads} There is an analogue $\smash{\otimes_\bfQ}$ of the tensor product functor $\otimes \colon \Opd\times\Opd \ra \Opd$ from \cref{sec:operads-tensor-product} for the category $\Opd^\red_\bfQ$, which is compatible with the tensor product on $\Opd$ in that it fits into a commutative square
\begin{equation}\label{equ:rational-operads-tensor-products}
\begin{tikzcd}
\Opd^\red\times\Opd^\red\rar{\otimes}\dar[swap]{(-)^\wedge_\bfQ\times (-)^\wedge_\bfQ}&\Opd^\red\dar{(-)^\wedge_\bfQ}\\
\Opd^\red_\bfQ\times \Opd^\red_\bfQ\arrow[r,dashed,"\otimes_\bfQ"] &\Opd^\red_\bfQ.
\end{tikzcd}
\end{equation}
To construct $\otimes_\bfQ$,  we will make use of a dendroidal description of the tensor product on $\Opd$: it follows from \cite[Theorem 1.1.1, Section 5]{HinichMoerdijk} that under the equivalence $\delta_\Phi\colon \PSh(\Phi)^{\seg,c} \simeq \Opd$ from \cref{sec:prelim-dendroidal}, the restriction of the tensor product $\otimes\colon \Opd\times\Opd\ra\Opd$ to reduced operads is given by the composition
\[\smash{\Opd^\red \times \Opd^\red \xlra{i\times i} \PSh(\Phi) \times \PSh(\Phi) \overset{\otimes^\pre}{\lra} \PSh(\Phi) \xlra{L} \Opd,}\]
where $\otimes^\pre \colon \PSh(\Phi) \times \PSh(\Phi) \to \PSh(\Phi)$ is the composition
\vspace{-0.15cm}
\[
{ \PSh(\Phi)^{\times 2}\xlra{(\subseteq_!)^{\times 2}} \PSh(\OpdSet)^{\times 2}\xlra{\times}  \PSh(\OpdSet^{\times 2})\xlra{\otimes_!}\PSh(\OpdSet)\xlra{\subseteq^*}\PSh(\Phi)}
\]
where the first functor is induced by left Kan extension along the full subcategory inclusion $\Phi\subseteq \OpdSet$, the second functor is induced by taking products in $\cS$, the third functor by left Kan extension along the tensor product in $\OpdSet$ given by the classical Boardman--Vogt tensor product of operads in sets, and the final functor by restriction along  $\Phi\subseteq \OpdSet$. The only properties of the target category in $\PSh(\Phi)=\Fun(\Phi^\op,\cS)$ this uses are that it has small colimits and finite products, so the same construction makes sense for $\Fun(\Phi^\op,\Pro(\cS)_\bfQ)$ using \cref{lem:properties-pro-rationalisation} \ref{enum:pro-prop-i} and yields a functor $\smash{\otimes_\bfQ^\pre \colon \Fun(\Phi^\op,\Pro(\cS)_\bfQ)^{\times 2}  \to \Fun(\Phi^\op,\Pro(\cS)_\bfQ)}$. Moreover, since $\smash{(-)^\wedge_\bfQ}$ preserves finite products by \cref{lem:properties-pro-rationalisation} \ref{enum:pro-prop-ii} and colimits since it is a left adjoint, the middle square in the  diagram
\[\hspace{-.3cm} \begin{tikzcd}[column sep=0.5cm] \Opd^\red \times \Opd^\red \rar{i \times i} \dar{(-)^\wedge_\bfQ \times (-)^\wedge_\bfQ} &[5pt] \Fun(\Phi^\op,\cS) \times \Fun(\Phi^\op,\cS) \rar{\otimes^\mathrm{pre}} \dar{(-)^\wedge_\bfQ \times (-)^\wedge_\bfQ} &[-2pt] \Fun(\Phi^\op,\cS) \rar{L} \dar{(-)^\wedge_\bfQ} &[-2pt] \Opd^\red \dar{(-)^\wedge_\bfQ} \\
	\Opd^\red_\bfQ \times \Opd^\red_\bfQ \rar{i_\bfQ \times i_\bfQ}& \Fun(\Phi^\op,\Pro(\cS)_\bfQ) \times \Fun(\Phi^\op,\Pro(\cS)_\bfQ) \rar{\otimes^\mathrm{pre}_\bfQ} & \Fun(\Phi^\op,\Pro(\cS)_\bfQ) \rar{L_\bfQ} & \Opd^\red_\bfQ \end{tikzcd}\]
	commutes. The left-hand and right-hand squares commute by the discussion in \cref{equ:pro}, so the bottom composition yields the promised dashed functor that makes \eqref{equ:rational-operads-tensor-products} commute.

\subsection{$2$-ary operations of tensor products of rational pro-operads} \label{sec:bv-complete-operads}Finally, we establish a partial analogue of \cref{thm:operad-truncation-tower-monoidal} \ref{enum:trunmon-iv} in the context of reduced rational pro-operads:

\begin{lem}\label{lem:2-ary-q-complete-partial}There is a natural transformation of functors $\smash{\Opd^\red_\bfQ\times\Opd^\red_\bfQ\ra \cS^{\Sigma_2}}$
\[\smash{\lvert \cX\rvert(2)\ast\lvert \cY\rvert(2)\lra \lvert\cX\otimes_\bfQ\cY\rvert(2)}
\]
which is an equivalence on $\smash{\cX=\cO^\wedge_\bfQ}$ and $\smash{\cY=\cP^\wedge_\bfQ}$ for $\cO,\cP\in\Opd^\red$ if $\cO(2)$ and $\cP(2)$ are nonempty nilpotent spaces of finite $\bfQ$-type and at least one of them is connected.
\end{lem}
\begin{proof}
The Beck--Chevalley transformation obtained from \eqref{equ:rational-operads-tensor-products} by taking vertical right adjoints gives a natural transformation $\varphi\colon \lvert \cX\rvert\otimes \lvert \cY\rvert\ra \lvert\cX\otimes_\bfQ\cY\rvert$ of functors $\smash{\Opd^\red_\bfQ\times\Opd^\red_\bfQ \ra \Opd^\red}$. Postcomposing it with the functor $(-)(2)\colon \Opd^\red\ra \cS^{\Sigma_2}$ and using the natural equivalence $(\cO\otimes\cP)(2)\simeq \cO(2)\ast\cP(2)$ from \cref{thm:operad-truncation-tower-monoidal} \ref{enum:trunmon-iv}, we get the natural transformation of the required form
\begin{equation}\label{equ:nat-traf-arity-2-rat}\smash{\lvert \cX\rvert(2)\ast \lvert \cY\rvert(2) \simeq (\lvert \cX\rvert\otimes \lvert \cY\rvert)(2)\xlra{\varphi} \lvert\cX\otimes_\bfQ\cY\rvert(2)}.\end{equation}
For $\smash{\cX=\cO^\wedge_\bfQ}$ and $\smash{\cY=\cP^\wedge_\bfQ}$, we consider the diagram
\[
\begin{tikzcd}[column sep=0.5cm]
\cO(2)\ast\cP(2)\arrow[r,phantom,"\simeq"]&[-0.3cm](\cO\otimes\cP)(2)
  \arrow[d,"\eta_\cO\otimes\eta_\cP",swap]
  \arrow[drr,"\eta_{\cO\otimes\cP}",bend left=10] & &[-0.3cm]&[-0.3cm] \\
\lvert\cO(2)^\wedge_\bfQ\rvert\ast\lvert\cP(2)^\wedge_\bfQ\rvert\arrow[r,phantom,"\simeq"]&[-0.3cm](\lvert\cO^\wedge_\bfQ\rvert\otimes\lvert\cP^\wedge_\bfQ\rvert)(2)
  \arrow[r,"\varphi",swap] &
\lvert\cO^\wedge_\bfQ\otimes_\bfQ\cP^\wedge_\bfQ\rvert(2)
  \arrow[r,phantom,"\simeq"] &
\lvert(\cO\otimes\cP)(2)^\wedge_\bfQ\rvert \arrow[r,phantom,"\simeq"] &
\lvert(\cO(2)\ast\cP(2))^\wedge_\bfQ\rvert
\end{tikzcd}
\]
whose inner triangle involving the units $\eta$ of the adjunction $\smash{(-)_\bfQ^\wedge \vdash \lvert-\rvert}$  commutes by the triangle identities. The outer equivalences hold by \cref{thm:operad-truncation-tower-monoidal} \ref{enum:trunmon-iv}. The task is to show that $\varphi$ is an equivalence under the assumptions on $\cO(2)$ and $\cP(2)$ in the claim. We start with some observations:
\begin{enumerate}[label=(\alph*)]
\item\label{observation-q-com-i} The join $X\ast Y$ is $1$-connected for $X\neq\varnothing$ and $Y$ connected. In particular  $\cO(2)\ast\cP(2)$ is $1$-connected.
\item\label{observation-q-com-ii} As $\cO(2)$, $\cP(2)$, and $\cO(2)\ast\cP(2)$ are nilpotent of finite $\bfQ$-type (the first two by assumption, and the third by \ref{observation-q-com-i} and the Mayer--Vietoris sequence), we have $\lvert\cO(2)^\wedge_\bfQ\rvert\ast\lvert\cP(2)^\wedge_\bfQ\rvert\simeq \cO(2)_\bfQ\ast\cP(2)_\bfQ$ and $\lvert(\cO(2)\ast\cP(2))^\wedge_\bfQ\rvert\simeq (\cO(2)\ast\cP(2))_\bfQ$ by \cref{lem:properties-pro-rationalisation}  \ref{enum:pro-prop-ii}.
\item\label{observation-q-com-iii} The spaces $\lvert\cO(2)^\wedge_\bfQ\rvert\ast\lvert\cP(2)^\wedge_\bfQ\rvert$ and $\lvert(\cO(2)\ast\cP(2))^\wedge_\bfQ\rvert$ are $1$-connected, by \ref{observation-q-com-i}, \ref{observation-q-com-ii}, and the fact that $(-)_\bfQ$ preserves connectivity (see \cref{sec:rat-inheritance}).
\end{enumerate}
By \ref{observation-q-com-iii} it suffices to show that the map $\varphi$ is in the diagram is a homology isomorphism. Moreover, by \ref{observation-q-com-ii}, the Mayer--Vietoris sequence, and since $\oH_*((-)_\bfQ;\bfZ)\cong \oH_*((-)_\bfQ;\bfQ)$ on nilpotent spaces (see \cref{sec:nilpotent-completion}), it suffices to show that $\varphi$ is a $\bfQ$-homology isomorphism. By commutativity, it suffices to show that $\eta_\cO\otimes\eta_\cP$ and $\eta_{\cO\otimes\cP}$ are $\bfQ$-homology isomorphisms. By \ref{observation-q-com-ii} and naturality of the equivalence in \cref{thm:operad-truncation-tower-monoidal} \ref{enum:trunmon-iv}, these maps are equivalent to the two maps $\smash{\cO(2)\ast\cP(2)\ra \cO(2)_\bfQ\ast\cP(2)_\bfQ}$ and $\smash{\cO(2)\ast\cP(2)\ra (\cO(2)\ast\cP(2))_\bfQ}$ induced by rationalisation. Using the Mayer--Vietoris sequence and that $X\ra X_\bfQ$ is a $\bfQ$-homology isomorphism on nilpotent spaces (see \cref{sec:nilpotent-completion}), it follows that these maps are indeed $\bfQ$-homology isomorphisms, so the claim follows.
\end{proof}

\subsection{The construction of the filler in \eqref{equ:filler-rational-bv}}\label{sec:the-rational-filler}Equipped with the previous discussion on tensor products of rational pro-operads, we construct the asserted filler in the diagram \eqref{equ:filler-rational-bv}. To this end, we first consider the commutative diagram
\vspace{-0.15cm}
\[\hspace{-0.5cm}\begin{tikzcd}[column sep=0.7cm,row sep=0.55cm]
			\Aut_{\Opd^\red}(E_{n})\times \Aut_{\Opd^\red}(E_{m})\rar{(-)^\wedge_\bfQ}\arrow[d,"{(-)\otimes(-)}",swap]&\Aut_{\Opd^\red_\bfQ}((E_n)_\bfQ^\wedge)\times \Aut_{\Opd^\red_\bfQ}((E_m)^\wedge_{\bfQ})\rar{\lvert-\rvert(2)}\arrow[d,"{(-)\otimes_\bfQ(-)}",swap]& \dar{(-)\ast(-)}\Aut_\cS(S^{n-1}_\bfQ)\times \Aut_\cS(S^{m-1}_\bfQ)\\
			\Aut_{\Opd^\red}(E_{n+m})\rar{(-)^\wedge_\bfQ}&\Aut_{\Opd^\red_\bfQ}((E_{n+m})^\wedge_\bfQ) \rar{\lvert-\rvert(2)}&\Aut_\cS(S^{n+m-1}_\bfQ)\end{tikzcd}\]
where the commutativity of the left-hand square is induced by  \eqref{equ:rational-operads-tensor-products} and the additivity equivalence $E_n\otimes E_m\simeq E_{n+m}$ from \cref{sec:dl-add} and commutativity of the right-hand square is induced by the natural transformation from \cref{lem:2-ary-q-complete-partial} which is an equivalence in the relevant case  since $E_d(2)\simeq S^{d-1}$ is nilpotent of finite $\bfQ$-type. Extending the diagram to the right by taking reduced homology of the rational spheres to map from the leftmost column to $\smash{\otimes\colon \GL_1(\bfQ)\times \GL_1(\bfQ)\ra \GL_1(\bfQ)}$, we see that the middle vertical map almost provides a filler as in \eqref{equ:filler-rational-bv}, except that the diagram now involves $\smash{\Aut(\lvert (E_d)^\wedge_\bfQ\rvert)}$ instead of $\smash{\Aut(E_{d,\bfQ})}$. They are, however, equivalent by \cref{ex:ed-auts-all-the-same}.
	
\bibliographystyle{amsalpha}
\bibliography{literatureSquare}

\end{document}